\documentclass[11pt,twoside,a4paper,final]{report}

\usepackage[DM,newLogo]{uaThesis}

\def\ThesisYear{2013}

\usepackage[portuguese,english]{babel}
\usepackage{hyperref}
\usepackage{enumerate}
\usepackage{amsmath}
\usepackage{amsthm}
\usepackage{amssymb}
\usepackage{xspace}
\usepackage{cases}
\usepackage{mathtools}
\usepackage{mathrsfs}

\usepackage{fancyhdr}

\usepackage{makeidx}

\usepackage{cite}

\addto\captionsenglish{}

\def\topfigrule{\kern 7.8pt \hrule width\textwidth\kern -8.2pt\relax}
\def\dblfigrule{\kern 7.8pt \hrule width\textwidth\kern -8.2pt\relax}
\def\botfigrule{\kern -7.8pt \hrule width\textwidth\kern 8.2pt\relax}

\newcommand{\N}{\mathbb{N}}
\newcommand{\R}{\mathbb{R}}
\newcommand{\PP}{\mathscr{P}}
\newcommand{\CC}{\mathscr{C}}
\newcommand{\W}{W^{1,p}(a,b;\R)}

\newcommand{\Il}{{_{a}}\textsl{I}_t^\alpha}
\newcommand{\Ir}{{_{t}}\textsl{I}_b^\alpha}
\newcommand{\Ilc}{{_{a}}\textsl{I}_t^{1-\alpha}}
\newcommand{\Irc}{{_{t}}\textsl{I}_b^{1-\alpha}}

\newcommand{\Ilct}{{_{a}}\textsl{I}_{\tau}^{1-\alpha}}

\newcommand{\Dr}{{_{t}}\textsl{D}_b^\alpha}
\newcommand{\Dl}{{_{a}}\textsl{D}_t^\alpha}
\newcommand{\Dcl}{{^{C}_{a}}\textsl{D}_t^\alpha}
\newcommand{\Dcr}{{^{C}_{t}}\textsl{D}_b^\alpha}

\newcommand{\Dclt}{{^{C}_{a}}\textsl{D}_{\tau}^\alpha}

\newcommand{\K}{K_P}

\newcommand{\Hl}{{_{a}}\textsl{J}_t^\alpha}
\newcommand{\Hr}{{_{t}}\textsl{J}_b^\alpha}

\newcommand{\Ilp}{{_{a_i}}\textsl{I}_{t_i}^{\alpha_i}}
\newcommand{\Irp}{{_{t_i}}\textsl{I}_{b_i}^{\alpha_i}}
\newcommand{\Ilcp}{{_{a_i}}\textsl{I}_{t_i}^{1-\alpha_i}}
\newcommand{\Ircp}{{_{t}}\textsl{I}_{b_i}^{1-\alpha_i}}
\newcommand{\Drp}{{_{t_i}}\textsl{D}_{b_i}^{\alpha_i}}
\newcommand{\Dlp}{{_{a_i}}\textsl{D}_{t_i}^{\alpha_i}}
\newcommand{\Dclp}{{^{C}_{a_i}}\textsl{D}_{t_i}^{\alpha_i}}
\newcommand{\Dcrp}{{^{C}_{t_i}}\textsl{D}_{b_i}^{\alpha_i}}
\newcommand{\PKi}{K_{P_{i}}}
\newcommand{\PBi}{B_{P_{i}}}
\newcommand{\PAi}{A_{P_{i}}}

\newtheorem{definition}{Definition}
\newtheorem{property}{Property}
\newtheorem{example}{Example}
\newtheorem{remark}{Remark}
\newtheorem{theorem}{Theorem}
\newtheorem{corollary}{Corollary}
\newtheorem{lemma}{Lemma}
\newtheorem{proposition}{Proposition}
\newcommand*{\hooktwoheadrightarrow}{\lhook\joinrel\twoheadrightarrow}
\newcommand{\fonction}[5]{\begin{array}[t]{lrcl}#1 :&#2 &\longrightarrow &#3\\&#4& \longmapsto &#5 \end{array}}

\newcommand{\fonctionsansdef}[3]{\begin{array}[t]{lrcl}#1 :&#2 &\longrightarrow &#3 \end{array}}

\newcommand{\y}{\bar{y}}
\newcommand{\seq}{y_m}

\def\AddVMargin#1{\setbox0=\hbox{#1}%
                  \dimen0=\ht0\advance\dimen0 by 2pt\ht0=\dimen0%
                  \dimen0=\dp0\advance\dimen0 by 2pt\dp0=\dimen0%
                  \box0}
\def\Header#1#2{\setbox1=\hbox{#1}\setbox2=\hbox{#2}%
           \ifdim\wd1>\wd2\dimen0=\wd1\else\dimen0=\wd2\fi%
           \AddVMargin{\parbox{\dimen0}{\centering #1\\#2}}}

\makeindex
\begin{document}


\fancyhf{}

\fancyhead[LO]{\slshape \rightmark}
\fancyhead[RE]{\slshape \leftmark}
\fancyfoot[C]{\thepage}


\TitlePage
  \HEADER{\BAR\FIG{\includegraphics[height=60mm]{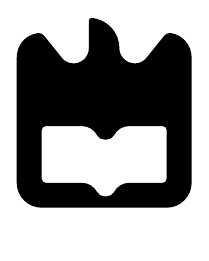}}}
         {\ThesisYear}
  \TITLE{Tatiana \newline Odzijewicz}
        {Generalized Fractional Calculus of Variations}
\EndTitlePage
\titlepage\ \endtitlepage


\TitlePage
  \HEADER{}{\ThesisYear}
  \TITLE{Tatiana \newline Odzijewicz}
        {C\'{a}lculo das Varia\c{c}\~{o}es Fraccional Generalizado}
  \vspace*{15mm}
\TEXT{}{Tese de doutoramento apresentada \`a Universidade de Aveiro para cumprimento dos requisitos
necess\'arios \`a obten\c c\~ao do grau de Doutor em Matem\'atica,
Programa Doutoral em Matem\'{a}tica e Aplica\c{c}\~{o}es (PDMA 2009-2013)
da Universidade de Aveiro e Universidade do Minho, rea\-li\-za\-da sob a orienta\c c\~ao
cient\'\i fica do Doutor Delfim Fernando Marado Torres, Professor Associado com Agrega\c{c}\~{a}o
do Departamento de Matem\'atica da Universidade de Aveiro, e co-orienta\c c\~ao da Doutora Agnieszka Barbara Malinowska,
Professora Auxiliar do Depar\-ta\-men\-to de Matem\'atica da Universidade T\'{e}cnica de Bia\l ystok, Pol\'onia.}
\vfill
\TEXT{}{
\begin{flushright}
\includegraphics[scale=0.4]{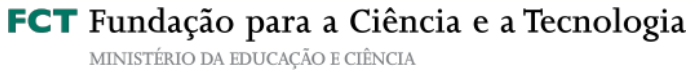}\newline
\includegraphics [scale=0.32]{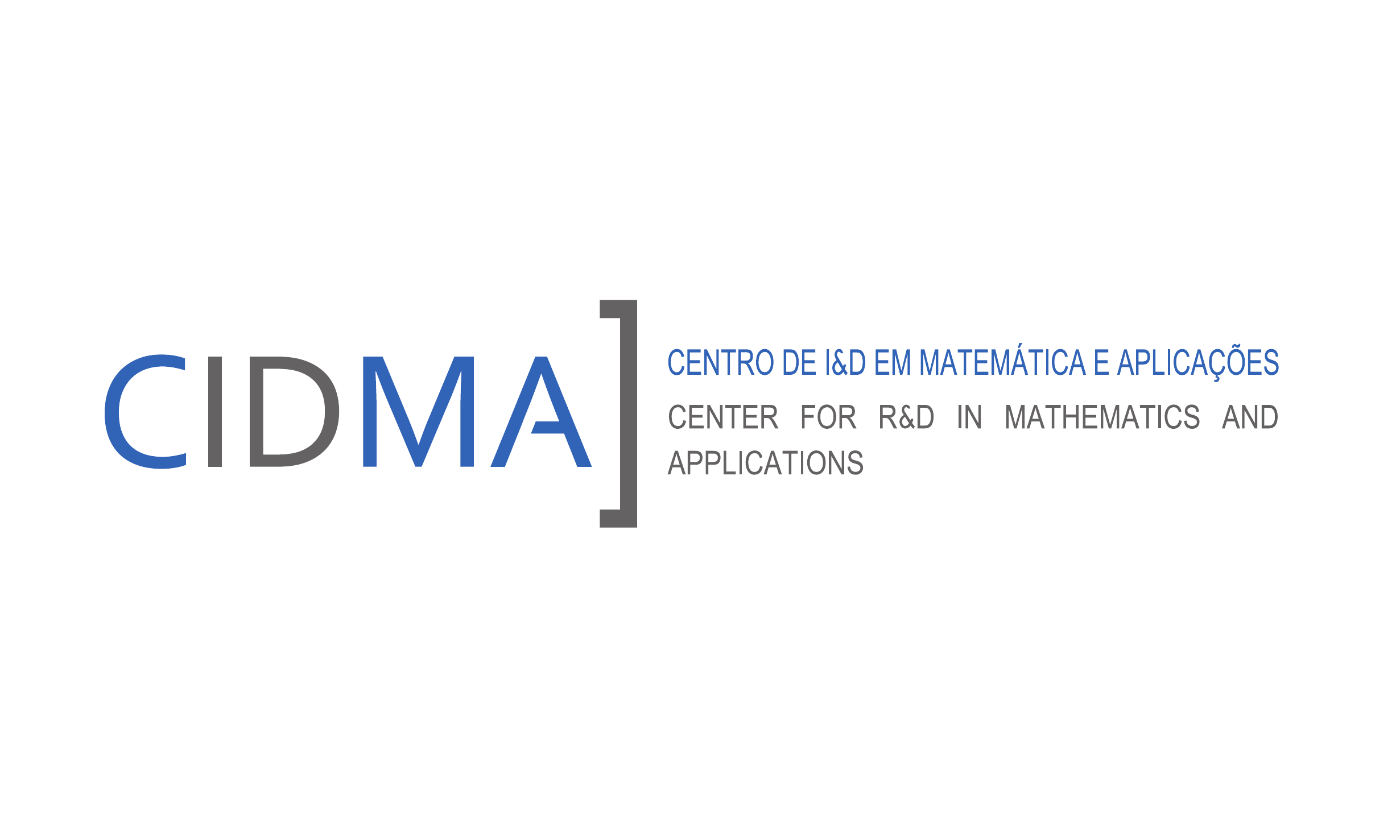}
\end{flushright}
}
\EndTitlePage
\titlepage\ \endtitlepage


\TitlePage
  \vspace*{55mm}
  \TEXT{\textbf{O j\'uri~/~The jury\newline}}
       {}
  \TEXT{Presidente~/~President}
       {\textbf{Doutora Nilza Maria Vilhena Nunes da Costa}\newline {\small
        Professora Catedr\'atica da Universidade de Aveiro}}
  \vspace*{5mm}
  \TEXT{Vogais~/~Examiners committee}
       {\textbf{Doutor Gueorgui Vitalievitch Smirnov}\newline {\small
        Professor Catedr\'atico da Universidade do Minho}}
  \vspace*{5mm}
  \TEXT{}
       {\textbf{Doutor Delfim Fernando Marado Torres}\newline {\small
        Professor Associado com Agrega\c{c}\~{a}o da Universidade de Aveiro (Orientador)}}
  \vspace*{5mm}
  \TEXT{}
       {\textbf{Doutora Agnieszka Barbara Malinowska}\newline {\small
        Professora Auxiliar da Bialystok University of Technology, Pol\'onia (Co-orientadora)}}
	\vspace*{5mm}			
	\TEXT{}
       {\textbf{Doutora Maria Margarida Amorim Ferreira}\newline {\small
        Professora Auxiliar da Universidade do Porto}}
	\vspace*{5mm}			
	\TEXT{}
       {\textbf{Doutor Ricardo Miguel Moreira de Almeida}\newline {\small
        Professor Auxiliar da Universidade de Aveiro}}						
\EndTitlePage
\titlepage\ \endtitlepage


\TitlePage
\vspace*{55mm}
\TEXT{\textbf{agradecimentos~/\newline acknowledgements}}{}
\TEXT{}{Uma tese de Doutoramento \'{e} um processo solit\'{a}rio.
S\~{a}o quatro anos de trabalho que s\~{a}o mais pass\'{\i}veis
de suportar gra\c{c}as ao apoio de v\'{a}rias pessoas e institui\c{c}\~{o}es.
Assim, e antes dos demais, gostaria de agradecer aos meus orientadores,
Professor Doutor Delfim F. M. Tor\-res e Professora Doutora Agnieszka B. Malinowska,
pelo apoio, pela partilha de saber e por estimularem o meu interesse pela Matem\'{a}tica.
Estou igualmente grata aos meus colegas e aos meus Professores do Programa Doutoral
pelo constante incentivo e pela boa disposi\c{c}\~{a}o que me transmitiram durante estes anos.
Gostaria de agradecer \`{a} FCT (Fundac\~ao para a Ci\^encia e a Tecnologia)
o apoio financeiro atribu\'{\i}do atrav\'{e}s da bolsa de Doutoramento
com a refer\^{e}ncia SFRH/BD/33865/2009.\\
Por \'{u}ltimo, mas sempre em primeiro lugar, agrade\c{c}o \`{a} minha fam\'{i}lia.}

\EndTitlePage


\titlepage\ \endtitlepage


\TitlePage
  \vspace*{55mm}
  \TEXT{\textbf{Resumo}}
       {Nesta tese de doutoramento apresentamos um c\'{a}lculo das varia\c{c}\~{o}es fraccional generalizado.
Consideramos problemas variacionais com derivadas e integrais fraccionais generalizados
e estudamo-los usando m\'{e}todos directos e indirectos.
Em particular, obtemos condi\c{c}\~{o}es necess\'{a}rias de optimalidade de Euler--Lagrange
para o problema fundamental e isoperim\'{e}trico, condi\c{c}\~{o}es de transversalidade e teoremas de Noether.
Demonstramos a exist\^{e}ncia de solu\c{c}\~{o}es, num espa\c{c}o de fun\c{c}\~{o}es apropriado,
sob condi\c{c}\~{o}es do tipo de Tonelli. Terminamos mostrando a exist\^{e}ncia de valores pr\'{o}prios,
e correspondentes fun\c{c}\~{o}es pr\'{o}prias ortogonais, para problemas de Sturm--Liouville.\\ \\ \\}

\TEXT{\textbf{Palavras chave}}{c\'{a}lculo das varia\c{c}\~{o}es,
condi\c{c}\~{o}es necess\'{a}rias de optimalidade do tipo de Euler--Lagrange,
m\'{e}todos directos, problemas isoperim\'{e}tricos, teorema de Noether,
c\'{a}lculo fraccional, problema de Sturm--Liouville.\\ \\ \\}

\TEXT{\textbf{2010 Mathematics Subject Classification:}}{\\26A33; 49K05; 49K21.}

\EndTitlePage

\titlepage\ \endtitlepage


\TitlePage
\vspace*{55mm}
\TEXT{\textbf{Abstract}}{In this thesis we introduce a generalized fractional calculus of variations.
We consider variational problems containing generalized fractional integrals
and derivatives and study them using standard (indirect) and direct methods.
In particular, we obtain necessary optimality conditions of Euler--Lagrange type
for the fundamental and isoperimetric problems, natural boundary conditions, and Noether theorems.
Existence of solutions is shown under Tonelli type conditions. Moreover, we apply our results
to prove existence of eigenvalues, and corresponding orthogonal eigenfunctions,
to fractional Sturm--Liouville problems.\\ \\ \\}

\TEXT{\textbf{Keywords}}{calculus of variations,
necessary optimality conditions of Euler--Lagrange type,
direct methods, isoperimetric problem, Noether's theorem, fractional calculus, Sturm--Liouville problem.\\ \\ \\}

\TEXT{\textbf{2010 Mathematics Subject Classification:}}{26A33; 49K05; 49K21.}
\EndTitlePage


\titlepage\ \endtitlepage


\pagenumbering{roman}
\tableofcontents


\cleardoublepage
\pagenumbering{arabic}


\chapter*{Introduction}\markboth{INTRODUCTION}{}

This thesis is dedicated to the generalized fractional calculus of variations and its main task
is to unify and extend results concerning the standard fractional variational calculus, that are
available in the literature. My adventure with the subject started on the first year
of my PhD Doctoral Programme, when I studied the course, given by my present supervisor
Delfim F. M. Torres, called {\it Calculus of Variations and Optimal Control}. He described
me an idea of the fractional calculus and showed that one can consider variational problems
with non-integer operators. Fractional integrals and derivatives can be defined in different ways,
and consequently in each case one must consider different variational problems. Therefore,
my supervisor suggested me to study more general operators, that by choosing special kernels,
reduce to the standard fractional integrals and derivatives. Finally, this interest resulted
in my PhD thesis entitled {\it Generalized Fractional Calculus of Variations}.

The calculus of variations is a mathematical research field that was born in 1696 with the solution
to the brachistochrone problem (see, e.g., \cite{book:Brunt}) and is focused on finding extremal values
of functionals \cite{book:Brunt,book:Jost,book:Giaquinta,book:Dacorogna}. Usually, considered functionals
are given in the form of an integral that involves an unknown function and its derivatives.
Variational problems are particularly attractive because of their many-fold applications, e.g.,
in physics, engineering, and economics; the variational integral may represent an action, energy,
or cost functional \cite{book:Ewing,book:Weinstock}. The calculus of variations possesses
also important connections with other fields of mathematics, e.g., with the
particularly important in this work--- fractional calculus.

Fractional calculus, i.e., the calculus of non-integer order derivatives, has also its origin in the 1600s.
It is a generalization of (integer) differential calculus, allowing to define derivatives (and integrals)
of real or complex order \cite{book:Kilbas,book:Podlubny,book:Samko}. During three centuries the theory of fractional
derivatives developed as a pure theoretical field of mathematics, useful only for mathematicians. However,
in the last few decades, fractional problems have received an increasing attention of many researchers.
As mentioned in \cite{isi}, {\it Science Watch of Thomson Reuters} identified the subject as an
{\it Emerging Research Front} area. Fractional derivatives are non-local operators and are historically
applied in the study of non-local or time dependent processes \cite{book:Podlubny}. The first and well
established application of fractional calculus in Physics was in the framework of anomalous diffusion,
which is related to features observed in many physical systems. Here we can mention the report \cite{MK}
demonstrating that fractional equations work as a complementary tool in the description of anomalous
transport processes. Within the fractional approach it is possible to include external fields
in a straightforward manner. As a consequence, in a short period of time the list of applications expanded.
Applications include chaotic dynamics \cite{Zaslavsky}, material sciences \cite{book:Mainardi},
mechanics of fractal and complex media \cite{Carpinteri,Li}, quantum mechanics \cite{book:Hilfer,Laskin},
physical kinetics \cite{Edelman}, long-range dissipation \cite{Tarasov3}, long-range interaction \cite{Tarasov2,Tarasov1},
just to mention a few. This diversity of applications makes the fractional calculus an important subject,
which requires serious attention and strong interest.

The calculus of variations and the fractional calculus are connected since the XIX century.
Indeed, in 1823 Niels Heinrik Abel applied the fractional calculus to the solution of an integral
equation that arises in the formulation of the tautochrone problem. This problem, sometimes also
called the isochrone problem, is that of finding the shape of a frictionless wire lying in a vertical
plane such that the time of a bead placed on the wire slides to the lowest point of the wire in the same
time regardless of where the bead is placed. It turns out that the cycloid is the isochrone as well
as the brachistochrone curve, solving simultaneously the brachistochrone problem of the calculus
of variations and Abel's fractional problem \cite{Abel}. It was however only in the XX century
that both areas joined in a unique research field: the fractional calculus of variations.

The fractional calculus of variations consists in extremizing (minimizing or maximizing)
functionals whose Lagrangians contain fractional integrals and derivatives. It was born
in 1996-97, when Riewe derived Euler--Lagrange fractional differential equations
and showed how non-conservative systems in mechanics can be described using fractional
derivatives \cite{CD:Riewe:1996,CD:Riewe:1997}. It is a remarkable result since frictional
and non-conservative forces are beyond the usual macroscopic variational treatment and,
consequently, beyond the most advanced methods of classical mechanics \cite{book:Lanczos}.
Recently, several different approaches have been developed to generalize the least action
principle and the Euler--Lagrange equations to include fractional derivatives.
Results include problems depending on Caputo fractional derivatives, Riemann--Liouville fractional derivatives,
Riesz fractional derivatives and others \cite{Almeida:AML,MyID:182,MyID:152,MyID:179,Cresson,jmp,gastao,mal,%
comBasia:Frac1,comDorota,MyID:181,MyID:207,Sha,BCGI,gastao2,Blaszczyk:et:al,Malgorzata1,Lazo,tatiana,DerInt,MyID:209,Shakoor:01}.
For the state of the art of the fractional calculus of variations we refer the reader to the recent book \cite{book:AD}.

A more general unifying perspective to the subject is, however, possible, by considering fractional operators
depending on general kernels \cite{OmPrakashAgrawal,MyID:226,FVC_Gen_Int,Lupa}. In this work we follow such
an approach, developing a generalized fractional calculus of variations. We consider problems,
where the Lagrangians depend not only on classical derivatives but also on generalized fractional operators.
Moreover, we discuss even more general problems, where also classical integrals are substituted by generalized
fractional integrals and obtain general theorems, for several types of variational problems, which are valid
for rather arbitrary operators and kernels. As special cases, one obtains the recent results available in the
literature of fractional variational calculus \cite{book:Klimek,Nabulsi,Nabulsi3,book:AD,Herrera}.

This thesis consists of two parts. The first one, named Synthesis, gives preliminary definitions and properties of fractional
operators under consideration (Chapter~\ref{sec:FC:1}). Moreover, it briefly describes recent results on the fractional calculus
of variations (Chapter~\ref{ch:CV}). The second one, called Original Work, contains new results published during my PhD project
in peer reviewed international journals, as chapters in books, or in the conference proceedings
\cite{LTDExist,GenExist,variable,GreenThm,MyID:226,FVC_Gen_Int,MyID:207,Hawaii,FVC_Sev,CLandFR,tatiana,NoetherTD,NoetherVO,NoetherGen}.
It is divided in three chapters. We begin with Chapter~\ref{ch:st}, where we apply standard methods to solve several problems of the
generalized fractional calculus of variations. We consider problems with Lagrangians depending on classical derivatives, generalized
fractional integrals and generalized fractional derivatives. We obtain necessary optimality conditions for the basic
and isoperimetric problems, as well as natural boundary conditions for free boundary value problems. In addition, we prove
a generalized fractional counterpart of Noether's theorem. We consider the case of one and several independent variables.
Moreover, each section contains illustrative optimization problems. Chapter~\ref{ch:di} is dedicated to direct methods
in the fractional calculus of variations. We prove a generalized fractional Tonelli's theorem, showing existence of minimizers
for fractional variational functionals. Then we obtain necessary optimality conditions for minimizers. Several illustrative examples
are presented. In the last Chapter~\ref{ch:SL} we show a certain application of the fractional variational calculus. More precisely,
we prove existence of eigenvalues and corresponding eigenfunctions for the fractional Sturm--Liouville problem using variational methods.
Moreover, we show two theorems concerning the lowest eigenvalue and illustrate our results through an example. We finish the thesis with
a conclusion, pointing out important directions of future research.


\clearpage{\thispagestyle{empty}\cleardoublepage}


\part{Synthesis}
\label{ch:synt:1}


\clearpage{\thispagestyle{empty}\cleardoublepage}


\chapter{Fractional Calculus}
\label{sec:FC:1}

Fractional calculus is a generalization of (integer) differential calculus,
in the sense that it deals with derivatives of real or complex order.
It was introduced on 30th September 1695.
On that day, Leibniz wrote a letter to L'H\^{o}pital,
raising the possibility of generalizing
the meaning of derivatives from integer order
to non-integer order derivatives.
L'H\^{o}pital wanted to know the result for the derivative of order $n=1/2$.
Leibniz replied that ``\emph{one day, useful consequences will be drawn}''
and, in fact, his vision became a reality.  However, the study of non-integer
order derivatives did not appear in the literature until 1819, when Lacroix presented
a definition of fractional derivative based on the usual expression for the $n$th
derivative of the power function \cite{Lacroix}. Within years the fractional
calculus became a very attractive subject to mathematicians,
and many different forms of fractional (\textrm{i.e.}, non-integer)
differential operators were introduced: the Grunwald--Letnikow, Riemann--Liouville,
Hadamard, Caputo, Riesz \cite{book:Kilbas, book:Hilfer,book:Podlubny,book:Samko}
and the more recent notions of Cresson \cite{Cresson}, Katugampola \cite{Katugampola},
Klimek \cite{Malgorzata1}, Kilbas \cite{Kilbas} or variable order fractional operators
introduced by Samko and Ross in 1993 \cite{SamkoRoss}.

In 2010, an interesting perspective to the subject, unifying all mentioned notions
of fractional derivatives and integrals, was introduced in \cite{OmPrakashAgrawal}
and later studied in \cite{GenExist,MyID:226,FVC_Gen_Int,GreenThm,FVC_Sev,Lupa,NoetherGen}.
Precisely, authors considered general operators, which by choosing special kernels, reduce
to the standard fractional operators. However, other nonstandard kernels
can also be considered as particular cases.

This chapter presents preliminary definitions and facts of classical,
variable order and generalized fractional operators.


\section{One-dimensional Fractional Calculus}

We begin with basic facts on the one-dimensional classical,
variable order, and generalized fractional operators.


\subsection{Classical Fractional Operators}
\label{subsec:1}

In this section, we present definitions and properties of the one-dimensional fractional
integrals and derivatives under consideration. The reader interested in the subject
is refereed to the books \cite{book:Kilbas,book:Samko,book:Podlubny,book:Klimek}.

\begin{definition}[Left and right Riemann--Liouville fractional integrals]
We define the left and ~the right Riemann--Liouville fractional
integrals\index{Riemann--Liouville!fractional integrals}
$\Il$ and $\Ir$ of order $\alpha\in\R$ ($\alpha >0$) by
\begin{equation}\label{eq:def:lRLI}
\Il [f](t):=\frac{1}{\Gamma(\alpha)}\int\limits_a^t \frac{f(\tau)d\tau}{(t-\tau)^{1-\alpha}},~~t\in (a,b],
\end{equation}
and
\begin{equation}\label{eq:def:rRLI}
\Ir [f](t):=\frac{1}{\Gamma(\alpha)}\int\limits_t^b \frac{f(\tau)d\tau}{(\tau-t)^{1-\alpha}},~~t\in [a,b),
\end{equation}
respectively. Here $\Gamma(\alpha)$ denotes Euler's Gamma function. Note that,
$\Il [f]$ and $\Ir [f]$ are defined a.e. on $(a,b)$ for $f\in L^1(a,b;\R)$.
\end{definition}

One can also define fractional integral operators in the frame of Hadamard setting. In the following,
we present definitions of Hadamard fractional integrals.

\begin{definition}[Left and right Hadamard fractional integrals]
We define the left-sided and right-sided Hadamard integrals of fractional order
$\alpha\in\R$ ($\alpha>0$) \index{Hadamard fractional integrals} by
\begin{equation*}
\Hl [f](t):=\frac{1}{\Gamma(\alpha)}\int\limits_a^t
\left(\log\frac{t}{\tau}\right)^{\alpha-1}\frac{f(\tau)d\tau}{\tau},~~t>a
\end{equation*}
and
\begin{equation*}
\Hr [f](t):=\frac{1}{\Gamma(\alpha)}\int\limits_t^b
\left(\log\frac{\tau}{t}\right)^{\alpha-1}\frac{f(\tau)d\tau}{\tau},~~t<b,
\end{equation*}
respectively.
\end{definition}

\begin{definition}[Left and right Riemann--Liouville fractional derivatives]
The left Riemann--Liouville fractional derivative\index{Riemann--Liouville!fractional derivatives}
of order $\alpha\in\R$ ($0<\alpha <1$) of a function $f$, denoted by $\Dl [f]$,
is defined by
\begin{equation*}
\forall t\in(a,b],~~\Dl [f](t)
:= \frac{d}{dt}{_{a}}\textsl{I}_t^{1-\alpha}[f](t).
\end{equation*}
Similarly, the right Riemann--Liouville fractional derivative of order $\alpha$
of a function $f$, denoted by $\Dr [f]$,
is defined by
\begin{equation*}
\forall t\in[a,b),~~\Dr [f](t)
:= -\frac{d}{dt}{_{t}}\textsl{I}_b^{1-\alpha}[f](t).
\end{equation*}
\end{definition}

As we can see below, Riemann--Liouville fractional integral and differential
operators of power functions return power functions.

\begin{property}[cf. Property 2.1 \cite{book:Kilbas}]
Now, let $1>\alpha,\beta>0$. Then the following identities hold:
\begin{equation*}
\Il[(\tau-a)^{\beta-1}](t)=\frac{\Gamma(\beta)}{\Gamma(\beta+\alpha)}(t-a)^{\beta+\alpha-1},
\end{equation*}
\begin{equation*}
\Dl[(\tau-a)^{\beta-1}](t)=\frac{\Gamma(\beta)}{\Gamma(\beta-\alpha)}(t-a)^{\beta-\alpha-1},
\end{equation*}
\begin{equation*}
\Ir[(b-\tau)^{\beta-1}](t)=\frac{\Gamma(\beta)}{\Gamma(\beta+\alpha)}(b-t)^{\beta+\alpha-1},
\end{equation*}
and
\begin{equation*}
\Dr[(b-\tau)^{\beta-1}](t)=\frac{\Gamma(\beta)}{\Gamma(\beta-\alpha)}(b-t)^{\beta-\alpha-1}.
\end{equation*}
\end{property}

\begin{definition}[Left and right Caputo fractional derivatives]
The left and the right Caputo fractional derivatives of order $\alpha\in\R$ ($0<\alpha <1$)
are given by \index{Caputo!fractional derivatives}
\begin{equation*}
\forall t\in(a,b],~~\Dcl [f](t):=\Ilc \left[\frac{d}{dt}f\right](t)
\end{equation*}
and
\begin{equation*}
\forall t\in[a,b),~~\Dcr [f](t):=-\Irc \left[\frac{d}{dt}f\right](t),
\end{equation*}
respectively.
\end{definition}

Let $0<\alpha<1$ and $f\in AC([a,b];\R)$. Then the Riemann--Liouville and Caputo fractional
derivatives satisfy relations \index{Relation between!Riemann--Liouville and Caputo fractional derivatives}
\begin{equation}\label{eq:Rel:CRL2}
\Dcl [f](t)=\Dl [f](t)-\frac{f(a)}{(t-a)^{\alpha}\Gamma(1-\alpha)},
\end{equation}
\begin{equation}\label{eq:Rel:CRL4}
\Dcr [f](t)=-\Dr [f](t)+\frac{f(b)}{(b-t)^{\alpha}\Gamma(1-\alpha)},
\end{equation}
that can be found in \cite{book:Kilbas}. Moreover, for Riemann--Liouville fractional
integrals and derivatives, the following composition rules hold
\begin{equation}\label{eq:2c}
\left(\Il\circ\Dl\right) [f](t)=f(t),
\end{equation}
\begin{equation}\label{eq:2rc}
\left(\Ir \circ\Dr\right) [f](t)=f(t).
\end{equation}
Note that, if $f(a)=0$, then \eqref{eq:Rel:CRL2} and \eqref{eq:2c} give
\begin{equation}\label{eq:3c}
\left(\Il\circ\Dcl\right) [f](t)=\left(\Il\circ\Dl\right) [f](t)=f(t),
\end{equation}
and if $f(b)=0$, then \eqref{eq:Rel:CRL4} and \eqref{eq:2rc} imply that
\begin{equation}\label{eq:3rc}
\left(\Ir\circ\Dcr\right) [f](t)=\left(\Ir\circ\Dr\right) [f](t)=f(t).
\end{equation}

The following assertion shows that Riemann--Liouville fractional
integrals satisfy semigroup property.

\begin{property}[cf. Lemma 2.3 \cite{book:Kilbas}]
Let $1>\alpha,\beta>0$ and $f\in L^r(a,b;\R)$, ($1\leq r\leq\infty$). Then, equations
\begin{equation*}
\left(\Il\circ{_{a}}\textsl{I}_{t}^{\beta}\right) [f](t)
={_{a}}\textsl{I}_{t}^{\alpha+\beta} [f](t),
\end{equation*}
and
\begin{equation*}
\left(\Ir\circ{_{t}}\textsl{I}_{b}^{\beta}\right) [f](t)
={_{t}}\textsl{I}_{b}^{\alpha+\beta} [f](t)
\end{equation*}
are satisfied.
\end{property}

Next results show that, for certain classes of functions, Riemann--Liouville fractional
derivatives and Caputo fractional derivatives are left inverse
operators of Riemann--Liouville fractional integrals.

\begin{property}[cf. Lemma 2.4 \cite{book:Kilbas}]
If $1>\alpha>0$ and $f\in L^r(a,b;\R)$, ($1\leq r\leq\infty$),
then the following is true:
\begin{equation*}
\left(\Dl\circ\Il\right) [f](t)=f(t),
\end{equation*}
\begin{equation*}
\left(\Dr\circ\Ir\right) [f](t)=f(t).
\end{equation*}
\end{property}

\begin{property}[cf. Lemma 2.21 \cite{book:Kilbas}]
\label{prop:4}
Let $1>\alpha>0$. If $f$ is continuous on the interval $[a,b]$, then
\begin{equation*}
\left(\Dcl\circ\Il\right) [f](t)=f(t),
\end{equation*}
\begin{equation*}
\left(\Dcr\circ\Ir\right) [f](t)=f(t).
\end{equation*}
\end{property}

For $r$-Lebesgue integrable functions, Riemann--Liouville fractional integrals
and derivatives satisfy the following composition properties.

\begin{property}[cf. Property 2.2 \cite{book:Kilbas}]
\label{prop:5}
Let $1>\alpha>\beta>0$ and $f\in L^r(a,b;\R)$, ($1\leq r\leq\infty$). Then, relations
\begin{equation*}
\left({_{a}}\textsl{D}_{t}^{\beta}\circ\Il\right) [f](t)
={_{a}}\textsl{I}_{t}^{\alpha-\beta} [f](t),
\end{equation*}
and
\begin{equation*}
\left({_{t}}\textsl{D}_{b}^{\beta}\circ{_{t}}\textsl{I}_{b}^{\alpha}\right) [f](t)
={_{t}}\textsl{I}_{b}^{\alpha-\beta} [f](t)
\end{equation*}
are satisfied.
\end{property}

In classical calculus, integration by parts formula relates the integral of a product
of functions to the integral of their derivative and antiderivative. As we can see below,
this formula works also for fractional derivatives, however it changes the type
of differentiation: left Riemann--Lioville fractional derivatives are transformed
to right Caputo fractional derivatives. \index{Integration by parts formula!for fractional derivatives}

\begin{property}[cf. Lemma 2.19 \cite{book:Klimek}]
Assume that $0<\alpha<1$, $f\in AC([a,b];\R)$ and $g\in L^r(a,b;\R)$ ($1\leq r\leq\infty$).
Then, the following integration by parts formula holds:
\begin{equation}
\label{eq:IBP}
\int_a^b f(t)\Dl [g](t)\;dt=\int_a^b g(t)\Dcr [f](t)\;dt
+\left.f(t)\Ilc [g](t)\right|_{t=a}^{t=b}.
\end{equation}
\end{property}

Let us recall the following property yielding boundedness of Riemann--Liouville fractional integral
\index{Boundedness!of Riemann--Liouville fractional integral} in the space $L^r(a,b;\R)$
(cf. Lemma 2.1, formula 2.1.23, from the monograph by Kilbas et al. \cite{book:Kilbas}).

\begin{property}\label{prop:K}
The fractional integral $\Il$ is bounded in space $L^{r}(a,b;\R)$ for $\alpha \in (0,1)$ and $r\geq 1$
\begin{equation}\label{K}
||\Il [f]||_{L^{r}}\leq K_{\alpha}||f||_{L^{r}},\quad K_{\alpha}= \frac{(b-a)^{\alpha}}{\Gamma(\alpha+1)}.
\end{equation}
\end{property}


\subsection{Variable Order Fractional Operators}

In 1993, Samko and Ross \cite{SamkoRoss} proposed an interesting generalization of fractional operators.
They introduced the study of fractional integration and differentiation when the order is not a constant
but a function. Afterwards, several works were dedicated to variable order fractional operators, their
applications and interpretations \cite{AlmeidaSamko,Coimbra,Lorenzo}. In particular, Samko's variable
order fractional calculus turns out to be very useful in mechanics and in the theory of viscous flows
\cite{Coimbra,Diaz,Lorenzo,Pedro,Ramirez,Ramirez2}. Indeed, many physical processes exhibit fractional-order
behavior that may vary with time or space \cite{Lorenzo}. The paper \cite{Coimbra} is devoted to the study
of a variable-order fractional differential equation that characterizes some problems in the theory
of viscoelasticity. In \cite{Diaz} the authors analyze the dynamics and control of a nonlinear variable
viscoelasticity oscillator, and two controllers are proposed for the variable order differential equations
that track an arbitrary reference function. The work \cite{Pedro} investigates the drag force acting
on a particle due to the oscillatory flow of a viscous fluid. The drag force is determined using the
variable order fractional calculus, where the order of derivative vary according to the dynamics of the flow.
In \cite{Ramirez2} a variable order differential equation for a particle in a quiescent viscous liquid is developed.
For more on the application of variable order fractional operators to the modeling of dynamic systems,
we refer the reader to the recent review article \cite{Ramirez}.

Let us introduce the following triangle:
\begin{equation*}
\Delta:=\left\{(t,\tau)\in\R^2:~a\leq \tau<t\leq b\right\},
\end{equation*}
and let $\alpha(t,\tau):\Delta\rightarrow[0,1]$
be such that $\alpha\in C^1\left(\bar{\Delta};\R\right)$.

\begin{definition}[Left and right Riemann--Liouville integrals of variable order]
Operator \index{Riemann--Liouville!integrals of variable order}
\begin{equation*}
{_{a}}\textsl{I}^{\alpha(\cdot,\cdot)}_{t}[f](t)
:= \int\limits_a^t\frac{1}{\Gamma(\alpha(t,\tau))}
(t-\tau)^{\alpha(t,\tau)-1}f(\tau)d\tau \quad (t>a)
\end{equation*}
is the left Riemann--Liouville integral
of variable fractional order $\alpha(\cdot,\cdot)$, while
\begin{equation*}
{_{t}}\textsl{I}^{\alpha(\cdot,\cdot)}_{b}[f](t)
:=\int\limits_t^b \frac{1}{\Gamma(\alpha(\tau,t))}
(\tau-t)^{\alpha(\tau,t)-1}f(\tau)d\tau \quad (t<b)
\end{equation*}
is the right Riemann--Liouville integral
of variable fractional order $\alpha(\cdot,\cdot)$.
\end{definition}

The following example gives a variable order fractional integral
for the power function $(t-a)^{\gamma}$.

\begin{example}[cf. Equation~4 of \cite{SamkoRoss}]
Let $\alpha(t,\tau) = \alpha(t)$ be a function
depending only on variable $t$,
$0<\alpha(t)<1$ for almost all $t \in (a,b)$
and $\gamma>-1$. Then,
\begin{equation}\label{eq:power}
{_{a}}\textsl{I}^{\alpha(\cdot)}_{t} (t-a)^{\gamma}
=\frac{\Gamma(\gamma+1)(t-a)^{\gamma+\alpha(t)}}{\Gamma(\gamma+\alpha(t)+1)}.
\end{equation}
\end{example}

Next we define two types of variable order fractional derivatives.

\begin{definition}[Left and right Riemann--Liouville derivatives of variable order]
The left Riemann--Liouville derivative of variable fractional
\index{Riemann--Liouville!derivatives of variable order}
order $\alpha(\cdot,\cdot)$ of a function $f$ is defined by
\begin{equation*}
\forall t\in(a,b],~~{_{a}}\textsl{D}^{\alpha(\cdot,\cdot)}_{t} [f](t)
:= \frac{d}{dt} {_{a}}\textsl{I}^{1-\alpha(\cdot,\cdot)}_{t} [f](t),
\end{equation*}
while the right Riemann--Liouville derivative of variable fractional order
$\alpha(\cdot,\cdot)$ is defined by
\begin{equation*}
\forall t\in[a,b),~~{_{t}}\textsl{D}^{\alpha(\cdot,\cdot)}_{b}[f](t)
:= -\frac{d}{dt} {_{t}}\textsl{I}^{1-\alpha(\cdot,\cdot)}_{b}[f](t).
\end{equation*}
\end{definition}

\begin{definition}[Left and right Caputo derivatives of variable fractional order]
\label{definition:Caputo}\index{Caputo!derivatives of variable order}
The left Caputo derivative \index{Caputo!derivatives of variable order}
of variable fractional order $\alpha(\cdot,\cdot)$ is defined by
\begin{equation*}
\forall t\in(a,b],~~{^{C}_{a}}\textsl{D}^{\alpha(\cdot,\cdot)}_{t}[f](t)
:={_{a}}\textsl{I}^{1-\alpha(\cdot,\cdot)}_{t} \left[\frac{d}{dt} f\right](t),
\end{equation*}
while the right Caputo derivative of variable fractional order $\alpha(\cdot,\cdot)$ is given by
\begin{equation*}
\forall t\in[a,b),~~{^{C}_{t}}\textsl{D}^{\alpha(\cdot,\cdot)}_{b}[f](t)
:=-{_{t}}\textsl{I}^{1-\alpha(\cdot,\cdot)}_{b} \left[\frac{d}{dt} f\right ](t).
\end{equation*}
\end{definition}


\subsection{Generalized Fractional Operators}

This section presents definitions of one-dimensional generalized fractional operators.
In special cases, these operators simplify to the classical Riemann--Liouville fractional integrals,
and Riemann--Liouville and Caputo fractional derivatives. As before,
\begin{equation*}
\Delta:=\left\{(t,\tau)\in\R^2:~a\leq \tau<t\leq b\right\}.
\end{equation*}

\begin{definition}[Generalized fractional integrals of Riemann--Liouville type]
Let us consider a function $k$ defined almost everywhere on $\Delta$ with values in $\R$.
For any function $f$ defined almost everywhere on $(a,b)$ with value in $\R$,
the generalized fractional integral operator $K_P$ is defined for almost all $t \in (a,b)$
by: \index{Generalized fractional!integrals of Riemann--Liouville type}
\begin{equation}
K_{P}[f](t) = \lambda \int_a^t k(t,\tau)
f(\tau)d\tau + \mu \int_t^b k(\tau,t) f(\tau) d\tau,
\end{equation}
with $P=\langle a,t,b,\lambda,\mu \rangle$, $ \lambda$, $\mu\in \R$.
\end{definition}

In particular, for suitably chosen kernels $k(t,\tau)$ and sets $P$, kernel operators $\K$,
reduce to the classical or variable order fractional integrals of Riemann--Liouville type,
and classical fractional integrals of Hadamard type.

\begin{example}
\begin{enumerate}[(a)]
\item Let $k^{\alpha}(t-\tau)=\frac{1}{\Gamma(\alpha)}(t-\tau)^{\alpha-1}$ and $0<\alpha<1$.
If $P=\langle a,t,b,1,0\rangle$, then
\begin{equation*}
K_{P}[f](t)=\frac{1}{\Gamma(\alpha)}
\int\limits_a^t(t-\tau)^{\alpha-1}f(\tau)d\tau
=: {_{a}}\textsl{I}^{\alpha}_{t} [f](t)
\end{equation*}
is the left Riemann--Liouville fractional integral
of order $\alpha$; if $P=\langle a,t,b,0,1\rangle$, then
\begin{equation*}
K_{P}[f](t)=\frac{1}{\Gamma(\alpha)}
\int\limits_t^b(\tau-t)^{\alpha-1}f(\tau)d\tau
=: {_{t}}\textsl{I}^{\alpha}_{b} [f](t)
\end{equation*}
is the right Riemann--Liouville fractional integral
of order $\alpha$.

\item For $k^{\alpha}(t,\tau)=\frac{1}{\Gamma(\alpha(t,\tau))}(t-\tau)^{\alpha(t,\tau)-1}$
and $P=\langle a,t,b,1,0\rangle$,
\begin{equation*}
K_{P}[f](t)=
\int\limits_a^t\frac{1}{\Gamma(\alpha(t,\tau)}(t-\tau)^{\alpha(t,\tau)-1}f(\tau)d\tau
=: {_{a}}\textsl{I}^{\alpha(\cdot,\cdot)}_{t} [f](t)
\end{equation*}
is the left Riemann--Liouville fractional integral
of order $\alpha(\cdot,\cdot)$ and for $P=\langle a,t,b,0,1\rangle$
\begin{equation*}
K_{P}[f](t)=
\int\limits_t^b\frac{1}{\Gamma(\alpha(\tau,t))}(\tau-t)^{\alpha(t,\tau)-1}f(\tau)d\tau
=: {_{t}}\textsl{I}^{\alpha(\cdot,\cdot)}_{b} [f](t)
\end{equation*}
is the right Riemann--Liouville fractional integral
of order $\alpha(t,\tau)$.

\item For any $0<\alpha<1$, kernel $k^{\alpha}(t,\tau)=\frac{1}{\Gamma(\alpha)}\left(
\log\frac{t}{\tau}\right)^{\alpha-1}\frac{1}{\tau}$ and $P=\langle a,t,b,1,0\rangle$,
the general operator $K_{P}$ reduces to the left Hadamard fractional integral:
\begin{equation*}
K_{P}[f](t)=
\frac{1}{\Gamma(\alpha)}\int_a^t \left(\log\frac{t}{\tau}\right)^{\alpha-1}\frac{f(\tau)d\tau}{\tau}
=: \Hl [f](t);
\end{equation*}
and for $P=\langle a,t,b,0,1\rangle$ operator $K_{P}$ reduces to the right Hadamard fractional integral:
\begin{equation*}
K_{P}[f](t)=
\frac{1}{\Gamma(\alpha)}\int_t^b \left(\log\frac{\tau}{t}\right)^{\alpha-1}\frac{f(\tau)d\tau}{\tau}
=:\Hr [f](t).
\end{equation*}

\item Generalized fractional integrals can be also reduced to, e.g., Riesz, Katugampola or Kilbas
fractional operators. Their definitions can be found in \cite{book:Kilbas,Kilbas,Katugampola}.
\end{enumerate}
\end{example}

The generalized differential operators $A_P$
and $B_P$ are defined with the help of the operator $K_P$.

\begin{definition}[Generalized fractional derivative of Riemann--Liouville type]
\label{def:GRL}
The generalized fractional derivative of Riemann--Liouville type, denoted by $A_P$,
is defined by \index{Generalized fractional!derivatives!of Riemann--Liouville type}
$$
A_P = \frac{d}{dt}\circ K_P.
$$
\end{definition}

The next differential operator is obtained by interchanging the order
of the operators in the composition that defines $A_P$.

\begin{definition}[Generalized fractional derivative of Caputo type]
\label{def:GC}
The general kernel differential operator of Caputo type, denoted by $B_P$,
\index{Generalized fractional!derivatives! of Caputo type} is given by
$$
B_P =K_P \circ\frac{d}{dt}.
$$
\end{definition}

\begin{example}
The standard Riemann--Liouville and Caputo fractional derivatives
(see, \textrm{e.g.}, \cite{book:Kilbas,book:Podlubny,book:Klimek,book:Samko})
are easily obtained from the general kernel operators $A_P $ and $B_P$, respectively.
Let $k^{\alpha}(t-\tau)=\frac{1}{\Gamma(1-\alpha)}(t-\tau)^{-\alpha}$,
$\alpha \in (0,1)$. If $P=\langle a,t,b,1,0\rangle$, then
\begin{equation*}
A_{P} [f](t)=\frac{1}{\Gamma(1-\alpha)}
\frac{d}{dt} \int\limits_a^t(t-\tau)^{-\alpha}f(\tau)d\tau
=: {_{a}}\textsl{D}^{\alpha}_{t} [f](t)
\end{equation*}
is the standard left Riemann--Liouville fractional derivative
of order $\alpha$, while
\begin{equation*}
B_{P} [f](t)=\frac{1}{\Gamma(1-\alpha)}
\int\limits_a^t(t-\tau)^{-\alpha} f'(\tau)d\tau
=: {^{C}_{a}}\textsl{D}^{\alpha}_{t} [f](t)
\end{equation*}
is the standard left Caputo fractional derivative of order $\alpha$;
if $P=\langle a,t,b,0,1\rangle$, then
\begin{equation*}
- A_{P} [f](t)
=- \frac{1}{\Gamma(1-\alpha)} \frac{d}{dt}
\int\limits_t^b(\tau-t)^{-\alpha}f(\tau)d\tau
=: {_{t}}\textsl{D}^{\alpha}_{b} [f](t)
\end{equation*}
is the standard right Riemann--Liouville
fractional derivative of order $\alpha$, while
\begin{equation*}
- B_{P} [f](t) = - \frac{1}{\Gamma(1-\alpha)}
\int\limits_t^b(\tau-t)^{-\alpha} f'(\tau)d\tau
=: {^{C}_{t}}\textsl{D}^{\alpha}_{b} [f](t)
\end{equation*}
is the standard right Caputo fractional derivative of order $\alpha$.
\end{example}


\section{Multidimensional Fractional Calculus}

In this section, we introduce notions of classical, variable order and generalized partial
fractional integrals and derivatives, in a multidimensional finite domain. They are natural
generalizations of the corresponding fractional operators of Section~\ref{subsec:1}. Furthermore,
similarly as in the integer order case, computation of partial fractional derivatives and integrals
is reduced to the computation of one-variable fractional operators. Along the work, for $i=1,\dots,n$,
let $a_i,b_i$ and $\alpha_i$ be numbers in $\R$ and $t=(t_1,\dots,t_n)$ be such that $t\in \Omega_n$,
where $\Omega_n=(a_1,b_1)\times\dots\times(a_n,b_n)$ is a subset of $\R^n$.
Moreover, let us define the following sets:
\begin{equation*}
\Delta_i:=\left\{(t_i,\tau)\in\R^2:~a_i\leq \tau <t_i\leq b_i\right\},~i=1\dots,n.
\end{equation*}


\subsection{Classical Partial Fractional Integrals and Derivatives}
\label{subsec:Cpd}

In this section we present definitions of classical partial fractional integrals and derivatives.
Interested reader can find more details in Section~24.1 of the book \cite{book:Samko}.

\begin{definition}[Left and right Riemann--Liouville partial fractional integrals]
\index{Riemann--Liouville!partial fractional integrals}
Let $t\in\Omega_n$. The left and the right partial Riemann--Liouville fractional integrals
of order $\alpha_i\in\R$ ($\alpha_i>0$) with respect to the $i$th variable $t_i$ are defined by
\begin{equation}
\label{eq:def:lRLIp}
\Ilp [f](t):=\frac{1}{\Gamma(\alpha_i)}\int_{a_i}^{t_i} \frac{f(t_1,\dots,t_{i-1},
\tau,t_{i+1},\dots,t_n)d\tau}{(t_i-\tau)^{1-\alpha_i}},~~t_i>a_i,
\end{equation}
and
\begin{equation}
\label{eq:def:rRLIp}
\Irp [f](t):=\frac{1}{\Gamma(\alpha_i)}\int_{t_i}^{b_i} \frac{f(t_1,\dots,t_{i-1},
\tau,t_{i+1},\dots,t_n)d\tau}{(\tau-t_i)^{1-\alpha}},~~t_i<b_i,
\end{equation}
respectively.
\end{definition}

\begin{definition}[Left and right Riemann--Liouville partial fractional derivatives]
\index{Riemann--Liouville!partial fractional derivatives}
Let $t\in\Omega_n$. The left partial Riemann--Liouville fractional derivative
of order $\alpha_i\in\R$ ($0<\alpha_i<1$) of a function $f$, with respect to the $i$th variable $t_i$
is defined by
\begin{equation*}
\forall t_i\in (a_i,b_i],~~\Dlp [f](t)
:= \frac{\partial}{\partial t_i}{_{a_i}}\textsl{I}_{t_i}^{1-\alpha_i}[f](t).
\end{equation*}
Similarly, the right partial Riemann--Liouville fractional derivative of order $\alpha_i$
of a function $f$, with respect to the $i$th variable $t_i$
is defined by
\begin{equation*}
\forall t_i\in[a_i,b_i),~~\Drp [f](t)
:= -\frac{\partial}{\partial t_i}{_{t_i}}\textsl{I}_{b_i}^{1-\alpha_i}[f](t).
\end{equation*}
\end{definition}

\begin{definition}[Left and right Caputo partial fractional derivatives]
\index{Caputo!partial fractional derivatives}
Let $t\in\Omega_n$. The left and the right partial Caputo fractional derivatives of order
$\alpha_i\in\R$ ($0<\alpha_i <1$) of a function $f$, with respect
to the $i$th variable $t_i$ are given by
\begin{equation*}
\forall t_i\in (a_i,b_i],~~\Dclp [f](t):=\Ilcp \left[\frac{\partial}{\partial t_i}f\right](t),
\end{equation*}
and
\begin{equation*}
\forall t_i\in[a_i,b_i),~~\Dcrp [f](t):=-\Ircp \left[\frac{\partial}{\partial t_i}f\right](t),
\end{equation*}
respectively.
\end{definition}


\subsection{Variable Order Partial Fractional Integrals and Derivatives}

In this section, we introduce the notions of partial fractional operators of variable order.
In the following let us assume that $\alpha_i:\Delta_i\rightarrow [0,1]$,
$\alpha_i\in C^1\left(\bar{\Delta};\R\right)$,
$i=1,\dots,n$, $t\in\Omega_n$ and $f:\Omega_n\rightarrow\R$.
\begin{definition}
\label{def:VOPI}
\index{Riemann--Liouville!partial!integrals of variable order}
The left Riemann--Liouville partial integral of variable fractional order
$\alpha_i(\cdot,\cdot)$ with respect to the $i$th variable $t_i$, is given by
\begin{equation*}
{_{a_i}}\textsl{I}^{\alpha_i(\cdot,\cdot)}_{t_i}[f](t)
:= \int\limits_{a_i}^{t_i}\frac{1}{\Gamma(\alpha_i(t_i,\tau))}
(t_i-\tau)^{\alpha_i(t_i,\tau)-1}f(t_1,\dots,t_{i-1},\tau,t_{i+1},\dots,t_n)d\tau,
\end{equation*}
$t_i>a_i$, while
\begin{equation*}
{_{t_i}}\textsl{I}^{\alpha_i(\cdot,\cdot)}_{b_i}[f](t)
:=\int\limits_{t_i}^{b_i} \frac{1}{\Gamma(\alpha_i(\tau,t_i))}
(\tau-t_i)^{\alpha_i(\tau,t_i)-1}f(t_1,\dots,t_{i-1},\tau,t_{i+1},\dots,t_n)d\tau,
\end{equation*}
$t_i<b_i$, is the right Riemann--Liouville partial integral of variable
fractional order $\alpha_i(\cdot,\cdot)$ with respect to variable $t_i$.
\end{definition}

\begin{definition}
\label{def:VOPRL}\index{Riemann--Liouville!partial!derivatives of variable order}
The left Riemann--Liouville partial derivative of variable fractional
order $\alpha_i(\cdot,\cdot)$, with respect to the $i$th variable $t_i$, is given by
\begin{equation*}
\forall t_i\in (a_i,b_i],~~{_{a_i}}\textsl{D}^{\alpha_i(\cdot,\cdot)}_{t_i} [f](t)
= \frac{\partial}{\partial t_i} {_{a_i}}\textsl{I}^{1-\alpha_i(\cdot,\cdot)}_{t_i} [f](t)
\end{equation*}
while the right Riemann--Liouville partial derivative of variable fractional order
$\alpha_i(\cdot,\cdot)$, with respect to the $i$th variable $t_i$,
is defined by
\begin{equation*}
\forall t_i\in[a_i,b_i),~~{_{t_i}}\textsl{D}^{\alpha_i(\cdot,\cdot)}_{b_i}[f](t)
= -\frac{\partial}{\partial t_i} {_{t_i}}\textsl{I}^{1-\alpha_i(\cdot,\cdot)}_{b_i}[f](t).
\end{equation*}
\end{definition}

\begin{definition}
\label{def:VOPC}
\index{Caputo!partial!derivatives of variable order}
The left Caputo partial derivative
of variable fractional order $\alpha_i(\cdot,\cdot)$,
with respect to the $i$th variable $t_i$, is defined by
\begin{equation*}
\forall t_i\in(a_i,b_i],~~{^{C}_{a_i}}\textsl{D}^{\alpha_i(\cdot,\cdot)}_{t_i} [f](t)
= {_{a_i}}\textsl{I}^{1-\alpha_i(\cdot,\cdot)}_{t_i} \left[\frac{\partial}{\partial t_i}f\right](t),
\end{equation*}
while the right Caputo partial derivative of variable fractional order
$\alpha_i(\cdot,\cdot)$, with respect to the $i$th variable $t_i$, is given by
\begin{equation*}
\forall t_i\in[a_i,b_i),~~{^{C}_{t_i}}\textsl{D}^{\alpha_i(\cdot,\cdot)}_{b_i} [f](t)
=-{_{t_i}}\textsl{I}^{1-\alpha_i(\cdot,\cdot)}_{b_i}\left[\frac{\partial}{\partial t_i} f\right](t).
\end{equation*}
\end{definition}

Note that, if $\alpha_i(\cdot,\cdot)$ is a constant function,
then the partial operators of variable fractional order
are reduced to corresponding partial integrals and derivatives
of constant order introduced in Section \ref{subsec:Cpd}.


\subsection{Generalized Partial Fractional Operators}

Let us assume that $\lambda=(\lambda_1,\dots,\lambda_n)$ and $\mu=(\mu_1,\dots,\mu_n)$
are in $\mathbb{R}^n$. We shall present definitions of generalized partial fractional
integrals and derivatives. Let $k_i:\Delta_i\rightarrow\R$, $i=1\dots,n$ and $t\in\Omega_n$.

\begin{definition}[Generalized partial fractional integral]
\label{def:GPI}\index{Generalized partial fractional!integral}
For any function $f$ defined almost everywhere on $\Omega_n$ with value in $\R$,
the generalized partial integral $K_{P_i}$ is defined for almost all $t_i \in (a_i,b_i)$ by:
\begin{multline*}
K_{P_{i}}[f](t):=\lambda_i\int\limits_{a_i}^{t_i}k_i(t_i,\tau)f(t_1,\dots,t_{i-1},\tau,t_{i+1},\dots,t_n)d\tau \\
+\mu_i\int\limits_{t_i}^{b_i}k_i(\tau,t_i)f(t_1,\dots,t_{i-1},\tau,t_{i+1},\dots,t_n)d\tau,
\end{multline*}
where $P_{i}=\langle a_i,t_i,b_i,\lambda_i,\mu_i \rangle $.
\end{definition}

\begin{definition}[Generalized partial fractional derivative of Riemann--Liouville type]
\index{Generalized partial fractional!derivative!of Riemann--Liouville type}
The generalized partial fractional derivative of Riemann--Liouville
type with respect to the $i$th variable $t_i$ is given by
\begin{equation*}
A_{P_{i}}:=\frac{\partial}{\partial t_i}\circ K_{P_{i}}.
\end{equation*}
\end{definition}

\begin{definition}[Generalized partial fractional derivative of Caputo type]
\index{Generalized partial fractional!derivative!of Caputo type}
The generalized partial fractional derivative of Caputo type
with respect to the $i$th variable $t_i$ is given by
\begin{equation*}
B_{P_{i}}:=K_{P_{i}}\circ\frac{\partial}{\partial t_i}.
\end{equation*}
\end{definition}

\begin{example}
Similarly, as in the one-dimensional case, partial operators $K$, $A$ and $B$ reduce
to the standard partial fractional integrals and derivatives. The left- or right-sided
Riemann--Liouville partial fractional integral with respect to the $i$th variable $t_i$
is obtained by choosing the kernel
$k_i^{\alpha}(t_i,\tau)=\frac{1}{\Gamma(\alpha_i)}(t_i-\tau)^{\alpha_i-1}$. That is,
\begin{equation*}
K_{P_{i}}[f](t)=\frac{1}{\Gamma(\alpha_i)}\int\limits_{a_i}^{t_i}(t_i-\tau)^{\alpha_i-1}
f(t_1,\dots,t_{i-1},\tau,t_{i+1},\dots,t_n)d\tau \\=: {_{a_i}}\textsl{I}^{\alpha_i}_{t_i} [f](t),
\end{equation*}
for $P_{i}=\langle a_i,t_i,b_i,1,0\rangle$, and
\begin{equation*}
K_{P_{i}}[f](t)=\frac{1}{\Gamma(\alpha_i)}\int\limits_{t_i}^{b_i}(\tau-t_i)^{\alpha_i-1}
f(t_1,\dots,t_{i-1},\tau,t_{i+1},\dots,t_n)d\tau \\=: {_{t_i}}\textsl{I}^{\alpha_i}_{b_i} [f](t),
\end{equation*}
for $P_{i}=\langle a_i,t_i,b_i,0,1\rangle$. The standard left- and right-sided Riemann--Liouville
and Caputo partial fractional derivatives with respect to $i$th variable $t_i$ are received
by choosing the kernel $k_i^{\alpha}(t_i,\tau)=\frac{1}{\Gamma(1-\alpha_i)}(t_i-\tau)^{-\alpha_i}$.
If $P_{i}=\langle a_i,t_i,b_i,1,0\rangle$, then
\begin{equation*}
A_{P_{i}}[f](t)=\frac{1}{\Gamma(1-\alpha_i)}\frac{\partial}{\partial t_i}
\int\limits_{a_i}^{t_i}(t_i-\tau)^{-\alpha_i}f(t_1,\dots,t_{i-1},\tau,t_{i+1},\dots,t_n)d\tau\\
=:{_{a_i}}\textsl{D}^{\alpha_i}_{t_i} [f](t),
\end{equation*}
\begin{equation*}
B_{P_{i}}[f](t)=\frac{1}{\Gamma(1-\alpha_i)}\int\limits_{a_i}^{t_i}(t_i-\tau)^{-\alpha_i}
\frac{\partial}{\partial \tau}f(t_1,\dots,t_{i-1},\tau,t_{i+1},\dots,t_n)d\tau\\
=:{^{C}_{a_i}}\textsl{D}^{\alpha_i}_{t_i} [f](t).
\end{equation*}
If $P_{i}=\langle a_i,t_i,b_i,0,1\rangle$, then
\begin{equation*}
-A_{P_{i}}[f](t)=\frac{-1}{\Gamma(1-\alpha_i)}\frac{\partial}{\partial t_i}
\int\limits_{t_i}^{b_i}(\tau-t_i)^{-\alpha_i}f(t_1,\dots,t_{i-1},\tau,t_{i+1},
\dots,t_n)d\tau \\=:{_{t_i}}\textsl{D}^{\alpha_i}_{b_i} [f](t),
\end{equation*}
\begin{equation*}
-B_{P_{i}}[f](t)=\frac{-1}{\Gamma(1-\alpha_i)}\int\limits_{t_i}^{b_i}(\tau-t_i)^{-\alpha_i}
\frac{\partial}{\partial \tau}f(t_1,\dots,t_{i-1},\tau,t_{i+1},\dots,t_n)d\tau\\
=:{^{C}_{t_i}}\textsl{D}^{\alpha_i}_{b_i} [f](t).
\end{equation*}
Moreover, one can easily check, that also variable order partial fractional integrals
and derivatievs are particular cases of operators $\PKi$, $\PAi$ and $\PBi$.
\end{example}


\clearpage{\thispagestyle{empty}\cleardoublepage}


\chapter{Fractional Calculus of Variations}
\label{ch:CV}

The calculus of variations is a beautiful and useful field of mathematics
that deals with problems of determining extrema (maxima or minima)
of functionals. For the first time, serious attention of scientists was directed
to the variational calculus in 1696, when Johann Bernoulli asked about the curve
with specified endpoints, lying in a vertical plane, for which the time taken
by a material point sliding without friction and under gravity from one end
to the other is minimal. This problem gained interest of such scientists as Leibniz,
Newton or L'Hospital and was called brachystochrone problem. Afterwards,
a student of Bernoulli, the brilliant Swiss mathematician Leonhard Euler,
considered the problem of finding a function extremizing
(minimizing or maximizing) an integral
\begin{equation}
\label{eq:funct:1}
\mathcal{J}(y)=\int\limits_a^b F(y(t),\dot{y}(t),t)dt
\end{equation}
subject to the boundary conditions
\begin{equation}\label{eq:bound:1}
y(a)=y_a~~\textnormal{and}~~y(b)=y_b
\end{equation}
with $y\in C^2([a,b];\R)$, $a,b,y_a,y_b\in\R$ and $F(u,v,t)$ satisfying some smoothness properties.
He proved that curve $y(t)$ must satisfy the following necessary condition, so-called Euler--Lagrange equation:
\begin{equation}\label{eq:EL:1}
\frac{\partial F(y(t),\dot{y}(t),t)}{\partial u}-\frac{d}{dt}\left(\frac{\partial F(y(t),\dot{y}(t),t)}{\partial v}\right)=0.
\end{equation}
Solutions of equation \eqref{eq:EL:1} are usually called extremals.
It is important to remark that the calculus of variations is a very interesting topic
because of its numerous applications in geometry and differential equations,
in mechanics and physics, and in areas as diverse as engineering, medicine,
economics, and renewable resources \cite{Clarke}.

In the next example we give a simple application of the calculus of variations.
Precisely, we present the soap bubble problem, stated by Euler in 1744.

\begin{example}[cf. Example 14.1 \cite{Clarke}]
In the soap bubble problem \index{The soap bubble problem} we want to find a surface
of rotation, spanned by two concentric rings of radius $A$ and $B$, which has the minimum area.
This wish is confirmed by  experiment and is based on d'Alembert principle. In the sense
of the calculus of variations, we can formulate the soap bubble problem in the following way:
we want to minimize the variational functional
\begin{equation*}
\mathcal{J}(y)=\int\limits_a^b y(t)\sqrt{1+\dot{y}(t)^2} dt~~subject~~to~~y(a)=A,~~y(b)=B.
\end{equation*}
This is a special case of problem \eqref{eq:funct:1}-\eqref{eq:bound:1} with $F(u,v,t)=u\sqrt{1+v^2}$.
Let $y(t)>0$ $\forall t$. It is not difficult to verify that the Euler--Lagrange equation is given by
\begin{equation*}
\ddot{y}(t)=\frac{1+\dot{y}(t)^2}{y(t)}
\end{equation*}
and its solution is the catenary curve given by
\begin{equation*}
y(t)=k\cosh\left(\frac{t+c}{k}\right),
\end{equation*}
where $c,k$ are certain constants.
\end{example}

This thesis is devoted to the fractional calculus of variations and its generalizations. Therefore,
in the next sections we present basic results of the non-integer variational calculus. Let us precise,
that along the work we will understand $\partial_i F$ as the partial derivative
of function $F$ with respect to its $i$th argument.


\section{Fractional Euler--Lagrange Equations}
\label{sec:FM}

Within the years, several methods were proposed to solve mechanical problems with nonconservative forces,
e.g., Rayleigh dissipation function method, technique introducing an auxiliary coordinate or approach
including the microscopic details of the dissipation directly in the Lagrangian. Although, all mentioned
methods are correct, they are not as direct and simple as it is in the case of conservative systems.
In the notes from 1996-1997, Riewe presented a new approach to nonconservative forces
\cite{CD:Riewe:1996,CD:Riewe:1997}. He claimed that friction forces follow from Lagrangians containing
terms proportional to fractional derivatives. Precisely, for $y:[a,b]\rightarrow \R^r$ and
$\alpha_i,\beta_j\in [0,1]$, $i=1,\dots,N$, $j=1,\dots,N'$, he considered the following energy functional:
\begin{equation*}
\mathcal{J}(y)=\int\limits_a^b F\left(_{a}\textsl{D}_t^{\alpha_1}[y](t),
\dots,_{a}\textsl{D}_t^{\alpha_N}[y](t),_{t}\textsl{D}_b^{\beta_1}[y](t),
\dots,_{t}\textsl{D}_b^{\beta_{N'}}[y](t),y(t),t\right)\;dt,
\end{equation*}
with $r$, $N$ and $N'$ being natural numbers. Using the fractional variational principle
he obtained the following Euler--Lagrange equation:
\begin{equation}
\label{eq:EL:Riewe}
\sum\limits_{i=1}^{N}{_{t}}\textsl{D}_b^{\alpha_i}\left[\partial_i F\right]
+\sum\limits_{i=1}^{N'}{_{a}}\textsl{D}_t^{\beta_i}\left[\partial_{i+N} F\right]+\partial_{N'+N+1}F=0.
\end{equation}
Riewe illustrated his results through the classical problem of linear friction.

\begin{example}[\cite{CD:Riewe:1997}]
Let us consider the following Lagrangian:
\begin{equation}\label{ex:L:Riewe}
F=\frac{1}{2}m\dot{y}^2-V(y)+\frac{1}{2}\gamma i\left({_{a}}\textsl{D}_t^{\frac{1}{2}}[y]\right)^2,
\end{equation}
where the first term in the sum represents kinetic energy, the second one potential energy,
the last one is linear friction energy and $i^2=-1$. Using \eqref{eq:EL:Riewe}
we can obtain the Euler--Lagrange equation for a Lagrangian containing derivatives
of order one and order $\frac{1}{2}$:
$$
\frac{\partial F}{\partial y}+{_{t}}\textsl{D}_b^{\frac{1}{2}}\left[
\frac{\partial F}{\partial {_{a}}\textsl{D}_t^{\frac{1}{2}}[y]}\right]
-\frac{d}{dt}\frac{\partial F}{\partial \dot{y}}=0,
$$
which, in the case of Lagrangian \eqref{ex:L:Riewe}, becomes
$$
m\ddot{y}=-\gamma i\left({_{t}}\textsl{D}_b^{\frac{1}{2}}
\circ{_{a}}\textsl{D}_t^{\frac{1}{2}}\right)[y]-\frac{\partial V(y)}{\partial y}.
$$
In order to obtain the equation with linear friction, $m\ddot{y}+\gamma\dot{y}
+\frac{\partial V}{\partial y}=0$, Riewe suggested considering an infinitesimal
time interval, that is, the limiting case $a\rightarrow b$, while keeping $a<b$.
\end{example}

After the works of Riewe several authors contributed to the theory of the fractional variational calculus.
First, let us point out the approach discussed by Klimek in \cite{Malgorzata1}. It was suggested to study
symmetric fractional derivatives \index{Symmetric fractional derivatives} of order
$\alpha$ ($0<\alpha<1$) defined as follows:
\begin{equation*}
\mathcal{D}^{\alpha}:=\frac{1}{2}\Dl+\frac{1}{2}\Dr.
\end{equation*}
In contrast to the left and right fractional derivatives, operator
$\mathcal{D}^{\alpha}$ is symmetric for the scalar product given by
$$
\langle f|g\rangle:=\int\limits_a^b \overline{f(t)}g(t)\; dt,
$$
that is,
$$
\langle \mathcal{D}^{\alpha}[f]|g\rangle=\langle f|\mathcal{D}^{\alpha}[g]\rangle.
$$
With this notion for the fractional derivative, for $\alpha_i\in (0,1)$ and
$y:[a,b]\rightarrow\R^r$, $i=1,\dots,N$, Klimek considered the following action functional:
\begin{equation}
\label{eq:F:Klimek}
\mathcal{J}(y)=\int\limits_a^b F\left(\mathcal{D}^{\alpha_1}[y](t),
\dots,\mathcal{D}^{\alpha_N}[y](t),y(t),t\right)\;dt.
\end{equation}
Using the fractional variational principle, she derived the Euler--Lagrange equation given by
\begin{equation}
\label{eq:EL:Klimek}
\partial_{N+1} F+\sum\limits_{i=1}^N \mathcal{D}^{\alpha_i}\left[\partial_{i} F\right]=0.
\end{equation}
As an example Klimek considered the following variational functional
\begin{equation*}
\mathcal{J}(y)=\int\limits_a^b 2m\dot{y}^2(t)
-\gamma i\left(\mathcal{D}^{\frac{1}{2}}[y](t)\right)^2-V(y(t))\;dt
\end{equation*}
and under appropriate assumptions arrived to the equation with linear friction
\begin{equation}
m\ddot{y}=-\frac{\partial V}{\partial y}-\gamma\dot{y}.
\end{equation}

Another type of problems, containing Riemann--Liouville fractional derivatives,
was discussed by Klimek in \cite{book:Klimek}:
\begin{equation*}
\mathcal{J}(y)=\int\limits_a^b F(_{a}\textsl{D}_t^{\alpha_1}[y](t),
\dots,_{a}\textsl{D}_t^{\alpha_N}[y](t),y(t),t)dt
\end{equation*}
and the Euler--Lagrange equation
\begin{equation}\label{eq:EL:2}
\partial_{N+1} F +\sum\limits_{i=1}^N
{^{C}_{t}}\textsl{D}_b^{\alpha_i}\left[\partial_i F\right]=0
\end{equation}
including fractional derivatives of the Caputo type was obtained.

The next examples are borrowed from \cite{book:Klimek}.
\begin{example}[cf. Example~4.1.1 of \cite{book:Klimek}]
Let $0<\alpha<1$ and $y$ be a minimizer of the functional
\begin{equation*}
\mathcal{J}(y)=\int\limits_a^b \frac{1}{2}y(t)\Dl[y](t)dt.
\end{equation*}
Then $y$ is a solution to the following Euler--Lagrange equation:
\begin{equation*}
\frac{1}{2}\left(\Dl [y]+\Dcr [y]\right)=0.
\end{equation*}
\end{example}

\begin{example}[cf. Example~4.1.2 of \cite{book:Klimek}]
Let $0<\alpha<1$. The model of harmonic oscillator,
in the framework of classical mechanics, is connected to an action
\begin{equation}
\label{eq:funct:2}
\mathcal{J}(y)=\int\limits_a^b \left[-\frac{1}{2}y'^2(t)
+\frac{\omega^2}{2}y^2(t)\right]dt,
\end{equation}
and is determined by the following equation
\begin{equation}\label{eq:EL:3}
y''+\omega^2y=0.
\end{equation}
If in functional \eqref{eq:funct:2} instead of derivative of order one
we put a derivative of fractional order $\alpha$, then
\begin{equation*}
\mathcal{J}(y)=\int\limits_a^b \left[-\frac{1}{2}\left(\Dl[y](t)\right)^2
+\frac{\omega^2}{2}y^2(t)\right]dt
\end{equation*}
and by \eqref{eq:EL:2} the Euler--Lagrange equation has the following form:
\begin{equation}\label{eq:EL:4}
-\Dcr\left[\Dl[y]\right]+\omega^2y=0.
\end{equation}
If $\alpha\rightarrow 1^+$, then equation \eqref{eq:EL:4} reduces to \eqref{eq:EL:3}.
The proof of this fact, as well as solutions to fractional harmonic oscillator
equation \eqref{eq:EL:4}, can be found in \cite{book:Klimek}.
\end{example}


\section{Fractional Embedding of Euler--Lagrange Equations}

The notion of embedding introduced in \cite{cd} is an algebraic procedure
providing an extension of classical differential equations over
an arbitrary vector space. This formalism is developed in the framework
of stochastic processes \cite{cd}, non-differentiable functions \cite{cft},
and fractional equations \cite{Cresson}.
The general scheme of embedding theories is the following:
(i) fix a vector space $V$ and a mapping
$\iota : C^0 ([a,b] ,\mathbb{R}^n ) \rightarrow V$;
(ii) extend differential operators over $V$;
(iii) extend the notion of integral over $V$.
Let $(\iota , D,J)$ be a given embedding formalism, where a linear operator
$D : V \rightarrow V$ takes place for a generalized derivative on $V$, and
a linear operator $J: V \rightarrow \mathbb{R}$ takes place for a generalized
integral on $V$. The embedding procedure gives two different ways, a priori, to generalize
Euler--Lagrange equations. The first (pure algebraic) way is to make a direct embedding
of the Euler--Lagrange equation. The second (analytic) is to embed the Lagrangian
functional associated to the equation and to derive, by the associated calculus of variations,
the Euler--Lagrange equation for the embedded functional. A natural
question is then the problem of coherence between these two extensions:

{\sc Coherence problem}.
{\it Let $(\iota , D,J)$ be a given embedding formalism.
Do we have equivalence between the Euler--Lagrange equation which gives the direct embedding
and the one received from the embedded Lagrangian system?}

As shown in the work \cite{Cresson} for standard fractional differential calculus,
the answer to the question above is known to be negative. To be more precise,
let us define the following operator first introduced in \cite{Cresson}.

\begin{definition}[Fractional operator of order $(\alpha,\beta)$]
\index{Fractional operator of order $(\alpha,\beta)$}
Let $a,b\in\R$, $a<b$ and $\mu\in\mathbb{C}$. We define the fractional operator
of order $(\alpha,\beta)$, with $\alpha>0$ and $\beta>0$, by
\begin{equation}
\label{eq:def:1}
\mathcal{D}_\mu^{\alpha,\beta}=\frac{1}{2}\left[_{a}\textsl{D}_t^{\alpha}
-_{t}\textsl{D}_b^{\beta}\right]+\frac{i\mu}{2}\left[_{a}\textsl{D}_t^{\alpha}
+_{t}\textsl{D}_b^{\beta}\right].
\end{equation}
\end{definition}

In particular, for $\alpha=\beta=1$ one has $\mathcal{D}_\mu^{1,1}=\frac{d}{dt}$.
Moreover, for $\mu=-i$ we recover the left Riemann--Liouville
fractional derivative of order $\alpha$,
\begin{equation*}
\mathcal{D}_{-i}^{\alpha,\beta}=\Dl,
\end{equation*}
and for $\mu=i$ the right Riemann--Liouville fractional derivative of order $\beta$:
\begin{equation*}
\mathcal{D}_{i}^{\alpha,\beta}=-\Dr.
\end{equation*}

Now, let us consider the following variational functional:
\begin{equation*}
\mathcal{J}(y)=\int\limits_a^b F(\mathcal{D}_\mu^{\alpha,\beta}[y](t),y(t),t)dt
\end{equation*}
defined on the space of continuous functions such that $_{a}\textsl{D}_t^{\alpha}[y]$
together with $_{t}\textsl{D}_b^{\beta}[y]$ exist and $y(a)=y_a$, $y(b)=y_b$. Using
the direct embedding procedure, the Euler--Lagrange equation derived by Cresson is
\begin{equation}
\label{eq:EL:7}
\mathcal{D}_{-\mu}^{\beta,\alpha}\left[\partial_1 F\right]=\partial_2 F.
\end{equation}
Using the variational principle in derivation of the Euler--Lagrange equation, one has
\begin{equation}\label{eq:EL:8}
\mathcal{D}_{\mu}^{\alpha,\beta}\left[\partial_1 F\right]=\partial_2 F.
\end{equation}
Reader can easily notice that, in general, there is a difference between equations
\eqref{eq:EL:7} and \eqref{eq:EL:8} i.e., they are not coherent. Cresson claimed
\cite{Cresson} that this lack of coherence has the following sources:
\begin{itemize}
\item the set of variations in the method of variational principle is to large
and therefore it does not give correct answer; one should find the corresponding
constraints for the variations;
\item there is a relation between lack of coherence and properties of the operator
used to generalize the classical derivative.
\end{itemize}
Let us observe that coherence between \eqref{eq:EL:7} and \eqref{eq:EL:8} is restored
in the case when $\alpha=\beta$ and $\mu=0$. This type of coherence is called time
reversible coherence \index{Time reversible coherence}. For a deeper discussion
of the subject we refer the reader to \cite{Cresson}.

In this chapter we presented few results of the fractional calculus of variations.
A comprehensive study of the subject can be found in the books \cite{book:Klimek,book:AD}.


\clearpage{\thispagestyle{empty}\cleardoublepage}


\part{Original Work}


\clearpage{\thispagestyle{empty}\cleardoublepage}


\chapter{Standard Methods in Fractional Variational Calculus}
\label{ch:st}

The model problem of this chapter is to find an admissible function giving a minimum value
to the integral functional, which depends on an unknown function (or functions) of one or several
variables and its generalized fractional derivatives and/or generalized fractional integrals.
In order to answer this question, we will make use of the standard methods in the fractional
calculus of variations (see e.g., \cite{book:AD}). Namely, by analogy to the classical
variational calculus  (see e.g., \cite{book:Dacorogna}), the approach that we call standard,
is first to prove Euler--Lagrange equations, find their solutions and then to check if they
are minimizers. It is important to remark that standard methods suffer an important disadvantage.
Precisely, solvability of Euler--Lagrange equations is assumed, which is not the case in direct methods
that are going to be presented later (see Chapter~\ref{ch:di}).

Now, before we describe briefly an arrangement of this chapter, we define the concept of minimizer.
Let $(X,\left\|\cdot\right\|)$ be normed linear space and $\mathcal{I}$ be a functional defined
on a nonempty subset $\mathcal{A}$ of $X$. Moreover, let us introduce the following set:
if $\y\in\mathcal{A}$ and $\delta>0$, then
\begin{equation*}
\mathcal{N}_{\delta}(\y):=\left\{y\in\mathcal{A}:\left\|y-\y\right\|<\delta\right\}
\end{equation*}
is called neighborhood of $\y$ in $\mathcal{A}$.
\begin{definition}
Function $\y\in\mathcal{A}$ is called minimizer of $\mathcal{I}$ if there exists
a neighborhood $\mathcal{N}_{\delta}(\y)$ of $\y$ such that
\begin{equation*}
\mathcal{I}(\y)\leq\mathcal{I}(y),~~\textnormal{for all }y\in \mathcal{N}_{\delta}(\y).
\end{equation*}
\end{definition}

Note that any function $y\in\mathcal{N}_{\delta}(\y)$ can be represented
in a convenient way as a perturbation of $\y$. Precisely,
\begin{equation*}
\forall y\in\mathcal{N}_{\delta}(\y),~~\exists\eta\in\mathcal{A}_0,
~~ y=\y+h\eta,~~\left|h\right|\leq\varepsilon,
\end{equation*}
where $0<\varepsilon<\frac{\delta}{\left\|\eta\right\|}$
and $\mathcal{A}_0$ is a suitable set of functions $\eta$ such that
\begin{equation*}
\mathcal{A}_0=\left\{\eta\in X:\y+h\eta\in\mathcal{A},
~~\left|h\right|\leq\varepsilon\right\}.
\end{equation*}

We begin the chapter with Section~\ref{sec:PGFI}, where we prove generalized integration
by parts formula and boundedness of generalized fractional integral from $L^p(a,b;\R)$to $L^q(a,b;\R)$.

In Section~\ref{sec:fp} we consider the one-dimensional fundamental problem with generalized fractional
operators and obtain an appropriate Euler--Lagrange equation. Then, we prove that under some convexity
assumptions on Lagrangian, every solution to the Euler--Lagrange equation is automatically a solution
to our problem. Moreover, as corollaries, we obtain results for problems of the constant and variable
order fractional variational calculus and discuss some illustrative examples.

In Section~\ref{sec:fib} we study variational problems with free end points and
besides Euler--Lagrange equations we prove natural boundary conditions. As particular cases
we obtain natural boundary conditions for problems with constant and variable order fractional operators.

Section~\ref{sec:ip} is devoted to generalized fractional isoperimetric problems. We want to find functions
that minimize an integral functional subject to given boundary conditions and isoperimetric constraints.
We prove necessary optimality conditions and, as corollaries, we obtain Euler--Lagrange equations
for isoperimetric problems with constant and variable order fractional operators. Furthermore,
we illustrate our results through several examples.

In Section~\ref{sec:NTH:sing} we prove a generalized fractional counterpart of Noether's theorem.
Assuming invariance of the functional, we prove that any extremal must satisfy a certain generalized
fractional equation. Corresponding results are obtained for functionals with constant
and variable order fractional operators.

Section~\ref{sec:fpT} is dedicated to variational problems defined by the use of the generalized
fractional integral instead of the classical integral. We obtain Euler--Lagrange equations
and discuss several examples.

Finally, in Section~\ref{sec:SEV} we study multidimensional fractional variational problems
with generalized partial operators. We begin with the proofs of integration by parts formulas
for generalized partial fractional integrals and derivatives. Next, we use these results
to show Euler--Lagrange equations for the fundamental problem. Moreover, we prove a generalized
fractional Dirichlet's principle, necessary optimality condition for the isoperimetric problem
and Noether's theorem. We finish the chapter with some conclusions.


\section{Properties of Generalized Fractional Integrals}
\label{sec:PGFI}

This section is devoted to properties of generalized fractional operators. We begin by proving
in Section~\ref{sec:GKO:bnd} that the generalized fractional operator $K_P$ is bounded and linear.
Later, in Section~\ref{sec:GKO:ibp}, we give integration by parts formulas for generalized fractional operators.


\subsection{Boundedness of Generalized Fractional Operators}
\label{sec:GKO:bnd}

Along the work, we assume that $1<p<\infty$ and that $q$
is an adjoint of $p$, that is $\frac{1}{p}+\frac{1}{q}=1$.
Let us prove the following theorem yielding boundedness of the generalized
fractional integral $K_P$ from $L^p(a,b;\R)$ to $L^q(a,b;\R)$.
\index{Boundedness!of generalized fractional integral $K_P$}

\begin{theorem}
\label{thm:bnd:K}
Let us assume that $k \in L^q (\Delta;\R)$. Then, $K_P$
is a linear bounded operator from $L^p(a,b;\R)$ to $L^q(a,b;\R)$.
\end{theorem}

\begin{proof}
The linearity is obvious. We will show that $K_P$ is bounded from $L^p(a,b;\R)$ to $L^q(a,b;\R)$.
Considering only the first term of $K_P$, let us prove that the following inequality
holds for any $f \in L^p(a,b;\R)$:
\begin{equation}
\label{eq01}
\left( \int\limits_a^b \left| \int\limits_a^t k(t,\tau) f(\tau) \;
d\tau \right|^q \; dt \right)^{1/q} \leq \left\| k \right\|_{L^q (\Delta,\R)} \left\| f \right\|_{L^p}.
\end{equation}
Using Fubini's theorem, we have $k(t,\cdot) \in L^q (a,t;\R)$ for almost all $t \in (a,b)$.
Then, applying H\"older's inequality, we have
\begin{equation}
\label{eq02}
\left| \int\limits_a^t k(t,\tau) f(\tau) \; d\tau \right|^q \leq \left[\left(\int\limits_a^t
\left| k(t,\tau) \right|^q \; d\tau\right)^{\frac{1}{q}}\left(\int\limits_a^t\left| f(\tau)\right|^p\;
d\tau\right)^{\frac{1}{p}}\right]^q \leq  \int\limits_a^t \vert k(t,\tau) \vert^q \; d\tau \; \Vert f \Vert_{L^p}^q
\end{equation}
for almost all $t \in (a,b)$. Hence, integrating equation \eqref{eq02} on the interval $(a,b)$,
we obtain inequality \eqref{eq01}. The proof is completed using the same strategy
on the second term in the definition of $K_P$.
\end{proof}

\begin{corollary}
\label{cor:fr:bd}
\index{Boundedness!of Riemann--Liouville fractional integral}
If $\frac{1}{p}<\alpha<1$, then $\Il$ is a linear bounded operator
from $L^p(a,b;\R)$ to $L^q(a,b;\R)$.
\end{corollary}

\begin{proof}
Let us denote $k^\alpha(t,\tau)=\frac{1}{\Gamma(\alpha)}(t-\tau)^{\alpha-1}$.
For $\frac{1}{p}<\alpha<1$ there exist a constant $C\in\R$
such that for almost all $t\in (a,b)$
\begin{equation}
\label{eq:bnd:1}
\int\limits_a^t\left|k^\alpha(t,\tau)\right|^q\;d\tau \leq C.
\end{equation}
Integrating \eqref{eq:bnd:1} on the $(a,b)$ we have $k^\alpha(t,\tau)\in L^q(\Delta;\R)$.
Therefore, applying Theorem~\ref{thm:bnd:K} with $P=\langle a,t,b,1,0 \rangle$ operator
$\Il$ is linear bounded from $L^p(a,b;\R)$ to $L^q(a,b;\R)$.
\end{proof}

Next result shows that with the use of Theorem~\ref{thm:bnd:K}, one can prove that
variable order fractional integral is a linear bounded operator.

\begin{corollary}
\label{col:var:bd}
\index{Boundedness!of variable order fractional integral}
Let $\alpha:\Delta\rightarrow [\delta,1]$ with $\delta>\frac{1}{p}$. Then
${_{a}}\textsl{I}^{\alpha(\cdot,\cdot)}_{t}$ is linear bounded operator
from $L^p(a,b;\R)$ to $L^q(a,b;\R)$.
\end{corollary}

\begin{proof}
Let us denote $k^\alpha(t,\tau)=(t-\tau)^{\alpha(t,\tau)-1}/\Gamma(\alpha(t,\tau))$.
We have just to prove that $k^\alpha\in L^q(\Delta,\R)$ in order to use Theorem~\ref{thm:bnd:K}.
Let us note that since $\alpha$ is with values in $[\delta,1]$ with $\delta>0$,
then $1/(\Gamma\circ\alpha)$ is bounded. Hence, we have just to prove that
$(\Gamma\circ\alpha)k^\alpha\in L^q(\Delta,\R)$. We have two different cases: $b-a\leq 1$ and $b-a>1$.

In the first case, for any $(t,\tau)\in\Delta$, we have $0<t-\tau\leq 1$ and $q(\delta-1)>-1$. Then:
\begin{equation*}
\int\limits_a^t \; (t-\tau)^{q(\alpha(t,\tau)-1)}\; d\tau \leq \int\limits_a^t \; (t-\tau)^{q(\delta-1)} \; d\tau
=\frac{(t-a)^{q(\delta-1)+1}}{q(\delta-1)+1}\leq \frac{1}{q(\delta-1)+1}.
\end{equation*}
In the second case, for almost all $(t,\tau)\in\Delta\cap (a,a+1)\times (a,b)$, we have $0<t-\tau\leq 1$.
Consequently, we conclude in the same way that:
\begin{equation*}
\int\limits_a^t \; (t-\tau)^{q(\alpha(t,\tau)-1)}\; d\tau \leq \frac{1}{q(\delta-1)+1}.
\end{equation*}
Still in the second case, for almost all $(t,\tau)\in\Delta\cap (a+1,b)\times (a,b)$,
we have $\tau<t-1$ or $t-1\leq\tau\leq t$. Then:
\begin{equation*}
\int\limits_a^t \; (t-\tau)^{q(\alpha(t,\tau)-1)}\; d\tau =\int\limits_a^{t-1} \;
(t-\tau)^{q(\alpha(t,\tau)-1)}\; d\tau+\int\limits_{t-1}^t \; (t-\tau)^{q(\alpha(t,\tau)-1)}\;
d\tau\leq b-a-1+\frac{1}{q(\delta-1)+1}.
\end{equation*}
Consequently, in any case, there exist a constant $C\in\R$ such that for almost all $t\in (a,b)$:
\begin{equation}
\int\limits_a^t\left|k^\alpha(t,\tau)\right|^q\;d\tau \leq C.
\end{equation}
Finally, $k^\alpha\in L^q(\Delta,\R)$.
\end{proof}


\subsection{Generalized Fractional Integration by Parts}
\label{sec:GKO:ibp}

In this section we obtain
formula of integration by parts
for the generalized fractional calculus.
Our results are particularly useful
with respect to applications in dynamic optimization,
where the derivation of the Euler--Lagrange equations
uses, as a key step in the proof, integration by parts
(see e.g., the proof of Theorem~\ref{thm:El:OCM} in Section~\ref{sec:fp}).

In our setting, integration by parts
changes a given parameter set $P$ into its dual $P^{*}$.
The term \emph{duality} comes from the fact that $P^{**} =P$.

\begin{definition}[Dual parameter set]
Given a parameter set $P=\langle a,t,b,\lambda,\mu\rangle$
we denote by $P^{*}$ the parameter set
$P^{*} = \langle a,t,b,\mu,\lambda\rangle$.
We say that $P^{*}$ is the dual of $P$.
\end{definition}

Our first formula of fractional integration by parts
involves the operator $K_P$.

\begin{theorem}
\label{prop4}
\index{Integration by parts formula!for operator $K_P$}
Let us assume that $k \in L^q (\Delta;\R)$. Then, the operator $K_{P^*}$ defined by
\begin{equation}\label{eq03}
K_{P^*} [f](t) = \mu \int\limits_a^t k(t,\tau) f(\tau) \; d\tau
+ \lambda \int\limits_t^b k(\tau,t) f(\tau) \; d\tau
\end{equation}
is a linear bounded operator from $L^{p}(a,b;\R)$ to $L^{q}(a,b;\R)$.
Moreover, the following integration by parts formula holds:
\begin{equation}
\label{eq:IBP:K}
\int\limits_a^b f(t)\cdot K_P[g](t) \; dt=\int\limits_a^b g(t)\cdot K_{P^*}[f](t) \; dt,
\end{equation}
for any $f,g \in L^{p}(a,b;\R)$.
\end{theorem}

\begin{proof}
Using Theorem~\ref{thm:bnd:K}, we obtain that $K_{P^*}$ is a linear bounded operator
from $L^{p}(a,b;\R)$ to $L^{q}(a,b;\R)$. The second part is easily proved using Fubini's theorem.
Indeed, considering only the first term of $K_P$,
the following equality holds for any $f,g \in L^{p}(a,b;\R)$:
\begin{equation*}
\lambda\int\limits_a^b f(t) \cdot \int\limits_a^t k(t,\tau) g(\tau) \; d\tau \; dt
= \lambda\int\limits_a^b g (\tau) \cdot \int\limits_\tau^b k(t,\tau) f(t) \; dt \; d\tau.
\end{equation*}
The proof is completed by using the same strategy on the second part of the definition of $K_P$.
\end{proof}

The next example shows that one cannot relax
the hypotheses of Theorem~\ref{prop4}.

\begin{example}
Let $P=\langle 0,t,1,1,-1\rangle$, $f=g=1$,
and $k(t,\tau)=\frac{t^2-\tau^2}{(t^2+\tau^2)^2}$.
Direct calculations show that
\begin{equation*}
\begin{split}
\int\limits_0^1 K_P [1] dt
&=\int\limits_0^1\left(\int_0^t\frac{t^2-\tau^2}{(t^2+\tau^2)^2}d\tau
-\int\limits_t^1\frac{\tau^2-t^2}{(t^2+\tau^2)^2}d\tau\right)dt\\
&=\int\limits_0^1\left(\int_0^1\frac{t^2-\tau^2}{(t^2+\tau^2)^2}d\tau\right)dt
=\int\limits_0^1 \frac{1}{t^2+1}dt=\frac{\pi}{4}
\end{split}
\end{equation*}
and
\begin{equation*}
\begin{split}
\int\limits_0^1 K_{P^*} [1] d\tau
&=\int\limits_0^1\left(-\int_0^\tau\frac{\tau^2-t^2}{(t^2+\tau^2)^2}dt
+\int\limits_\tau^1\frac{t^2-\tau^2}{(t^2+\tau^2)^2}dt\right)d\tau\\
&=-\int\limits_0^1\left(\int_0^1\frac{\tau^2-t^2}{(t^2+\tau^2)^2}dt\right)d\tau
=-\int\limits_0^1 \frac{1}{\tau^2+1}d\tau=-\frac{\pi}{4}.
\end{split}
\end{equation*}
Therefore, the integration by parts formula \eqref{eq:IBP:K} does not hold.
Observe that in this case
$\int\limits_0^1 \int\limits_0^1 \left|k(t,\tau)\right|^2 d\tau dt=\infty$.
\end{example}

For the classical Riemann--Liouville fractional integrals
the following result holds.

\begin{corollary}
\index{Integration by parts formula!for fractional integrals of Riemann--Liouville type}
Let $\frac{1}{p}<\alpha<1$. If $f, g\in L_p(a,b;\R)$, then
\begin{equation}\label{eq:PartsFrac}
\int\limits_{a}^{b} g(t)\cdot\Il [f](t)\;dt =\int\limits_a^b f(t)\cdot\Ir [g](t)\;dt.
\end{equation}
\end{corollary}

\begin{proof}
Let $k^{\alpha}(t,\tau)=\frac{1}{\Gamma(\alpha)}(t-\tau)^{\alpha-1}$.
Using the same reasoning as in the proof of Corollary~\ref{cor:fr:bd},
one has $k^{\alpha}\in L^q(\Delta;\R)$. Therefore,
\eqref{eq:PartsFrac} follows from \eqref{eq:IBP:K}.
\end{proof}

Furthermore, Riemann--Liouville integrals of variable order
satisfy the following integration by parts formula.

\begin{corollary}
\index{Integration by parts formula!for variable order fractional integrals}
Let $\alpha:\Delta\rightarrow [\delta,1]$ with $\delta>\frac{1}{p}$
and let $f, g\in L_p(a,b;\R)$. Then
\begin{equation}
\label{eq:IBP:varI}
\int\limits_{a}^{b} g(t)\cdot{_{a}}\textsl{I}^{\alpha(\cdot,\cdot)}_{t}[f](t)\; dt
=\int\limits_a^b f(t)\cdot{_{t}}\textsl{I}^{\alpha(\cdot,\cdot)}_{b}[g](t)\;dt.
\end{equation}
\end{corollary}

\begin{proof}
Let $k^{\alpha}(t,\tau)=\frac{1}{\Gamma(\alpha(t,\tau))}(t-\tau)^{\alpha(t,\tau)-1}$.
Similarly as in the proof of Corollary~\ref{col:var:bd}, for $\alpha:\Delta\rightarrow [\delta,1]$
with $\delta>\frac{1}{p}$ one has $k^{\alpha}\in L^q(\Delta;\R)$. Therefore,
\eqref{eq:IBP:varI} follows from \eqref{eq:IBP:K}.
\end{proof}


\section{Fundamental Problem}
\label{sec:fp}

For $P=\langle a,t,b,\lambda,\mu\rangle$,
let us consider the following functional:
\begin{equation}
\label{eq:F:1}
\fonction{\mathcal{I}}{\mathcal{A}(y_a,y_b)}{\R}{y}{\displaystyle
\int\limits_a^b F(y(t),K_P[y](t),\dot{y}(t),B_P[y](t),t) \; dt ,}
\end{equation}
where
$$
\mathcal{A}(y_a,y_b):=\left\{y\in C^1 ([a,b];\R):\; y(a)=y_a,~y(b)=y_b,\;
\textnormal{and}\; K_P[y],B_P[y]\in C([a,b];\R)\right\},
$$
$\dot{y}$ denotes the classical derivative of $y$, $K_P$ is the generalized fractional
integral operator with kernel belonging to $L^q(\Delta;\R)$, $B_P=K_P\circ\frac{d}{dt}$
and $F$ is the Lagrangian function, of class $C^1$:
\begin{equation}
\label{eq:Lagr:Sing}
\fonction{F}{\R^4 \times [a,b]}{\R}{(x_1,x_2,x_3,x_4,t)}{F(x_1,x_2,x_3,x_4,t).}
\end{equation}
Moreover, we assume that
\begin{itemize}
\item $K_{P^*}\left[\tau\mapsto\partial_2
F(y(\tau),K_P[y](\tau),\dot{y}(\tau),B_P[y](\tau),\tau)\right]\in C([a,b];\R)$,
\item $t\mapsto\partial_3 F (y(t),K_P[y](t),\dot{y}(t),B_P[y](t),t)\in C^1([a,b];\R)$,
\item $K_{P^*}\left[\tau\mapsto\partial_4
F(y(\tau),K_P[y](\tau),\dot{y}(\tau),B_P[y](\tau),\tau)\right]\in C^1([a,b];\R)$.
\end{itemize}

The next result gives a necessary optimality condition of the Euler--Lagrange type
for the problem of finding a function minimizing functional \eqref{eq:F:1}.

\begin{theorem}
\label{thm:El:OCM}
\index{Euler--Lagrange equation!for problems with generalized fractional operators}
Let $\bar{y}\in\mathcal{A}(y_a,y_b)$ be a minimizer of functional \eqref{eq:F:1}.
Then, $\bar{y}$ satisfies the following Euler--Lagrange equation:
\begin{equation}\label{eq:EL:OCM}
\frac{d}{dt}\left[\partial_3 F\left(\star_{\bar{y}}\right)(t)\right]
+A_{P^*}\left[\tau\mapsto\partial_4 F\left(\star_{\bar{y}}\right)(\tau)\right](t)
=\partial_1 F\left(\star_{\bar{y}}\right)(t)+K_{P^*}\left[\tau\mapsto\partial_2
F\left(\star_{\bar{y}}\right)(\tau)\right](t),
\end{equation}
where $\left(\star_{\bar{y}}\right)(t)
=(\y(t),K_P[\y](t),\dot{\y}(t),B_P[\y](t),t)$, for $t\in (a,b)$.
\end{theorem}

\begin{proof}
Because $\bar{y}\in\mathcal{A}(y_a,y_b)$ is a minimizer of \eqref{eq:F:1}, we have
$$\mathcal{I}(\bar{y})\leq\mathcal{I}(\bar{y}+h\eta),$$
for any $\left|h\right|\leq\varepsilon$ and every $\eta\in\mathcal{A}(0,0)$.
Let us define the following function
\begin{equation*}
\fonction{\phi_{\bar{y},\eta}}{[-\varepsilon,\varepsilon]}{\R}{h}{\mathcal{I}(\bar{y}+h\eta)
=\displaystyle\int\limits_a^b\; F(\bar{y}(t)+h\eta(t),K_P[\bar{y}+h\eta](t),
\dot{\bar{y}}(t)+h\dot{\eta}(t),B_P[\bar{y}+h\eta](t),t)\;dt.}
\end{equation*}
Since $\phi_{\bar{y},\eta}$ is of class $C^1$ on $[-\varepsilon,\varepsilon]$ and
\begin{equation*}
\phi_{\bar{y},\eta}(0)\leq\phi_{\bar{y},\eta}(h),~~\left|h\right|\leq\varepsilon,
\end{equation*}
we deduce that
\begin{equation*}
\phi_{\bar{y},\eta}'(0)=\left.\frac{d}{dh}\mathcal{I}(\bar{y}+h\eta)\right|_{h=0}=0.
\end{equation*}
Hence, by the theorem of differentiation under an integral sign and by the chain rule we get
\begin{equation*}
\int\limits_a^b\; \left(\partial_1 F(\star_{\bar{y}})(t)\cdot\eta(t)+\partial_2
F(\star_{\bar{y}})(t)\cdot K_P[\eta](t)+\partial_3 F(\star_{\bar{y}})(t)\cdot\dot{\eta}(t)
+\partial_4 F(\star_{\bar{y}})(t)\cdot B_P[\eta](t)\right)\; dt=0.
\end{equation*}
Finally, Theorem~\ref{prop4} yields
\begin{multline*}
\int\limits_a^b\; \left(\partial_1 F(\star_{\bar{y}})(t)
+K_{P^*}\left[\tau\mapsto\partial_2 F(\star_{\bar{y}})(\tau)\right](t)\right)\cdot\eta(t)\\
+\left(\partial_3 F(\star_{\bar{y}})(t)+K_{P^*}\left[\tau\mapsto\partial_4
F(\star_{\bar{y}})(\tau)\right](t)\right)\cdot\dot{\eta}(t)\; dt=0,
\end{multline*}
and applying the classical integration by parts formula and the fundamental
lemma of the calculus of variations (see e.g., \cite{book:GF}) we obtain \eqref{eq:EL:OCM}.
\end{proof}

\begin{remark}
From now, in order to simplify the notation, for $T$,
$S$ being fractional operators we will write shortly
$$
T\left[\partial_i F(y(\tau),T[y](\tau),\dot{y}(\tau),S[y](\tau),\tau)\right]
$$
instead of
$$
T\left[\tau\mapsto\partial_i F(y(\tau),T[y](\tau),\dot{y}(\tau),S[y](\tau),\tau)\right],
~i=1,\dots,5.
$$
\end{remark}

\begin{example}
\label{eq:GEN}
Let $P=\langle 0,t,1,1,0 \rangle$. Consider the problem of minimizing the following functional:
\begin{equation*}
\mathcal{I}(y)=\int\limits_0^1 \left(K_P[y](t)+t\right)^2\; dt
\end{equation*}
subject to the given boundary conditions
\begin{equation*}
y(0)=-1~~\textnormal{and}~~y(1)=-1-\int\limits_0^1u(1-\tau)\;d\tau,
\end{equation*}
where the kernel $k$ of $K_P$ is such that $k(t,\tau)=h(t-\tau)$ with $h\in C^1([0,1];\R)$
and $h(0)=1$. Here the resolvent $u$ is related to the kernel $h$ by
$u(t)=\mathcal{L}^{-1}\left[\frac{1}{s\tilde{h}(s)}-1\right](t)$,
$\tilde{h}(s)=\mathcal{L}[h](s)$, where $\mathcal{L}$ and $\mathcal{L}^{-1}$
are the direct and the inverse Laplace operators, respectively. We apply
Theorem~\ref{thm:El:OCM} with Lagrangian $F$ given by
$F(x_1,x_2,x_3,x_4,t)=(x_2+t)^2$. Because
\begin{equation*}
y(t)=-1-\int\limits_0^t u(t-\tau)\;d\tau
\end{equation*}
is the solution to the Volterra integral equation of the first kind
(see, e.g.,~Equation 16, p.114 of \cite{book:Polyanin})
\begin{equation*}
K_P[y](t)+t=0,
\end{equation*}
it satisfies our generalized Euler--Lagrange equation \eqref{eq:EL:OCM}, that is,
\begin{equation*}
K_{P^*}\left[K_P[y](\tau)+\tau\right](t)=0,~t\in (a,b).
\end{equation*}
In particular, for the kernel $h(t-\tau)=e^{-(t-\tau)}$ and the boundary conditions
are $y(0)=-1$, $y(1)=-2$, the solution is $y(t)=-1-t$.
\end{example}

\begin{remark}[cf. Theorem 2.2.3 of \cite{book:Brunt}]
If the functional \eqref{eq:F:1} does not depend on $K_P$ and $B_P$,
then Theorem~\ref{thm:El:OCM} reduces to the classical result:
if $\bar{y}\in C^2([a,b];\R)$ is a solution to the problem of minimizing the functional
\begin{equation*}
\mathcal{I}(y)=\int\limits_a^b F\left(y(t),\dot{y}(t),t\right)dt,
~~\textnormal{subject to}~~y(a)=y_a, \quad y(b)=y_b,
\end{equation*}
then $\bar{y}$ satisfies the Euler--Lagrange equation
\begin{equation*}
\label{eq:CEL}
\partial_1 F\left(\bar{y}(t),\dot{\bar{y}}(t),t\right)
- \frac{d}{dt} \partial_2 F\left(\bar{y}(t),\dot{\bar{y}}(t),t\right)=0,
~~\textnormal{for all }t\in (a,b).
\end{equation*}
\end{remark}

\begin{remark}
In the particular case when functional \eqref{eq:F:1} does not depend on the integer derivative
of function $y$, we obtain from Theorem~\ref{thm:El:OCM} the following result:
if $\bar{y}\in \mathcal{A}(y_a,y_b)$ is a solution to the problem of minimizing the functional
\begin{equation*}
\mathcal{I}(y)=\int\limits_a^b F\left(y(t),K_P[y](t),B_P[y](t),t\right)dt,
\end{equation*}
subject to $y(a)=y_a$ and $y(b)=y_b$, then $\bar{y}$ satisfies the Euler--Lagrange equation
\begin{multline*}
A_{P^*}\left[\partial_4 F\left(\bar{y}(\tau),K_P[\bar{y}](\tau),
B_P[\bar{y}](\tau),\tau\right)\right](t)\\
=\partial_1 F\left(\bar{y}(t),K_P[\bar{y}](t),B_P[\bar{y}](t),t\right)
+K_{P^*}\left[\partial_2 F\left(\bar{y}(\tau),K_P[\bar{y}](\tau),
B_P[\bar{y}](\tau),\tau\right)\right](t),~t\in (a,b).
\end{multline*}
This extends some of the recent results of \cite{OmPrakashAgrawal}.
\end{remark}

\begin{corollary}
\label{cor:FP:CL}
\index{Euler--Lagrange equation!for problems with Riemann--Liouville fractional integrals and Caputo fractional derivatives}
Let $0<\alpha<\frac{1}{q}$ and let $\bar{y}\in C^1([a,b];\R)$
be a solution to the problem of minimizing the functional
\begin{equation}
\label{eq:F:cor}
\mathcal{I}(y)=\int\limits_a^b\; F(y(t),\Ilc [y](t),\dot{y}(t),\Dcl [y](t),t)\; dt,
\end{equation}
subject to the boundary conditions $y(a)=y_a$ and $y(b)=y_b$, where
\begin{itemize}
\item $F\in C^1(\R^4\times [a,b];\R)$,
\item functions $t\mapsto\partial_1 F(y(t),\Ilc [y](t),\dot{y}(t),\Dcl [y](t),t)$,\newline
$\Irc\left[\partial_2 F(y(\tau),\Ilct [y](\tau),\dot{y}(\tau),\Dclt [y](\tau),\tau)\right]$
are continuous on $[a,b]$,
\item functions $t\mapsto\partial_3 F(y(t),\Ilc [y](t),\dot{y}(t),\Dcl [y](t),t)$, \newline
$\Irc\left[\partial_4 F(y(\tau),\Ilct [y](\tau),\dot{y}(\tau),\Dclt [y](\tau),\tau)\right]$
are continuously differentiable on $[a,b]$.
\end{itemize}
Then, the following Euler--Lagrange equation holds
\begin{multline}\label{eq:EL:cor}
\frac{d}{dt}\left(\partial_3 F(\bar{y}(t),\Ilc [\bar{y}](t),\dot{\bar{y}}(t),
\Dcl [\bar{y}](t),t)\right)\\
-\Dr\left[\partial_4 F(\bar{y}(\tau),\Ilct [\bar{y}](\tau),\dot{\bar{y}}(\tau),
\Dclt [\bar{y}](\tau),\tau)\right](t)\\
=\partial_1 F(\bar{y}(t),\Ilc [\bar{y}](t),\dot{\bar{y}}(t),\Dcl [\bar{y}](t),t)\\
+\Ir\left[\partial_2 F(\bar{y}(\tau),\Ilct [\bar{y}](\tau),\dot{\bar{y}}(\tau),
\Dclt [\bar{y}](\tau),\tau)\right](t),~t\in (a,b).
\end{multline}
\end{corollary}

\begin{proof}
The intended Euler--Lagrange equation follows from \eqref{eq:EL:OCM}
by choosing $P=\langle a,t,b,1,0\rangle$ and the kernel
$k^\alpha(t,\tau)=\frac{1}{\Gamma(1-\alpha)}(t-\tau)^{-\alpha}$.
Note that for $0<\alpha<\frac{1}{q}$, we have $k^\alpha\in L^q(\Delta;\R)$.
\end{proof}

In Example~\ref{ex:FP:CF} we make use of the Mittag--Leffler function
of one parameter. Let $\alpha>0$. We recall that
the Mittag--Leffler function is defined by
\begin{equation*}
E_{\alpha}(z)
=\sum_{k=0}^\infty\frac{z^k}{\Gamma(\alpha k+1)}\, .
\end{equation*}
This function appears naturally in the solutions
of fractional differential equations,
as a generalization of the exponential function
\cite{CapelasOliveira}.
Indeed, while a linear second order
ordinary differential equation
with constant coefficients presents an exponential function as solution,
in the fractional case the Mittag--Leffler functions emerge \cite{book:Kilbas}.

\begin{example}\label{ex:FP:CF}
Let $0<\alpha<\frac{1}{q}$. Consider problem of minimizing the functional
\begin{equation}\label{eq:funct:ex1}
\mathcal{I}(y)=\int\limits_0^1 \sqrt{1+\left(\dot{y}(t)+\Dcl [y](t)-1\right)^2}dt
\end{equation}
subject to the following boundary conditions:
\begin{equation}\label{eq:bc:ex1}
y(0)=0~\textnormal{and}~y(1)=\int\limits_0^1\;E_{1-\alpha}\left[-(1-\tau)^{1-\alpha}\right]d\tau.
\end{equation}
Function $F$ of Corollary~\ref{cor:FP:CL} is given by $F(x_1,x_2,x_3,x_4,t)=\sqrt{1+(x_3+x_4-1)^2}$.
One can easily check that (see~\cite{book:Kilbas} p.324)
\begin{equation}\label{eq:sol:5}
y(t)=\int\limits_0^t\;E_{1-\alpha}\left[-(t-\tau)^{1-\alpha}\right]d\tau
\end{equation}
satisfies $\dot{y}(t)+\Dcl[y](t)\equiv 1$. Moreover, it satisfies
$$
\frac{d}{dt}\left(\frac{\dot{y}(t)+\Dcl [y](t)-1}{\sqrt{1+\left(\dot{y}(t)+\Dcl [y](t)-1\right)^2}}\right)
-\Dr\left[\frac{\dot{y}(\tau)+\Dclt [y](\tau)-1}{\sqrt{1+\left(\dot{y}(\tau)+\Dclt [y](\tau)-1\right)^2}}\right](t)=0,
$$
for all $t\in (a,b)$. We conclude that \eqref{eq:sol:5} is a candidate for function
giving a minimum to problem \eqref{eq:funct:ex1}--\eqref{eq:bc:ex1}.
\end{example}

\begin{corollary}[cf. \cite{OmPrakashAgrawal2}]
Let $0<\alpha<\frac{1}{q}$ and let $\mathcal{I}$ be the functional
\begin{equation}
\label{eq:mix}
\index{Euler--Lagrange equation!for problems with Caputo fractional derivatives}
\mathcal{I}(y) = \int\limits_a^b F\left(y(t),\dot{y}(t),
\lambda \, _{a}^{C}\textsl{D}_t^\alpha [y](t)
+\mu \, _{t}^{C}\textsl{D}_b^\alpha [y](t),t\right)dt,
\end{equation}
where $\lambda$ and $\mu$ are real numbers,
and let $\bar{y}\in C^1 ([a,b];\R)$ be a minimizer
of $\mathcal{I}$ among all functions satisfying boundary conditions
$y(a)=y_a$, $y(b)=y_b$. Moreover, we assume that
\begin{itemize}
\item $F\in C^1(\R^3\times [a,b];\R)$,
\item functions $t\mapsto\partial_2 F\left(y(t),\dot{y}(t),\lambda
\, _{a}^{C}\textsl{D}_t^\alpha [y](t)+\mu \, _{t}^{C}\textsl{D}_b^\alpha [y](t),t\right)$,\newline
$\Ilc\left[\partial_3 F\left(y(\tau),\dot{y}(\tau),\lambda\,
_{a}^{C}\textsl{D}_{\tau}^\alpha [y](\tau)+\mu \,
_{\tau}^{C}\textsl{D}_b^\alpha [y](\tau),\tau\right)\right]$,\newline
and $\Irc\left[\partial_3 F\left(y(\tau),\dot{y}(\tau),\lambda \,
_{a}^{C}\textsl{D}_{\tau}^\alpha [y](\tau)+\mu \,_{\tau}^{C}\textsl{D}_b^\alpha [y](\tau),
\tau\right)\right]$ are continuously differentiable on $[a,b]$.
\end{itemize}
Then, $\bar{y}$ satisfies the Euler--Lagrange equation
\begin{multline}
\label{eq:35}
\lambda \, {_{t}}\textsl{D}_b^\alpha \left[\partial_3 F\left(\bar{y}(\tau),\dot{\bar{y}}(\tau),
\lambda \, _{a}^{C}\textsl{D}_{\tau}^\alpha [\bar{y}](\tau)
+\mu \, _{\tau}^{C}\textsl{D}_b^\alpha [\bar{y}](\tau),\tau\right)\right](t)\\
+ \mu \, {_{a}}\textsl{D}_t^\alpha \left[\partial_3 F\left(\bar{y}(\tau),\dot{\bar{y}}(\tau),
\lambda \, _{a}^{C}\textsl{D}_{\tau}^\alpha [\bar{y}](\tau)
+\mu \, _{\tau}^{C}\textsl{D}_b^\alpha [\bar{y}](\tau),\tau\right)\right](t)\\
+\partial_1 F \left(\bar{y}(t),\dot{\bar{y}}(t),
\lambda \, _{a}^{C}\textsl{D}_t^\alpha [\bar{y}](t)
+\mu \, _{t}^{C}\textsl{D}_b^\alpha [\bar{y}](t),t\right)\\
- \frac{d}{dt}\left(\partial_2 F\left(\bar{y}(t),\dot{\bar{y}}(t),
\lambda \, _{a}^{C}\textsl{D}_t^\alpha [\bar{y}](t)
+\mu \, _{t}^{C}\textsl{D}_b^\alpha [\bar{y}](t),t\right)\right) = 0
\end{multline}
for all $t\in (a,b)$.
\end{corollary}

\begin{proof}
Choose $P=\langle a,t,b,\lambda,-\mu\rangle$ and $k^{\alpha}(t-\tau)
=\frac{1}{\Gamma(1-\alpha)}(t-\tau)^{-\alpha}$.
Then, for $0<\alpha<\frac{1}{q}$ kernel $k^\alpha$ is in $L^q(\Delta;\R)$,
the operator $B_P$ reduces to the sum of the left and right Caputo fractional derivatives
and \eqref{eq:35} follows from \eqref{eq:EL:OCM}.
\end{proof}

\begin{corollary}
\index{Euler--Lagrange equation!for problems with variable order fractional integrals and derivatives}
Let us consider the problem of minimizing a functional
\begin{equation}
\label{eq:31}
\mathcal{I}(y)=\int\limits_a^b
F\left(y(t),{_{a}}\textsl{I}^{1-\alpha(\cdot,\cdot)}_{t}[y](t),\dot{y}(t),
{^{C}_{a}}\textsl{D}^{\alpha(\cdot,\cdot)}_{t}[y](t), t\right)dt
\end{equation}
subject to boundary conditions
\begin{equation}
\label{eq:32}
y(a)=y_a, \quad y(b)=y_b,
\end{equation}
where $\dot{y}, {_{a}}\textsl{I}^{1-\alpha(\cdot,\cdot)}_{t}[y],
{^{C}_{a}}\textsl{D}^{\alpha(\cdot,\cdot)}_{t}[y]\in C([a,b];\R)$
and $\alpha:\Delta\rightarrow [0,1-\delta]$ with $\delta>\frac{1}{p}$.
Moreover, we assume that
\begin{itemize}
\item $F\in C^1(\R^4\times[a,b],\R)$
\item function ${_t}\textsl{I}^{1-\alpha(\cdot,\cdot)}_{b}\left[\partial_2
F\left(y(\tau),{_{a}}\textsl{I}^{1-\alpha(\cdot,\cdot)}_{\tau}[y](\tau),\dot{y}(\tau),
{^{C}_{a}}\textsl{D}^{\alpha(\cdot,\cdot)}_{\tau}[y](\tau),\tau\right)\right]$ is continuous on $[a,b]$,
\item functions $t\mapsto\partial_3 F\left(y(t),{_{a}}\textsl{I}^{1-\alpha(\cdot,\cdot)}_{t}[y](t),
\dot{y},{^{C}_{a}}\textsl{D}^{\alpha(\cdot,\cdot)}_{t}[y](t),t\right)$ \newline
and ${_t}\textsl{I}^{1-\alpha(\cdot,\cdot)}_{b}\left[\partial_4 F\left(y(\tau),
{_{a}}\textsl{I}^{1-\alpha(\cdot,\cdot)}_{\tau}[y](\tau),\dot{y}(\tau),
{^{C}_{a}}\textsl{D}^{\alpha(\cdot,\cdot)}_{\tau}[y](\tau),\tau\right)\right]$
are continuously differentiable on $[a,b]$.
\end{itemize}
Then, if $\bar{y}\in C^1([a,b];\R)$ is a solution to problem
\eqref{eq:31}--\eqref{eq:32}, it necessarily satisfies
\begin{multline}
\label{eq:eqELCaputo}
\partial_1 F\left(y(t),{_{a}}\textsl{I}^{1-\alpha(\cdot,\cdot)}_{t}[y](t),
\dot{y},{^{C}_{a}}\textsl{D}^{\alpha(\cdot,\cdot)}_{t}[y](t),t\right)\\
-\frac{d}{dt}\partial_3 F\left(y(t),{_{a}}\textsl{I}^{1-\alpha(\cdot,\cdot)}_{t}[y](t),
\dot{y},{^{C}_{a}}\textsl{D}^{\alpha(\cdot,\cdot)}_{t}[y](t),t\right)\\
+{_{t}}\textsl{I}^{1-\alpha(\cdot,\cdot)}_{b}\left[\partial_2 F\left(y(\tau),
{_{a}}\textsl{I}^{1-\alpha(\cdot,\cdot)}_{\tau}[y](\tau),\dot{y}(\tau),
{^{C}_{a}}\textsl{D}^{\alpha(\cdot,\cdot)}_{\tau}[y](\tau),\tau\right)\right](t)\\
+{_{t}}\textsl{D}^{\alpha(\cdot,\cdot)}_{b}\left[\partial_4 F\left(y(\tau),
{_{a}}\textsl{I}^{1-\alpha(\cdot,\cdot)}_{\tau}[y](\tau),\dot{y}(\tau),
{^{C}_{a}}\textsl{D}^{\alpha(\cdot,\cdot)}_{\tau}[y](\tau),\tau\right)\right](t)=0
\end{multline}
for all $t\in (a,b)$.
\end{corollary}

\begin{proof}
For  $\alpha:\Delta\rightarrow [0,1-\delta]$ with $\delta>\frac{1}{p}$ we have that
$k^{\alpha}(t,\tau)=\frac{1}{\Gamma(1-\alpha(t,\tau))}(t-\tau)^{-\alpha(t,\tau)}$
is in $L^q (\Delta;\R)$. Therefore, from \eqref{eq:EL:OCM} follows \eqref{eq:eqELCaputo}.
\end{proof}

In the next example $\alpha(t)$ is a function defined on the interval $[a,b]$ and taking values in the set
$[0,1-\delta]$, where $\delta>\frac{1}{p}$. As before, we assume that $\alpha\in C^1([a,b];\R)$

\begin{example}
Consider the following problem:
\begin{gather*}
\mathcal{J}(y)=\int\limits_a^b
\left({^{C}_{a}}\textsl{D}^{\alpha(\cdot)}_{t}[y](t)\right)^2
+ \left({_{a}}\textsl{I}^{1-\alpha(\cdot)}_{t}[y](t)
- \frac{\xi(t-\tau)^{1-\alpha(t)}}{\Gamma(2-\alpha(t))}\right)^2 dt \longrightarrow \min,\\
y(a) = \xi, \quad y(b) = \xi,
\end{gather*}
for a given real $\xi$. Because $\mathcal{J}(y) \geq 0$ for any function $y$
and $\mathcal{J}(\y) = 0$ for the admissible function $\bar{y} = \xi$
(use relation \eqref{eq:power} for $\gamma=0$, the linearity of operator
${_{a}}\textsl{I}^{1-\alpha(\cdot)}_{t}$, and the definition of left Caputo derivative
of a variable fractional order), we conclude that $\bar{y}$
is the global minimizer to the problem. It is straightforward
to check that $\bar{y}$ satisfies our variable order fractional
Euler--Lagrange equation \eqref{eq:eqELCaputo}.
\end{example}

Next result gives a sufficient condition assuring that solution of
\eqref{eq:EL:OCM} is indeed minimizer of \eqref{eq:F:1}.

\begin{theorem}
Let $\y\in\mathcal{A}(y_a,y_b)$ satisfies \eqref{eq:EL:OCM} and
$(x_1,x_2,x_3,x_4)\longmapsto F(x_1,x_2,x_3,x_4,t)$ be convex for every
$t\in [a,b]$. Then $\y$ is a minimizer of \eqref{eq:F:1}.
\end{theorem}

\begin{proof}
Let us assume that $\y\in\mathcal{A}(y_a,y_b)$ satisfies equation \eqref{eq:EL:OCM}
and that $(x_1,x_2,x_3,x_4)\longmapsto F(x_1,x_2,x_3,x_4,t)$ is convex for every
$t\in [a,b]$. Then, for every $y\in\mathcal{A}(y_a,y_b)$ we have
\begin{multline*}
\mathcal{I}(y)\geq\mathcal{I}(\y)+\int\limits_a^b \left(\partial_1
F\cdot(y-\y)+\partial_2 F\cdot(K_P[y]-K_P[\y])\right.\\
\left.+\partial_3 F\cdot(\dot{y}-\dot{\y})+\partial_4 F\cdot(B_P[y]-B_P[\y])\right)\; dt,
\end{multline*}
where $\partial_i F$ are taken in $(\y,K_P[\y],\dot{\y},B_P[\y],t)$, $i=1,2,3,4$.
Having in mind that $y(a)-\y(a)=y(b)-\y(b)=0$, and using the classical integration
by parts formula as well as Theorem~~\ref{prop4} one has
\begin{equation*}
\mathcal{I}(y)\geq\mathcal{I}(\y)+\int\limits_a^b \left(\partial_1 F
+K_{P^*}[\partial_2 F]-\frac{d}{dt}(\partial_3 F)-A_{P^*}[\partial_4 F]\right)(y-\y)\; dt,
\end{equation*}
where, as before, $\partial_i F$ are taken in $(\y,K_P[\y],\dot{\y},B_P[\y],t)$, $i=1,2,3,4$.
Finally, by Euler--Lagrange equation \eqref{eq:EL:OCM}, we have $\mathcal{I}(y)\geq\mathcal{I}(\y)$.
\end{proof}


\section{Free Initial Boundary}
\label{sec:fib}

Let us define the following set
$$\mathcal{A}(y_b):=\left\{y\in C^1 ([a,b];\R):\; y(a)\;\textnormal{is free},~y(b)
=y_b,\; \textnormal{and}\; K_P[y],B_P[y]\in C([a,b];\R)\right\},$$
and let $\y$ be a minimizer of functional \eqref{eq:F:1} on $\mathcal{A}(y_b)$, i.e., now
\begin{equation}\label{eq:F:2}
\fonction{\mathcal{I}}{\mathcal{A}(y_b)}{\R}{y}{\displaystyle\int\limits_a^b
F(y(t),K_P[y](t),\dot{y}(t),B_P[y](t),t) \; dt .}
\end{equation}
Because
$$\mathcal{I}(\bar{y})\leq\mathcal{I}(\bar{y}+h\eta),$$
for any $\left|h\right|\leq\varepsilon$ and every $\eta\in\mathcal{A}(0)$,
we obtain as in the proof of Theorem~\ref{thm:El:OCM} that
\begin{multline*}
\int\limits_a^b\; \left(\partial_1 F(\star_{\bar{y}})(t)+K_{P^*}\left[\partial_2
F(\star_{\bar{y}})(\tau)\right](t)\right)\cdot\eta(t)\\
+\left(\partial_3 F(\star_{\bar{y}})(t)+K_{P^*}\left[\partial_4
F(\star_{\bar{y}})(\tau)\right](t)\right)\cdot\dot{\eta}(t)\; dt=0,\forall\eta\in\mathcal{A}(0),
\end{multline*}
where $(\star_{\bar{y}})(t)=(\bar{y}(t),K_P[\bar{y}](t),\dot{\bar{y}}(t),B_P[\bar{y}](t),t)$.
Moreover, having in mind that $\eta(b)=0$ and using the classical integration by parts formula, we find that
\begin{multline*}
\int\limits_a^b\; \left(\partial_1 F(\star_{\bar{y}})(t)+K_{P^*}\left[\partial_2
F(\star_{\bar{y}})(\tau)\right](t)-\frac{d}{dt}\left(\partial_3 F(\star_{\bar{y}})(t)
+K_{P^*}\left[\partial_4 F(\star_{\bar{y}})(\tau)\right](t)\right)\right)\cdot\eta(t)\; dt\\
+\left.\partial_3 F(\star_{\bar{y}})(t)\cdot\eta(t)\right|_a+\left.K_{P^*}\left[\partial_4
F(\star_{\y})(\tau)\right](t)\cdot\eta(t)\right|_a=0,\forall\eta\in\mathcal{A}(0).
\end{multline*}
Now, using the fundamental lemma of calculus of variations (see e.g.,~\cite{book:GF})
and the fact that $\eta(a)$ is arbitrary, we obtain
\begin{equation*}
\begin{cases}
\frac{d}{dt}\left[\partial_3 F(\star_{\bar{y}})(t)\right]+A_{P^*}\left[\partial_4
F(\star_{\bar{y}})(\tau)\right](t)=\partial_1 F(\star_{\bar{y}})(t)
+K_{P^*}\left[\partial_2 F(\star_{\bar{y}})(\tau)\right](t),\\
\left.\partial_3 F(\star_{\bar{y}})(t)\right|_a+\left.K_{P^*}\left[\partial_4
F(\star_{\y})(\tau)\right](t)\right|_a=0.
\end{cases}
\end{equation*}
We have just proved the following result.
\begin{theorem}
\label{thm:Nat:1}
\index{Natural boundary conditions!for generalized fractional operators}
If $\y\in\mathcal{A}(y_b)$ is a solution to the problem of minimizing functional \eqref{eq:F:2}
on the set $\mathcal{A}(y_b)$, then $\y$ satisfies the Euler--Lagrange equation \eqref{eq:EL:OCM}.
Moreover, the extra boundary condition
\begin{equation}
\label{eq:Nat:1}
\left.\partial_3 F(\star_{\bar{y}})(t)\right|_a
+\left.K_{P^*}\left[\partial_4 F(\star_{\y})(\tau)\right](t)\right|_a=0
\end{equation}
holds with $(\star_{\bar{y}})(t)=(\bar{y}(t),K_P[\bar{y}](t),\dot{\bar{y}}(t),B_P[\bar{y}](t),t)$.
\end{theorem}

Similarly as it is in the theory of the classical calculus of variations we will call
\eqref{eq:Nat:1} the generalized fractional natural boundary condition.

\begin{corollary}[cf. \cite{OmPrakashAgrawal3}]\label{cor:Nat:1}
\index{Natural boundary conditions!for problems with Caputo fractional derivatives}
Let $0<\alpha<\frac{1}{q}$ and $\mathcal{I}$ be the functional given by
\begin{equation*}
\mathcal{I}(y)=\int\limits_a^b
F\left(y(t),_{a}^{C}\textsl{D}_t^\alpha [y](t),t\right)dt,
\end{equation*}
where $F\in C^1(\R^2\times [a,b];\R)$,
and $\Ilc\left[\partial_2 F\left(y(\tau),_{a}^{C}\textsl{D}_{\tau}^\alpha [y](\tau),\tau\right)\right]$
is continuously differentiable on $[a,b]$. If $\y\in C^1([a,b];\R)$ is a minimizer
of $\mathcal{I}$ among all functions satisfying the boundary
condition $y(b)=y_b$, then $\y$ satisfies the Euler--Lagrange equation
\begin{equation*}
\partial_1 F\left(\y(t),_{a}^{C}\textsl{D}_t^\alpha [\y](t),t\right)
+_{t}\textsl{D}_b^\alpha\left[\partial_2 F\left(\y(\tau),
_{a}^{C}\textsl{D}_{\tau}^\alpha [\y](\tau),\tau\right)\right](t)=0
\end{equation*}
and the fractional natural boundary condition
\begin{equation*}
\left._{t}\textsl{I}_b^{1-\alpha}\left[\partial_2
F\left(\y(\tau),_{a}^{C}\textsl{D}_{\tau}^\alpha [\y](\tau),
\tau\right)\right](t)\right|_{a}=0
\end{equation*}
holds for all $t\in (a,b)$.
\end{corollary}

\begin{proof}
Corollary~\ref{cor:Nat:1} follows from Theorem~\ref{thm:Nat:1} with $P=\langle a,t,b,1,0\rangle$
and $k^\alpha(t,\tau)=\frac{1}{\Gamma(1-\alpha)}(t-\tau)^{-\alpha}$.
\end{proof}

\begin{corollary}
\label{cor:Nat:2}
\index{Natural boundary conditions!for problems with variable order fractional integrals and derivatives}
Suppose that $\alpha:\Delta\rightarrow [0,1-\delta]$, $\delta>\frac{1}{p}$ and $\mathcal{I}$
is the functional given by \eqref{eq:31}. If $\y\in C^1([a,b];\R)$ is a minimizer to $\mathcal{I}$
satisfying boundary condition $y(b)=y_b$ and being such that $_{a}\textsl{I}^{1-\alpha(\cdot,\cdot)}_{t}[y],
{^{C}_{a}}\textsl{D}^{\alpha(\cdot,\cdot)}_{t}[y]\in C([a,b];\R)$, then $\y$ is a solution to the Euler--Lagrange equation
\begin{multline*}
\partial_1 F\left(\y(t),{_{a}}\textsl{I}^{1-\alpha(\cdot,\cdot)}_{t}[\y](t),
\dot{\y},{^{C}_{a}}\textsl{D}^{\alpha(\cdot,\cdot)}_{t}[\y](t),t\right)
-\frac{d}{dt}\partial_3 F\left(\y(t),{_{a}}\textsl{I}^{1-\alpha(\cdot,\cdot)}_{t}[\y](t),
\dot{\y}(t),{^{C}_{a}}\textsl{D}^{\alpha(\cdot,\cdot)}_{t}[\y](t), t\right)\\
+{_{t}}\textsl{I}^{1-\alpha(\cdot,\cdot)}_{b}\left[\partial_2 F\left(\y(\tau),
{_{a}}\textsl{I}^{1-\alpha(\cdot,\cdot)}_{\tau}[\y](\tau),\dot{y}(\tau),
{^{C}_{a}}\textsl{D}^{\alpha(\cdot,\cdot)}_{\tau}[\y](\tau), \tau\right)\right](t)\\
+{_{t}}\textsl{D}^{\alpha(\cdot,\cdot)}_{b}\left[\partial_4 F\left(\y(\tau),
{_{a}}\textsl{I}^{1-\alpha(\cdot,\cdot)}_{\tau}[\y](\tau),\dot{\y}(\tau),
{^{C}_{a}}\textsl{D}^{\alpha(\cdot,\cdot)}_{\tau}[\y](\tau), \tau\right)\right](t)=0
\end{multline*}
and the natural boundary condition
\begin{multline*}
\left.\partial_3 F\left(\y(t),{_{a}}\textsl{I}^{1-\alpha(\cdot,\cdot)}_{t}[\y](t),
\dot{\y}(t),{^{C}_{a}}\textsl{D}^{\alpha(\cdot,\cdot)}_{t}[\y](t),
t\right)\right|_a\\
+\left.{_{t}}\textsl{I}^{1-\alpha(\cdot,\cdot)}_{b}\left[\partial_4 F\left(\y(\tau),
{_{a}}\textsl{I}^{1-\alpha(\cdot,\cdot)}_{\tau}[\y](\tau),\dot{\y}(\tau),
{^{C}_{a}}\textsl{D}^{\alpha(\cdot,\cdot)}_{\tau}[\y](\tau),
\tau\right)\right](t)\right|_a=0
\end{multline*}
holds for all $t\in (a,b)$.
\end{corollary}

\begin{proof}
Corollary~\ref{cor:Nat:2} is an easy consequence of Theorem~\ref{thm:Nat:1}.
\end{proof}

\begin{remark}
Observe that if the functional \eqref{eq:F:2} is independent of the operator $K_P$,
that is we have the following problem:
\begin{equation*}
\int\limits_a^b F\left(y(t),\dot{y}(t),B_{P} [y](t),t\right)dt
\longrightarrow \min, \quad y(b)=y_b
\end{equation*}
($y(a)$ free), then the optimality conditions \eqref{eq:EL:OCM}
and \eqref{eq:Nat:1} reduce, respectively, to
\begin{multline*}
\partial_1 F \left(\y(t),\dot{\y}(t),B_{P} [\y](t),t\right)
-\frac{d}{dt}\partial_2 F\left(\y(t),\dot{\y}(t),B_{P}[\y](t),t\right)\\
- A_{P^*}\left[\partial_3 F\left(\y(\tau),\dot{\y}(\tau),B_{P}[\y](\tau),\tau\right)\right](t)=0
\end{multline*}
and $\left.\partial_2 F\left(\y(t),\dot{\y}(t),B_{P}[\y](t),t\right)\right|_a
+\left.K_{P^*}\left[\partial_3
F\left(\y(\tau),\dot{\y}(\tau),B_{P}[\y](\tau),\tau\right)\right](t)\right|_{a}=0$ for all $t\in (a,b)$.
\end{remark}


\section{Isoperimetric Problem}
\label{sec:ip}

One of the earliest problems in geometry is the isoperimetric problem,
already considered by the ancient Greeks. It consists to find,
among all closed curves of a given length, the one which encloses the
maximum area. The general problem for which one integral is to be given
a fixed value, while another is to be made a maximum or a minimum,
is nowadays part of the calculus of variations \cite{book:Brunt,book:Giaquinta}.
Such \emph{isoperimetric problems} have found a broad class of important applications
throughout the centuries, with numerous useful implications
in astronomy, geometry, algebra, analysis,
and engineering \cite{Viktor,Curtis}.
For recent advancements on the study of isoperimetric problems
see \cite{isoJMAA,isoNabla,iso:ts} and references therein.
Here we consider isoperimetric problems
with generalized fractional operators.
Similarly to Section~\ref{sec:fp} and \ref{sec:fib},
we deal with integrands involving both the generalized
Caputo fractional derivative and the generalized fractional integral,
as well as the classical derivative.

Let $P=\langle a,t,b,\lambda,\mu\rangle$. We define the following functional:
\begin{equation}
\label{eq:F:3}
\fonction{\mathcal{J}}{\mathcal{A}(y_a,y_b)}{\R}{y}{\int\limits_a^b
G(y(t),K_P[y](t),\dot{y}(t),B_P[y](t),t) \; dt ,}
\end{equation}
where by $\dot{y}$ we understand the classical derivative of $y$, $K_P$ is the generalized
fractional integral operator with kernel belonging to $L^q(\Delta;\R)$,
$B_P=K_P\circ\frac{d}{dt}$ and $G$ is a Lagrangian of class $C^1$:
\begin{equation*}
\fonction{G}{\R^4 \times [a,b]}{\R}{(x_1,x_2,x_3,x_4,t)}{G(x_1,x_2,x_3,x_4,t).}
\end{equation*}
Moreover, we assume that
\begin{itemize}
\item $K_{P^*}\left[\partial_2 G (y(\tau),K_P[y](\tau),\dot{y}(\tau),B_P[y](\tau),\tau)\right]\in C([a,b];\R)$,
\item $t\mapsto\partial_3 G (y(t),K_P[y](t),\dot{y}(t),B_P[y](t),t)\in C^1([a,b];\R)$,
\item $K_{P^*}\left[\partial_4 G (y(\tau),K_P[y](\tau),\dot{y}(\tau),B_P[y](\tau),\tau)\right]\in C^1([a,b];\R)$.
\end{itemize}

The first problem in this section, is to find a minimizer of functional \eqref{eq:F:1}
subject to the isoperimetric constraint $\mathcal{J}(y)=\xi$. In the next theorem we
provide necessary optimality condition for this type of problems.

\begin{theorem}
\label{thm:EL:ISO1}
\index{Euler--Lagrange equation!for problems with generalized fractional operators}
Suppose that $\y$ is a minimizer of functional $\mathcal{I}$ in the class
$$
\mathcal{A}_{\xi}(y_a,y_b):=\left\{y\in\mathcal{A}(y_a,y_b):\mathcal{J}(y)=\xi\right\}.
$$
Then there exists a real constant $\lambda_0$, such that, for $H=F-\lambda_0 G$, equation
\begin{equation}
\label{eq:EL:ISO1}
\frac{d}{dt}\left[\partial_3 H(\star_{\bar{y}})(t)\right]+A_{P^*}\left[\partial_4
H(\star_{\bar{y}})(\tau)\right](t)=\partial_1 H(\star_{\bar{y}})(t)
+K_{P^*}\left[\partial_2 H(\star_{\bar{y}})(\tau)\right](t),~t\in (a,b)
\end{equation}
holds, provided that
\begin{equation}
\label{eq:EL:eml}
\frac{d}{dt}\left[\partial_3 G(\star_{\bar{y}})(t)\right]+A_{P^*}\left[\partial_4
G(\star_{\bar{y}})(\tau)\right](t)\neq\partial_1 G(\star_{\bar{y}})(t)
+K_{P^*}\left[\partial_2 G(\star_{\bar{y}})(\tau)\right](t),~t\in (a,b)
\end{equation}
where $(\star_{\bar{y}})(t)=(\bar{y}(t),K_P[\bar{y}](t),\dot{\bar{y}}(t),B_P[\bar{y}](t),t)$.
\end{theorem}

\begin{proof}
By hypothesis \eqref{eq:EL:eml} and the fundamental lemma of the calculus of variations
(see e.g.,~\cite{book:GF}) we can choose $\eta_2\in\mathcal{A}(0,0)$ so that
\begin{multline*}
\int\limits_a^b\; \left(\partial_1 G(\star_{\bar{y}})(t)
+K_{P^*}\left[\partial_2 G(\star_{\bar{y}})(\tau)\right](t)\right)\cdot\eta_2(t)\\
+\left(\partial_3 G(\star_{\bar{y}})(t)+K_{P^*}\left[\partial_4
G(\star_{\bar{y}})(\tau)\right](t)\right)\cdot\dot{\eta_2}(t)\; dt=1.
\end{multline*}
With this function $\eta_2$ and an arbitrary $\eta_1\in\mathcal{A}(0,0)$, let us define functions
\begin{equation*}
\fonction{\phi}{[-\varepsilon_1,\varepsilon_1]
\times [-\varepsilon_2,\varepsilon_2]}{\R}{(h_1,h_2)}{\mathcal{I}(\bar{y}+h_1\eta_1+h_2\eta_2)}
\end{equation*}
and
\begin{equation*}
\fonction{\psi}{[-\varepsilon_1,\varepsilon_1]
\times [-\varepsilon_2,\varepsilon_2]}{\R}{(h_1,h_2)}{\mathcal{J}(\bar{y}+h_1\eta_1+h_2\eta_2)-\xi.}
\end{equation*}
Observe that $\psi(0,0)=0$ and
\begin{multline*}
\left.\frac{\partial\psi}{\partial h_2}\right|_{(0,0)}=\int\limits_a^b\;
\left(\partial_1 G(\star_{\bar{y}})(t)+K_{P^*}\left[\partial_2
G(\star_{\bar{y}})(\tau)\right](t)\right)\cdot\eta_2(t)\\
+\left(\partial_3 G(\star_{\bar{y}})(t)+K_{P^*}\left[\partial_4
G(\star_{\bar{y}})(\tau)\right](t)\right)\cdot\dot{\eta_2}(t)\; dt=1.
\end{multline*}
According to the implicit function theorem we can find $\epsilon_0>0$ and a function
$s\in C^1([-\varepsilon_0,\varepsilon_0];\R)$ with $s(0)=0$ such that
\begin{equation*}
\psi(h_1,s(h_1))=0,~~\forall h_1\in[-\varepsilon_0,\varepsilon_0]
\end{equation*}
which implies that $\y+h_1\eta_1+s(h_1)\eta_2\in\mathcal{A}_{\xi}(y_a,y_b)$.
We also have
\begin{equation*}
\frac{\partial\psi}{\partial h_1}+\frac{\partial \psi}{\partial h_2}\cdot s'(h_1)=0,
~~\forall h_1\in [-\varepsilon_0,\varepsilon_0]
\end{equation*}
and hence
\begin{equation*}
s'(0)=\left.-\frac{\partial \psi}{\partial h_1}\right|_{(0,0)}.
\end{equation*}
Because $\y\in\mathcal{A}(y_a,y_b)$ is a minimizer of $\mathcal{I}$ we have
\begin{equation*}
\phi(0,0)\leq\phi(h_1,s(h_1)),~~\forall h_1\in [-\varepsilon_0,\varepsilon_0]
\end{equation*}
and then
\begin{equation*}
\left.\frac{\partial\phi}{\partial h_1}\right|_{(0,0)}
+\left.\frac{\partial\phi}{\partial h_2}\right|_{(0,0)}\cdot s'(0)=0.
\end{equation*}
Letting $\left.\lambda_0=\frac{\partial\phi}{\partial h_2}\right|_{(0,0)}$
be the Lagrange multiplier we find
\begin{equation*}
\left.\frac{\partial\phi}{\partial h_1}\right|_{(0,0)}
-\lambda_0\left.\frac{\partial\psi}{\partial h_1}\right|_{(0,0)}=0
\end{equation*}
or, in other words,
\begin{multline*}
\int\limits_a^b\; \left\{\left[\left(\partial_1 F(\star_{\bar{y}})(t)
+K_{P^*}\left[\partial_2 F(\star_{\bar{y}})(\tau)\right](t)\right)\cdot\eta_1(t)
+\left(\partial_3 F(\star_{\bar{y}})(t)+K_{P^*}\left[\partial_4
F(\star_{\bar{y}})(\tau)\right](t)\right)\cdot\dot{\eta_1}(t)\right]\right.\\
\left.-\lambda_0\left[\left(\partial_1 G(\star_{\bar{y}})(t)+K_{P^*}\left[\partial_2
G(\star_{\bar{y}})(\tau)\right](t)\right)\cdot\eta_1(t)+\left(\partial_3 G(\star_{\bar{y}})(t)
+K_{P^*}\left[\partial_4 G(\star_{\bar{y}})(\tau)\right](t)\right)\cdot\dot{\eta_1}(t)\right]\right\}dt\\
=0.
\end{multline*}
Finally, applying one more time the fundamental
lemma of the calculus of variations we obtain \eqref{eq:EL:ISO1}.
\end{proof}

\begin{example}
Let $P=\langle 0,t,1,1,0\rangle$.
Consider the problem
\begin{equation*}
\begin{gathered}
\mathcal{I}(y)=\int\limits_0^1\left({\textsl{K}_P} [y](t)
+t\right)^2dt \longrightarrow \min,\\
\mathcal{J}(y)=\int\limits_0^1 t{\textsl{K}_P} [y](t)\;dt = \xi,\\
y(0)=\xi-1, \quad y(1)=(\xi-1)\left(1
+\int\limits_0^1 u(1-\tau) d\tau\right),
\end{gathered}
\end{equation*}
where the kernel $k$ is such that $k(t,\tau)=h(t-\tau)$ with $h\in C^1([0,1];\R)$, $h(0)=1$
and $\textsl{K}_{P^*}[id](t)\neq 0$ (notation $id$ means identity transformation i.e., $id(t)=t$).
Here the resolvent $u$ is related to the kernel $h$ in the same way as in Example~\ref{eq:GEN}.
Since $\textsl{K}_{P^*}[id](t)\neq 0$, there is no solution
to the Euler--Lagrange equation for functional $\mathcal{J}$.
The augmented Lagrangian $H$ of Theorem~\ref{thm:EL:ISO1} is given by
$H(x_1,x_2,t) =(x_2+t)^2 -\lambda_0 tx_2$. Function
\begin{equation*}
y(t) = \left(\xi-1\right) \left( 1 +\int\limits_0^t u(t-\tau)d\tau \right)
\end{equation*}
is the solution to the Volterra integral equation of the first kind
$\textsl{K}_{P}[y](t)=(\xi-1)t$ (see, e.g., Equation~16, p.~114
of \cite{book:Polyanin})
and for $\lambda_0=2\xi$ satisfies
our optimality condition \eqref{eq:EL:ISO1}:
\begin{equation}
\label{eq:noc:ex2}
\textsl{K}_{P^*}\left[2\left(\textsl{K}_P[y](\tau)+\tau\right)
-2\xi \tau\right](t)=0.
\end{equation}
The solution of \eqref{eq:noc:ex2} subject to the given boundary conditions
depends on the particular choice for the kernel. For example, let
$h^{\alpha}(t-\tau)=e^{\alpha(t-\tau)}$. Then the solution of \eqref{eq:noc:ex2}
subject to the boundary conditions $y(0)=\xi-1$ and $y(1)=(\xi-1)(1-\alpha)$
is $y(t)=(\xi-1)(1-\alpha t)$ (cf. p.~15 of \cite{book:Polyanin}).
If $h^{\alpha}(t-\tau)=\cos\left(\alpha(t-\tau)\right)$, then the boundary
conditions are $y(0)=\xi-1$ and $y(1)=(\xi-1)\left(1+\alpha^2/2\right)$,
and the extremal is $y(t)=(\xi-1)\left(1+\alpha^2 t^2/2\right)$
(cf. p.~46 of \cite{book:Polyanin}).
\end{example}

Borrowing different kernels from the book \cite{book:Polyanin},
many other examples of dynamic optimization problems can be
explicitly solved by application of the results of this section.

As particular cases of our problem \eqref{eq:F:1}, \eqref{eq:F:3},
one obtains previously studied fractional isoperimetric problems
with Caputo derivatives.

\begin{corollary}[cf. \cite{MyID:207}]
\label{cor:EL:ISOF}
\index{Euler--Lagrange equation!for problems with Caputo fractional derivatives}
Let $\y\in C^1([a,b];\R)$ be a minimizer to the functional
\begin{equation*}
\mathcal{I}(y)=\int\limits_a^b F\left(y(t),\dot{y}(t),\Dcl [y](t),t\right) dt
\end{equation*}
subject to an isoperimetric constraint of the form
\begin{equation*}
\mathcal{J}(y)=\int\limits_a^b G\left(y(t),\dot{y}(t),\Dcl [y](t),t\right)dt=\xi,
\end{equation*}
and boundary conditions
\begin{equation*}
y(a)=y_a,~~y(b)=y_b,
\end{equation*}
where $0<\alpha<\frac{1}{q}$, and functions $F$, $G$ are such that
\begin{itemize}
\item $F,G\in C^1(\R^3\times [a,b];\R)$,
\item $t\mapsto\partial_2 F\left(y(t),\dot{y}(t),\Dcl [y](t),t\right)$,
$t\mapsto\partial_2 G\left(y(t),\dot{y}(t),\Dcl [y](t),t\right)$,\newline
$\Irc\left[\partial_3 F\left(y(\tau),\dot{y}(\tau),\Dclt [y](\tau),\tau\right)\right]$ and
$\Irc\left[\partial_3 G\left(y(\tau),\dot{y}(\tau),\Dclt [y](\tau),\tau\right)\right]$
are continuously differentiable on $[a,b]$.
\end{itemize}
If $\y$ is  such that
\begin{multline}\label{eq:extr}
\partial_1 G\left(\y(t),\dot{\y}(t),\Dcl [\y](t),t\right)-\frac{d}{dt}\left(\partial_2
G\left(\y(t),\dot{\y}(t),\Dcl [\y](t),t\right)\right)\\
+\Dr\left[
\partial_3 G\left(\y(\tau),\dot{\y}(\tau),\Dclt [\y](\tau),\tau\right)\right](t)\neq 0,
\end{multline}
then there exists a constant $\lambda_0$ such that $\y$ satisfies
\begin{multline}\label{eq:EL:ISOF}
\partial_1 H\left(\y(t),\dot{\y}(t),\Dcl [\y](t),t\right)-\frac{d}{dt}\left(\partial_2
H\left(\y(t),\dot{\y}(t),\Dcl [\y](t),t\right)\right)\\
+\Dr\left[
\partial_3 H\left(\y(\tau),\dot{\y}(\tau),\Dclt [\y](\tau),\tau\right)\right](t) = 0.
\end{multline}
with $H=F-\lambda_0 G$.
\end{corollary}

\begin{proof}
The result follows from Theorem~~\ref{thm:EL:ISO1}.
\end{proof}

\begin{example}
Let $\alpha\in\left(0,\frac{1}{q}\right)$
and $\xi\in\mathbb{R}$. Consider the following
fractional isoperimetric problem:
\begin{equation}
\label{eq:ex}
\begin{gathered}
\mathcal{I}(y)=\int\limits_0^1\left(\dot{y}(t)
+ \, {^C_{0}D_t^{\alpha}} [y](t)(t)\right)^2dt \longrightarrow \min\\
\mathcal{J}(y)=\int\limits_0^1\left(\dot{y}(t)
+ \, {^C_{0}D_t^{\alpha}} [y](t)\right)dt = \xi\\
y(0)=0,~~
y(1)=\int\limits_0^1 E_{1-\alpha}\left(-\left(1-\tau\right)^{1-\alpha}\right)\xi d\tau.
\end{gathered}
\end{equation}
In our example \eqref{eq:ex} the function $H$
of Corollary~\ref{cor:EL:ISOF} is given by
$$
H(y(t),\dot{y}(t), {^C_{0}D_t^{\alpha}} [y](t),t)
=\left(\dot{y}(t)+{^C_{0}D_t^{\alpha}} [y](t)\right)^2 -\lambda_0 \left(\dot{y}(t)
+ {^C_{0}D_t^{\alpha}}[y](t)\right).
$$
One can easily check that function
\begin{equation}
\label{eq:y:ex} y(t)=\int_0^t
E_{1-\alpha}\left(-\left(t-\tau\right)^{1-\alpha}\right) \xi d\tau
\end{equation}
\begin{itemize}
\item is such that \eqref{eq:extr} holds;
\item satisfies $\dot{y}(t)+ {^C_{0}D_t^{\alpha}}[y](t)= \xi$
(see, e.g., p.~328, Example~5.24 \cite{book:Kilbas}).
\end{itemize}
Moreover, \eqref{eq:y:ex} satisfies the Euler--Lagrange equation
\eqref{eq:EL:ISOF} for $\lambda_0=2\xi$, i.e.,
\begin{equation*}
-\frac{d}{dt}\left(2\left(\dot{y}(t) + {^C_{0}D_t^{\alpha}}
[y](t)\right) -2\xi\right) + {_{t}D_1^{\alpha}}\left[
\left(2\left(\dot{y}(\tau) + {^C_{0}D_{\tau}^{\alpha}}
[y](\tau)\right)-2\xi\right)\right]=0
\end{equation*}
We conclude that \eqref{eq:y:ex} is an extremal for the
isoperimetric problem \eqref{eq:ex}.

Let us consider two cases.
\begin{enumerate}
\item Choose $\xi = 1$. When $\alpha \rightarrow 0$
one gets from \eqref{eq:ex} the classical isoperimetric problem
\begin{equation}
\label{eq:ex:alpha0}
\begin{gathered}
\mathcal{I}(y) =\int\limits_0^1\left(\dot{y}(t)
+ y(t)\right)^2 dt \longrightarrow \min\\
\mathcal{J}(y)
=\int\limits_0^1 y(t) dt = \frac{1}{e}\\
y(0)=0 \quad y(1)= 1-\frac{1}{\mathrm{e}}.
\end{gathered}
\end{equation}
Our extremal \eqref{eq:y:ex} is then reduced to the classical
extremal $y(t)=1 - \mathrm{e}^{-t}$ of the isoperimetric problem
\eqref{eq:ex:alpha0}.

\item Let $\alpha=\frac{1}{2}$.
Then \eqref{eq:ex} gives the following
fractional isoperimetric problem:
\begin{equation}
\label{eq:ex:alpha=1/2}
\begin{gathered}
\mathcal{I}(y)=\int\limits_0^1\left(\dot{y}(t)
+  {^C_{0}\textsl{D}_t^{\frac{1}{2}}} [y](t)\right)^2 dt\longrightarrow \min\\
\mathcal{J}(y)=\int\limits_0^1\left(\dot{y}(t)
+ {^C_{0}\textsl{D}_t^{\frac{1}{2}}} [y](t)\right)dt=\xi \\
y(0) =0\, , \quad y(1) = \xi\left(
\mathrm{erfc}(1)\mathrm{e}+\frac{2}{\sqrt{\pi}}-1\right),
\end{gathered}
\end{equation}
where $\mathrm{erfc}$ is the complementary error function defined by
\begin{equation*}
\mathrm{erfc}(z)=\frac{2}{\sqrt{\pi}}\int\limits_{z}^{\infty}exp(-t^2)dt.
\end{equation*}
The extremal \eqref{eq:y:ex} for the particular fractional
isoperimetric problem \eqref{eq:ex:alpha=1/2} is
$$
y(t)=\xi\left(\mathrm{e}^t \mathrm{erfc}(\sqrt{t})
-\frac{2\sqrt{t}}{\sqrt{\pi}}-1\right).
$$
\end{enumerate}
\end{example}

\begin{corollary}
\index{Euler--Lagrange equation!for problems with variable order fractional integrals and derivatives}
Let us consider the problem of minimizing functional \eqref{eq:31} subject to an isoperimetric constraint
\begin{equation}
\label{eq:vo:iso}
\mathcal{J}(y)=\int\limits_a^b
G\left(y(t),{_{a}}\textsl{I}^{1-\alpha(\cdot,\cdot)}_{t}[y](t),\dot{y}(t),
{^{C}_{a}}\textsl{D}^{\alpha(\cdot,\cdot)}_{t}[y](t),
t\right)dt=\xi
\end{equation}
and the boundary conditions
\begin{equation}
\label{eq:vo:iso:bd}
y(a)=y_a, \quad y(b)=y_b,
\end{equation}
where $\dot{y},{_{a}}\textsl{I}^{1-\alpha(\cdot,\cdot)}_{t}[y],
{^{C}_{a}}\textsl{D}^{\alpha(\cdot,\cdot)}_{t}[y]\in C([a,b];\R)$
and $\alpha:\Delta\rightarrow [0,1-\delta]$ with $\delta>\frac{1}{p}$.
Moreover, we assume that
\begin{itemize}
\item $G\in C^1(\R^4\times [a,b];\R)$,
\item ${_{t}}\textsl{I}^{1-\alpha(\cdot,\cdot)}_{b}\left[\partial_2 G\left(y(\tau),
{_{a}}\textsl{I}^{1-\alpha(\cdot,\cdot)}_{\tau}[y](\tau),\dot{y}(\tau),
{^{C}_{a}}\textsl{D}^{\alpha(\cdot,\cdot)}_{\tau}[y](\tau),
\tau\right)\right]$ is continuous on $[a,b]$,
\item $t\mapsto\partial_3 G \left(y(t),{_{a}}\textsl{I}^{1-\alpha(\cdot,\cdot)}_{t}[y](t),
\dot{y}(t),{^{C}_{a}}\textsl{D}^{\alpha(\cdot,\cdot)}_{t}[y](t),t\right)$ \newline
and $,{_{t}}\textsl{I}^{1-\alpha(\cdot,\cdot)}_{b}\left[\partial_4 G\left(y(\tau),
{_{a}}\textsl{I}^{1-\alpha(\cdot,\cdot)}_{\tau}[y](\tau),\dot{y}(\tau),
{^{C}_{a}}\textsl{D}^{\alpha(\cdot,\cdot)}_{\tau}[y](\tau),
\tau\right)\right]$ are continuously differentiable on $[a,b]$.
\end{itemize}

If $\bar{y}\in C^1([a,b];\R)$ is a solution
to problem \eqref{eq:31}, \eqref{eq:vo:iso},
\eqref{eq:vo:iso:bd}, then there exists real number
$\lambda_0$ such that, for $H=F-\lambda_0 G$, we have
\begin{equation*}
\partial_1 H
-\frac{d}{dt}\partial_3 H
+{_{t}}\textsl{I}^{1-\alpha(t,\cdot)}_{b}[\partial_2 H]
+{_{t}}\textsl{D}^{\alpha(t,\cdot)}_{b}[\partial_4 H]=0,
\end{equation*}
provided that
$$
\partial_1 G
-\frac{d}{dt}\partial_3 G
+{_{t}}\textsl{I}^{1-\alpha(t,\cdot)}_{b}[\partial_2 G]
+{_{t}}\textsl{D}^{\alpha(t,\cdot)}_{b}[\partial_4 G]\neq 0.
$$
Here, functions $\partial_i H$ and $\partial_i G$, $i=1,2,3,4$,
are evaluated at $\left(\y(t),{_{a}}\textsl{I}^{1-\alpha(\cdot,\cdot)}_{t}[\y](t),
\dot{\y},{^{C}_{a}}\textsl{D}^{\alpha(\cdot,\cdot)}_{t}[\y](t), t\right)$.
\end{corollary}

\begin{proof}
The result follows from Theorem~~\ref{thm:EL:ISO1}.
\end{proof}

Theorem~\ref{thm:EL:ISO1} can be easily extended to $r$ subsidiary conditions of integral type.
Let $G_k$, $k=1,\dots, r$, be Lagrangians of class $C^1$:
\begin{equation*}
\fonction{G_k}{\R^4 \times [a,b]}{\R}{(x_1,x_2,x_3,x_4,t)}{G_k(x_1,x_2,x_3,x_4,t).}
\end{equation*}
and let
\begin{equation}
\label{eq:F:4}
\fonction{\mathcal{J}_k}{\mathcal{A}(y_a,y_b)}{\R}{y}{\int\limits_a^b
G_k(y(t),K_P[y](t),\dot{y}(t),B_P[y](t),t) \; dt ,}
\end{equation}
where $\dot{y}$ denotes the classical derivative of $y$, $K_P$ is generalized
fractional integral operator with the kernel belonging to $L^q(\Delta;\R)$
and $B_P=K_P\circ\frac{d}{dt}$. Moreover, we assume that
\begin{itemize}
\item $K_{P^*}\left[\partial_2 G_k (y(\tau),K_P[y](\tau),\dot{y}(\tau),B_P[y](\tau),\tau)\right]\in C([a,b];\R)$,
\item $t\mapsto\partial_3 G_k (y(t),K_P[y](t),\dot{y}(t),B_P[y](t),t)\in C^1([a,b];\R)$,
\item $K_{P^*}\left[\partial_4 G_k (y(\tau),K_P[y](\tau),\dot{y}(\tau),B_P[y](\tau),\tau)\right]\in C^1([a,b];\R)$.
\end{itemize}

Suppose that $\xi=(\xi_1,\dots,\xi_r)$ and define
$$
\mathcal{A}_{\xi}(y_a,y_b):=\left\{y\in\mathcal{A}(y_a,y_b):\mathcal{J}_k[y]=\xi_k,\; k=1\dots,r\right\}.
$$
Next theorem gives necessary optimality condition for a minimizer of functional
$\mathcal{I}$ subject to $r$ isoperimetric constraints.
\begin{theorem}
\index{Euler--Lagrange equation!for problems with generalized fractional operators}
Let $\y$ be a minimizer of $\mathcal{I}$ in the class $\mathcal{A}_{\xi}(y_a,y_b)$.
If one can find functions $\eta_1,\dots,\eta_r\in\mathcal{A}(0,0)$
such that the matrix $A=\left(a_{kl}\right)$, with
\begin{multline*}
a_{kl}:=\int\limits_a^b\; \left(\partial_1 G_k(\star_{\bar{y}})(t)
+K_{P^*}\left[\partial_2 G_k(\star_{\bar{y}})(\tau)\right](t)\right)\cdot\eta_l(\tau)\\
+\left(\partial_3 G_k(\star_{\bar{y}})(t)+K_{P^*}\left[\partial_4
G_k(\star_{\bar{y}})(\tau)\right](t)\right)\cdot\dot{\eta_l}(t)\; dt,
\end{multline*}
has rank equal to $r$, then there exist $\lambda_1,\dots,\lambda_r\in\R$ such that,
for $H=F-\sum\limits_{k=1}^r\lambda_kG_k$, minimizer $\y$ satisfies
\begin{equation}
\label{eq:EL:ISO2}
\frac{d}{dt}\left[\partial_3 H(\star_{\bar{y}})(t)\right]+A_{P^*}\left[\partial_4
H(\star_{\bar{y}})(\tau)\right](t)=\partial_1 H(\star_{\bar{y}})(t)
+K_{P^*}\left[\partial_2 H(\star_{\bar{y}})(\tau)\right](t),~t\in (a,b),
\end{equation}
where $(\star_{\bar{y}})(t)=(\bar{y}(t),K_P[\bar{y}](t),\dot{\bar{y}}(t),B_P[\bar{y}](t),t)$.
\end{theorem}

\begin{proof}
Let us define
\begin{equation*}
\fonction{\phi}{[-\varepsilon_0,\varepsilon_0]\times[-\varepsilon_1,\varepsilon_1]\times\dots
\times [-\varepsilon_r,\varepsilon_r]}{\R}{(h_0,h_1,\dots,h_r)}{\mathcal{I}(\bar{y}
+h_0\eta_0+h_1\eta_1+\dots+h_r\eta_r)}
\end{equation*}
and
\begin{equation*}
\fonction{\psi_k}{[-\varepsilon_0,\varepsilon_0]\times[-\varepsilon_1,\varepsilon_1]
\times\dots\times [-\varepsilon_r,\varepsilon_r]}{\R}{(h_0,h_1,\dots,h_r)}{\mathcal{J}_k(\bar{y}
+h_0\eta_0+h_1\eta_1+\dots+h_r\eta_r)-\xi_k}
\end{equation*}
Observe that $\phi,\psi_k$ are functions of class $C^1\left([-\varepsilon_0.\varepsilon_0]
\times\dots\times[-\varepsilon_r,\varepsilon_r];\R\right)$,
$A=\left(\left.\frac{\partial\psi_k}{\partial h_l}\right|_{0}\right)$ and that
$\psi_k (0,0,\dots,0)=0$. Moreover, because $\y$ is a minimizer of functional $\mathcal{I}$ we have
$$
\phi(0,\dots,0)\leq\phi(h_0,h_1,\dots,h_r).
$$
From the classical multiplier theorem, there are $\lambda_1,\dots,\lambda_r\in\R$ such that
\begin{equation}\label{eq:ISOpf:1}
\nabla\phi_l(0,\dots,0)+\sum\limits_{k=1}^r \lambda_k\nabla\psi_{k}(0,\dots,0)=0,
\end{equation}
From \eqref{eq:ISOpf:1}, we can compute $\lambda_1,\dots,\lambda_r$, turning out to be independent
of the choice of $\eta_0\in\mathcal{A}(0,0)$. Finally, by the fundamental
lemma of the calculus of variations, we arrive to \eqref{eq:EL:ISO2}.
\end{proof}


\section{Noether's Theorem}
\label{sec:NTH:sing}

Emmy Noether's classical work \cite{Noether} from 1918 states that a conservation law
in variational mechanics follow whenever the Lagrangian function is invariant under
a one-parameter continuous group of transformations, that transform dependent and/or independent variables.
This result not only unifies conservation laws but also suggests a way to discover new ones.
In this section we consider variational problems that depend on generalized fractional integrals and derivatives.
Following the methods used in \cite{gastao2,gastao,book:Jost,Cresson}we apply Euler--Lagrange equations
to formulate a generalized fractional version of Noether's theorem without transformation of time.
We start by introducing the notion generalized fractional extremal and one-parameter family of infinitesimal transformations.

\begin{definition}
The function $y\in C^1\left([a,b];\R\right)$, with $K_P[y],B_P[y]\in C\left([a,b];\R\right)$,
that is a solution to equation \eqref{eq:EL:OCM} is said to be a generalized fractional extremal.
\end{definition}

We consider a one-parameter family of transformations of the form
$\hat{y}(t)=\phi(\theta,t,y(t))$, where $\phi$ is a map of class $C^2$:
\begin{equation*}
\fonction{\phi}{[-\varepsilon,\varepsilon]
\times [a,b]\times\R}{\R}{(\theta,t,x)}{\phi(\theta,t,x),}
\end{equation*}
such that $\phi(0,t,x)=x$.
Note that, using Taylor's expansion of $\phi(\theta,t,y(t))$ in a neighborhood of $0$ one has
\begin{equation*}
\hat{y}(t)=\phi(0,t,y(t))+\theta\frac{\partial}{\partial\theta}\phi(0,t,y(t))+o(\theta),
\end{equation*}
provided that $\theta\in [-\varepsilon,\varepsilon]$. Moreover, having in mind that
$\phi(0,t,y(t))=y(t)$ and denoting $\xi(t,y(t))=\frac{\partial}{\partial\theta}\phi(0,t,y(t))$, one has
\begin{equation}
\label{eq:Tr}
\hat{y}(t)=y(t)+\theta\xi(t,y(t))+o(\theta),
\end{equation}
so that, for $\theta\in [-\varepsilon,\varepsilon]$
the linear approximation to the transformation is
\begin{equation*}
\hat{y}(t)\approx y(t)+\theta\xi(t,y(t)).
\end{equation*}

Now, let us introduce the notion of invariance.
\begin{definition}
\label{def:IF:1}
\index{Invariant Lagrangian!for with generalized fractional operators}
We say that the Lagrangian $F$ is invariant under the one--parameter family
of infinitesimal transformations \eqref{eq:Tr}, where $\xi$ is such that
$t\mapsto\xi(t,y(t))\in C^1\left([a,b];\R\right)$ with $K_P\left[\tau\mapsto
\xi(\tau,y(\tau))\right],B_P\left[\tau\mapsto\xi(\tau,y(\tau))\right]
\in C\left([a,b];\R\right)$ if
\begin{equation}
\label{eq:CI:1}
F\left(y(t),K_P[y](t),\dot{y}(t),B_P[y](t),t\right)
= F\left(\hat{y}(t),K_P[\hat{y}](t),\dot{\hat{y}}(t),B_P[\hat{y}](t),t\right),
\end{equation}
for all $\theta\in [-\varepsilon,\varepsilon]$, and  all
$y\in C^1\left([a,b];\R\right)$ with $K_P[y],B_P[y]\in C\left([a,b];\R\right)$.
\end{definition}

In order to state Noether's theorem in a compact form,
we introduce the following two bilinear operators:
\begin{equation}
\label{eq:BD:1}
\mathbf{D}[f,g]:=f\cdot A_{P^*}[g]+g\cdot B_P[f],
\end{equation}
\begin{equation}
\label{eq:BI:1}
\mathbf{I}[f,g]:=-f\cdot K_{P^*}[g]+g\cdot K_P[f].
\end{equation}

\begin{theorem}[Generalized Fractional Noether's Theorem]
\label{thm:NTH:1}
\index{Noether's Theorem!for generalized fractional operators}
Let $F$ be invariant under the one parameter family of infinitesimal transformations \eqref{eq:Tr}.
Then for every generalized fractional extremal the following equality holds
\begin{equation}\label{eq:NTH:1}
\frac{d}{dt}\left(\xi(t,y(t))\cdot\partial_3 F(\star_y)(t)\right)
+\mathbf{D}\left[\xi(t,y(t)),\partial_4 F(\star_y)(t)\right]
+\mathbf{I}\left[\xi(t,y(t)),\partial_2 F(\star_y)(t)\right]=0,~~t\in (a,b),
\end{equation}
where $(\star_y)(t)=(y(t),K_P[y](t),\dot{y}(t),B_P[y](t),t)$.
\end{theorem}

\begin{proof}
By equation \eqref{eq:CI:1} one has
\begin{equation*}
\left.\frac{d}{d\theta}\left[F\left(\hat{y}(t),K_P[\hat{y}](t),
\dot{\hat{y}}(t),B_P[\hat{y}](t),t\right)\right]\right|_{\theta=0}=0
\end{equation*}
The usual chain rule implies
\begin{multline*}
\Biggl.\partial_1 F(\star_{\hat{y}})(t)\cdot\frac{d}{d \theta}\hat{y}(t)
+\partial_2 F(\star_{\hat{y}})(t)\cdot\frac{d}{d \theta}K_P[\hat{y}](t)\\
+\partial_3 F(\star_{\hat{y}})(t)\cdot\frac{d}{d\theta}\dot{\hat{y}}(t)
+\partial_4 F(\star_{\hat{y}})(t)\cdot\frac{d}{d\theta}B_P[\hat{y}](t)\Biggr|_{\theta=0}=0.
\end{multline*}
By linearity of $K_P$, $B_P$ differentiating with respect to $\theta$,  and putting $\theta=0$ we have
\begin{multline*}
\partial_1 F(\star_y)(t)\cdot\xi(t,y(t))+\partial_2
F(\star_y)(t)\cdot K_P[\tau\mapsto\xi(\tau,y(\tau))](t)\\
+\partial_3 F(\star_y)(t)\cdot\frac{d}{dt}\xi(t,y(t))
+\partial_4 F(\star_y)(t)\cdot B_P[\tau\mapsto\xi(\tau,y(\tau))](t)=0.
\end{multline*}
Now, using generalized Euler--Lagrange equation \eqref{eq:EL:OCM} and formulas
\eqref{eq:BD:1} and \eqref{eq:BI:1} one arrives to \eqref{eq:NTH:1}.
\end{proof}

\begin{example}
Let $P=\langle a,t,b,\lambda,\mu\rangle$ and $y\in C^1\left([a,b];\R\right)$
with $B_P[y]\in C\left([a,b];\R\right)$.
Consider Lagrangian $F\left(B_P[y](t),t\right)$ and transformations
\begin{equation}
\label{eq:Tr:2}
\hat{y}(t)=y(t)+\varepsilon c+o(\varepsilon),
\end{equation}
where $c$ is a constant. Then,  we have
\begin{equation*}
F\left(B_P[y](t),t\right)=
F\left(B_P[\hat{y}](t),t\right)
\end{equation*}
Therefore, $F$ is invariant under
\eqref{eq:Tr:2} and the generalized fractional Noether's theorem indicates that
\begin{equation}
\label{eq:ex:NTH}
A_{P^*}[\partial_1 F\left(B_P[y](\tau),\tau\right)](t)=0,~t\in (a,b),
\end{equation}
along any generalized fractional extremal $y$. Notice that equation \eqref{eq:ex:NTH} can be written in the form
\begin{equation}
\label{eq:ex:NTH:2}
\frac{d}{dt}\left(K_{P^*}[\partial_1 F\left(B_P[y](\tau),\tau\right)](t)\right)=0,
\end{equation}
that is, quantity $K_{P^*}[\partial_1 F\left(B_P[y](\tau),\tau\right)]$ is conserved along
all generalized fractional extremals and this quantity, following the classical approach,
can be called a generalized fractional constant of motion.
\end{example}

Similarly to previous sections, one can obtain from Theorem~\ref{thm:NTH:1} results regarding
to constant and variable order fractional integrals and derivatives.

\begin{corollary}
\label{cor:NTH:1}
\index{Noether's Theorem!for fractional Caputo derivatives}
If for any $y\in C^1([a,b];\R)$ the following equality is satisfied
$$F\left(y(t),\dot{y}(t),
\lambda \, _{a}^{C}\textsl{D}_t^\alpha [y](t)
+\mu \, _{t}^{C}\textsl{D}_b^\alpha [y](t),t\right)=F\left(\hat{y}(t),\dot{\hat{y}}(t),
\lambda \, _{a}^{C}\textsl{D}_t^\alpha [\hat{y}](t)
+\mu \, _{t}^{C}\textsl{D}_b^\alpha [\hat{y}](t),t\right),
$$
where $\hat{y}$ is the family \eqref{eq:Tr}, then we have
\begin{multline*}
\frac{d}{dt}\left(\xi(t,y(t))\cdot \partial_2 F\left(y(t),\dot{y}(t),\lambda \,
_{a}^{C}\textsl{D}_t^\alpha [y](t)+\mu \, _{t}^{C}\textsl{D}_b^\alpha [y](t),t\right)\right)\\
-\xi(t,y(t))\cdot\left(\lambda\Dr\left[\partial_3 F\left(y(\tau),\dot{y}(\tau),
\lambda \, _{a}^{C}\textsl{D}_{\tau}^\alpha [y](\tau)+\mu \,
_{\tau}^{C}\textsl{D}_b^\alpha [y](\tau),\tau\right)\right](t)\right.\\
\left.\mu\Dl\left[\partial_3 F\left(y(\tau),\dot{y}(\tau),\lambda \,
_{a}^{C}\textsl{D}_{\tau}^\alpha [y](\tau)+\mu \,
_{\tau}^{C}\textsl{D}_b^\alpha [y](\tau),\tau\right)\right](t)\right)\\
+\partial_3 F\left(y(t),\dot{y}(t),\lambda \, _{a}^{C}\textsl{D}_t^\alpha [y](t)
+\mu \, _{t}^{C}\textsl{D}_b^\alpha [y](t),t\right)\cdot\left(\lambda\Dcl[\xi(\tau,y(\tau))](t)
+\mu\Dcr[\xi(\tau,y(\tau))](t)\right)=0
\end{multline*}
along all solutions of the Euler--Lagrange equation \eqref{eq:35}.
\end{corollary}

\begin{proof}
Corollary~\ref{cor:NTH:1} is a simple consequence of Theorem~\ref{thm:NTH:1}.
\end{proof}

\begin{corollary}
\label{cor:NTH:2}
\index{Noether's Theorem!for variable order fractional integrals and derivatives}
Let $y\in C^1([a,b];\R)$ with $\textsl{I}^{1-\alpha(\cdot,\cdot)}_{t} [y],
{^{C}_{a}}\textsl{D}^{\alpha(\cdot,\cdot)}_{t} [y]\in C([a,b];\R)$ and suppose that
\begin{equation*}
F\left(y(t),{_{a}}\textsl{I}^{1-\alpha(\cdot,\cdot)}_{t}[y](t),
\dot{y}(t),{^{C}_{a}}\textsl{D}^{\alpha(\cdot,\cdot)}_{t}[y](t),t\right)
=F\left(\hat{y}(t),{_{a}}\textsl{I}^{1-\alpha(\cdot,\cdot)}_{t}[\hat{y}](t),
\dot{\hat{y}}(t),{^{C}_{a}}\textsl{D}^{\alpha(\cdot,\cdot)}_{t}[\hat{y}](t), t\right)
\end{equation*}
where $\hat{y}$ is the family \eqref{eq:Tr} such that $t\mapsto\xi (t,y(t))\in C^1([a,b];\R)$
and $\textsl{I}^{1-\alpha(\cdot,\cdot)}_{t} [\tau\mapsto\xi(\tau, y(\tau))]$,
${^{C}_{a}}\textsl{D}^{\alpha(\cdot,\cdot)}_{t} [\tau\mapsto\xi (\tau,y(\tau))]\in C([a,b];\R)$.
Then all solutions of the Euler--Lagrange equation \eqref{eq:eqELCaputo} must satisfy
\begin{multline*}
\frac{d}{dt}\left(\xi(t,y(t))\cdot \partial_3 F\left(y(t),{_{a}}\textsl{I}^{1-\alpha(\cdot,\cdot)}_{t}[y](t),
\dot{y}(t),{^{C}_{a}}\textsl{D}^{\alpha(\cdot,\cdot)}_{t}[y](t), t\right)\right)\\
-\xi(t,y(t))\cdot{_{t}}\textsl{D}^{\alpha(\cdot,\cdot)}_{b}\left[\partial_4 F\left(y(\tau),
{_{a}}\textsl{I}^{1-\alpha(\cdot,\cdot)}_{\tau}[y](\tau),\dot{y}(\tau),
{^{C}_{a}}\textsl{D}^{\alpha(\cdot,\cdot)}_{\tau}[y](\tau), \tau\right)\right](t)\\
+\partial_4 F\left(y(t),{_{a}}\textsl{I}^{1-\alpha(\cdot,\cdot)}_{t}[y](t),\dot{y}(t),
{^{C}_{a}}\textsl{D}^{\alpha(\cdot,\cdot)}_{t}[y](t), t\right)
\cdot{^{C}_{a}}\textsl{D}^{\alpha(\cdot,\cdot)}_{t}[\xi(\tau,y(\tau))](t)\\
-\xi(t,y(t))\cdot{_{t}}\textsl{I}^{1-\alpha(\cdot,\cdot)}_{b}\left[\partial_2
F\left(y(\tau),{_{a}}\textsl{I}^{1-\alpha(\cdot,\cdot)}_{\tau}[y](\tau),
\dot{y}(\tau),{^{C}_{a}}\textsl{D}^{\alpha(\cdot,\cdot)}_{\tau}[y](\tau), \tau\right)\right](t)\\
+\partial_2 F\left(y(t),{_{a}}\textsl{I}^{1-\alpha(\cdot,\cdot)}_{t}[y](t),\dot{y}(t),
{^{C}_{a}}\textsl{D}^{\alpha(\cdot,\cdot)}_{t}[y](t), t\right)\cdot
{_{a}}\textsl{I}^{1-\alpha(\cdot,\cdot)}_{t}[\xi(\tau,y(\tau))](t)=0,
~t\in (a,b).
\end{multline*}
\end{corollary}

\begin{proof}
Corollary~\ref{cor:NTH:2} can be easily obtained from Theorem~\ref{thm:NTH:1}.
\end{proof}

\begin{corollary}
\label{cor:NTH:3}
\index{Noether's Theorem!for Riemann--Lioville fractional integrals and Caputo fractional derivatives}
Suppose that $y\in C^1([a,b];\R)$ and for family \eqref{eq:Tr} one has
\begin{equation*}
F(y(t),\Ilc [y](t),\dot{y}(t),\Dcl [y](t),t)\; dt=F(\hat{y}(t),
\Ilc [\hat{y}](t),\dot{\hat{y}}(t),\Dcl [\hat{y}](t),t).
\end{equation*}
Then
\begin{multline*}
\frac{d}{dt}\left(\xi(t,y(t))\cdot \partial_2 F(y(t),\Ilc [y](t),\dot{y}(t),\Dcl [y](t),t)\right)\\
-\xi(t,y(t))\cdot\Dr[\partial_4 F(y(\tau),\Ilct [y](\tau),\dot{y}(\tau),\Dclt [y](\tau),\tau)]\\
+\partial_4 F(y(t),\Ilc [y](t),\dot{y}(t),\Dcl [y](t),t)\cdot\Dcl[\xi(\tau,y(\tau))]\\
-\xi(t,y(t))\cdot\Irc[\partial_2 F(y(\tau),\Ilct [y](\tau),\dot{y}(\tau),\Dclt [y](\tau),\tau)]\\
+\partial_2 F(y(t),\Ilc [y](t),\dot{y}(t),\Dcl [y](t),t)\cdot\Ilc[\xi(\tau,y(\tau))]=0
\end{multline*}
along any solution of the Euler--Lagrange equation \eqref{eq:EL:cor}.
\end{corollary}

\begin{proof}
Corollary~\ref{cor:NTH:3} can be easily obtained from Theorem~\ref{thm:NTH:1}.
\end{proof}


\section{Variational Calculus in Terms of a Generalized Integral}
\label{sec:fpT}

In this section, we develop a generalized fractional calculus of variations, by considering
very general problems, where the classical integrals are substituted by generalized fractional
integrals, and the Lagrangians depend not only on classical derivatives
but also on generalized fractional operators.
By choosing particular operators and kernels, one obtains
the recent results available in the literature of mathematical physics
\cite{Nabulsi3,Nabulsi,Herrera}.

Let $R=\langle a,b,b,1,0\rangle$, $P=\langle a,t,b,\lambda,\mu\rangle$ and consider
the problem of finding a function $\y$ that gives minimum value to the functional
\begin{equation}
\label{eq:F:5}
\fonction{\mathcal{I}}{\mathcal{A}(y_a,y_b)}{\R}{y}{K_R\left[F(y(t),
K_P[y](t),\dot{y}(t),B_P[y](t),t)\right](b) ,}
\end{equation}
where $K_R$ and $K_P$ are generalized fractional integrals with kernels $k(x,t)$
and $h(t,\tau)$, respectively, being elements of $L^q(\Delta;\R)$,
$B_P=K_P\circ\frac{d}{dt}$ and $F$ is a Lagrangian of class $C^1$:
\begin{equation*}
\fonction{F}{\R^4 \times [a,b]}{\R}{(x_1,x_2,x_3,x_4,t)}{F(x_1,x_2,x_3,x_4,t).}
\end{equation*}
Moreover, we assume that
\begin{itemize}
\item $t\mapsto k(b,t)\cdot\partial_1 F(y(t),K_P[y](t),
\dot{y}(t),B_P[y](t),t)\in C([a,b];\R)$,
\item $K_{P^*}\left[k(b,\tau)\partial_2
F(y(\tau),K_P[y](\tau),\dot{y}(\tau),B_P[y](\tau),\tau)\right]\in C([a,b];\R)$,
\item $t\mapsto k(b,t)\cdot\partial_3
F(y(t),K_P[y](t),\dot{y}(t),B_P[y](t),t)\in C^1([a,b];\R)$,
\item $K_{P^*}\left[k(b,\tau)\partial_4
F(y(\tau),K_P[y](\tau),\dot{y}(\tau),B_P[y](\tau),\tau)\right]\in C^1([a,b];\R)$.
\end{itemize}

\begin{theorem}
\label{thm:EL:TER}
\index{Euler--Lagrange equation!for problems with generalized fractional operators}
If $\y\in\mathcal{A}(y_a,y_b)$ is a minimizer of functional \eqref{eq:F:5},
then $\y$ satisfies the generalized Euler--Lagrange equation
\begin{multline}
\label{eq:EL:TER}
k(b,t)\cdot\partial_1 F(\star_{\y})(t)-\frac{d}{dt}\left(\partial_3 F(\star_{\y})(t)\cdot k(b,t)\right)\\
-A_{P^*}\left[k(b,\tau)\cdot\partial_4 F(\star_{\y})(\tau)\right](t)+K_{P^*}\left[
k(b,\tau)\cdot\partial_2 F(\star_{\y})(\tau)\right](t)=0,~~t\in(a,b),
\end{multline}
where $(\star_{\y})=(\y(t),K_P[\y](t),\dot{\y}(t),B_P[\y](t),t)$.
\end{theorem}

\begin{proof}
One can prove Theorem~\ref{thm:EL:TER} in the similar way as Theorem~\ref{thm:El:OCM}.
\end{proof}

\begin{example}\label{ex:1}
Let
$R=\langle0,1,1,1,0\rangle$, and $P=\langle 0,t,1,1,0\rangle$.
Consider the following problem:
\begin{equation*}
\begin{gathered}
\mathcal{J}(y)=
K_{R}\left[tK_{P}[y](t)
+\sqrt{1-\left( K_{P}[y](t)\right)^2}\right](1)
\longrightarrow \min,\\
y(0)=1 \, , \ y(1)=\frac{\sqrt{2}}{4}+\int\limits_0^1 u(1-\tau)
\frac{1}{\left(1+\tau^2\right)^\frac{3}{2}}d\tau,
\end{gathered}
\end{equation*}
with kernel $h$ such that $h(t,\tau)=l(t-\tau)$, $l\in C^1([0,1];\R)$ and $l(0)=1$.
Here, the resolvent $u(t)$ is related to the kernel $l(t)$
by $u(t)=\mathcal{L}^{-1}\left[\frac{1}{s\widetilde{l}(s)}-1\right](t)$,
$\widetilde{l}(s)=\mathcal{L}\left[l(t)\right](s)$, where $\mathcal{L}$
and $\mathcal{L}^{-1}$ are the direct and the inverse Laplace operators, respectively.
We apply Theorem~\ref{thm:EL:TER} with Lagrangian $F$ given by
$F(x_1,x_2,x_3,x_4,t) =tx_2+\sqrt{1-x_2^2}$. Because
\begin{equation*}
y(t) = \frac{1}{\left(1+t^2\right)^\frac{3}{2}}
+\int\limits_0^t u(t-\tau)\frac{1}{\left(1+\tau^2\right)^\frac{3}{2}}d\tau
\end{equation*}
is the solution to the Volterra integral equation of first kind
(see, e.g., Equation~16, p.~114 of \cite{book:Polyanin})
\begin{equation*}
\textsl{K}_{P} [y](t)=\frac{t\sqrt{1+t^2}}{1+t^2},
\end{equation*}
it satisfies our generalized Euler--Lagrange equation
\eqref{eq:EL:TER}, that is,
\begin{equation*}
\textsl{K}_{P^*}\left[k(1,\tau)\left(
\frac{-\textsl{K}_{P^*}[y](\tau)}{\sqrt{1
-\left(\textsl{K}_{P^*}[y](\tau)\right)^2}}
+\tau\right)\right](t)=0.
\end{equation*}
In particular, for the kernel
$l^{\beta}(t-\tau)=\cosh(\beta(t-\tau))$, the boundary conditions are
$y(0)=1$ and $y(1)=1+\beta^2(1-\sqrt{2})$, and the solution is
$y(t)=\frac{1}{(1+t^2)^\frac{3}{2}}+\beta^2\left(1-\sqrt{1+t^2}\right)$
(cf. p.~22 in \cite{book:Polyanin}).
\end{example}

The following corollary gives an extension of the main result of \cite{jmp}.

\begin{corollary}
\label{cor:EL:TER:1}
If $\y\in C^1\left([a,b];\R\right)$ is a solution to the problem of minimizing
\begin{equation*}
\mathcal{J}(y)={_{a}}\textsl{I}^{\alpha}_{b}\left[F\left(y(t),\dot{y}(t),
\, {^C_aD_t^\beta} \left[y\right](t),t\right)\right](b),
\end{equation*}
subject to the boundary
conditions\index{Euler--Lagrange equation!for problems with Caputo fractional derivatives}
\begin{equation*}
y(a)=y_a, \quad y(b)=y_b,
\end{equation*}
where $\alpha,\beta\in(0,\frac{1}{q})$, $F\in C^2\left(\mathbb{R}^3\times[a,b];\mathbb{R}\right)$, then
\begin{multline*}
\partial_1 F \left(\y(t),\dot{\y}(t), \, {^C_aD_t^\beta} \left[\y\right](t),t\right)\cdot (b-t)^{\alpha-1}\\
-\frac{d}{dt}\left(\partial_2 F \left(\y(t),\dot{\y}(t),
\, {^C_aD_t^\beta} \left[\y\right](t),t\right)\cdot (b-t)^{\alpha-1}\right)\\
+{_{t}D^{\beta}_{b}}\left[(b-\tau)^{\alpha-1} \cdot\partial_3
F\left(\y(\tau),\dot{\y}(\tau), \, {^C_aD_{\tau}^\beta} \left[\y\right](\tau),\tau\right)\right]=0,~t\in (a,b).
\end{multline*}
\end{corollary}

If the Lagrangian of functional \eqref{eq:F:5} does not depend on generalized fractional operators $B$ and $K$,
then Theorem~\ref{thm:EL:TER} gives the following result:
if $y\in C^1([a,b];\R)$ is a solution to the problem of extremizing
\begin{equation}
\label{FALVA:J}
\mathcal{I}(y)=\int_a^b F\left(y(t),\dot{y}(t),t\right) k(b,t) dt
\end{equation}
subject to $y(a)=y_a$ and $y(b)=y_b$, then
\begin{equation}
\label{eq:FALVA}
\partial_1 F\left(y(t),\dot{y}(t),t\right)-\frac{d}{dt}\partial_2
F\left(y(t),\dot{y}(t),t\right)=\frac{1}{k(b,t)}
\cdot \frac{d}{dt}k(b,t)\partial_2 F\left(y(t),\dot{y}(t),t\right).
\end{equation}
We recognize on the right hand side of \eqref{eq:FALVA}
the generalized weak dissipative parameter\index{Weak dissipative parameter}
\begin{equation*}
\delta(t)=\frac{1}{(b,t)}\cdot \frac{d}{dt}k(b,t).
\end{equation*}

Now, let us consider two examples of Lagrangians associated to functional \eqref{FALVA:J}.

\begin{example}
As a first example, let us consider kernel
$k^\alpha(b,t)=\mathrm{e}^{\alpha (b-t)}$ and the Lagrangian
\begin{equation*}
L\left(y(t),\dot{y}(t),t\right)=\frac{1}{2}m \dot{y}^2(t)-V(y(t)),
\end{equation*}
where $V(y)$ is the potential energy and $m$ stands for mass.
The Euler--Lagrange equation \eqref{eq:FALVA} gives the following
second-order ordinary differential equation:
\begin{equation}\label{mod}
\ddot{y}(t)-\alpha \dot{y}(t) = -\frac{1}{m}V'(y(t)).
\end{equation}
Equation \eqref{mod} coincides with (14) of \cite{Herrera},
obtained by modification of Hamilton's principle.
\end{example}

Next example extends some of the recent results of \cite{Nabulsi3,Nabulsi},
where the fractional action-like variational approach (FALVA)
was proposed to model dynamical systems. FALVA functionals
are particular cases of \eqref{FALVA:J}, where the fractional time integral
introduces only one parameter $\alpha$.

\begin{example}\index{Caldirola--Kanai Lagrangian}
Let us consider the Caldirola--Kanai Lagrangian \cite{Nabulsi3,Nabulsi}
\begin{equation}\label{eq:CKLagr}
L\left(y(t),\dot{y}(t),t\right)
= m(t)\left(\frac{\dot{y}^2(t)}{2}-\omega^2\frac{y^2(t)}{2}\right),
\end{equation}
which describes a dynamical oscillatory system with exponentially
increasing time-dependent mass, where $\omega$ is the frequency and
$m(t)=m_0 \mathrm{e}^{-\gamma b}\mathrm{e}^{\gamma t} = \bar{m}_0\mathrm{e}^{\gamma t}$,
$\bar{m}_0=m_0 \mathrm{e}^{-\gamma b}$. Using our generalized FALVA Euler--Lagrange
equation \eqref{eq:FALVA} with kernel $k(b,t)$
to Lagrangian \eqref{eq:CKLagr}, we obtain
\begin{equation}
\label{eq:FALVAEL}
\ddot{y}(t)+\left(\delta(t)+\gamma\right)\dot{y}(t)+\omega^2y(t)=0.
\end{equation}
\end{example}

We note that there is a small inconsistence in \cite{Nabulsi3},
regarding the coefficient of $\dot{y}(t)$ in \eqref{eq:FALVAEL}.


\section{Generalized Variational Calculus of Several Variables}
\label{sec:SEV}

Variational problems with functionals depending on several variables arise, for example,
in mechanics problems, which involve systems with infinite number of degrees of freedom,
like a vibrating elastic solid. Fractional variational problems involving multiple integrals
have been already studied in different contexts. We can mention here \cite{Cresson,Baleanu,MyID:182,tatiana}
where the mutlidimensional fractional Euler--Lagrange equations for a field were obtained,
or \cite{Agnieszka2} where first and second fractional Noether--type theorems are proved.
In this section we present a more general approach to the subject by considering
functionals depending on generalized fractional operators.


\subsection{Multidimensional Generalized Fractional Integration by Parts}

In this section, it is of our interest to obtain integration by parts formula
for generalized fractional operators. We shall denote by $t=(t_1,\dots,t_n)$
a point in $\Omega_n$, where $\Omega_n=(a_1,b_1)\times\dots\times (a_n,b_n)$,
and by $dt=dt_1\dots dt_n$. Throughout this subsection $i\in\left\{1,\dots,n\right\}$ is arbitrary but fixed.
\begin{theorem}
\index{Integration by parts formula!of several variables!for generalized fractional integrals}
\label{thm:IPRI}
Let $P_{i}=\langle a_i,t_i,b_i,\lambda_i,\mu_i\rangle$ be the parameter set and let $K_{P_i}$
be the generalized partial fractional integral with $k_i$ being a difference kernel such that
$k_i\in L^1(0,b_i-a_i;\R)$. If $f:\R^n\rightarrow\R$ and $\eta:\R^n\rightarrow\R$,
$f,\eta\in C\left(\bar{\Omega}_n;\R\right)$, then the generalized partial fractional
integrals satisfy the following identity:
\begin{equation*}
\int\limits_{\Omega_n}f(t)\cdot K_{P_i}[\eta](t)\;dt
=\int\limits_{\Omega_n} \eta(t)\cdot K_{P_i^*} [f](t)\;dt,
\end{equation*}
where $P_{i}^*$ is the dual of $P_{i}$.
\end{theorem}
\begin{proof}
Let $P_{i}=\langle a_i,t_i,b_i,\lambda_i,\mu_i \rangle$ and $f,\eta,\in C\left(\bar{\Omega}_n;\R\right)$.
Since $f$ and $\eta$, are continuous functions on $\bar{\Omega}_n$, they are bounded on
$\bar{\Omega}_n$ \textrm{i.e.}, there exist real numbers $C, D>0$ such that
$\left|f(t)\right|\leq C$, $\left|\eta(t)\right| \leq D$,
for all $t\in \bar{\Omega}_n$. Therefore,
\begin{multline*}
\int\limits_{\Omega_n}\left(\int\limits_{a_i}^{t_i} \left|\lambda_i k_i(t_i-\tau)\right|
\cdot\left|f(t)\right|\cdot \left|\eta(t_1,\dots,t_{i-1},\tau,t_{i+1},\dots,t_n)\right| d\tau\right. \\
\left.+\int\limits_{t_i}^{b_i}\left|\mu_i k_i(\tau-t_i)\right|
\cdot \left|f(t)\right|\cdot\left|\eta(t_1,\dots,t_{i-1},\tau,t_{i+1},\dots,t_n)\right| d\tau \right)dt
\end{multline*}
\begin{multline*}
\leq C\cdot D\int\limits_{\Omega_n}\left(\int\limits_{a_i}^{t_i} \left| \lambda_i k_i(t_i-\tau)\right| d\tau
+ \int\limits_{t_i}^{b_i}\left|\mu_i k_i(\tau-t_i)\right|d\tau \right)dt \\
\leq C\cdot D \left(\left|\mu_i\right|+\left|\lambda_i\right|\right)
\left\|k_i\right\|_{L^1(0,b_i-a_i;\R)}\int\limits_{\Omega_n}dt\\
=C\cdot D \left(\left|\mu_i\right|+\left|\lambda_i\right|\right)
\left\|k_i\right\|_{L^1(0,b_i-a_i;\R)}
\cdot\prod_{i=1}^n(b_i-a_i)<\infty
\end{multline*}
Hence, we can use Fubini's theorem to change
the order of integration in the iterated integrals:
\begin{equation*}
\begin{split}
&\int\limits_{\Omega_n}f(t)\cdot K_{P_i}[\eta](t)\;dt_n\dots dt_1
=\int\limits_{\Omega_n}\left( \lambda_i\int\limits_{a_i}^{t_i} f(t)k_i(t_i-\tau)
\eta(t_1,\dots,t_{i-1},\tau,t_{i+1},\dots,t_n)\;d\tau \right.\\
&\left.+\mu_i \int\limits_{t_i}^{b_i}f(t)k_i(\tau-t_i)\eta(t_1,
\dots,t_{i-1},\tau,t_{i+1},\dots,t_n)d\tau \right)\;dt_n\dots dt_1\\
&=\int\limits_{\Omega_n}\left(\lambda_i\int\limits_{\tau}^{b_i} f(t)k_i(t_i-\tau)
\eta(t_1,\dots,t_{i-1},\tau,t_{i+1},\dots,t_n)\;dt_i \right. \\
&\left.+\mu_i \int\limits_{a_i}^{\tau}f(t)k_i(\tau-t_i)\eta(t_1,\dots,t_{i-1},
\tau,t_{i+1},\dots,t_n)\;dt_i \right)dt_n\dots dt_{i-1} d\tau dt_{i+1}\dots dt_1\\
&=\int\limits_{\Omega_n}\eta(t_1,\dots,t_{i-1},\tau,t_{i+1},\dots,t_n)\left(\lambda_i
\int\limits_{\tau}^{b_i} f(t)k_i(t_i-\tau)\;dt_i \right. \\
&\left.+\mu_i\int\limits_{a_i}^{\tau}f(t)k_i(\tau-t_i)\;dt_i \right)dt_n
\dots dt_{i-1}d\tau dt_{i+1}\dots dt_1\\
&=\int\limits_{\Omega_n}\eta(t)\cdot K_{P_{i}^*} [f](t)\; dt_n\dots dt_1.
\end{split}
\end{equation*}
\end{proof}

As corollaries, we obtain the following integration by parts
formulas for constant and variable order fractional integrals.

\begin{corollary}
\index{Integration by parts formula!of several variables!for Riemann--Liouville fractional integrals}
Let $0<\alpha_i<1$, and let $f:\R^n\rightarrow \R$, $\eta:\R^n\rightarrow\R$ be such that
$f,\eta\in C(\bar{\Omega}_n;\R)$, then the following formula holds
\begin{equation}
\int\limits_{\Omega_n}f(t)\cdot\Ilp[\eta](t)\;dt
=\int\limits_{\Omega_n}\eta(t)\cdot\Irp[f](t)\;dt.
\end{equation}
\end{corollary}

\begin{corollary}
\index{Integration by parts formula!of several variables!for variable order fractional integrals}
Suppose that $\alpha:[0,b_i-a_i]\rightarrow [0,1]$, and that $f:\R^n\rightarrow \R$,
$\eta:\R^n\rightarrow\R$ are such that $f,\eta\in C(\bar{\Omega}_n;\R)$. Then,
\begin{equation}
\int\limits_{\Omega_n}f(t)\cdot{_{a_i}}\textsl{I}^{\alpha_i(\cdot)}_{t_i}[\eta](t)\;dt
=\int\limits_{\Omega_n}\eta(t)\cdot{_{t_i}}\textsl{I}^{\alpha_i(\cdot)}_{b_i}[f](t)\;dt.
\end{equation}
\end{corollary}

\begin{theorem}[Generalized Fractional Integration by Parts for Several Variables]
\label{thm:IBPD}
\index{Integration by parts formula!of several variables!for generalized fractional derivatives}
Let $P_{i}=\langle a_i,t_i,b_i,\lambda_i,\mu_i\rangle$ be the parameter set and
$f,\eta\in C^1\left(\bar{\Omega}_n;\R\right)$. Moreover, let $B_{P_i}=\frac{d}{dt}\circ K_{P_i}$,
where $K_{P_i}$ is the generalized partial fractional integral with difference kernel, i.e.,
$k_i=k_i(t_i-\tau)$ such that $k_i\in L_1(0,b_i-a_i;\R)$,
and $K_{P_{i}^*}[f]\in C^1\left(\bar{\Omega}_n;\R\right)$. Then
\begin{equation*}
\int\limits_{\Omega_n} f(t)\cdot B_{P_i}[\eta](t)\;dt
=\int\limits_{\partial\Omega_n} \eta(t)\cdot K_{P_{i}^*}[f](t)\cdot\nu^i\;d(\partial\Omega_n)
-\int\limits_{\Omega_n}\eta(t)\cdot A_{P_i^*} [f](t)~dt,
\end{equation*}
where $\nu^i$ is the outward pointing unit normal to $\partial\Omega_n$.
\end{theorem}

\begin{proof}
By the definition of generalized partial Caputo fractional derivative,
Theorem~\ref{thm:IPRI} and the standard integration by parts formula
(see, e.g., \cite{book:Evans}) one has
\begin{equation*}
\begin{split}
&\int\limits_{\Omega_n} f(t)\cdot  B_{P_i}[\eta](t)~dt
=\int\limits_{\Omega_n} f(t)\cdot K_{P_{i}}\left[\frac{\partial\eta}{\partial t_i}\right](t)~dt
=\int\limits_{\Omega_n}\frac{\partial\eta(t)}{\partial t_i}\cdot
 K_{P_{i}^*}[f](t)~dt\\
&=\int\limits_{\partial\Omega_n}\eta(t)\cdot K_{P_{i}^*}[f](t)\cdot\nu^i~d(\partial\Omega_n)
-\int\limits_{\Omega_n} \eta(t)\cdot\frac{\partial}{\partial t_i} K_{P_{i}^*}[f](t)~dt\\
&=\int\limits_{\partial\Omega_n}\eta(t)\cdot K_{P_{i}^*}[f](t)\cdot\nu^i~d(\partial\Omega_n)
-\int\limits_{\Omega_n} \eta(t)\cdot A_{P_i^*}[f](t)~dt.
\end{split}
\end{equation*}
\end{proof}

Next corollaries present multidimensional integration by parts formulas
for constant and variable order fractional derivatives.

\begin{corollary}
\index{Integration by parts formula!of several variables!for fractional derivatives}
If $0<\alpha_i<1$, functions $f:\R^n\rightarrow\R$ and $\eta:\R^n\rightarrow\R$ are
such that $f,\eta\in C^1(\bar{\Omega}_n;\R)$ and $\Ircp[f]\in C^1(\bar{\Omega}_n;\R)$, then
\begin{equation*}
\int\limits_{\Omega_n}f(t)\cdot\Dclp[\eta](t)\;dt=\int\limits_{\partial\Omega_n}\eta(t)
\cdot\Ircp[f](t)\cdot\nu^i\; d(\partial{\Omega_n})-\int\limits_{\Omega_n}\eta(t)\cdot \Drp[f](t)\;dt.
\end{equation*}
\end{corollary}

\begin{corollary}
\index{Integration by parts formula!of several variables!for variable order fractional derivatives}
If $\alpha_i:[0,b_i-a_i]\rightarrow [0,1]$, functions $f:\R^n\rightarrow\R$ and $\eta:\R^n\rightarrow\R$
are such that $f,\eta\in C^1(\bar{\Omega}_n;\R)$ and
${_{t_i}}\textsl{I}^{\alpha_i(\cdot)}_{b_i}[f]\in C^1(\bar{\Omega}_n;\R)$, then
\begin{equation*}
\int\limits_{\Omega_n}f(t)\cdot{^{C}_{a_i}}\textsl{D}^{\alpha_i(\cdot)}_{t_i}[\eta](t)\;dt
=\int\limits_{\partial \Omega_n}\eta(t)\cdot{_{t_i}}\textsl{I}^{\alpha_i(\cdot)}_{b_i}[f](t)
\cdot\nu^i\; d(\partial{\Omega_n})-\int\limits_{\Omega_n}\eta\cdot{_{t_i}}\textsl{D}^{\alpha_i(\cdot)}_{b_i}[f](t)\;dt.
\end{equation*}
\end{corollary}


\subsection{Fundamental Problem}

In this subsection, we use the notion of generalized fractional gradient.
\begin{definition}
Let $n\in\N$ and $P=(P_1,\dots,P_n)$, $P_i=\langle a_i,t_,b_i,\lambda_i,\mu_i\rangle$.
We define generalized fractional gradient of a function $f:\R^n\rightarrow\R$
with respect to the generalized fractional operator $T$ by
\begin{equation*}
\nabla_{T_P}[f]:=\sum\limits_{i=1}^n e_i\cdot T_{P_i}[f],
\end{equation*}
where $\left\{e_i:i=1,\dots,n\right\}$ denotes the standard basis in $\R^n$.
\end{definition}

For $n\in\N$ let us assume that $P_i=\langle a_i,t_i,b_i,\lambda_i,\mu_i \rangle$
and $P=(P_1,\dots,P_n)$, $y:\R^n\rightarrow\R$, and $\zeta:\partial\Omega_n\rightarrow\R$
is a given function. Consider the following functional:
\begin{equation}
\label{eq:F:SEV}
\fonction{\mathcal{I}}{\mathcal{A}(\zeta)}{\R}{y}{\int\limits_{\Omega_n}F(y(t),
\nabla_{K_P}[y](t),\nabla[y](t),\nabla_{B_P}[y](t),t)\;dt}
\end{equation}
where
$$
\mathcal{A}(\zeta):=\left\{y\in C^1(\bar{\Omega}_n;\R):\left.y\right|_{\partial\Omega_n}
=\zeta,~K_{P_i}[y],B_{P_i}[y]\in C(\bar{\Omega}_n;\R),i=1,\dots,n\right\},
$$
$\nabla$ denotes the classical gradient operator, $\nabla_{K_P}$ and $\nabla_{B_P}$ are
generalized fractional gradient operators such that $K_{P_i}$ is the generalized partial
fractional integral with the kernel $k_i=k_i(t_i-\tau)$, $k_i\in L^1(0,b_i-a_i;\R)$
and $B_{P_i}$ is the generalized partial fractional derivative of Caputo type satisfying
$B_{P_i}=K_{P_i}\circ\frac{\partial}{\partial t_i}$, for $i=1,\dots,n$. Moreover,
we assume that $F$ is a Lagrangian of class $C^1$:
\begin{equation*}
\fonction{F}{\R\times\R^{3n}\times\bar{\Omega}_n}{\R}{(x_1,x_2,x_3,x_4,t)}{F(x_1,x_2,x_3,x_4,t),}
\end{equation*}
and
\begin{itemize}
\item $K_{P_i^*}\left[\partial_{1+i} F(y(\tau),\nabla_{K_P}[y](\tau),\nabla[y](\tau),
\nabla_{B_P}[y](\tau),\tau)\right]\in C(\bar{\Omega}_n;\R)$,
\item $t\mapsto\partial_{1+n+i} F(y(t),\nabla_{K_P}[y](t),\nabla[y](t),
\nabla_{B_P}[y](t),t)\in C^1(\bar{\Omega}_n;\R)$,
\item $K_{P_i^*}\left[\partial_{1+2n+i} F(y(\tau),\nabla_{K_P}[y](\tau),
\nabla[y](\tau),\nabla_{B_P}[y](\tau),\tau)\right]\in C^1(\bar{\Omega}_n;\R)$,
\end{itemize}
where $i=1,\dots,n$.

The following theorem states that if a function minimizes functional \eqref{eq:F:SEV},
then it necessarily must satisfy \eqref{eq:EL:SEV}. This means that equation \eqref{eq:EL:SEV}
determines candidates to solve problem of minimizing functional \eqref{eq:F:SEV}.

\begin{theorem}
\label{thm:EL:SEV}
Suppose that $\y\in\mathcal{A}(\zeta)$ is a minimizer of \eqref{eq:F:SEV}. Then,
$\y$ satisfies the following generalized Euler--Lagrange equation:
\begin{multline}
\label{eq:EL:SEV}
\index{Euler--Lagrange equation!of several variables!for problems with generalized fractional operators}
\partial_1 F(\star_{\y})(t)+\sum\limits_{i=1}^n\Biggl(K_{P_i^*}[\partial_{1+i}F(\star_{\y})(\tau)](t)\\
-\frac{\partial}{\partial t_i}\left(\partial_{1+n+i}F(\star_{\y})(t)\right)
-A_{P_i^*}[\partial_{1+2n+i}F(\star_{\y})(\tau)](t)\Biggr)=0,~~~t\in\Omega_n,
\end{multline}
where $(\star_{\y})(t)=(\y(t),\nabla_{K_P}[\y](t),\nabla[\y](t),\nabla_{B_P}[\y](t),t)$.
\end{theorem}

\begin{proof}
Let $\y\in\mathcal{A}(\zeta)$ be a minimizer of \eqref{eq:F:SEV}. Then, for any
$\left|h\right|\leq\varepsilon$ and every $\eta\in\mathcal{A}(0)$, it satisfies
\begin{equation*}
\mathcal{I}(\y)\leq\mathcal{I}(\y+h\eta).
\end{equation*}
Now, let us define the following function
\begin{equation*}
\fonction{\phi_{\bar{y},\eta}}{[-\varepsilon,\varepsilon]}{\R}{h}{\mathcal{I}(\bar{y}+h\eta)
=\displaystyle\int\limits_{\Omega_n}F(\y(t)+h\eta(t),
\nabla_{K_P}[\y+h\eta](t),\nabla[\y+h\eta](t),\nabla_{B_P}[\y+h\eta](t),t)\;dt.}
\end{equation*}
Because $\phi_{\bar{y},\eta}\in C^1([-\varepsilon,\varepsilon];\R)$ and
\begin{equation*}
\phi_{\bar{y},\eta}(0)\leq\phi_{\bar{y},\eta}(h),~~\left|h\right|\leq\varepsilon,
\end{equation*}
one has
\begin{equation*}
\phi_{\bar{y},\eta}'(0)=\left.\frac{d}{dh}\mathcal{I}(\bar{y}+h\eta)\right|_{h=0}=0.
\end{equation*}
Moreover, using the chain rule, we obtain
\begin{multline*}
\int\limits_{\Omega_n}\; \partial_1 F(\star_{\bar{y}})(t)\cdot\eta(t)
+\sum\limits_{i=1}^n\Biggl(\partial_{1+i} F(\star_{\bar{y}})(t)\cdot K_{P_i}[\eta](t)\\
+\partial_{1+n+i} F(\star_{\bar{y}})(t)\cdot\frac{\partial\eta(t)}{\partial t_i}
+\partial_{1+2n+i} F(\star_{\bar{y}})(t)\cdot B_{P_i}[\eta](t)\Biggr)\; dt=0.
\end{multline*}
Finally, Theorem~\ref{thm:IPRI} implies that
\begin{multline*}
\int\limits_{\Omega_n}\; \Biggl(\partial_1 F(\star_{\bar{y}})(t)
+\sum\limits_{i=1}^n K_{P_i^*}\left[\partial_{1+i} F(\star_{\bar{y}})(\tau)\right](t)\Biggr)\cdot\eta(t)\\
+\sum\limits_{i=1}^n\Biggl(\partial_{1+n+i} F(\star_{\bar{y}})(t)+K_{P_i^*}\left[\partial_{1+2n+i}
F(\star_{\bar{y}})(\tau)\right](t)\Biggr)\cdot\frac{\partial\eta(t)}{\partial t_i}\; dt=0
\end{multline*}
and by the classical integration by parts formula (see~e.g.,\cite{book:Evans})
and the fundamental lemma of the calculus of variations (see e.g., Theorem~1.24,
\cite{book:Dacorogna}), we arrive to \eqref{eq:EL:SEV}.
\end{proof}

\begin{example}
Consider a motion of medium whose displacement may be described as a scalar function
$y(t,x)$, where $x=(x^1,x^2)$. For example, this function may represent the transverse
displacement of a membrane. Suppose that the kinetic energy $T$ and the potential
energy $V$ of the medium are given by
\begin{equation*}
T\left(\frac{\partial y(t,x)}{\partial t}\right)
=\frac{1}{2}\int\rho(x)\left(\frac{\partial y(t,x)}{\partial t}\right)^2\;dx,
\end{equation*}
\begin{equation*}
V(y)=\frac{1}{2}\int k(x)\left|\nabla [y](t,x)\right|^2\;dx,
\end{equation*}
where $\rho(x)$ is the mass density and $k(x)$ is the stiffness, both assumed positive.
Then, the classical action functional is
\begin{equation*}
\mathcal{I}(y)=\frac{1}{2}\int\limits_{\Omega}\left(\rho(x)\left(
\frac{\partial y(t,x)}{\partial t}\right)^2-k(x)\left|\nabla [y](t,x)\right|^2\right)\;dxdt.
\end{equation*}
We shall illustrate what are the Euler--Lagrange equations when the Lagrangian
density depends on generalized fractional derivatives. When we have the Lagrangian
with the kinetic term depending on the operator $B_{P_1}$,
with $P_1=\langle a_1,t,b_1,\lambda,\mu\rangle$, then the fractional action functional has the form
\begin{equation}
\label{eq:F:SEVex}
\mathcal{I}(y)=\frac{1}{2}\int\limits_{\Omega_3}\left[
\rho(x)\left(B_{P_1}[y](t,x)\right)^2-k(x)\left|\nabla [y](t,x)\right|^2\right]\;dxdt.
\end{equation}
The generalized fractional Euler--Lagrange equation satisfied by an extremal of \eqref{eq:F:SEVex} is
\begin{equation*}
-\rho(x) A_{P_1^*}\left[B_{P_1}[y](\tau,s)\right](t,x)-\nabla\left[k(s)\nabla[y](\tau,s)\right](t,x)=0.
\end{equation*}
If $\rho$ and $k$ are constants, then equation
\begin{equation*}
\rho A_{P_1^*}\left[B_{P_1}[y](\tau,s)\right](t,x)+c^2\Delta[y](t,x)=0,
\end{equation*}
where $c^2=k/\rho$, can be called the generalized time-fractional wave equation.
Now, assume that the kinetic energy and the potential energy depend on $B_{P_1}$
and $\nabla_{B_P}$ operators, respectively, where $P=(P_{2},P_{3})$. Then,
the action functional for the system has the form
\begin{equation}
\label{eq:F:SEVex2}
\mathcal{I}(y)=\frac{1}{2}\int\limits_{\Omega_3}\left[\rho\left(B_{P_1}[y](t,x)\right)^2
-k\left|\nabla_{B_P}[y](t,x)\right|^2\right]\;dxdt.
\end{equation}
The generalized fractional Euler--Lagrange equation
satisfied by an extremal of \eqref{eq:F:SEVex2} is
\begin{equation*}
-\rho  A_{P_1^*}\left[B_{P_1}[y](\tau,s)\right](t,x)+\sum\limits_{i=2}^3
A_{P_{i}^*}\left[B_{P_{i}}[y](\tau,s)\right](t,x)=0.
\end{equation*}
If $\rho$ and $k$ are constants, then
\begin{equation*}
A_{P_1^*}\left[B_{P_1}[y](\tau,s)\right](t,x)-c^2\left(\sum\limits_{i=2}^3
A_{P_{i}^*}\left[kB_{P_{i}}[y](\tau,s)\right](t,x)\right)=0
\end{equation*}
can be called the generalized  space- and time-fractional wave equation.
\end{example}

\begin{corollary}
\index{Euler--Lagrange equation!of several variables!for problems with Riemann--Liouville fractional integrals and Caputo fractional derivatives}
Let $\alpha=(\alpha_1,\dots,\alpha_n)\in(0,1)^n$ and let $\y\in C^1(\bar{\Omega}_n;\R)$ be a minimizer of the functional
\begin{equation}
\label{eq:F:cor:SEV1}
\mathcal{I}(y)=\int\limits_{\Omega_n} F(y(t),
\nabla_{I^{1-\alpha}}[y](t),\nabla[y](t),\nabla_{D^\alpha}[y](t),t)\;dt
\end{equation}
satisfying
\begin{equation}
\label{eq:B:cor:SEV1}
\left.y(t)\right|_{\partial \Omega_n}=\zeta(t),
\end{equation}
where $\zeta:\partial\Omega_n\rightarrow\R$ is a given function,
\begin{equation*}
\nabla_{I^{1-\alpha}}=\sum\limits_{i=1}^n e_i\cdot\Ilcp,
~~\nabla_{D^\alpha}=\sum\limits_{i=1}^n e_i\cdot\Dclp,
\end{equation*}
$F$is of class $C^1$ and
\begin{itemize}
\item $\Ircp\left[\partial_{1+i} F(y(\tau),\nabla_{I^{1-\alpha}}[y](\tau),
\nabla[y](\tau),\nabla_{D^\alpha}[y](\tau),\tau)\right]$ is continuous on $\bar{\Omega}_n$,
\item $t\mapsto\partial_{1+n+i} F(y(t),\nabla_{I^{1-\alpha}}[y](t),\nabla[y](t),
\nabla_{D^\alpha}[y](t),t)$ is continuously differentiable on $\bar{\Omega}_n$,
\item $\Ircp\left[\partial_{1+2n+i} F(y(\tau),\nabla_{I^{1-\alpha}}[y](\tau),\nabla[y](\tau),
\nabla_{D^\alpha}[y](\tau),\tau)\right]$ is continuously differentiable on $\bar{\Omega}_n$.
\end{itemize}
Then, $\y$ satisfies the following fractional Euler--Lagrange equation:
\begin{multline}\label{eq:M:EL}
\partial_1 F(\y(t),\nabla_{I^{1-\alpha}}[\y](t),\nabla[\y](t),\nabla_{D^\alpha}[\y](t),t)\\
+\sum\limits_{i=1}^n\Biggl(\Ircp\left[\partial_{1+i}
F(\y(\tau),\nabla_{I^{1-\alpha}}[\y](\tau),\nabla[\y](\tau),\nabla_{D^\alpha}[\y](\tau),\tau)\right](t)\\
-\frac{\partial}{\partial t_i}\left(\partial_{1+n+i}
F(\y(t),\nabla_{I^{1-\alpha}}[\y](t),\nabla[\y](t),\nabla_{D^\alpha}[\y](t),t)\right)\\
+\Drp\left[\partial_{1+2n+i}F(\y(\tau),\nabla_{I^{1-\alpha}}[\y](\tau),\nabla[\y](\tau),
\nabla_{D^\alpha}[\y](\tau),\tau)\right](t)\Biggr)=0,~~~t\in\Omega_n.
\end{multline}
\end{corollary}

\begin{corollary}
\index{Euler--Lagrange equation!of several variables!for problems with variable order fractional operators}
For $i=1,\dots,n$, suppose that $\alpha_i:[0,b_i-a_i]\rightarrow [0,1]$ and let
\begin{equation*}
\nabla_{I}=\sum\limits_{i=1}^n e_i\cdot{_{a_i}}\textsl{I}^{1-\alpha_i(\cdot)}_{t_i},
~~\nabla_{D}=\sum\limits_{i=1}^n e_i\cdot{^{C}_{a_i}}\textsl{D}^{\alpha_i(\cdot)}_{t_i}.
\end{equation*}
If $\y\in C^1(\bar{\Omega}_n;\R)$ minimizes the functional
\begin{equation}
\label{eq:F:cor:SEV2}
\mathcal{I}(y)=\int\limits_{\Omega_n} F(y(t),\nabla_{I}[y](t),\nabla[y](t),\nabla_{D}[y](t),t)\;dt,
\end{equation}
subject to the boundary condition
\begin{equation*}
\left.y(t)\right|_{\partial \Omega_n}=\zeta(t),
\end{equation*}
where $\zeta:\partial\Omega_n\rightarrow\R$ is a given function,
${_{a_i}}\textsl{I}^{1-\alpha_i(\cdot)}_{t_i}[y],{^{C}_{a_i}}\textsl{D}^{\alpha_i(\cdot)}_{t_i}
\in C(\bar{\Omega}_n;\R)$, $F$ is of class $C^1$
and
\begin{itemize}
\item ${_{t_i}}\textsl{I}^{1-\alpha_i(\cdot)}_{b_i}\left[\partial_{1+i} F(y(\tau),
\nabla_{I}[y](\tau),\nabla[y](\tau),\nabla_{D}[y](\tau),\tau)\right]$ is continuous on $\bar{\Omega}_n$,
\item $t\mapsto\partial_{1+n+i} F(y(t),\nabla_{I}[y](t),\nabla[y](t),
\nabla_{D}[y](t),t)$ is continuously differentiable on $\bar{\Omega}_n$,
\item ${_{t_i}}\textsl{I}^{1-\alpha_i(\cdot)}_{b_i}\left[\partial_{1+2n+i}
F(y(\tau),\nabla_{I}[y](\tau),\nabla[y](\tau),\nabla_{D}[y](\tau),\tau)\right]$
is continuously differentiable on $\bar{\Omega}_n$,
\end{itemize}
then $\y$ satisfies the following equation
\begin{multline}
\label{eq:M:EL:2}
\partial_1 F(\y(t),\nabla_{I}[\y](t),\nabla[\y](t),\nabla_{D}[\y](t),t)\\
+\sum\limits_{i=1}^n\Biggl({_{t_i}}\textsl{I}^{1-\alpha_i(\cdot)}_{b_i}\left[
\partial_{1+i}F(\y(\tau),\nabla_{I}[\y](\tau),\nabla[\y](\tau),
\nabla_{D}[\y](\tau),\tau)\right](t)\\
-\frac{\partial}{\partial t_i}\left(\partial_{1+n+i}F(\y(t),
\nabla_{I}[\y](t),\nabla[\y](t),\nabla_{D}[\y](t),t)\right)\\
+{_{t_i}}\textsl{D}^{\alpha_i(\cdot)}_{b_i}\left[\partial_{1+2n+i}F(\y(\tau),
\nabla_{I}[\y](\tau),\nabla[\y](\tau),\nabla_{D}[\y](\tau),\tau)\right](t)\Biggr)=0,
~~~t\in\Omega_n.
\end{multline}
\end{corollary}

\begin{theorem}
\label{thm:Suff:SEV}
Suppose that $\y\in\mathcal{A}(\zeta)$ satisfies \eqref{eq:EL:SEV} and function
$(x_1,x_2,x_3,x_4)\rightarrow F(x_1,x_2,x_3,x_4,t)$ is convex for every $t\in\bar{\Omega}_n$.
Then, $\y$ is a minimizer of functional \eqref{eq:F:SEV}.
\end{theorem}

\begin{proof}
Let $\y\in\mathcal{A}(\zeta)$ be a function satisfying equation \eqref{eq:EL:SEV} and such that
$(x_1,x_2,x_3,x_4)\rightarrow F(x_1,x_2,x_3,x_4,t)$ is convex for every $t\in\bar{\Omega}_n$.
Then, the following inequality holds:
\begin{multline*}
\mathcal{I}(y)\geq \mathcal{I}(\y)\\
+\int\limits_{\Omega_n}\left(\partial_1 F(y-\y)+\sum\limits_{i=1}^n\left[\partial_{1+i}
F\cdot K_{P_i}[y-\y]+\partial_{1+n+i}F \frac{\partial}{\partial t_i}[y-\y]+\partial_{1+2n+i}
F\cdot B_{P_i}[y-\y]\right]\right)\;dt,
\end{multline*}
where functions $\partial_{i}F$ are evaluated at $(\y,\nabla_{K_P}[\y],\nabla[\y],
\nabla_{B_P}[\y],t)$, for $i=1,\dots,3n+1$. Moreover, using the classical integration
by parts formula, as well as Theorem~\ref{thm:IPRI} and the fact that
$\left.y-\y\right|_{\partial\Omega_n}=0$, we obtain
\begin{equation*}
\mathcal{I}(y)\geq \mathcal{I}(\y)+\int\limits_{\Omega_n}\left(\partial_1 F
+\sum\limits_{i=1}^n\left[K_{P_i^*}\left[\partial_{1+i} F\right]
+\frac{\partial}{\partial t_i}\left(\partial_{1+n+i} F\right)
+A_{P_i^*}\left[\partial_{1+2n+i}F\right]\right]\right)(y-\y)\;dt.
\end{equation*}
Finally, applying equation \eqref{eq:EL:SEV}, we have $\mathcal{I}(y)\geq \mathcal{I}(\y)$
for any $y\in\mathcal{A}(\zeta)$ and the proof is complete.
\end{proof}


\subsection{Dirichlet's Principle}

One of the most important variational principles for PDEs is Dirichlet's principle
for the Laplace equation. We shall present its generalized fractional counterpart.

We show that the solution of the generalized fractional
boundary value problem\index{Dirichlet's principle}
\begin{numcases}{ }
\sum_{i=1}^n A_{P_{i}^{*}}\left[B_{P_{i}}[y]\right]=0
& \text{ in } $\Omega_n$, \label{eq:BVP}\\
y=\zeta & \text{ on } $\partial{\Omega_n}$\label{b:BVP},
\end{numcases}
can be characterized as a minimizer of the following variational functional
\begin{equation}\label{eq:F:DirFunct}
\mathcal{I}(y)=\int\limits_{\Omega_n} \left|\nabla_{B_{P}}[y]\right|^2 dt
\end{equation}
on the set $\mathcal{A}(\zeta)$, where
$$
\nabla_{B_{P}}=\sum\limits_{i=1}^n e_i\cdot B_{P_i}
$$
is the generalized fractional gradient operator such that the partial derivatives $B_{P_i}$
have kernels $k_i=k_i(t_i-\tau)$, $k_i\in L^1(0,b_i-a_i;\R)$, and parameter sets given by
$P_i=\langle a_i,t_i,b_i,\lambda_i,\mu_i \rangle$, $i=1,\dots,n$.

\begin{remark}
In the following we assume that both problems, \eqref{eq:BVP}--\eqref{b:BVP} and minimization
of \eqref{eq:F:DirFunct} on the set $\mathcal{A}(\zeta)$, have solutions.
\end{remark}

\begin{theorem}[Generalized Fractional Dirichlet's Principle]
\label{thm:DP}
\index{Generalized fractional Dirichlet's principle}
Suppose that $\y\in\mathcal{A}(\zeta)$.
Then $\y$ solves the boundary value problem \eqref{eq:BVP}--\eqref{b:BVP}
if and only if $\y$ satisfies
\begin{equation}
\label{eq:3}
\mathcal{I}(\y)=\min\limits_{y\in \mathcal{A}(\zeta)}\mathcal{I}(y).
\end{equation}
\end{theorem}

\begin{proof}
Theorem~\ref{thm:DP} is a simple consequence
of Theorem~\ref{thm:EL:SEV} and Theorem~\ref{thm:Suff:SEV}.
\end{proof}

\begin{theorem}
There exists at most one solution $\y\in\mathcal{A}(\zeta)$ to problem \eqref{eq:BVP}--\eqref{b:BVP}.
\end{theorem}

\begin{proof}
Let $\y\in\mathcal{A}(\zeta)$ be a solution to problem \eqref{eq:BVP}--\eqref{b:BVP}.
Assume that $\hat{y}$ is another solution to problem  \eqref{eq:BVP}--\eqref{b:BVP}.
Then, $w=\y-\hat{y}\neq 0$ and
\begin{equation*}
0=-\int\limits_{\Omega_n}w\cdot\sum\limits_{i=1}^n A_{P_{i}^{*}}\left[B_{P_{i}}[w]\right]\;dt.
\end{equation*}
By classical integration by parts formula and Theorem~\ref{thm:IPRI}
and since $\left.w\right|_{\partial\Omega_n}=0$ we have
\begin{equation*}
0=\int\limits_{\Omega_n}w\cdot\sum\limits_{i=1}^n \left(B_{P_{i}}[w]\right)^2\;dt
=\int\limits_{\Omega_n}\left|\nabla_{B_{P}}[w]\right|^2 dt.
\end{equation*}
Note that $\left|\nabla_{B_{P}}[w]\right|^2$ is a positive definite quantity. The volume integral
of a positive definite quantity is equal to zero only in the case when this quantity is zero itself
throughout the volume. Thus $\nabla_{B_{P}}[w]= 0$. Since $w$ is twice continuously differentiable
and $k_i\in L^1(0,b_i-a_i;\R)$ we have
\begin{equation*}
\frac{\partial}{\partial t_i}w (t)= 0,~i=1,\dots,n
\end{equation*}
i.e., $\nabla [w]=0$. Because $w=0$ on $\partial\Omega_n$, we deduce that $w=0$. In other words $\y=\hat{y}$.
\end{proof}


\subsection{Isoperimetric Problem}

Suppose that $y:\R^n\rightarrow\R$, $P=\langle a_i,t_i,b_i,\lambda_i,\mu_i \rangle$,
$P=(P_1,\dots,P_n)$ and $\zeta:\partial\Omega_n\rightarrow\R$ is a given curve.
Let us define the following functional
\begin{equation}
\label{eq:F:SEV:IC}
\fonction{\mathcal{J}}{\mathcal{A}(\zeta)}{\R}{y}{\displaystyle\int\limits_{\Omega_n}
G(y(t),\nabla_{K_P}[y](t),\nabla[y](t),\nabla_{B_P}[y](t),t)\;dt,}
\end{equation}
where operators $\nabla_{K_P}$,$\nabla$, $\nabla_{B_P}$ and function $G$
are the same as in the case of functional \eqref{eq:F:SEV}.
The next theorem gives a necessary optimality condition for a function to be a minimizer
of functional \eqref{eq:F:SEV} subject to isoperimetric constraint $\mathcal{J}(y)=\xi$.

\begin{theorem}
\label{thm:EL:SEV:IC}
Let us assume that $\y$ minimizes functional \eqref{eq:F:SEV} on the following set:
$$\mathcal{A}_{\xi}(\zeta):=\left\{y\in\mathcal{A}(\zeta):\mathcal{J}(y)=\xi\right\}.$$
Then, one can find a real constant $\lambda_0$ such that, for $H=F-\lambda_0 G$, equation
\begin{multline}
\label{eq:EL:SEV:IC}
\index{Euler--Lagrange equation!of several variables!for problems with generalized fractional operators}
\partial_1 H(\star_{\y})(t)+\sum\limits_{i=1}^n\Biggl(K_{P_i^*}[\partial_{1+i}H(\star_{\y})(\tau)](t)\\
-\frac{\partial}{\partial t_i}\left(\partial_{1+n+i}H(\star_{\y})(t)\right)
-A_{P_i^*}[\partial_{1+2n+i}H(\star_{\y})(\tau)](t)\Biggr)=0
\end{multline}
holds, provided that
\begin{multline}\label{eq:EL:SEV:eml}
\partial_1 G(\star_{\y})(t)+\sum\limits_{i=1}^n\Biggl(K_{P_i^*}[\partial_{1+i}G(\star_{\y})(\tau)](t)\\
-\frac{\partial}{\partial t_i}\left(\partial_{1+n+i}G(\star_{\y})(t)\right)
-A_{P_i^*}[\partial_{1+2n+i}G(\star_{\y})(\tau)](t)\Biggr)\neq 0,
\end{multline}
where $(\star_{\bar{\y}})(t)=(\y(t),\nabla_{K_P}[\y](t),\nabla[\y](t),\nabla_{B_P}[\y](t),t)$.
\end{theorem}

\begin{proof}
The fundamental lemma of the calculus of variations, and hypothesis \eqref{eq:EL:SEV:eml},
imply that there exists $\eta_2\in\mathcal{A}(0)$ so that
\begin{multline*}
\int\limits_{\Omega_n}\; \Biggl(\partial_1 G(\star_{\bar{y}})(t)+\sum\limits_{i=1}^n
K_{P_i^*}\left[\partial_{1+i} G(\star_{\bar{y}})(\tau)\right](t)\Biggr)\cdot\eta_2(t)\\
+\sum\limits_{i=1}^n\Biggl(\partial_{1+n+i} G(\star_{\bar{y}})(t)+K_{P_i^*}\left[
\partial_{1+2n+i} G(\star_{\bar{y}})(\tau)\right](t)\Biggr)\cdot\frac{\partial\eta_2(t)}{\partial t_i}\; dt=1.
\end{multline*}
Now, with function $\eta_2$ and an arbitrary $\eta_1\in\mathcal{A}(0)$, let us define
\begin{equation*}
\fonction{\phi}{[-\varepsilon_1,\varepsilon_1]\times
[-\varepsilon_2,\varepsilon_2]}{\R}{(h_1,h_2)}{\mathcal{I}(\bar{y}+h_1\eta_1+h_2\eta_2)}
\end{equation*}
and
\begin{equation*}
\fonction{\psi}{[-\varepsilon_1,\varepsilon_1]\times
[-\varepsilon_2,\varepsilon_2]}{\R}{(h_1,h_2)}{\mathcal{J}(\bar{y}+h_1\eta_1+h_2\eta_2)-\xi}
\end{equation*}
Notice that, $\psi(0,0)=0$ and that
\begin{multline*}
\left.\frac{\partial\psi}{\partial h_2}\right|_{(0,0)}
=\int\limits_{\Omega_n}\; \Biggl(\partial_1 G(\star_{\bar{y}})(t)
+\sum\limits_{i=1}^n K_{P_i^*}\left[\partial_{1+i} G(\star_{\bar{y}})(\tau)\right](t)\Biggr)\cdot\eta_2(t)\\
+\sum\limits_{i=1}^n\Biggl(\partial_{1+n+i} G(\star_{\bar{y}})(t)+K_{P_i^*}\left[\partial_{1+2n+i}
G(\star_{\bar{y}})(\tau)\right](t)\Biggr)\cdot\frac{\partial\eta_2(t)}{\partial t_i}\; dt=1.
\end{multline*}
The implicit function theorem implies, that there is $\epsilon_0>0$ and a function
$s\in C^1([-\varepsilon_0,\varepsilon_0];\R)$ with $s(0)=0$ such that
\begin{equation*}
\psi(h_1,s(h_1))=0,~~\left|h_1\right|\leq\varepsilon_0,
\end{equation*}
and then $\y+h_1\eta_1+s(h_1)\eta_2\in\mathcal{A}_{\xi}(\zeta)$.
Moreover,
\begin{equation*}
\frac{\partial\psi}{\partial h_1}+\frac{\partial \psi}{\partial h_2}
\cdot s'(h_1)=0,~~\left|h_1\right|\leq\varepsilon_0,
\end{equation*}
and then
\begin{equation*}
s'(0)=\left.-\frac{\partial \psi}{\partial h_1}\right|_{(0,0)}.
\end{equation*}
Because $\y\in\mathcal{A}(\zeta)$ is a minimizer of $\mathcal{I}$ we have
\begin{equation*}
\phi(0,0)\leq\phi(h_1,s(h_1)),~~\left|h_1\right|\leq\varepsilon_0,
\end{equation*}
and hence
\begin{equation*}
\left.\frac{\partial\phi}{\partial h_1}\right|_{(0,0)}
+\left.\frac{\partial\phi}{\partial h_2}\right|_{(0,0)}\cdot s'(0)=0.
\end{equation*}
Letting $\left.\lambda_0=\frac{\partial\phi}{\partial h_2}\right|_{(0,0)}$
be the Lagrange multiplier we find
\begin{equation*}
\left.\frac{\partial\phi}{\partial h_1}\right|_{(0,0)}
-\lambda_0\left.\frac{\partial\psi}{\partial h_1}\right|_{(0,0)}=0
\end{equation*}
or, in other words,
\begin{multline*}
\int\limits_{\Omega_n}\; \left[\Biggl(\partial_1 F(\star_{\bar{y}})(t)
+\sum\limits_{i=1}^n K_{P_i^*}\left[\partial_{1+i}
F(\star_{\bar{y}})(\tau)\right](t)\Biggr)\cdot\eta_1(t)\right.\\
\left.+\sum\limits_{i=1}^n\Biggl(\partial_{1+n+i}
F(\star_{\bar{y}})(t)
+K_{P_i^*}\left[\partial_{1+2n+i} F(\star_{\bar{y}})(\tau)\right](t)\Biggr)
\cdot\frac{\partial\eta_1(t)}{\partial t_i}\right]\\
-\lambda_0\left[\Biggl(\partial_1 G(\star_{\bar{y}})(t)
+\sum\limits_{i=1}^n K_{P_i^*}\left[\partial_{1+i}
G(\star_{\bar{y}})(\tau)\right](t)\Biggr)\cdot\eta_1(t)\right.\\
\left.+\sum\limits_{i=1}^n\Biggl(\partial_{1+n+i}
G(\star_{\bar{y}})(t)
+K_{P_i^*}\left[\partial_{1+2n+i} G(\star_{\bar{y}})(\tau)\right](t)\Biggr)
\cdot\frac{\partial\eta_1(t)}{\partial t_i}\right]\;dt=0.
\end{multline*}
Finally, applying the integration by parts formula for classical and fractional derivatives,
and by the fundamental lemma of the calculus of variations we obtain \eqref{eq:EL:SEV:IC}.
\end{proof}

\begin{corollary}
\index{Euler--Lagrange equation!of several variables!for problems with Riemann--Liouville fractional integrals and Caputo fractional derivatives}
Let us assume that $\alpha=(\alpha_1,\dots,\alpha_n)\in (0,1)^n$ and $\y\in C^1(\bar{\Omega}_n;\R)$
is minimizer of functional \eqref{eq:F:cor:SEV1} subject to the isoperimetric constraint $\mathcal{J}(y)=\xi$, where
\begin{equation}
\label{eq:F:cor:SEVI}
\mathcal{J}(y)=\int\limits_{\Omega_n} G(y(t),\nabla_{I^{1-\alpha}}[y](t),\nabla[y](t),\nabla_{D^\alpha}[y](t),t)\;dt
\end{equation}
and boundary condition \eqref{eq:B:cor:SEV1}. Moreover
\begin{itemize}
\item $G$ is of class $C^1$,
\item ${_{t_i}}\textsl{I}^{1-\alpha_i}_{b_i}\left[\partial_{1+i}
G(y(\tau),\nabla_{I}[y](\tau),\nabla[y](\tau),\nabla_{D}[y](\tau),\tau)\right]$ is continuous on $\bar{\Omega}_n$,
\item $t\mapsto\partial_{1+n+i} G(y(t),\nabla_{I}[y](t),\nabla[y](t),\nabla_{D}[y](t),t)$
is continuously differentiable on $\bar{\Omega}_n$,
\item ${_{t_i}}\textsl{I}^{1-\alpha_i}_{b_i}\left[\partial_{1+2n+i}
G(y(\tau),\nabla_{I}[y](\tau),\nabla[y](\tau),\nabla_{D}[y](\tau),\tau)\right]$ is continuously differentiable on $\bar{\Omega}_n$.
\end{itemize}
Then, if $\y$ is not an extremal for functional \eqref{eq:F:cor:SEVI},
we can find $\lambda_0\in\R$ such that the following equation:
\begin{multline*}
\partial_1 H(\y(t),\nabla_{I^{1-\alpha}}[\y](t),\nabla[\y](t),\nabla_{D^\alpha}[\y](t),t)\\
+\sum\limits_{i=1}^n\Biggl(\Ircp\left[\partial_{1+i}H(\y(\tau),
\nabla_{I^{1-\alpha}}[\y](\tau),\nabla[\y](\tau),\nabla_{D^\alpha}[\y](\tau),\tau)\right](t)\\
-\frac{\partial}{\partial t_i}\left(\partial_{1+n+i}H(\y(t),
\nabla_{I^{1-\alpha}}[\y](t),\nabla[\y](t),\nabla_{D^\alpha}[\y](t),t)\right)\\
+\Drp\left[\partial_{1+2n+i}H(\y(\tau),\nabla_{I^{1-\alpha}}[\y](\tau),
\nabla[\y](\tau),\nabla_{D^\alpha}[\y](\tau),\tau)\right](t)\Biggr)=0,
~~~t\in\Omega_n,
\end{multline*}
is satisfied, where $H=F-\lambda_0 G$.
\end{corollary}

\begin{corollary}
\index{Euler--Lagrange equation!of several variables!for problems with variable order fractional operators}
Suppose that $\alpha_i:[0,b_i-a_i]\rightarrow [0,1]$. If $\y\in C^1(\bar{\Omega}_n;\R)$
minimizes \eqref{eq:F:cor:SEV2} subject to \eqref{eq:B:cor:SEV1} and
\begin{equation*}
\mathcal{J}(y)=\int\limits_{\Omega_n} G(y(t),\nabla_{I}[y](t),\nabla[y](t),\nabla_{D}[y](t),t)\;dt=\xi,
\end{equation*}
where ${_{a_i}}\textsl{I}^{1-\alpha_i(\cdot)}_{t_i}[y],
{^{C}_{a_i}}\textsl{D}^{\alpha_i(\cdot)}_{t_i}\in C(\bar{\Omega}_n;\R)$, and
\begin{itemize}
\item $G$ is of class $C^1$,
\item ${_{t_i}}\textsl{I}^{1-\alpha_i(\cdot)}_{b_i}\left[\partial_{1+i}
G(y(\tau),\nabla_{I}[y](\tau),\nabla[y](\tau),\nabla_{D}[y](\tau),\tau)\right]$
is continuous on $\bar{\Omega}_n$,
\item $t\mapsto\partial_{1+n+i} G(y(t),\nabla_{I}[y](t),\nabla[y](t),\nabla_{D}[y](t),t)$
is continuously differentiable on $\bar{\Omega}_n$,
\item ${_{t_i}}\textsl{I}^{1-\alpha_i(t_i-\cdot)}_{b_i}\left[\partial_{1+2n+i}
G(y(\tau),\nabla_{I}[y](\tau),\nabla[y](\tau),\nabla_{D}[y](\tau),\tau)\right]$
is continuously differentiable on $\bar{\Omega}_n$,
\end{itemize}
for $i=1,\dots,n$. Then, there is $\lambda_0\in\R$ such that, for $H=F-\lambda_0 G$,
$\y$ is a solution to the following equation:
\begin{multline*}
\partial_1 H(\y(t),\nabla_{I}[\y](t),\nabla[\y](t),\nabla_{D}[\y](t),t)\\
+\sum\limits_{i=1}^n\Biggl({_{t_i}}\textsl{I}^{1-\alpha_i(\cdot)}_{b_i}\left[\partial_{1+i}
H(\y(\tau),\nabla_{I}[\y](\tau),\nabla[\y](\tau),\nabla_{D}[\y](\tau),\tau)\right](t)\\
-\frac{\partial}{\partial t_i}\left(\partial_{1+n+i}H(\y(t),\nabla_{I}[\y](t),\nabla[\y](t),\nabla_{D}[\y](t),t)\right)\\
+{_{t_i}}\textsl{D}^{\alpha_i(\cdot)}_{b_i}\left[\partial_{1+2n+i}
H(\y(\tau),\nabla_{I}[\y](\tau),\nabla[\y](\tau),\nabla_{D}[\y](\tau),\tau)\right](t)\Biggr)=0,
~~~t\in\Omega_n,
\end{multline*}
provided that $\y$ is not a solution to Euler--Lagrange equation associated to $\mathcal{J}$.
\end{corollary}


\subsection{Noether's Theorem}

In Section~\ref{sec:NTH:sing} of this thesis, we have proved a generalized fractional version of Noether's theorem.
That is, assuming invariance of Lagrangian under changes in the coordinate system, we showed that its extremal must
satisfy equation \eqref{eq:NTH:1}. In this section, we prove a generalized multidimensional fractional Noether's
theorem. As before, we start with definitions of extremal and invariance.

\begin{definition}
A function $y\in C^1(\bar{\Omega}_n;\R)$ such that $K_{P_i}[y],B_{P_i}[y]\in C(\bar{\Omega}_n;\R)$,
$i=1,\dots,n$ satisfying equation \eqref{eq:EL:SEV} is said to be a generalized fractional extremal.
\end{definition}

We consider a one--parameter family of transformations of the form
$\hat{y}(t)=\phi(\theta,t,y(t))$, where $\phi$ is a map of class $C^2$:
\begin{equation*}
\fonction{\phi}{[-\varepsilon,\varepsilon]
\times \bar{\Omega}_n\times\R}{\R}{(\theta,t,x)}{\phi(\theta,t,x),}
\end{equation*}
such that $\phi(0,t,x)=x$.
Note that, using Taylor's expansion of $\phi(\theta,t,y(t))$ in a neighborhood of $0$ one has
\begin{equation*}
\hat{y}(t)=\phi(0,t,y(t))+\theta\frac{\partial}{\partial\theta}\phi(0,t,y(t))+o(\theta),
\end{equation*}
provided that $\theta\in [-\varepsilon,\varepsilon]$. Moreover, having in mind that $\phi(0,t,y(t))=y(t)$
and denoting $\xi(t,y(t))=\frac{\partial}{\partial\theta}\phi(0,t,y(t))$, one has
\begin{equation}\label{eq:Tr:Mult}
\hat{y}(t)=y(t)+\theta\xi(t,y(t))+o(\theta),
\end{equation}
so that, for $\theta\in [-\varepsilon,\varepsilon]$ the linear approximation to the transformation is
\begin{equation*}
\hat{y}(t)\approx y(t)+\theta\xi(t,y(t)).
\end{equation*}

Now, let us introduce the notion of invariance.
\begin{definition}
\label{def:IF:SEV}
\index{Invariant Lagrangian!of several variables with generalized fractional operators}
We say that the Lagrangian $F$ is invariant under the one--parameter family of infinitesimal
transformations \eqref{eq:Tr:Mult}, where $\xi$ is such that
$t\mapsto\xi(t,y(t))\in C^1\left(\bar{\Omega}_n;\R\right)$ with
$K_{P_i}\left[\tau\mapsto\xi(\tau,y(\tau))\right],B_{P_i}\left[\tau\mapsto\xi(\tau,y(\tau))\right]
\in C\left(\bar{\Omega}_n;\R\right)$, $i=1,\dots,n$, if
\begin{equation}
\label{eq:CI:2}
F(y(t),\nabla_{K_P}[y](t),\nabla[y](t),\nabla_{B_P}[y](t),t)
= F(\hat{y}(t),\nabla_{K_P}[\hat{y}](t),\nabla[\hat{y}](t),\nabla_{B_P}[\hat{y}](t),t),
\end{equation}
for all $\theta\in [-\varepsilon,\varepsilon]$, and  all $y\in C^1\left(\bar{\Omega}_n;\R\right)$
with $K_{P_i}[y],B_{P_i}[y]\in C\left(\bar{\Omega}_n;\R\right)$, $i=1,\dots,n$.
\end{definition}

Similarly to Section~\ref{sec:NTH:sing}, we want to state Noether's theorem in a compact form.
For that we introduce the following two bilinear operators:
\begin{equation}\label{eq:BD:SEV}
\mathbf{D}_i[f,g]:=f\cdot A_{P_i^*}[g]+g\cdot B_{P_i}[f],
\end{equation}
\begin{equation}\label{eq:BI:SEV}
\mathbf{I}_i[f,g]:=-f\cdot K_{P_i^*}[g]+g\cdot K_{P_i}[f],
\end{equation}
where $i=1,\dots,n$.

Now we are ready to state the generalized fractional Noether's theorem.

\begin{theorem}[Generalized Multidimensional Fractional Noether's Theorem]
\label{thm:NTH:SEV}
\index{Noether's Theorem!of several variables!for generalized fractional operators}
Let $F$ be invariant under the one--parameter family of infinitesimal transformations
\eqref{eq:Tr:Mult}. Then for every generalized fractional extremal the following equality holds
\begin{multline}\label{eq:NTH:SEV}
\sum\limits_{i=1}^n \mathbf{I}_i\left[\xi(t,y(t)),\partial_{1+i}F(\star_y)(t)\right]
+\frac{\partial}{\partial t_i}\left(\xi(t,y(t))\cdot \partial_{1+n+i}F(\star_y)(t)\right)\\
+\mathbf{D}_i\left[\xi(t,y(t)),\partial_{1+2n+i}F(\star_y)(t)\right]=0,~~t\in\Omega_n,
\end{multline}
where $(\star_y)(t)=(y(t),\nabla_{K_P}[y](t),\nabla[y](t),\nabla_{B_P}[y](t),t)$.
\end{theorem}

\begin{proof}
By equation \eqref{eq:CI:2} one has
\begin{equation*}
\left.\frac{d}{d\theta}\left[F(\hat{y}(t),\nabla_{K_P}[\hat{y}](t),
\nabla[\hat{y}](t),\nabla_{B_P}[\hat{y}](t),t)\right]\right|_{\theta=0}=0
\end{equation*}
The usual chain rule implies
\begin{multline*}
\Biggl.\partial_1 F(\star_{\hat{y}})(t)\cdot\frac{d}{d \theta}\hat{y}(t)
+\sum\limits_{i=1}^n\partial_{1+i} F(\star_{\hat{y}})(t)\cdot\frac{d}{d \theta}K_{P_i}[\hat{y}](t)\\
+\partial_{1+n+i} F(\star_{\hat{y}})(t)\cdot\frac{d}{d\theta}\frac{\partial}{\partial t_i}\hat{y}(t)
+\partial_{1+2n+i} F(\star_{\hat{y}})(t)\cdot\frac{d}{d\theta}B_{P_i}[\hat{y}](t)\Biggr|_{\theta=0}=0.
\end{multline*}
By linearity of $K_{P_i}$, $B_{P_i}$, $i=1,\dots,n$ differentiating with respect to $\theta$,
and putting $\theta=0$ we have
\begin{multline*}
\partial_1 F(\star_y)(t)\cdot\xi(t,y(t))+\sum\limits_{i=1}^n
\partial_{1+i} F(\star_y)(t)\cdot K_{P_i}[\tau\mapsto\xi(\tau,y(\tau))](t)\\
+\partial_{1+n+i} F(\star_y)(t)\cdot\frac{\partial}{\partial t_i}\xi(t,y(t))
+\partial_{1+2n+i} F(\star_y)(t)\cdot B_{P_i}[\tau\mapsto\xi(\tau,y(\tau))](t)=0.
\end{multline*}
Now, using generalized Euler--Lagrange equation \eqref{eq:EL:SEV} and formulas
\eqref{eq:BD:SEV} and \eqref{eq:BI:SEV} we arrive to \eqref{eq:NTH:SEV}.
\end{proof}

\begin{example}
Let $c\in\R$, $P=(P_1,\dots,P_n)$ with $P_i=\langle a_i,t_i,b_i,\lambda_i,\mu_i\rangle$
and $y\in C^1(\bar{\Omega}_n;\R)$ with $B_{P_i}[y]\in C(\bar{\Omega}_n;\R)$, $i=1,\dots,n$.
We consider one--parameter family of infinitesimal transformations
\begin{equation}
\label{eq:SEV:2}
\hat{y}(t)=y(t)+\varepsilon c+o(\varepsilon),
\end{equation}
and the Lagrangian $F\left(\nabla_{B_P}[y](t),t\right)$
Then we have
\begin{equation*}
F\left(\nabla_{B_P}[y](t),t\right)
=F\left(\nabla_{B_P}[\hat{y}](t),t\right).
\end{equation*}
Hence, $F$ is invariant under \eqref{eq:SEV:2} and Theorem~\ref{thm:NTH:SEV} asserts that
\begin{equation}
\sum\limits_{i=1}^n\mathbf{D}_i\left[c,\partial_{n+i}F\left(\nabla_{B_P}[y](t),t\right)\right]=0.
\end{equation}
\end{example}

Similarly to the previous sections, one can obtain from Theorem~\ref{thm:NTH:SEV} results regarding
to constant and variable order fractional integrals and derivatives.

\begin{corollary}
\index{Noether's Theorem!of several variables!for problems with Riemann--Liouville fractional integrals and Caputo fractional derivatives}
Suppose that $\alpha=(\alpha_1,\dots,\alpha_n)\in(0,1)^n$ and that
\begin{equation*}
F(y(t),\nabla_{I^{1-\alpha}}[y](t),\nabla[y](t),\nabla_{D^{\alpha}}[y](t),t)\\
=F(\hat{y}(t),\nabla_{I^{1-\alpha}}[\hat{y}](t),\nabla[\hat{y}](t),\nabla_{D^{\alpha}}[\hat{y}](t),t),
\end{equation*}
where $y\in C^1(\bar{\Omega}_n;\R)$ and $\hat{y}$ is the family \eqref{eq:Tr:Mult}.
Then all solutions of the Euler--Lagrange equation \eqref{eq:M:EL} satisfy
\begin{equation*}
\sum\limits_{i=1}^n \mathbf{I}_i^{1-\alpha_i}\left[\xi,\partial_{1+i}F\right]
+\frac{\partial}{\partial t_i}\left(\xi\cdot \partial_{1+n+i}F\right)
+\mathbf{D}_i^{\alpha_i}\left[\xi,\partial_{1+2n+i}F\right]=0,
\end{equation*}
where
\begin{equation*}
\mathbf{D}_i^{\alpha_i}[f,g]:=-f\cdot \Drp[g]+g\cdot \Dclp[f],
\end{equation*}
\begin{equation*}
\mathbf{I}_i^{1-\alpha_i}[f,g]:=-f\cdot \Ircp[g]+g\cdot \Ilcp[f],~i=1,\dots,n,
\end{equation*}
function $\xi$ is taken in $(t,y(t))$ and functions $\partial_j F$ are evaluated at
$(y(t),\nabla_{I^{1-\alpha}}[y](t),\nabla[y](t),\nabla_{D^{\alpha}}[y](t),t)$
for $j=1,\dots,3n$.
\end{corollary}

\begin{corollary}
\index{Noether's Theorem!of several variables!%
for problems with variable order Riemann--Liouville fractional integrals and variable order Caputo fractional derivatives}
Let us assume that $y\in C^1(\bar{\Omega}_n;\R)$ with ${_{a_i}}\textsl{I}^{1-\alpha_i(\cdot)}_{t_i}[y]$,
${^{C}_{a_i}}\textsl{D}^{\alpha_i(\cdot)}_{t_i}[y]\in C(\bar{\Omega}_n;\R)$,
$\alpha_i:[0,b_i-a_i]\rightarrow [0,1]$, and that
\begin{equation*}
F(y(t),\nabla_{I}[y](t),\nabla[y](t),\nabla_{D}[y](t),t)
=F(\hat{y}(t),\nabla_{I}[\hat{y}](t),\nabla[\hat{y}](t),\nabla_{D}[\hat{y}](t),t),
\end{equation*}
where $\hat{y}$ is the family \eqref{eq:Tr:Mult} such that $t\mapsto\xi(t,y(t))\in C^1(\bar{\Omega}_n;\R)$
with ${_{a_i}}\textsl{I}^{1-\alpha_i(\cdot)}_{t_i}[\tau\mapsto\xi(\tau,y(\tau))]$,
${^{C}_{a_i}}\textsl{D}^{\alpha_i(\cdot)}_{t_i}[\tau\mapsto\xi(\tau,y(\tau))]\in C(\bar{\Omega}_n;\R)$. Then
\begin{equation*}
\sum\limits_{i=1}^n \mathbf{I}_i^{1-\alpha_i(\cdot,\cdot)}\left[\xi,\partial_{1+i}F\right]
+\frac{\partial}{\partial t_i}\left(\xi\cdot \partial_{1+n+i}F\right)
+\mathbf{D}_i^{\alpha_i(\cdot,\cdot)}\left[\xi,\partial_{1+2n+i}F\right]=0,
\end{equation*}
along any solution of Euler--Lagrange equation \eqref{eq:M:EL:2}, where
\begin{equation*}
\mathbf{D}_i^{\alpha_i(\cdot,\cdot)}[f,g]:=-f\cdot {_{t_i}}\textsl{D}^{\alpha_i(\cdot,\cdot)}_{b_i}[g]
+g\cdot {^{C}_{a_i}}\textsl{D}^{\alpha_i(\cdot,\cdot)}_{t_i}[f],
\end{equation*}
\begin{equation*}
\mathbf{I}_i^{1-\alpha_i(\cdot,\cdot)}[f,g]:=-f\cdot {_{t_i}}\textsl{I}^{1-\alpha_i(\cdot,\cdot)}_{b_i}[g]
+g\cdot {_{a_i}}\textsl{I}^{1-\alpha_i(\cdot,\cdot)}_{t_i}[f],~i=1,\dots,n,
\end{equation*}
function $\xi$ is taken in $(t,y(t))$ and functions $\partial_j F$ are evaluated at
$(y(t),\nabla_{I}[y](t),\nabla[y](t),\nabla_{D}[y](t),t)$ for $j=1,\dots,3n$.
\end{corollary}


\section{Conclusion}

In this chapter we unified, subsumed and significantly extended the necessary optimality conditions available
in the literature of the fractional calculus of variations. It should be mentioned, however, that since fractional
operators are nonlocal, it can be extremely challenging to find analytical solutions to fractional problems
of the calculus of variations and, in many cases, solutions may not exist. Here, we gave several examples
with analytic solutions, and many more can be found borrowing different kernels from the book \cite{book:Polyanin}.
On the other hand, one can easily choose examples for which the fractional Euler--Lagrange differential equations
are hard to solve, and in that case one needs to use numerical methods \cite{Shakoor:01}. However, in the absence
of existence, the necessary conditions for extremality are vacuous: one can not characterize an entity that does
not exist in the first place. For solving a problem of the fractional calculus of variations one should proceed
along the following three steps: (i) first, prove that a solution to the problem exists; (ii) second, verify
the applicability of necessary optimality conditions; (iii) finally, apply the necessary conditions
which identify the extremals (the candidates). Further elimination, if necessary, identifies the minimizer(s)
of the problem. All three steps in the above procedure are crucial. As mentioned by Young in \cite{Young},
the calculus of variations has born from the study of necessary optimality conditions, but any such theory
is ``naive'' until the existence of minimizers is verified. Therefore, in the next chapter,
we shall follow the direct approach, first proving that a solution exists and next finding candidates
with the help of the necessary optimality conditions.


\section{State of the Art}

The results of this chapter were published in
\cite{variable,Hawaii,GreenThm,MyID:226,FVC_Gen_Int,MyID:207,FVC_Sev,CLandFR,tatiana,Ja}
and were presented by the author at several international conferences:
\begin{itemize}
\item International Mathematical Symposium on Mathematical Theory of
Networks and Systems, 9th July-- 13th July 2012, Melbourne, Australia. Presentation:
{\it A Generalized Fractional Calculus with Applications} \cite{MyID:233}. Talk at invited session.
\item 5th Symposium on Fractional Differentiation and Its Applications,
14th May -- 17th May 2012, Nanjing, China.
Presentation: {\it Generalized Fractional Green's Theorem}. Talk at invited session.
\item Transform methods and special functions, 20th October-- 23rd October 2011, Sofia, Bulgaria.
Presentation: {\it Fractional Variational Calculus of Variable Order}.
\item SIAM Conference on Control and Its Applications, 8th July-- 12th July 2011,
Baltimore, United States of America. Presentation:
{\it Fractional Variational Calculus with Classical and Combined Caputo Derivatives}.
Talk at invited session.
\item International Conference on Isoperimetric Problem of Queen Dido and its Mathematical Ramifications,
24th May-- 29th May 2010, Tunis, Tunisia. Presentation: {\it Fractional Isoperimetric Problems}.
\item Workshop on Variational Analysis and Applications, October 2011, \'Evora, Portugal.
Presentation: {\it Fractional Variational Calculus with Classical and Combined Caputo Derivatives}.
\item The 4th IFAC Workshop on Fractional Differentiation and its Applications,
18th October-- 20th October 2010, Badajoz, Spain. Presentation: {\it Calculus of Variations
with Fractional and Classical Derivatives} \cite{Tatiana:Spain2010}.
\item 8th EUROPT Workshop on Advances in Continuous Optimization, 9th July -- 10th July 2010,
Aveiro, Portugal. Presentation: {\it Necessary Optimality Conditions for Fractional Variational Problem}.
\end{itemize}
Paper \cite{GreenThm} entitled {\it Green's theorem for generalized fractional derivatives} \cite{MyID:236}
received the Grunvald-Letnikov Award, Best Student Paper (Theory) at the conference 'Fractional Differentiation
and its Applications' in May 2012, Nanjing, China. More detailed information
can be found at \url{http://em.hhu.edu.cn/fda12/Awards.html}.


\clearpage{\thispagestyle{empty}\cleardoublepage}


\chapter{Direct Methods in Fractional Calculus of Variations}
\label{ch:di}

In this chapter we study the fundamental problem of calculus of variations with a Lagrangian depending
on generalized fractional integrals and generalized fractional derivatives. In contrast with the standard
approach presented in Chapter~\ref{ch:st}, we use direct methods to address the problem of finding minima
to generalized fractional functionals. First, we prove existence of solutions in an appropriate space
of functions and under suitable  assumptions of regularity, coercivity and convexity. Next, we proceed
with application of an optimality condition, and finish examining the candidates to arrive to solution.

Let us briefly describe the main contents of the chapter. In Section~\ref{section1} we prove a Tonelli
type theorem ensuring existence of minimizers for generalized fractional functionals. We also give sufficient
conditions for a regular Lagrangian and for a coercive functional. Section~\ref{sec:sob} is devoted to a
necessary optimality condition for minimizers. In the last Section~\ref{sec:reg} we improve our results
assuming more regularity of the Lagrangian and generalized fractional operators.


\section{Existence of a Minimizer for a Generalized Functional}
\label{section1}

Let us recall that $1<p,q<\infty$ and $\frac{1}{p}+\frac{1}{q}=1$. In this section,
our aim is to give sufficient conditions ensuring the existence of a minimizer
for the following generalized Lagrangian functional:
\begin{equation*}
\fonction{\mathcal{I}}{\mathcal{A}}{\R}{y}{\int\limits_a^b F(y(t),K_P[y](t),\dot{y}(t),B_P[y](t),t) \; dt ,}
\end{equation*}
where $\mathcal{A}$ is a weakly closed subset of $W^{1,p}(a,b;\R)$, $\dot{y}$ denotes the derivative
of $y$, $K_P$ is the generalized fractional integral with a kernel in $L^q(\Delta;\R)$,
$B_P=\frac{d}{dt}\circ K_P$ is the generalized fractional Caputo derivative,
$P=\langle a,t,b,\lambda,\mu\rangle$ is a set of parameters and $F$ is a Lagrangian of class $C^1$:
\begin{equation*}
\fonction{F}{\R^4 \times [a,b]}{\R}{(x_1,x_2,x_3,x_4,t)}{F(x_1,x_2,x_3,x_4,t).}
\end{equation*}


\subsection{A Tonelli-type Theorem}
\label{section11}

In this section we state a Tonelli-type theorem ensuring the existence of a minimizer for
$\mathcal{I}$ with the help of general assumptions of regularity, coercivity and convexity.
These three hypothesis are usual in the classical case, see \cite{book:Dacorogna}. Precisely:

\begin{definition}
\label{defreg}
A Lagrangian $F$ is said to be regular if it satisfies:
\begin{itemize}
\item $t\mapsto F (y(t),K_P[y](t),\dot{y}(t),B_P[y](t),t) \in L^1(a,b;\R)$;
\item $t\mapsto \partial_1 F (y(t),K_P[y](t),\dot{y}(t),B_P[y](t),t) \in L^1(a,b;\R)$;
\item $t\mapsto \partial_2 F (y(t),K_P[y](t),\dot{y}(t),B_P[y](t),t) \in L^{p}(a,b;\R)$;
\item $t\mapsto \partial_3 F (y(t),K_P[y](t),\dot{y}(t),B_P[y](t),t) \in L^{q}(a,b;\R)$;
\item $t\mapsto \partial_4 F (y(t),K_P[y](t),\dot{y}(t),B_P[y](t),t) \in L^{p}(a,b;\R)$,
\end{itemize}
for any $y \in W^{1,p}(a,b;\R)$:
\end{definition}

\begin{definition}\label{defcoer}
Functional $\mathcal{I}$ is said to be coercive on $\mathcal{A}$ if it satisfies:
\begin{equation*}
\lim\limits_{\substack{\left\| y \right\|_{W^{1,p}} \to \infty \\
y \in \mathcal{A} }} \mathcal{I} (y) = +\infty.
\end{equation*}
\end{definition}

We are now in position to state the following result:

\begin{theorem}[Tonelli-type Theorem]
\label{thm:tonelli}
\index{Tonelli-type theorem}
Let us assume that:
\begin{itemize}
\item $F$ is regular;
\item $\mathcal{I}$ is coercive on $\mathcal{A}$;
\item $F(\cdot,t)$ is convex on $(\R^d)^4$ for any $t \in [a,b]$.
\end{itemize}
Then, there exists a minimizer for $\mathcal{I}$.
\end{theorem}

\begin{proof}
Since $F$ is regular, for any $y \in \mathcal{A}$, $t\mapsto F (y(t),K_P[y](t),\dot{y}(t),B_P[y](t),t) \in L^1(a,b;\R)$
and then $\mathcal{I} (y)$ exists in $\R$. Let us introduce a minimizing sequence
$(y_n)_{n \in \N} \subset \mathcal{A}$ satisfying:
\begin{equation*}
\mathcal{I} (y_n) \longrightarrow \inf\limits_{y \in \mathcal{A} }\mathcal{I}(y) < +\infty.
\end{equation*}
Since $\mathcal{I}$ is coercive, $(y_n)_{n \in \N}$ is bounded in $W^{1,p}(a,b;\R)$.
Since $W^{1,p}(a,b;\R)$ is a reflexive Banach space, it exists a subsequence of $(y_n)_{n \in \N}$
weakly convergent in $W^{1,p}(a,b;\R)$. In the following, we still denote this subsequence by
$(y_n)_{n \in \N}$ and we denote by $\bar{y}$ its weak limit. Since $\mathcal{A}$ is a weakly
closed subset of $W^{1,p}(a,b;\R)$, $\y\in \mathcal{A}$. Finally, using the convexity of $F$,
we have for any $n \in \N$:
\begin{multline}\label{eq:1}
\mathcal{I} (y_n) \geq \mathcal{I} (\y) + \int\limits_a^b \partial_1 F \cdot (y_n(t)
-\bar{y}(t)) + \partial_2 F \cdot (K_P[y_n](t)-K_P[\bar{y}](t))\\
+ \partial_3 F \cdot (\dot{y}_n(t)-\dot{\bar{y}}(t))
+ \partial_4 F \cdot (B_P[y_n](t)-B_P[\y](t))  \; dt,
\end{multline}
where $\partial_i F$ are taken in $(\y(t),K_P[\y](t),\dot{\y}(t),B_P[\y](t),t)$ for any $i=1,2,3,4$.

Now, from these four following facts:
\begin{itemize}
\item $F$ is regular;
\item $y_n \xrightharpoonup[]{W^{1,p}} \bar{y}$;
\item $K_P$ is linear bounded from $L^p(a,b;\R)$ to $L^q(a,b;\R)$;
\item the compact embedding $W^{1,p}(a,b;\R) \hooktwoheadrightarrow C([a,b];\R)$ holds;
\end{itemize}
one can easily conclude that:
\begin{itemize}
\item $t\mapsto\partial_3 F(\y(t),K_P[\y](t),\dot{\y}(t),B_P[\y](t),t)
\in L^{q}(a,b;\R)$ and $\dot{y}_n \xrightharpoonup[]{L^p} \dot{\bar{y}}$;
\item $t\mapsto\partial_4 F(\y(t),K_P[\y](t),\dot{\y}(t),B_P[\y](t),t)
\in L^{p}(a,b;\R)$ and $B_P[y_n] \xrightharpoonup[]{L^q} B_P[\bar{y}]$;
\item $t\mapsto\partial_1 F(\y(t),K_P[\y](t),\dot{\y}(t),B_P[\y](t),t)
\in L^1(a,b;\R)$ and $y_n \xrightarrow[]{L^\infty} \bar{y}$;
\item $t\mapsto\partial_2 F(\y(t),K_P[\y](t),\dot{\y}(t),B_P[\y](t),t)
\in L^{p}(a,b;\R)$ and $K_P[y_n] \xrightarrow[]{L^q} K[\bar{y}]$.
\end{itemize}
Finally, when $n$ tends to $\infty$ in the inequality \eqref{eq:1}, we obtain:
\begin{equation*}
\inf\limits_{y \in \mathcal{A}} \mathcal{I} (y) \geq \mathcal{I} (\bar{y}) \in \R,
\end{equation*}
which completes the proof.
\end{proof}

The first two hypothesis of Theorem~\ref{thm:tonelli} are very general. Consequently,
in Sections~\ref{section12} and \ref{section13}, we give concrete assumptions on $F$
ensuring its regularity and the coercivity of $\mathcal{I}$.

The last hypothesis of convexity is strong. Nevertheless, from more regularity assumptions
on $F$, on $K_P$ and on $B_P$, we prove in Section~\ref{sec:reg} that we can provide versions
of Theorem~\ref{thm:tonelli} with weaker convexity assumptions.


\subsection{Sufficient Condition for a Regular Lagrangians}
\label{section12}
\index{Set $\PP_M$}

In this section we give a sufficient condition on $F$ implying its regularity.
First, for any $M \geq 1$, let us define the set $\PP_M$ of maps $P$ such that
for any $(x_1,x_2,x_3,x_4,t) \in (\R)^4 \times [a,b]$:
\begin{equation*}
P(x_1,x_2,x_3,x_4,t) = \sum_{k=0}^{N} c_k(x_1,t) \left| x_2 \right|^{d_{2,k}}
\left| x_3 \right|^{d_{3,k}} \left| x_4 \right|^{d_{4,k}},
\end{equation*}
with $ N \in \N$ and where, for any $k=0,\ldots,N$,
$\fonctionsansdef{c_k}{\R \times [a,b]}{\R^+}$ is continuous and
$(d_{2,k},d_{3,k},d_{4,k}) \in [0,q] \times [0,p] \times [0,q]$
satisfies $d_{2,k}+ (q/p) d_{3,k} + d_{4,k} \leq (q/M)$.

The following lemma shows the interest of sets $\PP_M$.

\begin{lemma}
\label{lemregP}
Let $M \geq 1$ and $P \in \PP_M$. Then:
\begin{equation*}
\forall y \in \W, \; t\mapsto
P(y(t),K_P[y](t),\dot{y}(t),B_P[y](t),t) \in L^{M}(a,b;\R).
\end{equation*}
\end{lemma}

\begin{proof}
For any $k=0,\ldots,N$, $c_k(u,t)$ is continuous and then it is in $L^\infty(a,b;\R)$.
We also have $\left| K_P[y] \right|^{d_{2,k}} \in L^{q/d_{2,k}}(a,b;\R)$,
$\left| \dot{y} \right|^{d_{3,k}} \in L^{p/d_{3,k}}(a,b;\R)$ and
$\left| B_P[y] \right|^{d_{4,k}} \in L^{q/d_{4,k}}(a,b;\R)$. Consequently:
\begin{equation}
c_k(u,t) \left| K_P[y] \right|^{d_{2,k}}  \left| \dot{y} \right|^{d_{3,k}}
\left|  B_P[y] \right|^{d_{4,k}}  \in L^r(a,b;\R),
\end{equation}
with $r=q/(d_{2,k} + (q/p) d_{3,k} +d_{4,k}) \geq M$. The proof is complete.
\end{proof}

Then, from this previous Lemma, one can easily obtain the following proposition.

\begin{proposition}\label{prop1}\index{Sufficient condition for regular Lagrangian}
Let us assume that there exist $P_0 \in \PP_{1}$, $P_1 \in \PP_{1}$, $P_2 \in \PP_{p}$,
$P_3 \in \PP_{q}$ and $P_4 \in \PP_{p}$ such that for any $(x_1,x_2,x_3,x_4,t) \in (\R)^4 \times [a,b]$:
\begin{itemize}
\item $ \left| F(x_1,x_2,x_3,x_4,t) \right| \leq P_0(x_1,x_2,x_3,x_4,t) $;
\item $ \left| \partial_1 F(x_1,x_2,x_3,x_4,t) \right| \leq P_1(x_1,x_2,x_3,x_4,t) $;
\item $ \left| \partial_2 F(x_1,x_2,x_3,x_4,t) \right| \leq P_2(x_1,x_2,x_3,x_4,t) $;
\item $ \left| \partial_3 F(x_1,x_2,x_3,x_4,t) \right| \leq P_3(x_1,x_2,x_3,x_4,t) $;
\item $ \left| \partial_4 F(x_1,x_2,x_3,x_4,t) \right| \leq P_4(x_1,x_2,x_3,x_4,t) $.
\end{itemize}
Then, $F$ is regular.
\end{proposition}

This last proposition states that if the norms of $F$ and of its partial derivatives are controlled
from above by elements of $\PP_M$, then $F$ is regular. We will see some examples in Section~\ref{sec:ex:lagr}.


\subsection{Sufficient Condition for a Coercive Functionals}
\label{section13}

The definition of coercivity for functional $\mathcal{I}$ is strongly dependent on the considered
set $\mathcal{A}$. Consequently, in this section, we will consider an example of set $\mathcal{A}$
and we will give a sufficient condition on $F$ ensuring the coercivity of $\mathcal{I}$ in this case.

Precisely, let us consider $y_a \in \R$ and $\mathcal{A} = W_a^{1,p}(a,b;\R)$ where
$W_a^{1,p}(a,b;\R) := \{ y \in \W, \; y(a)=y_a \}$. From the compact embedding
$\W \hooktwoheadrightarrow C([a,b];\R)$, $W_a^{1,p}(a,b;\R)$ is weakly closed in $\W$.

An important consequence of such a choice of set $\mathcal{A}$ is given by the following lemma.

\begin{lemma}
\label{lemdom}
There exist $A_0$, $A_1 \geq 0$ such that for any $y \in W_a^{1,p}(a,b;\R)$:
\begin{itemize}
\item $\left\| y \right\|_{L^\infty} \leq A_0 \left\| \dot{y} \right\|_{L^p} + A_1$;
\item $\left\| K_P[y] \right\|_{L^q} \leq A_0 \left\| \dot{y} \right\|_{L^p} + A_1$;
\item $\left\| B_P[y] \right\|_{L^q} \leq A_0 \left\| \dot{y} \right\|_{L^p} + A_1$.
\end{itemize}
\end{lemma}

\begin{proof}
The last inequality comes from the boundedness of $K_P$. Let us consider the second one.
For any $y \in W_a^{1,p}(a,b;\R)$, we have $\left\| y\right\|_{L^p}
\leq \left\| y-y_a \right\|_{L^p} + \left\| y_a \right\|_{L^p} \leq (b-a) \left\| \dot{y} \right\|_{L^p}
+ (b-a)^{1/p} \left| y_a \right|$. We conclude using again the boundedness of $K_P$. Now, let us consider
the first inequality. For any $y \in W^{1,p}_a(a,b;\R)$, we have $\left\| y \right\|_{L^\infty}
\leq \left\| y-y_a \right\|_{L^\infty} + \left| y_a \right| \leq \left\| \dot{y} \right\|_{L^1}
+ \left| y_a \right| \leq (b-a)^{1/q} \left\| \dot{y} \right\|_{L^p} + \left| y_a \right|$.
Finally, we have just to define $A_0$ and $A_1$ as the maxima of the appearing constants.
The proof is complete.
\end{proof}

Precisely, Lemma~\ref{lemdom} states the {\it affine domination} of the term $\left\| \dot{y} \right\|_{L^p}$
on the terms $\left\| y \right\|_{L^\infty}$, $\left\| K_P[y] \right\|_{L^q}$ and $\left\| B_P[y] \right\|_{L^q}$
for any $y \in W_a^{1,p}(a,b;\R)$. This characteristic of $W_a^{1,p}(a,b;\R)$ leads us to give the following
sufficient condition for a coercive functional $\mathcal{I}$.

\begin{proposition}
\label{prop1b}
Let us assume that for any $(x_1,x_2,x_3,x_4,t) \in \R^4 \times [a,b]$:
\begin{equation*}
F(x_1,x_2,x_3,x_4,t) \geq c_0 \left| x_3 \right|^{p} + \sum_{k=1}^{N} c_k \left| x_1 \right|^{d_{1,k}}
\left| x_2 \right|^{d_{2,k}} \left| x_3 \right|^{d_{3,k}} \left| x_4 \right|^{d_{4,k}},
\end{equation*}
with $c_0 > 0$ and $N \in \N$ and where, for any $k=1,\ldots,N$, $c_k \in \R$ and
$(d_{1,k},d_{2,k},d_{3,k},d_{4,k}) \in \R^+
\times [0,q] \times [0,p] \times [0,q]$ satisfies:
\begin{equation*}
d_{2,k}+(q/p)d_{3,k}+d_{4,k} \leq q \quad \text{and}
\quad 0 \leq d_{1,k}+d_{2,k}+d_{3,k}+d_{4,k} < p.
\end{equation*}
Then, $\mathcal{I}$ is coercive on $W_a^{1,p}(a,b;\R)$.
\end{proposition}

\begin{proof}
Let us define $r_k = q/(d_{2,k}+d_{4,k}+(q/p)d_{3,k}) \geq 1$ and let $r_k'$
denote the adjoint of $r_k$ i.e., $r_k'=\frac{r_k}{r_{k-1}}$. Using H\"older's inequality,
one can easily prove that, for any $y \in W_a^{1,p}(a,b;\R)$, we have:
\begin{equation*}
\mathcal{I} (y) \geq c_0 \left\| \dot{y} \right\|_{L^p} - (b-a)^{1/r'}
\sum_{k=1}^N \left| c_k \right| \left\| y \right\|^{d_{1,k}}_{L^\infty}
\left\| K_P[y] \right\|^{d_{2,k}}_{L^q}  \left\| \dot{y} \right\|^{d_{3,k}}_{L^p}
\left\| B_P[y] \right\|^{d_{4,k}}_{L^q}.
\end{equation*}
From the "affine domination" of the term $ \left\|\dot{y} \right\|_{L^p}$ on the terms
$\left\|  y\right\|_{L^\infty}$, $\left\| K_P[y] \right\|_{L^q}$ and $\left\| B_P[y] \right\|_{L^q}$
for any $y \in W_a^{1,p}(a,b;\R)$ (see Lemma~\ref{lemdom}) and from the assumption
$0 \leq d_{1,k}+d_{2,k}+d_{3,k}+d_{4,k} < p$, we obtain that:
\begin{equation*}
\lim\limits_{\substack{\left\| \dot{y} \right\|_{L^p} \to \infty \\
y \in W_a^{1,p}(a,b;\R)}} \mathcal{I} (y) = +\infty.
\end{equation*}
Finally, from Lemma \ref{lemdom}, we also have in $W_a^{1,p}(a,b;\R)$:
\begin{equation*}
\left\| \dot{y} \right\|_{L^p} \to \infty
\Longleftrightarrow \left\| y \right\|_{W^{1,p}} \to \infty .
\end{equation*}
Consequently, $\mathcal{I}$ is coercive on $W_a^{1,p}(a,b;\R)$. The proof is complete.
\end{proof}

In this section, we have studied the case where $\mathcal{A}$ is the weakly closed subset of
$\W$ satisfying the initial condition $y(a)=y_a$. For other examples of set $\mathcal{A}$,
let us note that all the results of this section are still valid when:
\begin{itemize}
\item $\mathcal{A}$ is weakly closed subset of $\W$ satisfying a final condition in $t=b$;
\item $\mathcal{A}$ is weakly closed subset of $\W$ satisfying two boundary conditions in $t=a$ and in $t=b$.
\end{itemize}


\subsection{Examples of Lagrangians}
\label{sec:ex:lagr}

In this section, we give several examples of a convex Lagrangian $F$ satisfying assumptions
of Propositions~\ref{prop1} and~\ref{prop1b}. In consequence, they are examples of application
of Theorem~\ref{thm:tonelli} in the case $\mathcal{A}=W_a^{1,p}(a,b;\R)$.
\begin{example}
\label{ex:lagr1}
The most classical examples of a Lagrangian are the quadratic ones. Let us consider the following:
\begin{equation*}
F(x_1,x_2,x_3,x_4,t)=c(t)+\frac{1}{2}\sum\limits_{i=1}^4\left|x_i\right|^2,
\end{equation*}
where $c:[a,b]\rightarrow\R$ is of class $C^1$. One can easily check that $F$ satisfies
the assumptions of Propositions~\ref{prop1} and~\ref{prop1b} with $p=q=2$. Moreover,
$F$ satisfies the convexity hypothesis of Theorem~\ref{thm:tonelli}. Consequently,
one can conclude that there exists a minimizer of $\mathcal{I}$ defined on $W_a^{1,2}(a,b;\R)$.
\end{example}

\begin{example}
Let us consider $p=q=2$ and let us still denote by $F$ the Lagrangian defined in Example~\ref{ex:lagr1}.
To obtain a more general example, one can define a Lagrangian $F_1$ from $F$ as a time-dependent homothetic
transformation and/or translation of its variables. Precisely:
\begin{equation}
\label{ex:eq:lagr}
F_1(x_1,x_2,x_3,x_4,t)=F(c_1(t)x_1+c_1^0(t),c_2(t)x_2
+c_2^0(t),c_3(t)x_3+c_3^0(t),c_4(t)x_4+c_4^0(t),t),
\end{equation}
where $c_i:[a,b]\rightarrow\R$ and $c_i^0:[a,b]\rightarrow\R$ are of class $C^1$
for any $i=1,2,3,4$. In this case, $F_1$ also satisfies convexity hypothesis of
Theorem~\ref{thm:tonelli} and the assumptions of Proposition~\ref{prop1}. Moreover,
if $c_3$ is with values in $\R^{+}$, then $F_1$ also satisfies the assumptions
of Proposition~\ref{prop1b}.
\end{example}

One should be careful: this last remark is not available in more general context. Precisely,
if a general Lagrangian $F$ satisfies the convexity hypothesis of Theorem~\ref{thm:tonelli}
and assumptions of Propositions~\ref{prop1} and~\ref{prop1b}, then Lagrangian $F_1$ obtained
by \eqref{ex:eq:lagr} also satisfies the convexity hypothesis of Theorem~\ref{thm:tonelli}
and the assumptions of Proposition~\ref{prop1}. Nevertheless, the assumption
of Proposition~\ref{prop1b} can be lost by this process.

\begin{example}
\label{ex:lagr3}
We can also study quasi--linear examples given by Lagrangians of the type
\begin{equation*}
F(x_1,x_2,x_3,x_4,t)=c(t)+\frac{1}{p}\left|x_3\right|^p+\sum\limits_{i=1}^{4}f_i(t)\cdot x_i,
\end{equation*}
where $c:[a,b]\rightarrow\R$ and for any $i=1,2,3,4$, $f_i:[a,b]\rightarrow\R$ are of class $C^1$.
In this case, $F$ satisfies the assumptions of Propositions~\ref{prop1} and~\ref{prop1b}.
Consequently, since $F$ satisfies the convexity hypothesis of Theorem~\ref{thm:tonelli},
one can conclude that there exists a minimizer of $\mathcal{I}$ defined on $W_a^{1,p}(a,b;\R)$.
\end{example}

The most important constraint in order to apply Theorem~\ref{thm:tonelli} is the convexity hypothesis.
This is the reason why the previous examples concern convex quasi--polynomial Lagrangians. Nevertheless,
in Section~\ref{sec:reg}, we are going to provide some improved versions of Theorem~\ref{thm:tonelli}
with weaker convexity assumptions. This will be allowed by more regularity hypotheses on $F$ and/or
on $K_P$ and $B_P$. We refer to Section~\ref{sec:reg} for more details.


\section{Necessary Optimality Condition for a Minimizer}
\label{sec:sob}

In this section, we assume additionally that:
\begin{itemize}
\item $F$ satisfies the assumptions of Proposition \ref{prop1} (in particular, $F$ is regular);
\item $\mathcal{A}$ satisfies the following condition:
\begin{equation}
\forall y \in \mathcal{A}, \; \forall \eta \in \CC^\infty_c, \; \exists 0 < \varepsilon \leq 1, \;
\forall \left| h \right| \leq \varepsilon, \; y+h\eta \in \mathcal{A}.
\end{equation}
\end{itemize}
The assumption on $\mathcal{A}$ is satisfied if $\mathcal{A}+\CC^\infty_c \subset \mathcal{A}$
(for example $\mathcal{A} = W_a^{1,p}(a,b;\R)$ in Section~\ref{section13}). By $\CC^\infty_c$
we denote the space of infinitely differentiable functions compactly supported in $(a,b)$.

In the next theorem we will make use of the following Lemma.

\begin{lemma}
\label{lemdomP}
Let $M \geq 1$ and $P \in \PP_M$. Then, for any $y \in \mathcal{A}$ and any $\eta \in \CC^\infty_c$,
it exists $g \in L^M (a,b;\R^+)$ such that for any $ h \in [-\varepsilon,\varepsilon]$:
\begin{equation}
P(y+h\eta,K_P[y]+hK_P[\eta],\dot{y}+h\dot{\eta},B_P[y]+hB_P[\eta],t ) \leq g.
\end{equation}
\end{lemma}

\begin{proof}
Indeed, for any $k=0,\ldots,N$, for almost all $t \in (a,b)$ and for any
$h \in [-\varepsilon,\varepsilon]$, we have:
\begin{multline}
\label{eqproof}
c_k(y(t)+h\eta(t),t) \left| K_P[y](t) + h K_P[\eta](t) \right|^{d_{2,k}}
\left| \dot{y}(t)+h \dot{\eta}(t) \right|^{d_{3,k}} \left| B_P[y](t)
+ h B_P[\eta](t) \right|^{d_{4,k}} \\ \leq \bar{c}_k  (\underbrace{\left| K_P[y](t)\right|^{d_{2,k}}
+ \left| K_P[\eta](t) \right|^{d_{2,k}}}_{\in L^{q/d_{2,k}}(a,b;\R)})(\underbrace{\left| \dot{y}(t) \right|^{d_{3,k}}
+ \left| \dot{\eta}(t)\right| ^{d_{3,k}}}_{\in L^{p/d_{3,k}}(a,b;\R)})(\underbrace{\left| B_P[y](t) \right|^{d_{4,k}}
+ \left| B_P[\eta](t) \right|^{d_{4,k}}}_{\in L^{q/d_{4,k}}(a,b;\R)}),
\end{multline}
where $\bar{c}_k = 2^{d_{2,k}+d_{3,k}+d_{4,k}} \max\limits_{[a,b]\times[-\varepsilon,\varepsilon]} c_k(y(t)+h\eta(t),t)$
exists in $\R$ because $c_k$, $y$ and $\eta$ are continuous. Since $d_{2,k} + (q/p) d_{3,k} +d_{4,k} \leq (q/M)$,
the right term in inequality \eqref{eqproof} is in $L^M (a,b;\R^+)$ and is independent of $h$. The proof is complete.
\end{proof}

Finally, from this previous Lemma, we can prove the following result:

\begin{theorem}
\label{thm:EL:Sob}
Let us assume that $F$ satisfies assumptions of Proposition~\ref{prop1}, $\mathcal{I}$
is coercive and $F(\cdot,t)$ is convex on $\R^4$ for any $t \in [a,b]$. Then, the minimizer
$\y\in\mathcal{A}$ of $\mathcal{I}$ (given by Theorem~\ref{thm:tonelli})
satisfies the generalized Euler-Lagrange equation
\begin{equation}
\label{gel}
\tag{GEL}
\index{Euler--Lagrange equation!for problems with generalized fractional operators}
\dfrac{d}{dt} \big( \partial_3 F (\star_{\y})(t)+ K_{P^*} [\partial_4 F(\star_{\y})(\tau)](t) \big)
= \partial_1 F(\star_{\y})(t)+K_{P^*} [\partial_2 F(\star_{\y})(\tau)](t),
\end{equation}
for almost all $t\in (a,b)$, where $(\star_{\y})(t)=(\y(t),K_P[\y](t),\dot{\y}(t),B_P[\y](t),t)$.
\end{theorem}

\begin{proof}
Since $F$ satisfies the assumptions of Proposition \ref{prop1}, $F$ is regular. Consequently,
from Theorem~\ref{thm:tonelli}, we know that $\mathcal{I}$
admits a minimizer $\y \in \mathcal{A}$ and then
\begin{equation}
\label{eq:enq}
\mathcal{I}(\y)\leq\mathcal{I}(\y+h\eta),
\end{equation}
for any $\left|h\right|\leq\varepsilon$ and every $\eta\in\CC_c^\infty$.
Let us define the following map:
\begin{equation*}
\fonction{\phi_{\y,\eta}}{[-\varepsilon,\varepsilon]}{\R}{h}{\mathcal{I} (\y+h\eta)
= \int\limits_a^b \psi_{\y,\eta} (t,h) \; dt,}
\end{equation*}
where
\begin{equation*}
\psi_{\y,\eta} (t,h) : = F (\y(t)+h\eta(t),K_P[\y](t)+hK_P[\eta](t),\dot{\y}(t)
+h\dot{\eta}(t),B_P[\y](t)+hB_P[\eta](t),t ).
\end{equation*}
First we want to prove that the following term:
\begin{equation}
\lim\limits_{h \to 0} \dfrac{\mathcal{I} (\y+h\eta) - \mathcal{I} (\y)}{h}
= \lim\limits_{h \to 0} \dfrac{\phi_{\y,\eta} (h) - \phi_{\y,\eta}(0)}{h}
= \phi_{\y,\eta}'(0)
\end{equation}
exists in $\R$. In order to differentiate $\phi_{\y,\eta}$, we use the theorem of differentiation
under the integral sign. Indeed, we have for almost all $t \in (a,b)$ that $\psi_{\y,\eta}(t,\cdot)$
is differentiable on $[-\varepsilon,\varepsilon]$ with:
\begin{multline}
\forall h \in [-\varepsilon,\varepsilon], \; \dfrac{\partial \psi_{\y,\eta}}{\partial h} (t,h) \\
= \partial_1 F(\star_{\y+h\eta})(t) \cdot \eta(t) + \partial_2 F(\star_{\y+h\eta})(t)\cdot K_P[\eta](t)
+ \partial_3 F(\star_{\y+h\eta})(t) \cdot \dot{\eta}(t) + \partial_4 F(\star_{\y+h\eta})(t) \cdot B_P[\eta](t).
\end{multline}
Since $F$ satisfies the assumptions of Proposition \ref{prop1}, from Lemma \ref{lemdomP} there exist
$g_1 \in L^1 (a,b;\R^+)$, $g_2 \in L^{p} (a,b;\R^+)$, $g_3 \in L^{q} (a,b;\R^+)$ and $g_4 \in L^{p} (a,b;\R^+)$
such that for any $h \in [-\varepsilon,\varepsilon]$ and for almost all $t \in (a,b)$:
\begin{equation}
\label{eqproof2}
\left| \dfrac{\partial \psi_{\y,\eta}}{\partial h} (t,h) \right| \leq g_1 (t) \left| \eta(t) \right|
+ g_2 (t) \left| K_P[\eta](t) \right| + g_3 (t) \left| \dot{\eta}(t) \right|
+ g_4 (t) \left| B_P[\eta](t)\right| .
\end{equation}
Since $\eta \in L^\infty(a,b;\R)$, $K_P[\eta] \in L^q(a,b;\R)$, $\dot{\eta} \in L^p(a,b;\R)$
and $B_P[\dot{\eta}] \in L^q(a,b;\R)$, we can conclude that the right term in inequality \eqref{eqproof2}
is in $L^1 (a,b;\R^+)$ and is independent of $h$. Consequently, we can use the theorem of differentiation
under the integral sign and we obtain that $\phi_{\y,\eta}$ is differentiable with
\begin{eqnarray}
\forall h \in [-\varepsilon,\varepsilon], \; \phi_{\y,\eta}'(h)
= \int\limits_a^b \dfrac{\partial \psi_{\y,\eta}}{\partial h} (t,h) \; dt .
\end{eqnarray}
Furthermore, inequality \eqref{eq:enq} implies
\begin{equation*}
\left.\frac{d}{dh}\phi_{\y,\eta}(h)\right|_{h=0}=0,
\end{equation*}
what can be written as
\begin{equation*}
\int\limits_a^b \partial_1 F(\star_{\y})(t) \cdot \eta(t) + \partial_2 F(\star_{\y})(t)
\cdot K_P[\eta](t) + \partial_3 F(\star_{\y})(t) \cdot \dot{\eta}(t)
+ \partial_4 F(\star_{\y})(t) \cdot B_P[\eta] (t)\; dt=0.
\end{equation*}
Moreover, applying integration by parts formula \eqref{eq03} one has
\begin{equation*}
\int\limits_a^b \big( \partial_1 F(\star_{\y})(t) + K_{P^*} [ \partial_2 F(\star_{\y})(\tau) ] \big) \cdot \eta(t)
+ \big( \partial_3 F(\star_{\y})(t) + K_{P^*} [\partial_4 F(\star_{\y})(\tau)] \big) \cdot \dot{\eta}(t) \; dt=0.
\end{equation*}
Then, taking an absolutely continuous anti-derivative $w$ of $t\mapsto\partial_1 F(\star_{\y})(t)
+ K_{P^*} [ \partial_2 F(\star_{\y})(\tau) ](t) \in L^1(a,b;\R)$, we obtain using integration by parts that:
\begin{equation*}
\int\limits_a^b \big( \partial_3 F(\star_{\y})(t) + K_{P^*} [\partial_4 F(\star_{\y})(\tau)](t)
- w(t) \big) \cdot \dot{\eta}(t) \; dt=0.
\end{equation*}
From the du Bois Reymond Lemma there exists a constant $C \in \R$ such that for almost all $t \in (a,b)$, we have:
\begin{equation}
\partial_3 F(\star_{\y})(t) + K^* [\partial_4 F(\star_{\y})(\tau)](t) = C + w(t).
\end{equation}
Since the right term is absolutely continuous, we can differentiate it almost everywhere on $(a,b)$.
Finally, we obtain that $\y$ is a minimizer of $\mathcal{I}$.
Then the following equation holds almost everywhere on $(a,b)$:
\begin{equation}
\dfrac{d}{dt} \big( \partial_3 F(\star_{\y})(t) + K_{P^*} [\partial_4 F(\star_{\y})(\tau)](t) \big)
= \partial_1 F(\star_{\y})(t)+K_{P^*} [\partial_2 F(\star_{\y})(\tau)](t).
\end{equation}
The proof is complete.
\end{proof}

We would like to remark that in the particular case, when
$k^\alpha(t,\tau)=\frac{1}{\Gamma(1-\alpha)}(t-\tau)^{-\alpha}$,
$k^\alpha\in L^q(\Delta;\R)$ and $P=\langle a,t,b,1,0\rangle$,
then Sections~\ref{section11},~\ref{section12},~\ref{section13},~\ref{sec:ex:lagr}
and~\ref{sec:sob} recover the case of the following fractional variational functional
\begin{equation*}
\fonction{\mathcal{I}}{\mathcal{A}}{\R}{y}{\int\limits_a^b
F(y(t),\Ilc[y](t),\dot{y}(t),\Dcl[y](t),t) \; dt ,}
\end{equation*}
studied in \cite{LTDExist}. Moreover, if we choose the kernel $k^\alpha(t,\tau)
=\frac{1}{\Gamma(1-\alpha(t,\tau))}(t-\tau)^{-\alpha(t,\tau)}$, $k^\alpha\in L^q(\Delta;\R)$
and the parameter set $P=\langle a,t,b,1,0\rangle$, then
Sections~\ref{section11},~\ref{section12},~\ref{section13},~\ref{sec:ex:lagr}
and~\ref{sec:sob} restore the case of the variable order fractional variational functional
\begin{equation*}
\fonction{\mathcal{I}}{\mathcal{A}}{\R}{y}{\int\limits_a^b
F\left(y(t),{_{a}}\textsl{I}^{1-\alpha(\cdot,\cdot)}_{t}[y](t),
\dot{y}(t),{^{C}_{a}}\textsl{D}^{\alpha(\cdot,\cdot)}_{t}[y](t), t\right) \; dt .}
\end{equation*}


\section{Some Improvements}
\label{sec:reg}

In this section, we assume more regularity
of the Lagrangian $F$ and of the operators $K_P$ and $B_P$.
It allows to weaken the convexity assumption in Theorem~\ref{thm:tonelli}
and/or the assumptions of Propositions~\ref{prop1} and~\ref{prop1b}.


\subsection{A First Weaker Convexity Assumption}

Let us assume that $F$ satisfies the following condition:
\begin{equation}\label{eq:ueq}
\left(F(\cdot,x_2,x_3,x_4,t)\right)_{(x_2,x_3,x_4,t)\in\R^3\times[a,b]}
~~\textnormal{is uniformly equicontinuous on }\R.
\end{equation}
This condition has to be understood as:
\begin{multline}\label{eq:ueq1}
\forall\varepsilon>0,\exists\delta>0,\forall(u,v)\in\R^2,\;\left|u-v\right|
\leq\delta\Longrightarrow\forall(x_2,x_3,x_4,t)\in\R^3\times [a,b],\\
\left|F(u,x_2,x_3,x_4,t)-F(v,x_2,x_3,x_4,t)\right|\leq\varepsilon.
\end{multline}

For example, this condition is satisfied for a Lagrangian $F$ with bounded $\partial_1 F$.
In this case, we can prove the following improved version of Theorem~\ref{thm:tonelli}:

\begin{theorem}
\label{thm:tonelli2}
Let us assume that:
\begin{itemize}
\item $F$ satisfies the condition given by \eqref{eq:ueq};
\item $F$ is regular;
\item $\mathcal{I}$ is coercive on $\mathcal{A}$;
\item $L(x_1,\cdot,t)$ is convex on $\R^3$ for any $x_1\in\R$ and for any $t\in [a,b]$.
\end{itemize}
Then, there exists a minimizer for $\mathcal{I}$.
\end{theorem}

\begin{proof}
Indeed, with the same proof of Theorem~\ref{thm:tonelli}, we can construct a weakly
convergent sequence $(y_n)_{n\in\N}\subset\mathcal{A}$ satisfying
\begin{equation*}
y_n \xrightharpoonup[]{W^{1,p}} \bar{y}\in\mathcal{A}~~\textnormal{and}~~\mathcal{I} (y_n)
\longrightarrow \inf\limits_{y \in \mathcal{A} }\mathcal{I}(y) < +\infty.
\end{equation*}
Since the compact embedding $W^{1,p}(a,b;\R) \hooktwoheadrightarrow C([a,b];\R)$ holds,
we have $y_n\stackrel{C}{\rightarrow}\y$. Let $\varepsilon>0$ and let us consider $\delta>0$
given by \eqref{eq:ueq1}. There exists $N\in\N$ such that for any $n\geq N$, $\left\|y_n-\y\right\|_{\infty}\leq \delta$.
So, for any $n\geq N$ and for $t\in (a,b)$:
\begin{equation*}
\left|F(y_n(t),K_P[y_n](t),\dot{y}_n(t),B_P[y_n](t),t)
-F(\y(t),K_P[y_n](t),\dot{y}_n(t),B_P[y_n](t),t)\right|
\leq\varepsilon.
\end{equation*}
Consequently, for any $n\geq N$, we have
\begin{equation*}
\mathcal{I}(y_n)\geq\int_{a}^{b}
F(\y(t),K_P[y_n](t),\dot{y}_n(t),B_P[y_n](t),t)dt-(b-a)\varepsilon.
\end{equation*}
From the convexity hypothesis and using the same strategy as in the proof
of Theorem~\ref{thm:tonelli}, we have by passing to the limit on $n$
\begin{equation*}
\inf\limits_{y \in \mathcal{A} }\mathcal{I}(y)\geq \mathcal{I}(\y)-(b-a)\varepsilon.
\end{equation*}
The proof is complete since the previous inequality is true for any $\varepsilon>0$.
\end{proof}

Such an improvement allows to give examples of a Lagrangian $F$ without convexity on its first variable.
Taking inspiration from Example~\ref{ex:lagr1}, we can provide the following example.

\begin{example}
Let us consider $p=2$ and $\mathcal{A}=W_a^{1,2}(a,b;\R)$. Let us consider
\begin{equation*}
F(x_1,x_2,x_3,x_4,t)=f(x_1,t)+\frac{1}{2}\sum\limits_{i=2}^{4}\left|x_i\right|^2,
\end{equation*}
for any $f:\R\times[a,b]\rightarrow\R$ of class $C^1$ with $\partial_1 f$ bounded
(like sine and cosine functions). In this case, $F$ satisfies the hypothesis of
Theorem~\ref{thm:tonelli2} and we can conclude with the existence of a minimizer
of $\mathcal{I}$ defined on $\mathcal{A}$.
\end{example}


\subsection{A Second Weaker Convexity Assumption}

In this section, we assume that $K_P$ is moreover a linear bounded operator from
$C([a,b];\R)$ to $C([a,b];\R)$. For example, this condition is satisfied by
Riemann--Liouville fractional integrals (see \cite{book:Kilbas,book:Samko}
for detailed proofs). We also assume that $F$ satisfies the following condition:
\begin{equation}
\label{eq:ueqb}
\left(F(\cdot,\cdot,x_3,x_4,t)\right)_{(x_3,x_4,t)\in\R^2\times[a,b]}
~~\textnormal{is uniformly equicontinuous on }\R^2.
\end{equation}
This condition has to be understood as:
\begin{multline}\label{eq:ueqb1}
\forall\varepsilon>0,\exists\delta>0,\forall(u,v)\in\R^2,\forall (u_0,v_0)
\in\R^2\;\left|u-v\right|\leq\delta,\left|u_0-v_0\right|
\leq\delta\Longrightarrow\forall(x_3,x_4,t)\in\R^2\times [a,b],\\
\left|F(u,u_0,x_3,x_4,t)-F(v,v_0,x_3,x_4,t)\right|\leq\varepsilon.
\end{multline}

For example, this condition is satisfied for a Lagrangian $F$ with bounded $\partial_1 F$
and bounded $\partial_2 F$. In this case, we can prove the following
improved version of Theorem~\ref{thm:tonelli}:

\begin{theorem}\label{thm:tonelli3}
Let us assume that:
\begin{itemize}
\item $F$ satisfies the condition given by \eqref{eq:ueqb};
\item $F$ is regular;
\item $\mathcal{I}$ is coercive on $\mathcal{A}$;
\item $L(x_1,x_2,\cdot,t)$ is convex on $\R^2$
for any $(x_1,x_2)\in\R$ and for any $t\in [a,b]$.
\end{itemize}
Then, there exists a minimizer for $\mathcal{I}$.
\end{theorem}

\begin{proof}
Indeed, with the same proof of Theorem~\ref{thm:tonelli}, we can construct
a weakly convergent sequence $(y_n)_{n\in\N}\subset\mathcal{A}$ satisfying
\begin{equation*}
y_n \xrightharpoonup[]{W^{1,p}} \bar{y}
\in\mathcal{A}~~\textnormal{and}~~\mathcal{I} (y_n)
\longrightarrow \inf\limits_{y \in \mathcal{A} }\mathcal{I}(y) < +\infty.
\end{equation*}
Since the compact embedding $W^{1,p}(a,b;\R) \hooktwoheadrightarrow C([a,b];\R)$ holds,
we have $y_n\stackrel{C}{\rightarrow}\y$ and since $K_P$ is continuous from $C([a,b];\R)$
to $C([a,b];\R)$, we have $K_P[\y_n]\stackrel{C}{\rightarrow}K_P[\y]$. Let $\varepsilon>0$
and let us consider $\delta>0$ given by \eqref{eq:ueqb1}. There exists $N\in\N$ such that
for any $n\geq N$, $\left\|y_n-\y\right\|_{\infty}\leq \delta$ and $\left\|K_P[y_n]
-K_P[\y]\right\|_{\infty}\leq \delta$. So, for any $n\geq N$ and for $t\in (a,b)$:
\begin{equation*}
\left|F(y_n(t),K_P[y_n](t),\dot{y}_n(t),B_P[y_n](t),t)-F(\y(t),
K_P[\y](t),\dot{y}_n(t),B_P[y_n](t),t)\right|\leq\varepsilon.
\end{equation*}
Consequently, for any $n\geq N$, we have
\begin{equation*}
\mathcal{I}(y_n)\geq\int_{a}^{b} F(\y(t),K_P[\y](t),
\dot{y}_n(t),B_P[y_n](t),t)dt-(b-a)\varepsilon.
\end{equation*}
From the convexity hypothesis and using the same strategy as in the proof
of Theorem~\ref{thm:tonelli}, we have by passing to the limit on $n$:
\begin{equation*}
\inf\limits_{y \in \mathcal{A} }\mathcal{I}(y)\geq \mathcal{I}(\y)-(b-a)\varepsilon.
\end{equation*}
The proof is complete since the previous inequality is true for any $\varepsilon>0$.
\end{proof}

Such an improvement allows to give examples of a Lagrangian $F$ without convexity
on its two first variables. Taking inspiration from Example~\ref{ex:lagr3},
we can provide the following example.

\begin{example}
Let us consider
\begin{equation*}
F(x_1,x_2,x_3,x_4,t)=c(t)\cos (x_1)\cdot\sin (x_2)+\frac{1}{p}\left|x_3\right|^p+f(t)\cdot x_4,
\end{equation*}
where $c: [a,b]\rightarrow\R$, $f:[a,b]\rightarrow\R$ are of class $C^1([a,b];\R)$. In this case,
one can prove that $F$ satisfies all hypothesis of Theorem~\ref{thm:tonelli3} and then, we can
conclude with the existence of a minimizer of $\mathcal{I}$ defined on $W_a^{1,p}(a,b;\R)$
for any $1<p<\infty$ and $1<q<\infty$.
\end{example}


\section{Conclusion}

In \cite{LTDExist} existence results for fractional variational problems containing Caputo derivatives were given.
This chapter extends those results to any linear operator $K_P$ bounded from the space $L^p(a,b;\R)$ to $L^q(a,b;\R)$,
having in mind that $B_P:=K_P\circ\frac{d}{dt}$.


\section{State of the Art}

The results presented in this chapter are submitted \cite{GenExist}.
For an important particular case see \cite{LTDExist}.


\chapter{Application to the Sturm--Liouville Problem}
\label{ch:SL}

In 1836--1837 the French mathematicians Sturm (1803-1853) and Liouville (1809-1855)
published series of articles initiating a new subtopic of mathematical analysis --
the Sturm--Liouville theory. It deals with the general linear,
second order ordinary differential equation of the form
\begin{equation}
\label{eq:SLI}
\frac{d}{dt}\left(p(t)\frac{dy}{dt}\right)+q(t)y=\lambda w(t)y,
\end{equation}
where $t\in[a,b]$, and in any particular problem functions $p(t)$, $q(t)$ and $w(t)$ are known.
In addition, certain boundary conditions are attached to equation \eqref{eq:SLI}. For specific
choices of the boundary conditions, nontrivial solutions of \eqref{eq:SLI} exist only for particular
values of the parameter $\lambda=\lambda^{(m)}$, $m=1,2,\dots$. Constants $\lambda^{(m)}$
are called eigenvalues and corresponding solutions $y^{(m)}(t)$ are called eigenfunctions.
For a deeper discussion of the classical Sturm--Liouville theory we refer the reader
to \cite{book:GF,book:Brunt}.

Recently, many researchers focused their attention on certain generalizations of Sturm--Liouville problem.
Namely, they are interested in equations of the type \eqref{eq:SLI}, however with fractional differential operators
(see e.g., \cite{QMQ,Malgorzata2,Lin,QCH,Kli13,Kli13a,AlM09}). In this chapter, we develop the Sturm--Liouville
theory by studying the Sturm--Liouville eigenvalue problem with Caputo fractional derivatives. We show that
fractional variational principles are useful for the approximation of eigenvalues and eigenfunctions.
Traditional Sturm--Liouville theory does not depend upon the calculus of variations, but stems from
the theory of ordinary linear differential equations. However, the Sturm--Liouville eigenvalue problem
is readily formulated as a constrained variational principle, and this formulation can be used
to approximate the solutions. We emphasize that it has a special importance for the fractional
Sturm--Liouville equation since fractional operators are nonlocal and it can be extremely challenging
to find analytical solutions. Besides allowing convenient approximations many general properties
of the eigenvalues can be derived using the variational principle.


\section{Useful Lemmas}

In this section, we present three lemmas that are used
to prove existence of solutions for the fractional Sturm--Liouville problem.

\begin{lemma}
\label{lem:A}
Let $\alpha\in (0,1) $ and function $\gamma\in C([a,b];\R)$. If
\begin{equation*}
\int_{a}^{b}\gamma(t)\frac{d}{dt}\left(\Dcl [h](t)\right)\;dt=0
\end{equation*}
for each $h\in  C^{1}([a,b];\R)$ such that
$\frac{d}{dt}\Dcl [h]\in C([a,b];\R)$ and boundary conditions
$$
h(a)=\Ilc h(b)=0
$$
and
$$
\Dcl [h](t)|_{t=a}=\Dcr [h](t)|_{t=b}=0
$$
are fulfilled, then $ \gamma(t)=c_{0}+c_{1}t$,
where $c_{0}, c_{1}$ are some real constants.
\end{lemma}

\begin{proof} Let us define function $h$ as follows
\begin{equation} \label{defh1}
 h(t):={_{a}}\textsl{I}_t^{1+\alpha}\left[\gamma(\tau)-c_{0}-c_{1}\tau\right](t)
\end{equation}
with constants fixed by the conditions
\begin{eqnarray}
&&{_{a}}\textsl{I}_t^{2}\left[\gamma(\tau)-c_{0}-c_{1}\tau\right](t)|_{t=b}=0\label{cn1}\\
&& {_{a}}\textsl{I}_t^{1}\left[\gamma(\tau)-c_{0}-c_{1}\tau\right](t)|_{t=b}=0.\label{cn2}
\end{eqnarray}
Observe that function $h$ is continuous and fulfills the boundary conditions
$$
h(a)=0\hspace{2cm} \Ilc [h](t)|_{t=b}
={_{a}}\textsl{I}_t^{2}\left[\gamma(\tau)-c_{0}-c_{1}\tau\right](t)|_{t=b}=0
$$
and
$$
\Dcl [h](t)|_{t=a}=\Dl [h](t)|_{t=a}=\frac{d}{dt}{_{a}}\textsl{I}_t^{2}\left[
\gamma(\tau)-c_{0}-c_{1}\tau\right](t)|_{t=a}
$$
$$
={_{a}}\textsl{I}_t^{1}\left[\gamma(\tau)-c_{0}-c_{1}\tau\right](t)|_{t=a}=0,
$$
$$
\Dcl [h](t)|_{t=b}=\Dl [h](t)|_{t=b}=\frac{d}{dt}{_{a}}\textsl{I}_t^{2}\left[
\gamma(\tau)-c_{0}-c_{1}\tau\right](t)|_{t=b}
$$
$$
= {_{a}}\textsl{I}_t^{1}\left[\gamma(\tau)-c_{0}-c_{1}\tau\right](t)|_{t=b}=0.
$$
In addition,
$$
t\mapsto h'(t) = \Il [\gamma(\tau)-c_{0}-c_{1}\tau](t) \in C([a,b];\R)
$$
$$
t\mapsto\frac{d}{dt}\Dcl [h](t)=\gamma(t)-c_{0}-c_{1}t\in C([a,b];\R).
$$
We also have
$$
\int_{a}^{b}\left (\gamma(t)-c_{0}-c_{1}t\right)\frac{d}{dt}\left(\Dcl [h](t)\right)\;dt
$$
$$
=\int_{a}^{b}\left (-c_{0}-c_{1}t\right)\frac{d}{dt}\left(\Dcl [h](t)\right)\;dt
$$
$$
=-c_{0}\cdot\Dcl [h](t)|_{t=a}^{t=b}-c_{1}t\cdot\Dcl [h](t)|_{t=a}^{t=b}
+c_{1}\cdot \Ilc [h](t)|_{t=a}^{t=b}=0.
$$
On the other hand,
$$
\frac{d}{dt}\left(\Dcl [h](t)\right)=\frac{d}{dt}\Dcl \left[{_{a}}\textsl{I}_t^{1+\alpha}\left[
\gamma(\tau)-c_{0}-c_{1}\tau\right](s)\right](t)=\gamma(t)-c_{0}-c_{1}t
$$
and
$$
0= \int_{a}^{b}\left (\gamma(t)-c_{0}-c_{1}t\right)\frac{d}{dt}\left(\Dcl [h](t)\right)\;dt
= \int_{a}^{b}\left (\gamma(t)-c_{0}-c_{1}t\right)^{2}\;dt.
$$
Thus function $\gamma$ is
$$
\gamma(t)=c_{0}+c_{1}t.
$$
The proof is complete.
\end{proof}

\begin{lemma}
\label{lem:B}
Let $\alpha\in \left(\frac{1}{2},1\right) $, $\gamma\in C([a,b];\R)$
and ${_{a}}D^{1-\alpha}_{t}[\gamma]\in L^2(a,b;\R)$. If
\begin{equation*}
\int_{a}^{b}\gamma(t)\frac{d}{dt}\left(\Dcl [h](t)\right)\;dt=0
\end{equation*}
for each $ h\in  C^{1}([a,b];\R)$ such that
$h''\in L^{2}(a,b;\R)$, $\frac{d}{dt}\Dcl [h]\in C([a,b];\R)$
and boundary conditions
\begin{eqnarray}
&&h(a)=\Ilc [h](b)=0,\\
&&\Dcl [h](t)|_{t=a}=\Dcl [h](t)|_{t=b}=0
\end{eqnarray}
are fulfilled, then $ \gamma(t)=c_{0}+c_{1}t$,
where $c_{0}, c_{1}$ are some real constants.
\end{lemma}

\begin{proof}
We define function $h$ as in the proof of Lemma~\ref{lem:A}:
\begin{equation}
\label{defh2}
h(t):={_{a}}\textsl{I}_t^{1+\alpha}\left[\gamma(\tau)-c_{0}-c_{1}\tau\right](t)
\end{equation}
with constants fixed by the conditions (\ref{cn1}) and (\ref{cn2}).
The proof of the lemma is analogous to that of the Lemma~\ref{lem:A}.
In addition,
for the second order derivative we have
$$
h''(t)=\frac{d}{dt}\Il\left[\gamma(\tau)-c_{0}-c_{1}\tau\right](t)
$$
$$
= {_{a}}D^{1-\alpha}_{t}\left [\gamma(\tau)-c_{0}-c_{1}\tau\right](t)
$$
$$
={_{a}}D^{1-\alpha}_{t} [\gamma](t)-(c_{0}+c_{1}a)\frac{(t-a)^{\alpha-1}}{\Gamma(\alpha)}
-c_{1}\frac{(t-a)^{\alpha}}{\Gamma(\alpha+1)}.
$$
Let us observe that for $\alpha>1/2$
$$
t\mapsto\frac{(t-a)^{\alpha-1}}{\Gamma(\alpha)}\in L^{2}(a,b;\R)
$$
$$
t\mapsto\frac{(t-a)^{\alpha}}{\Gamma(\alpha+1)}\in C([a,b];\R)\subset L^{2}(a,b;\R).
$$
Thus, we conclude that $h''\in L^{2}(a,b;\R)$ and function $h$ constructed in this proof fulfills
all the assumptions of Lemma~\ref{lem:B}. The remaining part of the proof is analogous
to that for Lemma~\ref{lem:A}.
\end{proof}

\begin{lemma}
\label{lem:C}
\begin{enumerate}[(a)]
\item Let $\alpha\in \left(\frac{1}{2},1\right)$, functions $\gamma_{j}\in C([a,b];\R)$,
$j=1,2,3$ and ${_{a}}D^{1-\alpha}_{t}[\gamma_3]\in L^2(a,b;\R)$. If
\begin{equation}
\label{eq:lem:C}
\int\limits_{a}^{b}\left (\gamma_{1}(t)h(t)+\gamma_{2}(t)\Dcl [h](t)
+\gamma_{3}(t)\frac{d}{dt}\left(\Dcl [h](t)\right)\right)\;dt=0
\end{equation}
for each $ h\in  C^{1}([a,b];\R)$, such that $h''\in L^2(a,b;\R)$
and $\frac{d}{dt}\Dcl[h]\in C([a,b];\R)$, fulfilling boundary conditions
\begin{eqnarray}
&&h(a)=\Ilc [h](b)=0,\label{c1}\\
&& \Dcl[h](t)|_{t=a}=\Dcl[h](t)|_{t=b}=0 \label{c2}
\end{eqnarray}
then $\gamma_{3}\in C^{1}([a,b];\R)$.

\item Let $\alpha\in \left(\frac{1}{2},1\right)$
and functions $\gamma_1,\gamma_2\in C([a,b];\R)$. If
\begin{equation}
\label{eq:lem:C:b}
\int\limits_{a}^{b}\left (\gamma_{1}(t)h(t)+\gamma_{2}(t)\Dcl [h](t)\right)\;dt=0
\end{equation}
for each $ h\in  C^{1}([a,b];\R)$, such that $h''\in L^2(a,b;\R)$ and
$\frac{d}{dt}\Dcl[h]\in C([a,b];\R)$,
fulfilling boundary conditions \eqref{c1}, \eqref{c2}, then
\begin{equation*}
-\gamma_{1}(t)-\Dcr[\gamma_{2}](t)=0.
\end{equation*}
\end{enumerate}
\end{lemma}

\begin{proof}
Observe that integral \eqref{eq:lem:C} can be rewritten as follows:
$$
\int_{a}^{b}\left (\gamma_{1}(t)h(t)+\gamma_{2}(t)\Dcl [h](t)
+\gamma_{3}(t)\frac{d}{dt}\Dcl[h](t)\right)dt
$$
$$
= \int_{a}^{b}\left(-\left({_{a}}\textsl{I}^{1}_{t}\circ{_{t}}\textsl{I}^{\alpha}_{b}\right)[\gamma_1](t)
-{_{a}}\textsl{I}^{1}_{t}[\gamma_2](t)+\gamma_3(t)\right)\frac{d}{dt}\Dcl [h](t)\;dx = 0.
$$
Due to the fact that relations
$$
\left({_{a}}\textsl{I}^{\alpha}_{t}\circ{_{t}}\textsl{I}^{1}_{b}\circ\frac{d}{dt}\Dcl\right) [h](t)=h(t)
$$
and
$$
\left({_{t}}\textsl{I}_b^{1}\circ\frac{d}{d\tau}\Dcl\right)[h](t)=-\Dcl[h](t)
$$
are valid because function $h$ fulfills boundary conditions \eqref{c1},\eqref{c2}. Denote
$$
\gamma(t):= -\left({_{a}}\textsl{I}^{1}_{t}\circ {_{t}}\textsl{I}^{\alpha}_{b}\right)[\gamma_{1}](t)
-{_{a}}\textsl{I}^{1}_{t}[\gamma_{2}](t)+\gamma_{3}(t).
$$
It is clear that $\gamma\in C([a,b];\R)$ and ${_{a}}D^{1-\alpha}_{t}[\gamma] \in L^{2}(a,b;\R)$. Thus,
according to Lemma~\ref{lem:B}, there exist constants $c_{0}$ and $c_{1}$  such that
$$
-\left({_{a}}\textsl{I}^{1}_{t}\circ {_{t}}\textsl{I}^{\alpha}_{b}\right)[
\gamma_{1}](t)-{_{a}}\textsl{I}^{1}_{t}[\gamma_{2}](t)+\gamma_{3}(t)=c_{0}+c_{1}t.
$$
Let us note that function $\gamma_{3}$ is
$$
\gamma_{3}(t)=\left({_{a}}\textsl{I}^{1}_{t}\circ {_{t}}\textsl{I}^{\alpha}_{b}\right)[
\gamma_{1}](t)+{_{a}}\textsl{I}^{1}_{t}[\gamma_{2}](t)+c_{0}+c_{1}t.
$$
Hence its first order derivative is continuous in $[a,b]$ and $\gamma_{3}\in C^{1}([a,b];\R)$.

The proof of part (b) is similar. We write integral \eqref{eq:lem:C:b} as follows:
\begin{multline*}
\int\limits_{a}^{b}\left (\gamma_{1}(t)h(t)+\gamma_{2}(t)\Dcl[h](t)\right)dt\\
= \int\limits_{a}^{b}\left (-\left({_{a}}\textsl{I}^{1}_{t}
\circ {_{t}}\textsl{I}^{\alpha}_{b}\right)[\gamma_{1}](t)
-{_{a}}\textsl{I}^{1}_{t}[\gamma_{2}](t)\right )\frac{d}{dt}\Dcl [h](t)dt=0.
\end{multline*}
The function in brackets is continuous in $[a,b]$,
\begin{equation*}
{_{a}}D^{1-\alpha}_{t} \left [-\left({_{a}}\textsl{I}^{1}_{t}
\circ {_{t}}\textsl{I}^{\alpha}_{b}\right)[\gamma_{1}](\tau)
-{_{a}}\textsl{I}^{1}_{t}[\gamma_{2}](\tau)\right ](t)
=-\left({_{a}}\textsl{I}^{\alpha}_{t}\circ {_{t}}\textsl{I}^{\alpha}_{b}\right)[
\gamma_{1}](t)-{_{a}}\textsl{I}^{\alpha}_{t}[\gamma_{2}](t),
\end{equation*}
and
\begin{equation*}
-\left({_{a}}\textsl{I}^{\alpha}_{t}\circ {_{t}}\textsl{I}^{\alpha}_{b}\right)[
\gamma_{1}]-{_{a}}\textsl{I}^{\alpha}_{t}[\gamma_{2}]\in C([a,b];\R) \subset L^{2}(a,b;\R),
\end{equation*}
so we again can apply Lemma~\ref{lem:B} and obtain that there exist constants $c_{0}$
and $c_{1}$ such that
$$
\left({_{a}}\textsl{I}^{1}_{t}\circ {_{t}}\textsl{I}^{\alpha}_{b}\right)[\gamma_{1}](t)
+{_{a}}\textsl{I}^{1}_{t}[\gamma_{2}](t)=c_{0}+c_{1}t.
$$
Thus functions $\gamma_{1,2}$ fulfill equation:
$$
\Dcr[\gamma_{2}](t)+\gamma_{1}(t)=0.
$$
\end{proof}


\section{The Fractional Sturm--Liouville Problem}

The crucial idea in the proof of main result of this chapter (Theorem~\ref{thm:exist})
is to apply direct variational methods to the fractional Sturm--Liouville equation.
Starting from the fractional Sturm--Liouville equation the approach is to find
an associated functional and to use this to find approximations to the minimizers,
which are necessarily solutions to the original equation. In the case of the fractional
Sturm--Liouville equation an associated variational problem is the fractional
isoperimetric problem which is defined in the following way:
\begin{equation}
\label{izo:1}
\min \mathcal{I}(y)=\int\limits_a^b F(y(t),\Dcl[y](t),t)\;dt,
\end{equation}
subject to the boundary conditions
\begin{equation} \label{izo:2}
y(a)=y_a,\quad y(b)=y_b
\end{equation}
and the isoperimetric constraint
\begin{equation}\label{izo:3}
\mathcal{J}(y)=\int\limits_a^b G(y(t),\Dcl[y](t),t)\;dt=\xi,
\end{equation}
where $\xi\in\R$ is given, and
\begin{equation*}
\fonction{F}{[a,b]\times\R^2}{\R}{(y,u,t)}{F(y,u,t),}
\end{equation*}
\begin{equation*}
\fonction{G}{[a,b]\times\R^2}{\R}{(y,u,t)}{G(y,u,t)}
\end{equation*}
are functions of class $C^1$, such that
$\frac{\partial F}{\partial u}$, $\frac{\partial G}{\partial u}$
have continuous $\Dr$ derivatives.

\begin{theorem}[cf. Theorem~3.3 \cite{Isoperimetric}]
\label{thm:EL}
If $y\in C[a,b]$ with $\Dcl[y]\in C([a,b];\R)$ is a minimizer for problem
\eqref{izo:1}--\eqref{izo:3}, then there exists a real constant $\lambda$
such that, for $H=F+\lambda G$, the equation
\begin{equation}
\label{eq:EL}
\frac{\partial H}{\partial y}(y(t),\Dcl[y](t),t)
+\Dr\left[\frac{\partial H}{\partial u}(y(t),\Dcl[y](t),t)\right]=0
\end{equation}
holds, provided that
\begin{equation*}
\frac{\partial G}{\partial y}(y(t),\Dcl[y](t),t)
+\Dr\left[\frac{\partial G}{\partial u}(y(t),\Dcl[y](t),t)\right] \neq 0.
\end{equation*}
\end{theorem}


\subsection{Existence of Discrete Spectrum}

We show that, similarly to the classical case, for the fractional Sturm--Liouville problem
there exist an infinite monotonic increasing sequence of eigenvalues. Moreover, apart from
multiplicative factors, to each eigenvalue there corresponds precisely one eigenfunction
and eigenfunctions form an orthogonal set of solutions.

We shall use the following assumptions.\\
(H1) Let $\frac{1}{2}<\alpha<1$ and $p,q,w_{\alpha}$ be given functions such that:
$p$ is of $C^1$ class and $p(t)>0$; $q,w_{\alpha}$ are continuous, $w_{\alpha}(t)>0$
and $(\sqrt{w_{\alpha}})'$ is H\"{o}lderian of order $\beta\leq \alpha-\frac{1}{2}$.
Consider the fractional differential equation
\begin{equation}
\label{eq:SLE}
\index{Sturm--Liouville problem of fractional order}
\Dcr \left[p(\tau){^{C}_a}\textsl{D}_{\tau}^{\alpha}[y](\tau)\right](t)+q(t)y(t)
= \lambda w_{\alpha}(t)y(t),
\end{equation}
that will be called the fractional Sturm-Liouville equation,
subject to the boundary conditions
\begin{equation}
\label{eq:BC}
y(a)=y(b)=0.
\end{equation}

\begin{theorem}
\label{thm:exist}
Under assumptions (H1), the fractional Sturm--Liouville problem
(FSLP) \eqref{eq:SLE}--\eqref{eq:BC} has an infinite increasing
sequence of eigenvalues $\lambda^{(1)}, \lambda^{(2)},\dots$,
and to each eigenvalue $\lambda^{(n)}$ there corresponds an eigenfunction
$y^{(n)}$ which is unique up to a constant factor. Furthermore,
eigenfunctions $y^{(n)}$ form an orthogonal set of solutions.
\end{theorem}

\begin{proof}
The proof is similar in spirit to \cite{book:GF} and will be divided into 6 steps.
As in \cite{book:GF} at the same time we shall derive a method
for approximating the eigenvalues and eigenfunctions.\\
\textit{Step 1.} We shall consider problem of minimizing the functional
\begin{equation}
\label{eq:IF}
\mathcal{I}(y)= \int_{a}^{b}\limits\left[p(t) (\Dcl[y](t))^{2}+q(t)y^{2}(t)\right]\;dt
\end{equation}
subject to an isoperimetric constraint
\begin{equation}
\label{eq:IC}
\hspace{2cm}\mathcal{J}(y)= \int\limits_{a}^{b}w_{\alpha}(t)y^{2}(t)dt=1
\end{equation}
and boundary conditions \eqref{eq:BC}. First, let us point out that functional
\eqref{eq:IF} is bounded from below. Indeed, as $p(t)>0$ we have
\begin{multline*}
\mathcal{I}(y)= \int\limits_{a}^{b}\left[p(t) (\Dcl[y](t))^{2}+q(t)y^{2}(t)\right]\;dt\\
\geq\min\limits_{t\in[a,b]}\frac{q(t)}{w_{\alpha}(t)}\cdot\int\limits_{a}^{b}w_{\alpha}(t)y^{2}(t)\;dt
=\min\limits_{t\in[a,b]}\frac{q(t)}{w_{\alpha}(t)}=:M_0>-\infty.
\end{multline*}
From now on, for simplicity, we assume that $a=0$ and $b=\pi$. According to the Ritz method,
we approximate solution of \eqref{eq:BC}--\eqref{eq:IC}  using the following trigonometric
function with coefficient depending on $w_{\alpha}$:
\begin{equation}
\label{aprox}
y_{m}(t)= \frac{1}{\sqrt{w_{\alpha}}}\sum_{k=1}^{m} \beta_{k}\sin(kt).
\end{equation}
Observe that $y_{m}(0)=y_{m}(\pi)=0$. Substituting \eqref{aprox} into \eqref{eq:IF}
and \eqref{eq:IC} we obtain the problem of minimizing the function
\begin{multline}
\label{eq:F}
I(\beta_{1},\dots,\beta_{m})= I([\beta])\\
= \sum_{k,j=1}^{m}\beta_{k}\beta_{j}\int\limits_{0}^{\pi}\left[p(t) \left({^{C}_{0}}
\textsl{D}_t^{\alpha}\left[\frac{\sin (k\tau)}{\sqrt{w_{\alpha}}}\right](t)
\cdot {^{C}_{0}}\textsl{D}_t^{\alpha}\left[\frac{\sin (j\tau)}{\sqrt{w_{\alpha}}}\right](t)\right)
+ \frac{q(t)}{w_{\alpha}(t)}\sin(kt)\sin(jt)\right]dt
\end{multline}
subject to the condition
\begin{equation}
\label{eq:sphere}
J(\beta_{1},\dots,\beta_{m})= J([\beta])=\frac{\pi}{2}\sum_{k=1}^{m}(\beta_{k})^{2}=1.
\end{equation}
Since $I([\beta])$ is continuous and the set given by \eqref{eq:sphere} is compact,
function $I([\beta])$ attains minimum, denoted by $\lambda_m^{(1)}$, at some point
$[\beta^{(1)}]=(\beta_{1}^{(1)},\dots,\beta_{m}^{(1)})$. If this procedure is carried
out for $m=1,2,\ldots$, we obtain a sequence of numbers $\lambda_{1}^{(1)},\lambda_2^{(1)},\ldots$.
Because $\lambda_{m+1}^{(1)}\leq \lambda_m^{(1)}$ and  $\mathcal{I}(y)$ is bounded from below,
we can find the limit
\begin{equation*}
\lim\limits_{m\rightarrow\infty}\lambda_m^{(1)}=\lambda^{(1)}.
\end{equation*}

\textit{Step 2.} Let
\begin{equation*}
y_{m}^{(1)}(t)= \frac{1}{\sqrt{w_{\alpha}}}\sum_{k=1}^{m} \beta_{k}^{(1)}\sin(kt)
\end{equation*}
denote the linear combination \eqref{aprox} achieving the minimum $\lambda_m^{(1)}$.
We shall prove that sequence $(y_{m}^{(1)})_{m\in\N}$ contains a uniformly convergent subsequence.
From now on, for simplicity, we will write $y_m$ instead of $y_m^{(1)}$. Recall that
\begin{equation*}
\lambda_m^{(1)}=\int\limits_{0}^{\pi} \left[p(t)\left({^{C}_{0}}\textsl{D}_t^{\alpha}[\seq](t)\right)^2
+q(t)\seq^2(t)\right]dt
\end{equation*}
is convergent, so it must be bounded, i.e.,~ there exists constant $M>0$ such that
\begin{equation*}
\int\limits_{0}^{\pi} \left[p(t)\left({^{C}_{0}}\textsl{D}_t^{\alpha}[\seq](t)\right)^2
+q(t)\seq^2(t)\right]dt\leq M,~m\in\N.
\end{equation*}
Therefore, for all $m\in\N$ it hold the following
\begin{multline*}
\int\limits_0^{\pi} p(t)\left({^{C}_{0}}\textsl{D}_t^{\alpha}[\seq](t)\right)^2\;dt
\leq M+\left|\int\limits_0^{\pi} q(t)\seq^2(t)\;dt\right|\\
\leq M+\max\limits_{t\in[0,\pi]}\left|\frac{q(t)}{w_{\alpha}(t)}\right|
\int\limits_0^{\pi} w_{\alpha}(t)\seq^2(t)\;dt=M+\max\limits_{t\in[0,\pi]}\left|
\frac{q(t)}{w_{\alpha}(t)}\right| =: M_1.
\end{multline*}
Moreover, since $p(t)>0$ one has
\begin{equation*}
\min\limits_{t\in[0,\pi]}p(t)\int\limits_0^{\pi} \left({^{C}_{0}}\textsl{D}_t^{\alpha}[\seq](t)\right)^2\;dt
\leq \int\limits_0^{\pi} p(t)\left({^{C}_{0}}\textsl{D}_t^{\alpha}[\seq](t)\right)^2\;dt\leq M_1,
\end{equation*}
and hence
\begin{equation}\label{eq:4}
\int\limits_0^{\pi} \left({^{C}_{0}}\textsl{D}_t^{\alpha}[\seq](t)\right)^2\;dt
\leq\frac{M_1}{\min\limits_{t\in[0,\pi]}p(t)}=: M_2.
\end{equation}
Using \eqref{eq:3}, \eqref{eq:4}, condition $\seq(0)=0$ and Schwartz inequality, one has
\begin{multline*}
\left|\seq(t)\right|^2=\left|\left({_{0}}\textsl{I}_t^{\alpha}
\circ {^{C}_{0}}\textsl{D}_t^{\alpha}\right)[\seq](t)\right|^2
=\frac{1}{\Gamma(\alpha)^2}\left|\int\limits_0^{t}
(t-\tau)^{\alpha-1}{^{C}_{0}}\textsl{D}_{\tau}^{\alpha}\seq(\tau)d\tau\right|^2\\
\leq \frac{1}{\Gamma(\alpha)^2}\left(\int\limits_0^\pi
\left|{^{C}_{0}}\textsl{D}_{\tau}^{\alpha}[\seq](\tau)\right|^2d\tau\right)\left(
\int_0^t (t-\tau)^{2(\alpha-1)}d\tau\right)\\
\leq \frac{1}{\Gamma(\alpha)^2}M_2\int_0^t (t-\tau)^{2(\alpha-1)}d\tau
<\frac{1}{\Gamma(\alpha)}M_2\frac{1}{2\alpha-1}\pi^{2\alpha-1},
\end{multline*}
so that $(\seq)_{m\in\N}$ is uniformly bounded.
Now, using Schwartz inequality, equation \eqref{eq:4} and the fact that the following inequality holds
\begin{equation*}
\forall x_1\geq x_2\geq 0,~(x_1-x_2)^2\leq x_1^2-x_2^2,
\end{equation*}
we have for any $0< t_1< t_2\leq \pi $
\begin{multline*}
\left|\seq(t_2)-\seq(t_1)\right|=\left|\left({_{0}}\textsl{I}_t^{\alpha}
\circ {^{C}_{0}}\textsl{D}_t^{\alpha}\right)[\seq](t_2)
-\left({_{0}}\textsl{I}_t^{\alpha}\circ {^{C}_{0}}\textsl{D}_t^{\alpha}\right)[\seq](t_1)\right|\\
=\frac{1}{\Gamma(\alpha)}\left|\int\limits_0^{t_2}(t_2-\tau)^{\alpha-1}{^{C}_{0}}\textsl{D}_t^{\alpha}[\seq](\tau)d\tau
-\int\limits_0^{t_1}(t_1-\tau)^{\alpha-1}{^{C}_{0}}\textsl{D}_t^{\alpha}[\seq](\tau)d\tau\right|\\
=\frac{1}{\Gamma(\alpha)}\left|\int\limits_{t_1}^{t_2}(t_2-\tau)^{\alpha-1}{^{C}_{0}}\textsl{D}_t^{\alpha}[\seq](\tau)d\tau
-\int\limits_0^{t_1}\left((t_2-\tau)^{\alpha-1}-(t_1-\tau)^{\alpha-1}\right){^{C}_{0}}\textsl{D}_t^{\alpha}[\seq](\tau)d\tau\right|\\
\leq \frac{1}{\Gamma(\alpha)}\left[\left(\int\limits_{t_1}^{t_2}
(t_2-\tau)^{2(\alpha-1)}d\tau\right)^{\frac{1}{2}}\left(\int\limits_{t_1}^{t_2}\left[
\left({^{C}_{0}}\textsl{D}_t^{\alpha}[\seq](\tau)\right)^2\right]d\tau\right)^{\frac{1}{2}}\right.\\
\left.+\left(\int\limits_0^{t_1}\left((t_1-\tau)^{\alpha-1}-(t_2-\tau)^{\alpha-1}\right)^2 d\tau\right)^{\frac{1}{2}}\left(
\int\limits_{0}^{t_1}\left[\left({^{C}_{0}}\textsl{D}_t^{\alpha}[\seq](\tau)\right)^2\right]d\tau\right)^{\frac{1}{2}}\right]\\
\leq\frac{\sqrt{M_2}}{\Gamma(\alpha)}\left[\left(\int\limits_{t_1}^{t_2}(t_2-\tau)^{2(\alpha-1)}d\tau\right)^{\frac{1}{2}}
+\left(\int\limits_0^{t_1}\left((t_1-\tau)^{2(\alpha-1)}-(t_2-\tau)^{2(\alpha-1)}\right)d\tau\right)^{\frac{1}{2}}\right]\\
=\frac{\sqrt{M_2}}{\Gamma(\alpha)\sqrt{2\alpha-1}}\left[(t_2-t_1)^{\alpha-\frac{1}{2}}+\left[(t_2-t_1)^{2\alpha-1}
-t_2^{2\alpha-1}+t_1^{2\alpha-1}\right]^{\frac{1}{2}}\right]\\
\leq \frac{2\sqrt{M_2}}{\Gamma(\alpha)\sqrt{2\alpha-1}}(t_2-t_1)^{\alpha-\frac{1}{2}}.
\end{multline*}
Therefore, by Ascoli's theorem, there exists a uniformly convergent subsequence
$(y_{m_n})_{n\in\N}$ of sequence $(\seq)_{m\in\N}$. It means that
we can find $y^{(1)}\in C([a,b];\R)$ such that
\begin{equation*}
y^{(1)}=\lim\limits_{n\rightarrow\infty}y_{m_n}.
\end{equation*}

\textit{Step 3.} Observe that by the Lagrange multiplier rule at $[\beta]=[\beta^{(1)}]$ we have
\begin{equation*}
0=\frac{\partial}{\partial \beta_{j}}\left [I([\beta])
-\lambda^{(1)}_{m}J([\beta])\right]|_{[\beta]=[\beta^{(1)}]},~~j=1,\dots,m.
\end{equation*}
Multiplying equations by an arbitrary constant
$C^{j}$ and summing from 1 to $m$ we obtain
\begin{equation}
\label{lag:1}
0=\sum_{j=1}^{m}C^{j}\frac{\partial}{\partial \beta_{j}}\left [I([\beta])
-\lambda^{(1)}_{m}J([\beta])\right]|_{[\beta]=[\beta^{(1)}]}.
\end{equation}
Introducing
\begin{equation*}
h_{m}(x)=\frac{1}{\sqrt{w_{\alpha}}}\sum_{j=1}^{m} C^{j}\sin (jt)
\end{equation*}
we can rewrite \eqref{lag:1} in the form
\begin{equation}
\label{lag:2}
0=\int\limits_{0}^{\pi}\left[p(t){^{C}_{0}}\textsl{D}_t^{\alpha}[y_{m}](t){^{C}_{0}}\textsl{D}_t^{\alpha}[h_{m}](t)
+[q(t)-\lambda^{(1)}_{m}w_{\alpha}(t)]y_{m}(t)h_{m}(t)\right]dt,
\end{equation}
Using the differentiation properties and formula
${^{C}_{0}}\textsl{D}_t^{\alpha}[y_{m}]=\frac{d}{dt}{_{0}}\textsl{I}_t^{1-\alpha}[y_{m}]$ we write \eqref{lag:2} as
\begin{multline}
\label{intseq}
0=\int\limits_{0}^{\pi}\left [-p'(t){_{0}}\textsl{I}_t^{1-\alpha}[y_{m}](t){^{C}_{0}}\textsl{D}_t^{\alpha}[h_{m}](t)
-p(t){_{0}}\textsl{I}_t^{1-\alpha}[y_{m}](t)\frac{d}{dt}{^{C}_{0}}\textsl{D}_t^{\alpha}[h_{m}](t)\right]dt\\
+ p(t){_{0}}\textsl{I}_t^{1-\alpha}[y_{m}](t){^{C}_{0}}\textsl{D}_t^{\alpha}[h_{m}](t)|_{t=0}^{t=\pi}
+\int\limits_{0}^{\pi}[q(t)-\lambda^{(1)}_{m}w_{\alpha}(t)]y_{m}(t)h_{m}(t)dt:=I_{m}.
\end{multline}
By Lemma~\ref{lem:D} (with $w=1/\sqrt{w_{\alpha}}$) and Lemma~\ref{lem:E} (Appendix), for function $h$ fulfilling assumptions of
Lemma~\ref{lem:B}, we shall obtain in the limit (at least for the convergent subsequence $(y_{m_n})_{n\in\N}$) the relation
\begin{multline}
\label{int}
0=\int\limits_{0}^{\pi}\left [-p'(t){_{0}}\textsl{I}_t^{1-\alpha}[y^{(1)}](t)\;{^{C}_{0}}\textsl{D}_t^{\alpha}[h](t)
-p(t){_{0}}\textsl{I}_t^{1-\alpha}[y^{(1)}](t)\frac{d}{dt}{^{C}_{0}}\textsl{D}_t^{\alpha}[h](t)\right]dt\\
+ p(t){_{0}}\textsl{I}_t^{1-\alpha}[y^{(1)}](t)\;{^{C}_{0}}\textsl{D}_t^{\alpha}[h](t)|_{t=0}^{t=\pi}
+ \int\limits_{0}^{\pi}[q(t)-\lambda^{(1)}w_{\alpha}(t)]y^{(1)}(t)h(t)\;dt:=I.
\end{multline}
Let us check the convergence of integrals \eqref{intseq} explicitly
\begin{multline}
\label{estproof}
 \left |I_{m}-I\right|\leq \int\limits_{0}^{\pi}\left|
 -p'(t){_{0}}\textsl{I}_t^{1-\alpha}[y_{m}](t){^{C}_{0}}\textsl{D}_t^{\alpha}[h_{m}](t)
 +p'(t){_{0}}\textsl{I}_t^{1-\alpha}[y^{(1)}](t)\;{^{C}_{0}}\textsl{D}_t^{\alpha}[h](t)\right|dt\\
+ \int\limits_{0}^{\pi}\left|p(t){_{0}}\textsl{I}_t^{1-\alpha}[y_{m}](t)\frac{d}{dt}{^{C}_{0}}\textsl{D}_t^{\alpha}[h_{m}](t)
-p(t){_{0}}\textsl{I}_t^{1-\alpha}[y^{(1)}](t)\;\frac{d}{dt}{^{C}_{0}}\textsl{D}_t^{\alpha}[h](t)\right|dt\\
+\left|p(t){_{0}}\textsl{I}_t^{1-\alpha}[y_{m}](t){^{C}_{0}}\textsl{D}_t^{\alpha}[h_{m}](t)|_{x=0}
-p(t){_{0}}\textsl{I}_t^{1-\alpha}[y^{(1)}](t)\;{^{C}_{0}}\textsl{D}_t^{\alpha}[h](t)|_{t=0}\right|\\
+\left|p(t){_{0}}\textsl{I}_t^{1-\alpha}[y_{m}](t){^{C}_{0}}\textsl{D}_t^{\alpha}[h_{m}](t)|_{t=\pi}
-p(t){_{0}}\textsl{I}_t^{1-\alpha}[y^{(1)}](t)\;{^{C}_{0}}\textsl{D}_t^{\alpha}[h](t)|_{t=\pi}\right|\\
+ \int\limits_{0}^{\pi}\left|[q(t)-\lambda^{(1)}_{m}w_{\alpha}(t)]y_{m}(t)h_{m}(t)
-\left[q(t)-\lambda^{(1)}w_{\alpha}(t)\right]y^{(1)}(t)\; h(t)\right|dt.
\end{multline}
For the first integral we get
$$
\int\limits_{0}^{\pi}\left|-p'(t){_{0}}\textsl{I}_t^{1-\alpha}[y_{m}](t){^{C}_{0}}\textsl{D}_t^{\alpha}[h_{m}](t)
+p'(t){_{0}}\textsl{I}_t^{1-\alpha}[y^{(1)}](t)\;{^{C}_{0}}\textsl{D}_t^{\alpha}[h](t)\right|dt
$$
$$
\leq ||p'||\cdot\left [ ||{^{C}_{0}}\textsl{D}_t^{\alpha}[h]||\cdot ||{_{0}}\textsl{I}_t^{1-\alpha}[y_{m}
-y^{(1)}]||_{L^{1}}+M_3K_{1-\alpha}\sqrt{\pi}||{^{C}_{0}}\textsl{D}_t^{\alpha}[h_{m}-h]||_{L^{2}}\right],
$$
where constant $M_3=\sup\limits_{m\in \mathbb{N}} ||y_{m}||$ and $||\cdot||$
denotes the supremum norm in the $C([0,\pi];\R)$ space. Now, we estimate the second integral
$$
\int\limits_{0}^{\pi}\left|p(t){_{0}}\textsl{I}_t^{1-\alpha}[y_{m}](t)\frac{d}{dt}{^{C}_{0}}\textsl{D}_t^{\alpha}[h_{m}](t)
-p(t)I^{1-\alpha}_{0+}y^{(1)}(t)\;\frac{d}{dt}{^{C}_{0}}\textsl{D}_t^{\alpha}[h](t)\right|dt
$$
$$
\leq ||p||\cdot\left [||\frac{d}{dt}{^{C}_{0}}\textsl{D}_t^{\alpha}[h]||_{L^{2}}
\cdot ||{_{0}}\textsl{I}_t^{1-\alpha}[y_{m}-y^{(1)}]||_{L^{2}}+M_3K_{1-\alpha}
\cdot ||\frac{d}{dt}{^{C}_{0}}\textsl{D}_t^{\alpha}[h_{m}-h]||_{L^{1}}\right].
$$
For the next two terms we have
\begin{equation}
\label{point1}
{_{0}}\textsl{I}_t^{1-\alpha}[y_{m}](0)\longrightarrow {_{0}}\textsl{I}_t^{1-\alpha}[y](0),
\quad  {_{0}}\textsl{I}_t^{1-\alpha}[y_{m}](\pi)\longrightarrow {_{0}}\textsl{I}_t^{1-\alpha}[y](\pi)
\end{equation}
resulting from the convergence of sequence $y_{m}\stackrel{C}{\longrightarrow}y$.
For sequence $h_{m}=g_{m}/\sqrt{w_{\alpha}}$, we infer from Lemma~\ref{lem:E} that
$$
h'_{m}\stackrel{C}{\longrightarrow}h'.
$$
Hence, also
$$
{^{C}_{0}}\textsl{D}_t^{\alpha}[h_{m}]\stackrel{C}{\longrightarrow}{^{C}_{0}}\textsl{D}_t^{\alpha}[h],
\quad {_{0}}\textsl{I}_t^{1-\alpha}[h'_{m}]\stackrel{C}{\longrightarrow}{_{0}}\textsl{I}_t^{1-\alpha}[h']
$$
and at points $t=0,\pi$ we obtain
\begin{equation}
\label{point2}
{^{C}_{0}}\textsl{D}_t^{\alpha}[h_{m}](0)\longrightarrow {^{c}_{0}}\textsl{D}_t^{\alpha}[h](0),
\quad  {^{C}_{0}}\textsl{D}_t^{\alpha}[h_{m}](\pi)\longrightarrow {^{c}_{0}}\textsl{D}_t^{\alpha}[h](\pi).
\end{equation}
The above pointwise convergence \eqref{point1} and \eqref{point2} imply that
$$
\lim_{m\longrightarrow\infty} \left|
p(t){_{0}}\textsl{I}_t^{1-\alpha}[y_{m}](t){^{C}_{0}}\textsl{D}_t^{\alpha}[h_{m}](t)|_{t=0}
-p(t){_{0}}\textsl{I}_t^{1-\alpha}[y^{(1)}](t)\;{^{C}_{0}}\textsl{D}_t^{\alpha}[h](t)|_{t=0}\right|=0
$$
$$
\lim_{m\longrightarrow\infty}\left|p(t){_{0}}\textsl{I}_t^{1-\alpha}[y_{m}](t)
{^{C}_{0}}\textsl{D}_t^{\alpha}[h_{m}](t)|_{t=\pi}-p(t){_{0}}\textsl{I}_t^{1-\alpha}[y^{(1)}](t)\;
{^{C}_{0}}\textsl{D}_t^{\alpha}[h](t)|_{t=\pi}\right|=0.
$$
Finally, for the last term in estimation \eqref{estproof} we get
$$
\int\limits_{0}^{\pi}|[q(t)-\lambda^{(1)}_{m}w_{\alpha}(t)]y_{m}(t)h_{m}(t)
-[q(t)-\lambda^{(1)}w_{\alpha}(t)]y^{(1)}(t)\; h(t)|dt
$$
$$
\leq \int\limits_{0}^{\pi}|q(t)(y_{m}(t)h_{m}(t)-y^{(1)}(t)\; h(t))|dt
+\int_{0}^{\pi}|w_{\alpha}(t)(\lambda^{(1)}_{m}y_{m}(t)h_{m}(t)-\lambda^{(1)}y^{(1)}(t)\; h(t))|dt
$$
$$
\leq \pi\cdot||q||\cdot \left [M_3 \cdot ||h_{m}-h|| + ||h||\cdot ||y_{m}-y^{(1)}||\right]
$$
$$
+\pi\cdot||w_{\alpha}||\cdot\left  [\Lambda\left (M_3 \cdot ||h_{m}-h||
+ ||h||\cdot ||y_{m}-y^{(1)}||\right)
+ ||y^{(1)}h||\cdot |\lambda^{(1)}_{m}-\lambda^{(1)}|\right],
$$
where constants $M_3=\sup\limits_{m\in \mathbb{N}} ||y_{m}||$
and $\Lambda=\sup\limits_{m\in \mathbb{N}} |\lambda^{(1)}_{m}|$.
We conclude that
$$
0= \lim_{m\longrightarrow \infty}I_{m}=I
$$
and \eqref{int} is fulfilled for function $y^{(1)}$ being
the limit of subsequence $(y_{m_n})$ of the sequence $\left(y_{m}\right)_{m\in\N}$.

\textit{Step 4.}
Let us denote in relation \eqref{int}:
\begin{eqnarray*}
&& \gamma_1 (t):=(q(t)-\lambda^{(1)}w_{\alpha}(t))y^{(1)}(t),\\
&& \gamma_2 (t):= -p'(t){_{0}}I^{1-\alpha}_{t} [y^{(1)}](t),\\
&& \gamma_3 (t):= -p(t){_{0}}I^{1-\alpha}_{t}[y^{(1)}](t).
\end{eqnarray*}
We observe that $\gamma_{j} \in C([0,\pi];\R),
\;j=1,2,3 $ and ${_{0}}D^{1-\alpha}_{t}[\gamma_{3}]\in L^{2}(0,\pi;\R)$ because
$$
{_0}D^{1-\alpha}_{t}[\gamma_{3}]={_{0}}D^{1-\alpha}_{t}\left [p \cdot {_{0}}I^{1-\alpha}_{t}[y^{(1)}]\right]
= {_{0}}I^{\alpha}_{t} \left[\frac{d}{dt}\left (p \cdot {_{0}}I^{1-\alpha}_{t}[y^{(1)}]\right)\right]
$$
$$
= {_{0}}I^{\alpha}_{t} \left [p' \cdot{_{0}}I^{1-\alpha}_{t}[y^{(1)}] + p\cdot
\; {^{C}_{0}}D^{\alpha}_{t}[y^{(1)}]\right].
$$
Both parts of the above function belong to the $L^{2}(0,\pi;\R)$ space.\\
Assuming that function $h$ in \eqref{int} is an arbitrary function fulfilling assumptions
of Lemma~\ref{lem:C} and  applying Lemma~\ref{lem:C} part (a),  we conclude that
$\gamma_3 =-p\cdot {_{0}}I^{1-\alpha}_{t}[y^{(1)}]\in C^{1}([0,\pi];\R)$.
From this fact it follows that $p\cdot \frac{d}{dt} {_{0}}I^{\alpha}_{t}[y^{(1)}]\in C([0,\pi];\R)$
and integral \eqref{int} can be rewritten as
\begin{equation*}
0=\int\limits_{0}^{\pi}\left[p(t){^{C}_{0}}\textsl{D}_t^{\alpha}[y^{(1)}](t){^{C}_{0}}\textsl{D}_t^{\alpha}[h](t)
+(q(t)-\lambda^{(1)}w_{\alpha}(t))y^{(1)}(t)h(t)\right]dt.
\end{equation*}
Now, we apply Lemma~\ref{lem:C} part (b) defining
\begin{eqnarray*}
&&\bar{\gamma}_{1} (t) : =  \gamma_1 (t)=(q(t)-\lambda^{(1)}w_{\alpha})y^{(1)}(t),\\
&&\bar{\gamma}_{2} (t):= p(t)\frac{d}{dt}{_{0}}I^{1-\alpha}_{t}[y^{(1)}](t).
\end{eqnarray*}
This time $\bar{\gamma}_{1},\bar{\gamma}_2\in C([0,\pi];\R)$ and from Lemma~\ref{lem:C} part (b)
it follows that
\begin{equation*}
{^{C}_{t}}\textsl{D}_{\pi}^{\alpha}\left[p(\tau){^{c}_{0}}\textsl{D}_{\tau}^{\alpha}[y^{(1)}](\tau)\right](t)
+q(t)y^{(1)}(t)= \lambda^{(1)}w_{\alpha}(t)y^{(1)}(t).
\end{equation*}
By construction this solution fulfills the Dirichlet boundary conditions
\begin{equation*}
y^{(1)}(0)=y^{(1)}(\pi)=0
\end{equation*}
and is nontrivial because
\begin{equation*}
\mathcal{J}(y^{(1)})=\int\limits_{0}^{\pi}w_{\alpha}(t)\left (y^{(1)}(t)\right)^{2}\;dt=1.
\end{equation*}
In addition, we also have for the solution
$$ {_{0}}D^{\alpha}_{t}[y^{(1)}] ={^{C}_{0}}D^{\alpha}_{t}[y^{(1)}]\in C([0,\pi];\R). $$
Let us observe that from the Dirichlet boundary conditions
it follows that $y^{(1)}$ also solves the FSLP (\ref{eq:SLE})-(\ref{eq:BC}) in $[0,\pi]$.\\

\textit{Step 5.} Now, let us restore the superscript on $y_{m}^{(1)}$ and show that
$\left(y_{m}^{(1)}\right)_{m\in\N}$ itself converges to $y^{(1)}$. First,
let us point out that for given $\lambda$ the solution of
\begin{equation}
\label{eq:SL:2}
{^{C}_{t}}\textsl{D}_{\pi}^{\alpha}\left[
p(\tau){^{c}_{0}}\textsl{D}_{\tau}^{\alpha}[y](\tau)\right](t)+q(t)y(t)
= \lambda^{(1)}w_{\alpha}(t)y(t),
\end{equation}
subject to the boundary conditions
\begin{equation}\label{eq:BC:2}
y(a)=y(b)=0
\end{equation}
and the normalization condition
\begin{equation}\label{eq:IC:2}
\int\limits_0^{\pi} w_{\alpha}(t)y^2(t)\;dt=1
\end{equation}
is unique except for a sign. Next, let us assume that $y^{(1)}$ solves
Sturm--Liouville equation \eqref{eq:SL:2} and that corresponding eigenvalue
is $\lambda=\lambda^{(1)}$. In addition, suppose that $y^{(1)}$ is nontrivial, i.e.,
we can find $t_0\in [0,\pi]$ such that $y^{(1)}(t_0)\neq 0$ and choose the sign
so that $y^{(1)}(t_0)>0$. Similarly, for all $m\in\N$, let $y_m^{(1)}$ solve \eqref{eq:SL:2}
with corresponding eigenvalue $\lambda=\lambda_m^{(1)}$ and let us choose the signs so that
$y_m^{(1)}(t_0)\geq 0$. Now, suppose that $\left(y_{m}^{(1)}\right)_{m\in\N}$ does not converge
to $y^{(1)}$. It means that we can find another subsequence of $\left(y_{m}^{(1)}\right)_{m\in\N}$
such that it converges to another solution $\bar{y}^{(1)}$ of \eqref{eq:SL:2} with $\lambda=\lambda^{(1)}$.
We know that for $\lambda=\lambda^{(1)}$ solution of \eqref{eq:SL:2} subject to \eqref{eq:BC:2}
and \eqref{eq:IC:2} must be unique except for a sign, thence
\begin{equation*}
\bar{y}^{(1)}=-y^{(1)}
\end{equation*}
and we must have $\bar{y}^{(1)}(t_0)<0$. However, it is impossible because for all $m\in\N$
value of $y_m^{(1)}$ in $t_0$ is greater or equal zero. It means that we have contradiction
and hence, choosing each $y_m^{(1)}$ with adequate sign, we obtain $y_m^{(1)}\rightarrow y^{(1)}$.

\textit{Step 6.} In order to find eigenfunction $y^{(2)}$ and the corresponding eigenvalue $\lambda^{(2)}$,
we again minimize functional \eqref{eq:IF} subject to \eqref{eq:IC} and \eqref{eq:BC},
but now with an extra orthogonality condition
\begin{equation}
\label{eq:OC}
\int\limits_0^{\pi} w_{\alpha}(t)y(t) y^{(1)}(t)\;dt=0.
\end{equation}
If we approximate solution by
\begin{equation*}
y_{m}(t)= \frac{1}{\sqrt{w_{\alpha}}}\sum_{k=1}^{m} \beta_{k}\sin(kt),
\hspace{2cm} y_{m}(0)=y_{m}(\pi)=0,
\end{equation*}
then we again receive quadratic form \eqref{eq:F}. However, in this case admissible
solutions are points satisfying \eqref{eq:sphere} together with
\begin{equation}
\label{eq:hyperpl}
\frac{\pi}{2}\sum\limits_{k=1}^{m}\beta_k\beta_k^{(1)}=0,
\end{equation}
i.e., they lay in $(m-1)$-dimensional sphere. As before, we find that function
$I([\beta])$ has a minimum $\lambda_m^{(2)}$ and there exists $\lambda^{(2)}$ such that
\begin{equation*}
\lambda^{(2)}=\lim\limits_{m\rightarrow\infty}\lambda_m^{(2)},
\end{equation*}
because $J(y)$ is bounded from below. Moreover, it is clear that the following relation:
\begin{equation}
\label{eq:eig}
\lambda^{(1)}\leq\lambda^{(2)}
\end{equation}
holds. Now, let us denote by
\begin{equation*}
y_{m}^{(2)}(t)= \frac{1}{\sqrt{w_{\alpha}}}\sum_{k=1}^{m} \beta_{k}^{(2)}\sin(kt),
\end{equation*}
the linear combination achieving the minimum $\lambda_m^{(2)}$, where
$\beta^{(2)}=(\beta_1^{(2)},\dots,\beta_m^{(2)})$ is the point satisfying
\eqref{eq:sphere} and \eqref{eq:hyperpl}. By the same argument as before,
we can prove that the sequence $(y_m^{(2)})_{m\in\N}$ converges uniformly
to a limit function $y^{(2)}$, which satisfies the Strum-Liouville equation
\eqref{eq:SLE} with $\lambda^{(2)}$, the boundary conditions \eqref{eq:BC},
normalization condition \eqref{eq:IC} and the orthogonality condition \eqref{eq:OC}.
Therefore, solution $y^{(2)}$ of the FSLP corresponding to the eigenvalue $\lambda^{(2)}$ exists.
Furthermore, because orthogonal functions cannot be linearly dependent, and since only
one eigenfunction corresponds to each eigenvalue (except for a constant factor),
we have the strict inequality
\begin{equation*}
\lambda^{(1)}<\lambda^{(2)}
\end{equation*}
instead of \eqref{eq:eig}. Finally, if we repeat the above procedure, with similar modifications,
we can obtain eigenvalues $\lambda^{(3)},\lambda^{(4)},\dots$
and corresponding eigenfunctions $y^{(3)},y^{(4)},\dots$.
\end{proof}


\subsection{The First Eigenvalue}

In this section we prove two theorems showing that the first eigenvalue of problem
\eqref{eq:SLE}--\eqref{eq:BC} is a minimum value of certain functionals. As in the
proof of Theorem~\ref{thm:exist} in the sequel, for simplicity, we assume that
$a=0$ and $b=\pi$ in the problem \eqref{eq:SLE}--\eqref{eq:BC}.

\begin{theorem}
\label{thm:FE}
Let $y^{(1)}$ denote the eigenfunction, normalized to satisfy the isoperimetric constraint
\begin{equation}
\label{eq:IC2}
\mathcal{J}(y)=\int\limits_0^{\pi}w_{\alpha}(t)y^2(t)\;dt=1,
\end{equation}
associated to the first eigenvalue $\lambda^{(1)}$ of problem
\eqref{eq:SLE}--\eqref{eq:BC} and assume that function
${_{t}}D_{\pi}^{\alpha}\left[p\cdot {^{C}_0}D_{t}^{\alpha}[y]\right]$ is continuous.
Then, $y^{(1)}$ is a minimizer of the following variational functional
\begin{equation}
\label{eq:F2}
\mathcal{I}(y)=\int\limits_0^{\pi}\left[p(t)
({^{C}_{0}}\textsl{D}_{t}^{\alpha}[y](t))^{2}+q(t)y^{2}(t)\right]\;dt,
\end{equation}
in the class $C([0,\pi];\R)$ with ${^{C}_{0}}\textsl{D}_{t}^{\alpha}[y]\in C([0,\pi];\R)$
subject to the boundary conditions
\begin{equation}
\label{eq:BC2}
y(0)=y(\pi)=0
\end{equation}
and an isoperimetric constraint \eqref{eq:IC2}. Moreover,
\begin{equation*}
\mathcal{I}(y^{(1)})=\lambda^{(1)}.
\end{equation*}
\end{theorem}

\begin{proof}
Suppose that $y\in C([0,\pi];\R)$ is a minimizer of $\mathcal{I}$
and ${^{C}_{0}}\textsl{D}_{t}^{\alpha}[y]\in C([0,\pi];\R)$. Then,
by Theorem~\ref{thm:EL}, there is number $\lambda$ such that $y$ satisfies equation
\begin{equation}
\label{eq:SLE2}
{_{t}}\textsl{D}_{\pi}^{\alpha}\left[p(\tau){^{c}_{0}}\textsl{D}_{\tau}^{\alpha}[y](\tau)\right](t)
+q(t)y(t)= \lambda w_{\alpha}(t)y(t),
\end{equation}
and conditions \eqref{eq:IC2}, \eqref{eq:BC2}.
Since ${_{t}}D_{\pi}^{\alpha}\left[p\cdot {^{C}_0}D_{t}^{\alpha}[y]\right]$
and ${^{C}_{t}}D_{\pi}^{\alpha}\left[p\cdot {^{C}_0}D_{t}^{\alpha}[y]\right]$
are continuous, it follows that $\left.p(t)\cdot {^{C}_0}D_{t}^{\alpha}[y](t)\right|_{t=\pi}=0$.
Therefore, equation \eqref{eq:SLE2} is equivalent to
\begin{equation}
\label{eq:SLEn}
{^{C}_{t}}\textsl{D}_{\pi}^{\alpha}\left[p(\tau){^{c}_{0}}\textsl{D}_{\tau}^{\alpha}[y](\tau)\right](t)
+q(t)y(t)= \lambda w_{\alpha}(t)y(t).
\end{equation}
Let us multiply \eqref{eq:SLE2} by $y$ and integrate it on the interval $[0,\pi]$, then
\begin{equation*}
\int\limits_0^{\pi}\left(y(t)\cdot {_{t}}\textsl{D}_{\pi}^{\alpha}\left[
p(\tau){^{c}_{0}}\textsl{D}_{\tau}^{\alpha}[y](\tau)\right](t)+q(t)y^2(t)\right)\;dt
=\lambda\int\limits_0^{\pi}w_{\alpha}(t)y^2(t)\;dt.
\end{equation*}
Applying the integration by the parts formula for fractional derivatives
(cf. \eqref{eq:IBP}) and having in mind that conditions \eqref{eq:BC2},
\eqref{eq:IC2} and $\left.p(t)\cdot {^{C}_0}D_{t}^{\alpha}[y](t)\right|_{t=\pi}=0$ hold, one has
\begin{equation*}
\int\limits_0^{\pi}\left(\left({^{c}_{0}}\textsl{D}_{t}^{\alpha} [y](t)\right)^2 p(t)
+q(t)y^2(t)\right)\;dt=\lambda.
\end{equation*}
Hence
\begin{equation*}
\mathcal{I}(y)=\lambda.
\end{equation*}
Any solution to problem \eqref{eq:IC2}--\eqref{eq:BC2} satisfies equation \eqref{eq:SLEn}
must be nontrivial since \eqref{eq:IC2} holds, so $\lambda$ must be an eigenvalue. Moreover,
according to Theorem~\ref{thm:exist} there is the least element in the spectrum being eigenvalue
$\lambda^{(1)}$ and the corresponding eigenfunction $y^{(1)}$ normalized to meet the
isoperimetric condition. Therefore $J(y^{(1)})=\lambda^{(1)}$.
\end{proof}

\begin{definition}
We call to functional $\mathcal{R}$ defined by
\begin{equation*}
\mathcal{R}(y)=\frac{\mathcal{I}(y)}{\mathcal{J}(y)},
\end{equation*}
where $\mathcal{I}(y)$ is given by \eqref{eq:F2} and $\mathcal{J}(y)$ by \eqref{eq:IC2},
the Rayleigh quotient for the fractional Sturm--Liouville problem \eqref{eq:SLE}--\eqref{eq:BC}.
\end{definition}

\begin{theorem}
\label{thm:RQ}
Let us assume that function $y\in C([0,\pi];\R)$ with ${^{C}_{0}}\textsl{D}_{t}^{\alpha}[y]\in C([0,\pi];\R)$,
satisfying boundary conditions $y(0)=y(\pi)=0$ and being nontrivial, is a minimizer of Rayleigh quotient $\mathcal{R}$
for the Sturm--Liouville problem \eqref{eq:SLE}--\eqref{eq:BC}. Moreover, assume that function
${_{t}}D_{\pi}^{\alpha}\left[p\cdot {^{C}_0}D_{t}^{\alpha}[y]\right]$ is continuous. Then, value of
$\mathcal{R}$ in $y$ is equal to the first eigenvalue $\lambda^{(1)}$, i.e., $\mathcal{R}(y)=\lambda^{(1)}$.
\end{theorem}

\begin{proof}
Suppose that function $y\in C([0,\pi];\R)$ with ${^{C}_{0}}\textsl{D}_{t}^{\alpha}[y]\in C([0,\pi];\R)$,
satisfying $y(0)=y(\pi)=0$ and nontrivial, is a minimizer of Rayleigh quotient $\mathcal{R}$
and that value of $\mathcal{R}$ in $y$ is equal to $\lambda$, i.e.,
$$
\mathcal{R}(y)=\frac{\mathcal{I}(y)}{\mathcal{J}(y)}=\lambda.
$$
Consider one-parameter family of curves
\begin{equation*}
\hat{y}=y+h\eta,~~\left|h\right|\leq\varepsilon,
\end{equation*}
where $\eta\in C([0,\pi];\R)$ with ${^{C}_{0}}\textsl{D}_{t}^{\alpha}[\eta]\in C([0,\pi];\R)$
is such that $\eta(0)=\eta(\pi)=0$, $\eta\neq 0$ and define the following functions
\begin{equation*}
\fonction{\phi}{[-\varepsilon,\varepsilon]}{\R}{h}{\mathcal{J}(y+h\eta)
=\displaystyle\int\limits_0^{\pi}w_{\alpha}(t)(y(t)+h\eta(t))^2\;dt,}
\end{equation*}
\begin{equation*}
\fonction{\psi}{[-\varepsilon,\varepsilon]}{\R}{h}{\mathcal{I}(y+h\eta)
=\displaystyle\int\limits_0^{\pi}\left[p(t)
({^{C}_{0}}\textsl{D}_{t}^{\alpha}[y+h\eta](t))^{2}+q(t)(y(t)+h\eta(t))^{2}\right]dt}
\end{equation*}
and
\begin{equation*}
\fonction{\zeta}{[-\varepsilon,\varepsilon]}{\R}{h}{\mathcal{R}(y+h\eta)
=\frac{\mathcal{I}(y+h\eta)}{\mathcal{J}(y+h\eta)}.}
\end{equation*}
Since $\zeta$ is of class $C^1$ on $[-\varepsilon,\varepsilon]$ and
$$
\zeta(0)\leq\zeta (h),~~\left|h\right|\leq\varepsilon,
$$
we deduce that
\begin{equation*}
\zeta'(0)=\left.\frac{d}{dh}\mathcal{R}(y+h\eta)\right|_{h=0}=0.
\end{equation*}
Moreover, notice that
\begin{equation*}
\zeta'(h)=\frac{1}{\phi(h)}\left(\psi'(h)-\frac{\psi(h)}{\phi(h)}\phi'(h)\right)
\end{equation*}
and that
\begin{equation*}
\psi'(0)=\left.\frac{d}{dh}\mathcal{I}(y+h\eta)\right|_{h=0}=2\int\limits_0^{\pi}\left[
p(t)\cdot {^{C}_{0}}\textsl{D}_{t}^{\alpha}[y](t)
\cdot{^{C}_{0}}\textsl{D}_{t}^{\alpha}[\eta](t)+q(t)y(t)\eta(t)\right]\;dt,
\end{equation*}
\begin{equation*}
\phi'(0)=\left.\frac{d}{dh}\mathcal{J}(y+h\eta)\right|_{h=0}
=2\int\limits_0^{\pi}\left[w_{\alpha}(t)y(t)\eta(t)\right]\;dt.
\end{equation*}
Therefore
\begin{multline*}
\zeta'(0)=\left.\frac{d}{dh}\mathcal{R}(y+h\eta)\right|_{h=0}\\
=\frac{2}{\mathcal{J}(y)}\left(\int\limits_0^{\pi}p(t)
\cdot {^{C}_{0}}\textsl{D}_{t}^{\alpha}[y](t)\cdot{^{C}_{0}}\textsl{D}_{t}^{\alpha}[\eta](t)
+q(t)y(t)\eta(t)\;dt-\frac{\mathcal{I}(y)}{\mathcal{J}(y)}\int\limits_0^{\pi}
w_{\alpha}(t)y(t)\eta(t)\;dt\right)=0.
\end{multline*}
Having in mind that $\frac{\mathcal{I}(y)}{\mathcal{J}(y)}=\lambda$ and $\eta(0)=\eta(\pi)=0$,
using the integration by parts formula \eqref{eq:IBP} we obtain
\begin{equation*}
\int\limits_0^{\pi}\left({_{t}}\textsl{D}_{\pi}^{\alpha}
\left[p{^{c}_{0}}\textsl{D}_{\tau}^{\alpha}[y]\right](t)
+q(t)y(t)-\lambda w_{\alpha}(t)y(t)\right)\eta(t)\;dt=0.
\end{equation*}
Now, applying the fundamental lemma of the calculus of variations, we arrive at
\begin{equation}
\label{eq:SL:old}
{_{t}}\textsl{D}_{\pi}^{\alpha}\left[p(\tau)
\cdot{^{C}_{0}}\textsl{D}_{\tau}^{\alpha}[y](\tau)\right](t)
+q(t)y(t)=\lambda w_{\alpha}(t)y(t).
\end{equation}
Under our assumptions, $\left.p(t)\cdot {^{C}_0}D_{t}^{\alpha}[y](t)\right|_{t=\pi}=0$
and therefore equation \eqref{eq:SL:old} is equivalent to
\begin{equation}
\label{eq:SL:new}
{^{C}_{t}}\textsl{D}_{\pi}^{\alpha}\left[p(\tau)
\cdot{^{C}_{0}}\textsl{D}_{\tau}^{\alpha}[y](\tau)\right](t)+q(t)y(t)
=\lambda w_{\alpha}(t)y(t).
\end{equation}
Since $y\neq 0$ we deduce that number $\lambda$ is an eigenvalue of \eqref{eq:SL:new}.
On the other hand, let $\lambda^{(m)}$ be an eigenvalue and $y^{(m)}$
the corresponding eigenfunction, then
\begin{equation}
\label{eq:SLRa}
{^{C}_{t}}\textsl{D}_{\pi}^{\alpha}\left[p(\tau){^{C}_{0}}\textsl{D}_{\tau}^{\alpha}[y^{(m)}](\tau)\right](t)
+q(t)y^{(m)}(t)=\lambda^{(m)} w_{\alpha}(t)y^{(m)}(t).
\end{equation}
Similarly to the proof of Theorem~\ref{thm:FE}, we can obtain
\begin{equation*}
\frac{\int\limits_0^{\pi}\left(\left({^{C}_{0}}\textsl{D}_{t}^{\alpha} [y^{(m)}](t)\right)^2 p(t)
+q(t)(y^{(m)}(t))^2\right)\;dt}{\int\limits_0^{\pi}\lambda^{(m)} w_{\alpha}(t)(y^{(m)}(t))^2\;dt}
=\lambda^{(m)},
\end{equation*}
for any $m\in\N$. That is $\mathcal{R}(y^{(m)})=\frac{\mathcal{I}(y^{(m)})}{\mathcal{J}(y^{(m)})}
=\lambda^{(m)}$. Finally, since minimum value of $\mathcal{R}$ at $y$ is equal to $\lambda$, i.e.,
\begin{equation*}
\lambda\leq \mathcal{R}(y^{(m)})=\lambda^{(m)}~~\forall m\in\N,
\end{equation*}
we have $\lambda=\lambda^{(1)}$.
\end{proof}


\subsection{An Illustrative Example}

Let us consider the following fractional oscillator equation:
\begin{equation}
\label{eq:example}
{_{t}}\textsl{D}_{b}^{\alpha}\left[{^{c}_{a}}\textsl{D}_{\tau}^{\alpha}[y](\tau)\right](t)-\lambda y(t)=0,
\end{equation}
where $y(a)=y(b)=0$. One can easily check that problem of finding nontrivial solutions
to equation \eqref{eq:example} and corresponding values of parameter $\lambda$ is a
particular case of problem \eqref{eq:SLE}--\eqref{eq:BC} with $p(t)\equiv 1$, $q(t)\equiv 0$
and $w_{\alpha}(t)\equiv 1$. The corresponding variational functional is
\begin{equation*}
\mathcal{I}_{\alpha}(y)
= \int\limits_{a}^{b}p(t) \cdot (\Dcl[y](t))^{2}dt
= ||\sqrt{p}\;\;\Dcl[y]||^{2}_{L^{2}}
\end{equation*}
with the isoperimetric condition
\begin{equation*}
\int\limits_{a}^{b}y^{2}(t)dt = 1.
\end{equation*}
Let us fix the value of parameter $p$ and  assume that orders $\alpha_{1}, \alpha_{2}$ fulfill
the condition $\frac{1}{2}<\alpha_{1}<\alpha_{2}<1$. Then, we obtain for functionals
$\mathcal{I}_{\alpha_{1}}, \mathcal{I}_{\alpha_{2}}$ the following relation
\begin{eqnarray*}
&& \mathcal{I}_{\alpha_{1}}(y)= ||\sqrt{p}{^{C}_{a}}\textsl{D}_t^{\alpha_1}[y]||^{2}_{L^{2}}
= ||\sqrt{p}{_{a}}\textsl{I}_t^{1-\alpha_1}\left[\frac{d}{dt}y\right]||^{2}_{L^{2}}
=||\sqrt{p}{_{a}}I^{\alpha_2-\alpha_1}_{t}{_{a}}I^{1-\alpha_2}_{t}\left[\frac{d}{dt}y\right]||^{2}_{L^{2}}\\
&& \leq K_{\alpha_{2}-\alpha_{1}}^{2}\cdot ||\sqrt{p}{^{C}_{a}}\textsl{D}_t^{\alpha_2}[y]||^{2}_{L^{2}}
=K_{\alpha_{2}-\alpha_{1}}^{2}\mathcal{I}_{\alpha_{2}}(y),
\end{eqnarray*}
where we denoted
\begin{equation*}
K_{\alpha_{2}-\alpha_{1}}:= \frac{(b-a)^{\alpha_2-\alpha_1}}{\Gamma(\alpha_2-\alpha_1+1)}.
\end{equation*}
We observe that in the above estimation two cases occur:
\begin{eqnarray*}
&\textnormal{if }K_{\alpha_{2}-\alpha_{1}}\leq 1,
~\textnormal{then} & \mathcal{I}_{\alpha_{1}}(y) \leq \mathcal{I}_{\alpha_{2}}(y);\\
&\textnormal{if }K_{\alpha_{2}-\alpha_{1}}>1,~\textnormal{then}
& \mathcal{I}_{\alpha_{1}}(y) \leq K^{2}_{\alpha_{2}-\alpha_{1}}
\cdot \mathcal{I}_{\alpha_{2}}(y).
\end{eqnarray*}
The relations between functionals for different values of fractional order
lead to the set of inequalities for eigenvalues $\lambda_{j}$ valid
for any $j\in \mathbb{N}$:
\begin{eqnarray*}
&\textnormal{if }K_{\alpha_{2}-\alpha_{1}}\leq 1,~\textnormal{then}
& \lambda_{j}(\alpha_1)\leq \lambda _{j}(\alpha_2);\\
&\textnormal{if } K_{\alpha_{2}-\alpha_{1}}>1,~\textnormal{then}
& \lambda_{j}(\alpha_1)\leq K_{\alpha_{2}-\alpha_{1}}^{2} \cdot\lambda_{j}(\alpha_2).
\end{eqnarray*}
In particular when order $\alpha_{2}=1$ we get
\begin{eqnarray*}
&&\mathcal{I}_{\alpha_{1}}(y)=||\sqrt{p}{^{C}_{a}}\textsl{D}_t^{\alpha_1}[y]||^{2}_{L^{2}}
= ||\sqrt{p}{_{a}}\textsl{I}_t^{1-\alpha_1}\left[Dy\right]||^{2}_{L^{2}} \\
&& \leq  K^{2}_{1-\alpha_{1}}\cdot ||\sqrt{p}Dy||^{2}_{L^{2}}
=K^{2}_{1-\alpha_{1}}\mathcal{I}_{1}(y)
\end{eqnarray*}
and the following relations dependent on the value of constant $K_{1-\alpha_{1}}$
\begin{eqnarray*}
&\textnormal{if }K_{1-\alpha_{1}}\leq 1,~\textnormal{then}
& \mathcal{I}_{\alpha_{1}}(y) \leq \mathcal{I}_{1}(y);\\
&\textnormal{if } K_{1-\alpha_{1}}>1,~\textnormal{then}
& \mathcal{I}_{\alpha_{1}}(y) \leq K^{2}_{1-\alpha_{1}}\cdot \mathcal{I}_{1}(y).
\end{eqnarray*}
Thus, comparing the eigenvalues for the fractional and the classical harmonic oscillator
equation for boundary conditions $y(a)=y(b)=0$, we conclude that the respective classical eigenvalues are
higher than the ones resulting from the fractional problem for any  $j\in \mathbb{N}$. Namely,
\begin{eqnarray*}
&\textnormal{if } K_{1-\alpha_{1}}\leq 1,~\textnormal{then} & \lambda_{j}(\alpha_1)\leq \lambda _{j}(1)
=p\left (\frac{j\pi}{b-a}\right)^{2};\\
&\textnormal{if } K_{1-\alpha_{1}}>1,~\textnormal{then} & \lambda_{j}(\alpha_1)\leq K^{2}_{1-\alpha_{1}}
\cdot\lambda_{j}(1)=p\left (\frac{j\pi}{(b-a)^{\alpha_{1}}\Gamma(2-\alpha_{1})}\right)^{2}.
\end{eqnarray*}


\section{State of the Art}

The results of this chapter can be found in the paper \cite{Klimek:Ja}.


\clearpage{\thispagestyle{empty}\cleardoublepage}


\chapter*{Appendix}\markboth{APPENDIX}{}
\label{sec:ap}

In this appendix we prove two lemmas, concerning certain convergence properties of fractional
and classical derivatives, that are important in the proof of Theorem~\ref{thm:exist}.
Let us begin with the following definition of H\"{o}lder continuous functions.

\begin{definition}
Function $g$ is H\"{o}lder continuous in the interval $[a,b]$ with coefficient $0<\beta\leq 1$ if
\begin{equation}
\sup_{x,y\in [a,b], \; x\neq y}\frac{|g(x)-g(y)|}{|x-y|^{\beta}}<\infty.
\end{equation}
We denote this class of H\"{o}lder continuous functions as $C^{\beta}_{H}([a,b];\R)$.
\end{definition}

\begin{lemma}
\label{lem:D}
Let  $\alpha\in (0,1) $, functions $w,g\in C^{1}([0,\pi];\R) \cap C^{1}_{H}([-\pi,\pi];\R)$
be odd functions in  $ [-\pi,\pi]$ such that $w'',g''\in L^{2}(0,\pi;\R)$. If we denote
as $g_{m}$ the $m$-th sum of the Fourier series
of function $g$, then  the following
convergences are valid in $[0,\pi]$
\begin{eqnarray}
&& {^{C}_{0}}\textsl{D}_{t}^{\alpha}[g_{m}]
\stackrel{L^2}{\longrightarrow}{^{C}_{0}}\textsl{D}_{t}^{\alpha}[g] \label{con1} \\
&& \frac{d}{dt}{^{C}_{0}}\textsl{D}_{t}^{\alpha}[g_{m}]
\stackrel{L^1}{\longrightarrow}\frac{d}{dt}{^{C}_{0}}\textsl{D}_{t}^{\alpha}[g]\label{con2} \\
&& {^{C}_{0}}\textsl{D}_{t}^{\alpha}[w g_{m}]
\stackrel{L^2}{\longrightarrow}{^{C}_{0}}\textsl{D}_{t}^{\alpha}[w g]\label{con3}\\
&& \frac{d}{dt}{^{C}_{0}}\textsl{D}_{t}^{\alpha}[w g_{m}]
\stackrel{L^1}{\longrightarrow}\frac{d}{dt}{^{C}_{0}}\textsl{D}_{t}^{\alpha}[w g].\label{con4}
\end{eqnarray}
\end{lemma}

\begin{proof}
We can apply Property~\ref{prop:K} and  estimate the
$\left\|{^{C}_{0}}\textsl{D}_{t}^{\alpha}[g_{m}]-{^{C}_{0}}\textsl{D}_{t}^{\alpha}[g]\right\|_{L^{2}}$
norm in $[0,\pi]$ as follows:
$$
\left\|{^{C}_{0}}\textsl{D}_{t}^{\alpha}[g_{m}]-{^{C}_{0}}\textsl{D}_{t}^{\alpha}[g]\right\|_{L^{2}}
\leq \left\|{_{0}}\textsl{I}_{t}^{1-\alpha}[g'_{m}-g']|\right\|_{L^{2}}\leq K_{1-\alpha} \cdot  ||g'_{m}-g'||_{L^{2}}.
$$
For even functions from the $C^{1}([0,\pi];\R)\cap C^{1}_{H}([-\pi,\pi];\R)$ space,
$g'_{m}$ is the $m$th sum of the Fourier series of the derivative $g'$. Hence,
in interval $[0,\pi]$,
$$
g'_{m}\stackrel{L^2}{\longrightarrow}g'
$$
and from the above inequalities it follows that \eqref{con1} is valid on $[0,\pi]$:
$$
{^{C}_{0}}\textsl{D}_{t}^{\alpha}[g_{m}]\stackrel{L^2}{\longrightarrow}{^{C}_{0}}\textsl{D}_{t}^{\alpha}[g].
$$
Let us observe that for $t>0$
$$
\frac{d}{dt}{^{C}_{0}}\textsl{D}_{t}^{\alpha}[g_{m}](t)
= {_{0}}\textsl{D}_{t}^{\alpha}[g'_{m}](t)={^{C}_{0}}\textsl{D}_{t}^{\alpha}[g'_{m}](t)
+\frac{g_{m}'(0)\cdot t^{-\alpha}}{\Gamma(1-\alpha)}
$$
$$
\frac{d}{dt}{^{C}_{0}}\textsl{D}_{t}^{\alpha}[g] (t)
= {_{0}}\textsl{D}_{t}^{\alpha}[g'](t)={^{C}_{0}}\textsl{D}_{t}^{\alpha}[g'](t)
+\frac{g'(0)\cdot t^{-\alpha}}{\Gamma(1-\alpha)}.
$$
Therefore, we can  estimate the distance between
$\frac{d}{dt}{^{c}_{0}}\textsl{D}_{t}^{\alpha}[g_{m}]$
and $\frac{d}{dt}{^{c}_{0}}\textsl{D}_{t}^{\alpha}[g]$
in interval $[0,\pi]$ using \eqref{K} for $\beta=1-\alpha$ and $p=1$
$$
\left\|\frac{d}{dt}{^{C}_{0}}\textsl{D}_{t}^{\alpha}[g_{m}]
-\frac{d}{dt}{^{C}_{0}}\textsl{D}_{t}^{\alpha}[g]\right\|_{L^{1}}
$$
$$
\leq \left\|{^{C}_{0}}\textsl{D}_{t}^{\alpha}[g'_{m}-g']\right\|_{L^{1}}
+ \left\|(g'_{m}(0)-g'(0))\cdot \frac{t^{-\alpha}}{\Gamma(1-\alpha)}\right\|_{L^{1}}
$$
$$
= \left\|{_{0}}\textsl{I}_{t}^{1-\alpha}[g''_{m}-g'']\right\|_{L^{1}}
+|g'_{m}(0)-g'(0)|\cdot \left\|\frac{t^{-\alpha}}{\Gamma(1-\alpha)}\right\|_{L^{1}}
$$
$$
\leq K_{1-\alpha}\cdot ||g''_{m}-g''||_{L^{1}}+|g'_{m}(0)-g'(0)|
\cdot \frac{\pi^{1-\alpha}}{\Gamma(2-\alpha)}
$$
$$
\leq K_{1-\alpha}\cdot \sqrt{\pi}\cdot ||g''_{m}-g''||_{L^{2}}+|g'_{m}(0)-g'(0)|
\cdot \frac{\pi^{1-\alpha}}{\Gamma(2-\alpha)}.
$$
By assumptions we have in $[-\pi,\pi]$ (thence also in $[0,\pi]$)
$$
g''_{m}\stackrel{L^2}{\longrightarrow}g'',\quad g'_{m}(0)\longrightarrow g'(0).
$$
Hence we conclude that (\ref{con2}) is valid. \\
The convergence given in (\ref{con3}) follows from (\ref{con1}), namely
$$
\left\|{^{C}_{0}}\textsl{D}_{t}^{\alpha}[w g_{m}]
-{^{C}_{0}}\textsl{D}_{t}^{\alpha}[w g]\right\|_{L^{2}}
$$
$$
= \left\|{_{0}}\textsl{I}_{t}^{1-\alpha}\left[\left (w g_{m}\right)'
-\left (w g\right)'\right]\right\|_{L^{2}}
$$
$$
\leq \left\|{_{0}}\textsl{I}_{t}^{1-\alpha}[w(g'_{m}-g')]\right\|_{L^{2}}
+\left\|{_{0}}\textsl{I}_{t}^{1-\alpha}[w'(g_{m}-g)]\right\|_{L^{2}}
$$
$$
\leq K_{1-\alpha}\cdot ||w(g'_{m}-g')||_{L^{2}}
+K_{1-\alpha}\cdot ||w'(g_{m}-g)||_{L^{2}}
$$
$$
\leq K_{1-\alpha} \left (||w||\cdot ||g'_{m}-g'||_{L^{2}}
+||w'|| \cdot ||g_{m}-g||_{L^{2}}\right),
$$
where $||\cdot||$ denotes the supremum norm in the $C([0,\pi];\R)$ space.
From assumptions of our lemma, it follows that in $[0,\pi]$
$$
g'_{m}\stackrel{L^2}{\longrightarrow}g',\quad\quad g_{m}\stackrel{L^2}{\longrightarrow}g.
$$
Thus convergence \eqref{con3} is valid.\\
To prove convergence \eqref{con4} we start by observing that for $t>0$
$$
\frac{d}{dt}{^{C}_{0}}\textsl{D}_{t}^{\alpha}[w g_{m}] (t)
= {_{0}}\textsl{D}_{t}^{\alpha}\left[\left(w g_{m}\right)'\right](t)
={^{C}_{0}}\textsl{D}_{t}^{\alpha}\left[\left(w g_{m}\right)'\right](t)
+\frac{(w g_{m})'(0)\cdot t^{-\alpha}}{\Gamma(1-\alpha)}
$$
$$
\frac{d}{dt}{^{C}_{0}}\textsl{D}_{t}^{\alpha}[w g](t)
= {_{0}}\textsl{D}_{t}^{\alpha}\left[\left(w g\right)'\right](t)
={^{C}_{0}}\textsl{D}_{t}^{\alpha}\left[\left (w g\right)'\right](t)
+\frac{(w g)'(0)\cdot t^{-\alpha}}{\Gamma(1-\alpha)}.
$$
For the $L^{1}$-distance between $\frac{d}{dt}{^{C}_{0}}\textsl{D}_{t}^{\alpha}[w g_{m}]$
and $\frac{d}{dt}{^{C}_{0}}\textsl{D}_{t}^{\alpha}[w g] $ in interval $[0,\pi]$
we have
\begin{equation}
\label{est}
\left\|\frac{d}{dt}{^{c}_{0}}\textsl{D}_{t}^{\alpha}[w g_{m}]
-\frac{d}{dt}{^{c}_{0}}\textsl{D}_{t}^{\alpha}[w g]\right\|_{L^{1}}
\end{equation}
$$
\leq ||{^{c}_{0}}\textsl{D}_{t}^{\alpha}\left[\left (w g_{m}\right)'-\left (w g\right)'\right]||_{L^{1}}
+ \left\|\left [\left(w g_{m}\right)'(0)-\left(w g\right)'(0)\right]
\cdot \frac{t^{-\alpha}}{\Gamma(1-\alpha)}\right\|_{L^{1}}
$$
$$
\leq ||{_{0}}\textsl{I}_{t}^{1-\alpha}\left[\left (w g_{m}\right)''
-\left (w g\right)''\right]||_{L^{1}}+ \left|\left(w g_{m}\right)'(0)
-\left(w g\right)'(0)\right|\cdot \left\|\frac{t^{-\alpha}}{\Gamma(1-\alpha)}\right\|_{L^{1}}
$$
$$
\leq K_{1-\alpha}\cdot || \left (w g_{m}\right)''-\left (w g\right)''||_{L^{1}}
+\left\|\left(w g_{m}\right)'(0)-\left(w g\right)'(0)\right\|
\cdot \frac{\pi^{1-\alpha}}{\Gamma(2-\alpha)}
$$
$$
\leq K_{1-\alpha}\cdot \sqrt{\pi}\cdot || \left (w g_{m}\right)''
-\left (w g\right)''||_{L^{2}}+\left\|\left(w g_{m}\right)'(0)
-\left(w g\right)'(0) \right\|\cdot \frac{\pi^{1-\alpha}}{\Gamma(2-\alpha)}.
$$
Because
$$
\left (w g_{m}\right)''-\left (w g\right)''
$$
$$
= w(g''_{m}-g'') + 2 \left (w\right)'\cdot (g'_{m}-g') +\left (w\right)''\cdot (g_{m}-g)
$$
we have
$$
|| \left (w g_{m}\right)''-\left (w g\right)''||_{L^{2}}
$$
$$
\leq ||w||\cdot || g''_{m} - g''||_{L^{2}} + 2 \cdot ||\left (w\right)'||
\cdot || g'_{m} - g'||_{L^{2}}+ ||w''||_{L^{2}}\cdot || g_{m} - g||_{L^{2}}.
$$
From the assumptions of the lemma it follows that for $j=0,1,2$
$$
\lim_{m\longrightarrow \infty}|| g^{(j)}_{m} - g^{(j)}||_{L^{2}} =0.
$$
Hence,
$$
\lim_{m\longrightarrow \infty}|| \left (w g_{m}\right)''-\left (w g\right)''||_{L^{2}}=0.
$$
In addition,
$$
\lim_{m\longrightarrow \infty}\left|\left (w g_{m}\right)'(0)-\left(w g\right)'(0)\right|
$$
$$
= \lim_{m\longrightarrow \infty}\left|\left (w\right)'(0)(g_{m}(0)-g(0)) + w(0)(g'_{m}(0)-g'(0))\right|
$$
$$
\leq \lim_{m\longrightarrow \infty}\left|\left (w\right)'(0)(g_{m}(0)-g(0))\right|
+  \lim_{m\longrightarrow \infty}\left|w(0)(g'_{m}(0)-g'(0))\right|=0.
$$
Taking into account estimation \eqref{est} and the above inequalities,
we conclude that \eqref{con4} is valid.
\end{proof}

\begin{lemma}
\label{lem:E}
Let $\alpha\in \left(\frac{1}{2},1\right), \;\; \beta \leq \alpha -\frac{1}{2}$,
function $w$ be positive, even function in  $ [-\pi,\pi]$ and
$w'\in  C^{\beta}_{H}([-\pi,\pi];\R) $. Function $h'$ is the derivative of $h$
defined by assumptions of Lemma~\ref{lem:B} and formula \eqref{defh2}, function $g$ is defined as
$$
g(t):= h(t)w(t).
$$
If we denote as $g_{m}$ the $m$th sum of the Fourier series
of function $g$, then  the following convergences
are valid in interval $[0,\pi]$:
\begin{eqnarray}
&& g'_{m}\stackrel{C}{\longrightarrow}g', \label{cp1}\\
&& g'_{m}(0)\longrightarrow g'(0), \label{cp2}\\
&& g'_{m}(\pi)\longrightarrow g'(\pi).\label{cp3}
\end{eqnarray}
\end{lemma}

\begin{proof}
Definition \eqref{defh2} in interval $[0,\pi]$ implies for derivative $h'$
\begin{equation}
h'(t)= {_{0}}\textsl{I}^{\alpha}_{t}[\gamma](t)+At^{\alpha}+Bt^{1+\alpha},
\end{equation}
where $\gamma\in C[0,\pi]$ and constants $A,B\in \mathbb{R}$ are specified
by conditions \eqref{cn1}, \eqref{cn2} in the proof of Lemma~\ref{lem:A}.
Let us observe that $t^{1+\alpha}\in C^{1}([0,\pi];\R)$, function
$t^{\alpha}$ is H\"{o}lder continuous  in $[0,\pi]$ with coefficient
$\beta \leq \alpha$, thus it can be extended to an odd/even,
H\"{o}lder continuous function in interval $[-\pi,\pi]$.
In addition ${_{0}}\textsl{I}^{\alpha}_{t}[\gamma](t)$ is H\"{o}lder continuous function
in $[0,\pi]$ with coefficient $\beta \leq \alpha-\frac{1}{2}$ because:
$$
\frac{\left|{_{0}}\textsl{I}^{\alpha}_{t}[\gamma](t)-{_{0}}\textsl{I}^{\alpha}_{t}[\gamma](s)\right|}{|t-s|^{\beta}}
\leq \frac{2 \cdot ||\gamma||_{L^{2}}}{\Gamma(\alpha)\sqrt{2\alpha-1}}\cdot |t-s|^{\alpha-\frac{1}{2}-\beta}
\leq \frac{2 \cdot ||\gamma||_{L^{2}}}{\Gamma(\alpha)\sqrt{2\alpha-1}}\cdot \pi^{\alpha-\frac{1}{2}-\beta}<\infty
$$
and can be extended  to an odd/even, H\"{o}lder continuous function in interval $[-\pi,\pi]$.
Observe that for H\"{o}lder continuous functions in $[-\pi,\pi]$ we have the absolute convergence of their
Fourier series. For function $g' $ we obtain in $[0,\pi]$
$$
g'(t) = h'(t)w(t)+h(t)w'(t).
$$
Both terms on the right-hand side are, by assumption, functions from the
$C^{\beta}_{H}([0,\pi];\R)$-space and  can be extended to odd/even functions
in the $C^{\beta}_{H}([-\pi,\pi];\R)$ space. Hence, their Fourier series
are absolutely convergent in $[-\pi,\pi]$. Concluding, we have for function
$g'$ the convergence in interval $[-\pi,\pi]$
$$
g'_{m}\stackrel{C}{\longrightarrow}g'.
$$
Thus the sequence  $g'_{m}$ of partial sums is
also absolutely convergent in interval $[0,\pi]$.
Formulas \eqref{cp2}, \eqref{cp3} are
a straightforward consequence of this fact.
\end{proof}


\clearpage{\thispagestyle{empty}\cleardoublepage}


\chapter*{Conclusions and Future Work}\markboth{CONCLUSIONS AND FUTURE WORK}{}
\label{ch:Conclusions}

This thesis was dedicated to generalized fractional calculus of variations.
We extended standard fractional variational calculus, by considering problems
with generalized fractional operators, that by choosing special kernels reduce, e.g.,
to fractional operators of Riemann--Liouville, Caputo, Hadamard, Riesz or Katugampola types.
First, we proved several properties of generalized fractional operators,
including boundedness in the space of $p$-Lebesgue integrable functions,
and generalized fractional integration by parts formula. Next, we applied standard methods
of fractional variational calculus to find admissible functions giving minima to certain functionals.
We considered cases of one and several variables. However, because in standard methods
it is assumed that Euler--Lagrange equations are solvable, we presented certain results
according to direct methods, where it is not the case. We proved a Tonelli type theorem
ensuring existence of minimizers and then obtained necessary optimality condition
giving candidates for solutions. The last chapter was devoted to the fractional Sturm--Liouville problem.
Applying methods of fractional variational calculus we proved that there exists an infinite increasing
sequence of eigenvalues, to each eigenvalue corresponds an eigenfunction and all of them are orthogonal.
Moreover, we presented two theorems concerning the first eigenvalue.

Concluding, our results cover several variational problems with particular fractional operators
and give a compact and transparent view for the fractional calculus of variations. We trust that
our work will provide new insights to further research on the subject, where still much remains to be done.

This research can choose several directions. Here, we find important to mention the following ones. We can

\begin{itemize}
\item consider variational problems with higher order derivatives;
\item consider Lagrangians with different operators $K_{P^1},K_{P^2},\dots$ and $B_{P^1},B_{P^2},\dots$;
\item apply direct methods to multidimensional variational problems;
\item explore problems of generalized fractional calculus of variations with holonomic or non holonomic constraints;
\item with the help of fractional variational calculus show fractional counterpart of isoperimetric inequality;
\item continue with varational methods in problems of generalized fractional optimal control including
necessary conditions of optimality or maximum principle;
\end{itemize}

The results mentioned in this thesis were published to peer reviewed international
journals (see \cite{LTDExist,GenExist,GreenThm,MyID:226,FVC_Gen_Int,MyID:207,FVC_Sev,CLandFR,tatiana,Ja,Klimek:Ja}).


\clearpage{\thispagestyle{empty}\cleardoublepage}



\clearpage{\thispagestyle{empty}\cleardoublepage}


\printindex


\clearpage{\thispagestyle{empty}\cleardoublepage}

\ \
\newpage
\thispagestyle{empty}
\


\begin{thebibliography}{10}

\bibitem{Abel}
N. H. Abel,
{\em {E}uvres completes de Niels Henrik Abel},
Christiana: Imprimerie de Grondahl and Son; New York and London:
Johnson Reprint Corporation. VIII, 621 pp., 1965.

\bibitem{OmPrakashAgrawal3}
O. P. Agrawal,
Fractional variational calculus and the transversality conditions,
J. Phys. A: Math. Gen., {\bf 39} (2006), no.~33, 10375--10384.

\bibitem{OmPrakashAgrawal2}
O. P. Agrawal,
Generalized Euler-Lagrange equations and transversality
conditions for FVPs in terms of the Caputo derivative, J. Vib. Control, {\bf 13}
(2007), no.~9-10, 1217--1237.

\bibitem{OmPrakashAgrawal}
O. P. Agrawal,
Generalized variational problems and Euler-Lagrange equations,
Comput. Math. Appl. {\bf 59} (2010), no.~5, 1852--1864.

\bibitem{AlM09}
Q. M. Al-Mdallal,
An efficient method for solving fractional Sturm-Liouville problems,
Chaos Solitons Fractals {\bf 40} (2009), no.~1, 183--189.

\bibitem{QMQ}
Q. M. Al-Mdallal,
On the numerical solution of fractional Sturm-Liouville problems,
Int. J. Comput. Math. {\bf 87} (2010), no.~12, 2837--2845.

\bibitem{Almeida:AML}
R. Almeida,
Fractional variational problems with the Riesz-Caputo derivative,
Appl. Math. Lett. {\bf 25} (2012), no.~2, 142--148.

\bibitem{MyID:182}
R. Almeida, A. B. Malinowska, D. F. M. Torres,
A fractional calculus of variations for multiple integrals with application to vibrating string,
J. Math. Phys. {\bf 51} (2010), no.~3, 033503, 12~pp.
{\tt arXiv:1001.2722}

\bibitem{MyID:209}
R. Almeida, A. B. Malinowska, D. F. M. Torres,
Fractional Euler-Lagrange differential equations via Caputo derivatives,
In: Fractional Dynamics and Control,
Springer New York, 2012, Part~2, 109--118.
{\tt arXiv:1109.0658}

\bibitem{Shakoor:01}
R. Almeida, S. Pooseh, D. F. M. Torres,
Fractional variational problems depending on indefinite integrals,
Nonlinear Anal. {\bf 75} (2012), no.~3, 1009--1025.
{\tt arXiv:1102.3360}

\bibitem{AlmeidaSamko}
A. Almeida, S. Samko,
Fractional and hypersingular operators in variable exponent
spaces on metric measure spaces, Mediterr. J. Math. {\bf 6} (2009), 215--232.

\bibitem{DerInt}
R. Almeida, D. F. M. Torres,
Calculus of variations with fractional derivatives
and fractional integrals,
Appl. Math. Lett. {\bf 22} (2009), no.~12, 1816--1820.
{\tt arXiv:0907.1024}

\bibitem{isoJMAA}
R. Almeida, D. F. M. Torres,
Holderian variational problems subject to integral constraints,
J. Math. Anal. Appl. {\bf 359} (2009), no.~2, 674--681.
{\tt arXiv:0807.3076}

\bibitem{isoNabla}
R. Almeida, D. F. M. Torres,
Isoperimetric problems on time scales with nabla derivatives,
J. Vib. Control {\bf 15} (2009), no.~6, 951--958.
{\tt arXiv:0811.3650}

\bibitem{Isoperimetric}
R. Almeida, D. F. M. Torres,
Necessary and sufficient conditions for the fractional
calculus of variations with Caputo derivatives,
Commun. Nonlinear Sci. Numer. Simul. 16 (2011), no.~3, 1490--1500.
{\tt arXiv:1007.2937}

\bibitem{isi}
K. Balachandran, J. Y. Park, J. J. Trujillo,
Controllability of nonlinear fractional dynamical systems,
Nonlinear Anal. {\bf 75} (2012), no.~4, 1919--1926.

\bibitem{Baleanu}
D. Baleanu, I. S. Muslih,
Lagrangian formulation of classical fields within Riemann-Liouville fractional derivatives,
Phys. Scripta {\bf 72} (2005), no.~2-3, 119--121.

\bibitem{MyID:152}
N. R. O. Bastos, R. A. C. Ferreira, D. F. M. Torres,
Necessary optimality conditions for fractional
difference problems of the calculus of variations,
Discrete Contin. Dyn. Syst. {\bf 29} (2011), no.~2, 417--437.
{\tt arXiv:1007.0594}

\bibitem{MyID:179}
N. R. O. Bastos, R. A. C. Ferreira, D. F. M. Torres,
Discrete-time fractional variational problems,
Signal Process. {\bf 91} (2011), no.~3, 513--524.
{\tt arXiv:1005.0252}

\bibitem{Viktor}
V. Blasjo,
The isoperimetric problem,
Amer. Math. Monthly, {\bf 112} (2005), no.~6, 526--566.

\bibitem{Blaszczyk:et:al}
T. Blaszczyk, M. Ciesielski, M. Klimek, J. Leszczynski,
Numerical solution of fractional oscillator equation,
Appl. Math. and Comput. 218 (2011), no.~6, 2480--2488.

\bibitem{BCGI}
L. Bourdin,
Existence of a weak solution for fractional Euler-Lagrange equations,
J. Math. Anal. Appl. 399 (2013), no.~1, 239--251.
{\tt arXiv:1203.1414}

\bibitem{LTDExist}
L. Bourdin, \underline{T. Odzijewicz}, D. F. M. Torres,
Existence of minimizers for fractional variational problems containing Caputo derivatives,
Adv. Dyn. Syst. Appl. {\bf 8} (2013), no.~1, 3--12.
{\tt arXiv:1208.2363}

\bibitem{GenExist}
L. Bourdin, \underline{T. Odzijewicz}, D. F. M. Torres,
Existence of minimizers for generalized Lagrangian functionals and a necessary optimality condition
--- Application to fractional variational problems,
Differential Integral Equations, in press.
{\tt arXiv:1403.3937}

\bibitem{CapelasOliveira}
R. F. Camargo, A. O. Chiacchio, R. Charnet, E. Capelas de Oliveira,
Solution of the fractional Langevin equation and the Mittag--Leffler functions,
J. Math. Phys. {\bf 6} (2009) 063507. 8~pp.

\bibitem{Carpinteri}
A. Carpinteri, F. Mainardi,
{\it Fractals and fractional calculus in continuum mechanics},
CISM Courses and Lectures, 378, Springer, Vienna, 1997.

\bibitem{Clarke}
F. Clarke,
{\it Functional analysis, calculus of variations and optimal control},
Graduate Texts in Mathematics, Springer, 2013.

\bibitem{Coimbra}
C. F. M. Coimbra,
Mechanics with variable-order differential operators,
Ann. Phys. {\bf 12} (2003), no.~11--12, 692--703.

\bibitem{Cresson}
J. Cresson,
Fractional embedding of differential operators and Lagrangian systems,
J. Math. Phys. {\bf 48} (2007), no.~3, 033504, 34~pp.
{\tt arXiv:math/0605752}

\bibitem{cd}
J. Cresson, S. Darses,
Stochastic embedding of dynamical systems,
J. Math. Phys. {\bf 48} (2007), no.~7, 072703, 54~pp.
{\tt arXiv:math/0509713}

\bibitem{cft}
J. Cresson, G. S. F. Frederico\ and\ D. F. M. Torres,
Constants of motion for non-differentiable quantum variational problems,
Topol. Methods Nonlinear Anal. {\bf 33} (2009), no.~2, 217--231.
{\tt arXiv:0805.0720}

\bibitem{Curtis}
J. P. Curtis,
Complementary extremum principles for isoperimetric optimization problems,
Optim. Eng. {\bf 5} (2004), no.~4, 417--430.

\bibitem{book:Dacorogna}
B. Dacorogna,
{\it Introduction to the calculus of variations},
Imperial College Press, 2004.

\bibitem{Diaz}
G. Diaz, C. F. M. Coimbra,
Nonlinear dynamics and control of a variable order
oscillator with application to the van der Pol equation,
Nonlinear Dynam. {\bf 56} (2009), no.~1--2, 145--157.

\bibitem{Nabulsi3}
R. A. El-Nabulsi,
Fractional quantum Euler-Cauchy equation in the Schrodinger picture,
complexified harmonic oscillators and emergence
of complexified Lagrangian and Hamiltonian dynamics,
Mod. Phys. Lett. B {\bf 23} (2009), no.~28, 3369--3386.

\bibitem{Nabulsi}
R. A. El-Nabulsi,
Fractional variational problems from extended exponentially fractional integral,
Appl. Math. Comput. {\bf 217} (2011), no.~22, 9492--9496.

\bibitem{jmp}
R. A. El-Nabulsi, D. F. M. Torres,
Fractional actionlike variational problems,
J. Math. Phys. {\bf 49} (2008), no.~5, 053521, 7~pp.
{\tt arXiv:0804.4500}

\bibitem{book:Evans}
L. C. Evans,
\textit{Partial differential equations},
Gruaduate Studies in Mathematics, American Mathematical Society, United States of America, 1997.

\bibitem{book:Ewing}
G. M. Ewing,
\textit{Calculus of variations with applications},
Courier Dover Publications, New York, 1985.

\bibitem{iso:ts}
R. A. C. Ferreira, D. F. M. Torres,
Isoperimetric problems of the calculus of variations on time scales,
in Nonlinear Analysis and Optimization II (eds: A. Leizarowitz, B. S. Mordukhovich,
I. Shafrir, and A. J. Zaslavski), Contemporary Mathematics, vol. 514, Amer.
Math. Soc., Providence, RI, 2010, pp. 123--131.
{\tt arXiv:0805.0278}

\bibitem{NoetherTD}
G. S. F. Frederico, \underline{T. Odzijewicz}, D. F. M. Torres,
Noether's theorem for non-smooth extremals of variational problems with time delay,
Appl. Anal. {\bf 93} (2014), no.~1, 153--170.
{\tt arXiv:1212.4932}

\bibitem{gastao}
G. S. F. Frederico, D. F. M. Torres,
Fractional conservation laws in optimal control theory,
Nonlinear Dynam. {\bf 53} (2008), no.~3, 215--222.
{\tt arXiv:0711.0609}

\bibitem{gastao2}
G. S. F. Frederico, D. F. M. Torres,
Fractional Noether's theorem in the Riesz-Caputo sense,
Appl. Math. Comput. 217 (2010), no.~3, 1023--1033.
{\tt arXiv:1001.4507}

\bibitem{book:GF}
I. M. Gelfand, S.V. Fomin,
{\it Calculus of Variations},
Dover Publications Inc., New York, 2000.

\bibitem{book:Giaquinta}
M. Giaquinta, S. Hildebrandt,
{\it Calculus of variations I},
Springer-Verlag, Berlin, Heidelberg, 2004.

\bibitem{Herrera}
L. Herrera, L. Nunez, A. Patino, H. Rago,
A variational principle and the classical
and quantum mechanics of the damped harmonic oscillator,
Am. J. Phys. {\bf 54} (1986), no.~3, 273--277.

\bibitem{book:Hilfer}
R. Hilfer,
{\it Applications of fractional calculus in physics},
World Sci. Publishing, River Edge, NJ, 2000.

\bibitem{book:Jost}
J. Jost, X. Li-Jost,
\textit{Calculus of variations},
Cambridge Univ. Press, Cambridge, 1998.

\bibitem{Katugampola}
U. N. Katugampola,
New approach to a generalized fractional integral,
Appl. Math. Comput. {\bf 218} (2011), no.~3, 860--865.
{\tt arXiv:1010.0742}

\bibitem{Kilbas}
A. A. Kilbas, M. Saigo,
Generalized Mittag--Leffler function and generalized fractional calculus operators,
Integral Transform. Spec. Func., {\bf 15} (2004), no.~1, 31--49.

\bibitem{book:Kilbas}
A. A. Kilbas, H. M. Srivastava, J. J. Trujillo,
\textit{Theory and applications of fractional differential equations},
North-Holland Mathematics Studies, 204, Elsevier, Amsterdam, 2006.

\bibitem{Malgorzata1}
M. Klimek,
Lagrangian fractional mechanics---a noncommutative approach,
Czechoslovak J. Phys. {\bf 55} (2005), no.~11, 1447--1453.

\bibitem{book:Klimek}
M. Klimek,
\textit{On solutions of linear fractional differential equations of a variational type,}
The Publishing Office of Czestochowa University of Technology, Czestochowa, 2009.

\bibitem{Malgorzata2}
M. Klimek, O. P. Agrawal,
On a regular fractional Sturm--Liouville problem with derivatives of order in $(0,1)$,
Proceedings of the 13th International Carpathian Control Conference, 28-31 May 2012,
Vysoke Tatry (Podbanske), Slovakia.
\url{http://dx.doi.org/10.1109/CarpathianCC.2012.6228655}

\bibitem{Kli13}
M. Klimek, O.  P.  Agrawal,
Regular fractional Sturm-Liouville problem with generalized derivatives of order in (0,1).
Proceedings of the IFAC Joint Conference: 5th SSSC, 11th WTDA, 5th WFDA,
4-6 February 2013, Grenoble, France.

\bibitem{Kli13a}
M. Klimek, O. P. Agrawal,
Fractional Sturm-Liouville problem,
Comput. Math. Appl. {\bf 66} (2013), no.~5, 795--812.

\bibitem{Lupa}
M. Klimek, M. Lupa,
Reflection symmetric formulation of generalized fractional variational calculus,
Fract. Calc. Appl. Anal. {\bf 16} (2013), no.~1, 243--261.

\bibitem{Klimek:Ja}
M. Klimek, \underline{T. Odzijewicz}, A. B. Malinowska,
Variational methods for the fractional Sturm-Liouville problem,
J. Math. Anal. Appl., \url{http://dx.doi.org/10.1016/j.jmaa.2014.02.009}
{\tt arXiv:1304.6258}

\bibitem{Lacroix}
S. F. Lacroix,
Traite du calcul differentiel et du calcul integral,
Paris: Mme. VeCourcier, {\bf 3} (1819), second edition, 409--410.

\bibitem{book:Lanczos}
C. L\'anczos,
{\it The variational principles of mechanics},
fourth edition, Mathematical Expositions, No.~4,
Univ. Toronto Press, Toronto, ON, 1970.

\bibitem{Laskin}
N. Laskin,
Fractional quantum mechanics and L\'evy path integrals,
Phys. Lett. A {\bf 268} (2000), no.~4-6, 298--305.

\bibitem{Lazo}
M. J. Lazo, D. F. M. Torres,
The DuBois--Reymond fundamental lemma of the fractional calculus of variations
and an Euler--Lagrange equation involving only derivatives of Caputo,
J. Optim. Theory Appl. {\bf 156} (2013), no.~1, 56--67.
{\tt arXiv:1210.0705}

\bibitem{Li}
J. Li, M. Ostoja-Starzewski,
Micropolar continuum mechanics of fractal media,
Internat. J. Engrg. Sci. {\bf 49} (2011), no.~12, 1302--1310.

\bibitem{Lin}
Y. Lin, T. He, H. Shi,
Existence of positive solutions for Sturm--Liouville BVPs of singular fractional differential equations,
U. P. B. Sci. Bull. {\bf 74} (2012), no.~1, Series A.

\bibitem{Lorenzo}
C. F. Lorenzo, T. T. Hartley,
Variable order and distributed order fractional operators,
Nonlinear Dynam. {\bf 29} (2002), no.~1--4, 57--98.

\bibitem{book:Mainardi}
F. Mainardi,
{\it Fractional calculus and waves in linear viscoelasticity},
Imp. Coll. Press, London, 2010.

\bibitem{mal}
A. B. Malinowska,
Fractional variational calculus for non-differentiable functions,
In: Fractional Dynamics and Control, D. Baleanu, J. A. Tenreiro Machado and A. C. J. Luo (eds.),
Springer New York, 2012, Part~2, Chapter~8, 97--108.
{\tt arXiv:1201.6640}

\bibitem{Agnieszka2}
A. B. Malinowska,
A formulation of the fractional Noether-type theorem for multidimensional Lagrangians,
Appl. Math. Lett. {\bf 25} (2012), no.~11, 1941--1946.
{\tt arXiv:1203.2107}

\bibitem{comBasia:Frac1}
A. B. Malinowska, D. F. M. Torres,
Generalized natural boundary conditions for fractional
variational problems in terms of the Caputo derivative,
Comput. Math. Appl. {\bf 59} (2010), no.~9, 3110--3116.
{\tt arXiv:1002.3790}

\bibitem{book:AD}
A. B. Malinowska, D. F. M. Torres,
{\it Introduction to the fractional calculus of variations},
Imp. Coll. Press, London, 2012.

\bibitem{MK}
R. Metzler, J. Klafter,
The random walk's guide to anomalous diffusion:
a fractional dynamics approach,
Phys. Rep. {\bf 339} (2000), no.~1, 77~pp.

\bibitem{comDorota}
D. Mozyrska, D. F. M. Torres,
Minimal modified energy control for fractional
linear control systems with the Caputo derivative,
Carpathian J. Math. {\bf 26} (2010), no.~2, 210--221.
{\tt arXiv:1004.3113}

\bibitem{MyID:181}
D. Mozyrska, D. F. M. Torres,
Modified optimal energy and initial memory
of fractional continuous-time linear systems,
Signal Process. {\bf 91} (2011), no.~3, 379--385.
{\tt arXiv:1007.3946}

\bibitem{Noether}
E. Noether,
Invariante Variationsprobleme,
Nachr. v. d. Ges. d. Wiss. zu G\"{o}ttingen (1918), 235--257.

\bibitem{Ja}
\underline{T. Odzijewicz},
Variable Order Fractional Isoperimetric Problem of Several Variables,
Advances in the Theory and Applications of Non-integer Order Systems,
257 (2013), 133--139.

\bibitem{Tatiana:Spain2010}
\underline{T. Odzijewicz}, A. B. Malinowska, D. F. M. Torres,
Calculus of variations with fractional and classical derivatives,
Proceedings of FDA'10, The 4th IFAC Workshop on Fractional Differentiation and its Applications,
Badajoz, Spain, October 18-20, 2010 (Eds: I. Podlubny, B. M. Vinagre Jara, YQ. Chen,
V. Feliu Batlle, I. Tejado Balsera), Article no. FDA10-076, 5~pp.
{\tt arXiv:1007.0567}

\bibitem{tatiana}
\underline{T. Odzijewicz}, D. F. M. Torres,
Fractional calculus of variations for double integrals,
Balkan J. Geom. Appl. \textbf{16} 2 (2011), 102--113.
{\tt arXiv:1102.1337}

\bibitem{MyID:226}
\underline{T. Odzijewicz}, A. B. Malinowska, D. F. M. Torres,
Generalized fractional calculus with applications
to the calculus of variations,
Comput. Math. Appl. {\bf 64} (2012), no.~10, 3351--3366.
{\tt arXiv:1201.5747}

\bibitem{FVC_Gen_Int}
\underline{T. Odzijewicz}, A. B. Malinowska, D. F. M. Torres,
Fractional calculus of variations in terms of a generalized fractional integral
with applications to physics,
Abstr. Appl. Anal. {\bf 2012} (2012), Art. ID 871912, 24~pp.
{\tt arXiv:1203.1961}

\bibitem{MyID:207}
\underline{T. Odzijewicz}, A. B. Malinowska, D. F. M. Torres,
Fractional variational calculus with classical and combined Caputo derivatives,
Nonlinear Anal. {\bf 75} (2012), no.~3, 1507--1515.
{\tt arXiv:1101.2932}

\bibitem{Hawaii}
\underline{T. Odzijewicz}, A. B. Malinowska, D. F. M. Torres,
Variable order fractional variational calculus for double integrals,
Proceedings of the IEEE Conference on Decision and Control, 2012, art. no. 6426489, pp.~6873--6878.
{\tt arXiv:1209.1345}

\bibitem{FVC_Sev}
\underline{T. Odzijewicz}, A. B. Malinowska, D. F. M. Torres,
Fractional calculus of variations of several independent variables,
Eur. Phys. J. Special Topics {\bf 222} (2013),  no.~8, 1813--1826.
{\tt arXiv:1308.4585}

\bibitem{CLandFR}
\underline{T. Odzijewicz}, D. F. M. Torres,
Calculus of variations with classical and fractional derivatives,
Math. Balkanica {\bf 26} (2012), no.~1-2, 191--202.
{\tt arXiv:1007.0567}

\bibitem{MyID:233}
\underline{T. Odzijewicz}, A. B. Malinowska, D. F. M. Torres,
A Generalized fractional calculus of variations with applications,
Proceedings of The 20th International Symposium
on Mathematical Theory of Networks and Systems (MTNS), 9 - 13 July 2012,
University of Melbourne, Australia, Paper 159.

\bibitem{MyID:236}
\underline{T. Odzijewicz}, A. B. Malinowska, D. F. M. Torres,
Green's theorem for generalized fractional derivatives,
Proceedings of FDA'2012, The Fifth Symposium
on Fractional Differentiation and its Applications,
May 14-17, 2012, Hohai University, Nanjing, China.
Editors: Wen Chen, HongGuang Sun and Dumitru Baleanu. Paper \#084
{\tt arXiv:1205.4851}

\bibitem{variable}
\underline{T. Odzijewicz}, A. B. Malinowska, D. F. M. Torres,
Fractional variational calculus of variable order,
Operator Theory: Advances and Applications, {\bf 229} (2013), 291--301.
{\tt arXiv:1110.4141}

\bibitem{GreenThm}
\underline{T. Odzijewicz}, A. B. Malinowska, D. F. M. Torres,
Green's theorem for generalized fractional derivative,
Fract. Calc. Appl. Anal. {\bf 16} (2013), no.~1, 64--75.
{\tt arXiv:1205.4851}

\bibitem{NoetherVO}
\underline{T. Odzijewicz}, A. B. Malinowska, D. F. M. Torres,
Noether's theorem for fractional variational problems of variable order,
Cent. Eur. J. Phys. {\bf 11} (2013), no.~6, 691--701.
{\tt arXiv:1303.4075}

\bibitem{NoetherGen}
\underline{T. Odzijewicz}, A. B. Malinowska, D. F. M. Torres,
A generalized fractional calculus of variations,
Control Cybernet. {\bf 42} (2013), no.~2, 443--458.
{\tt arXiv:1304.5282}

\bibitem{Pedro}
H. T. C. Pedro, M. H. Kobayashi, J. M. C. Pereira, C. F. M. Coimbra,
Variable order modeling of diffusive-convective effects on the oscillatory flow past a sphere,
J. Vib. Control {\bf 14} (2008), no.~9--10, 1569--1672.

\bibitem{book:Podlubny}
I. Podlubny,
{\it Fractional differential equations},
Academic Press, San Diego, CA, 1999.

\bibitem{book:Polyanin}
A. D. Polyanin, A. V. Manzhirov,
{\it Handbook of integral equations},
CRC, Boca Raton, FL, 1998.

\bibitem{QCH}
J. Qi, S. Chen,
Eigenvalue problems of the model from nonlocal continuum mechanics,
J. Math. Phys. 52 (2011), no.~7, 073516, 14~pp.

\bibitem{Ramirez}
L. E. S. Ramirez, C. F. M. Coimbra,
On the selection and meaning of variable order operators for dynamic modeling,
Int. J. Differ. Equ. {\bf 2010} (2010), Art. ID 846107, 16~pp.

\bibitem{Ramirez2}
L. E. S. Ramirez, C. F. M. Coimbra,
On the variable order dynamics of the nonlinear wake caused by a sedimenting particle,
Phys. D {\bf 240} (2011), no.~13, 1111--1118.

\bibitem{CD:Riewe:1996}
F. Riewe,
Nonconservative Lagrangian and Hamiltonian mechanics,
Phys. Rev. E (3) {\bf 53} (1996), no.~2, 1890--1899.

\bibitem{CD:Riewe:1997}
F. Riewe,
Mechanics with fractional derivatives,
Phys. Rev. E (3) {\bf 55} (1997), no.~3, part B, 3581--3592.

\bibitem{book:Samko}
S. G. Samko, A. A. Kilbas, O. I. Marichev,
\textit{Fractional integrals and derivatives},
translated from the 1987 Russian original, Gordon and Breach, Yverdon, 1993.

\bibitem{SamkoRoss}
S. G. Samko, B. Ross,
Integration and differentiation to a variable fractional order,
Integral Transform. Spec. Funct. {\bf 1} (1993), no.~4, 277--300.

\bibitem{Sha}
Z. Sha, F. Jing-Li, L. Yong-Song,
Lagrange equations of nonholonomic systems with fractional derivatives,
Chin. Phys. B {\bf 19} (2010), no.~12, 120301, 5~pp.

\bibitem{Tarasov2}
V. E. Tarasov,
Continuous limit of discrete systems with long-range interaction,
J. Phys. A {\bf 39} (2006), no.~48, 14895--14910.

\bibitem{Tarasov3}
V. E. Tarasov,
Fractional statistical mechanics,
Chaos {\bf 16} (2006), no.~3, 033108, 7~pp.

\bibitem{Tarasov1}
V. E. Tarasov\ and\ G. M. Zaslavsky,
Fractional dynamics of coupled oscillators with long-range interaction,
Chaos {\bf 16} (2006), no.~2, 023110, 13~pp.

\bibitem{book:Brunt}
B. van Brunt,
{\it The calculus of cariations},
Springer, New York, 2004.

\bibitem{book:Weinstock}
R. Weinstock,
{\it Calculus of variations with applications to physics and engineering},
McGraw-Hill Book Company Inc., New York, 1952.

\bibitem{Young}
L. C. Young,
{\it Lectures on the calculus of variations and optimal control theory},
Foreword by Wendell H. Fleming Saunders, Philadelphia, 1969.

\bibitem{Zaslavsky}
G. M. Zaslavsky,
{\it Hamiltonian chaos and fractional dynamics},
reprint of the 2005 original,
Oxford Univ. Press, Oxford, 2008.

\bibitem{Edelman}
G. M. Zaslavsky, M. A. Edelman,
Fractional kinetics: from pseudochaotic dynamics to Maxwell's demon,
Phys. D {\bf 193} (2004), no.~1-4, 128--147.

\end{thebibliography}
\end{document}